\documentclass[11pt,a4paper,twoside]{amsart}
\usepackage[top=1in, bottom=1.25in, marginparwidth=0.8in, inner=1in, outer=1in]{geometry}
\usepackage{lmodern}
\usepackage[T2A,T1]{fontenc} 
\usepackage[utf8]{inputenc}
\usepackage[UKenglish]{babel}
\usepackage{yfonts}
\usepackage{etoolbox}

\usepackage{pinlabel}

\makeatletter
\newcommand*{\math@version@bold}{bold}
\DeclareMathOperator\DD{
  \textrm{%
    \usefont{T2A}{cmr}{\ifx\math@version\math@version@bold bx\else m\fi}{n}%
    \CYRD
  }%
}
\makeatother
\newcommand{\DDa}{\DD_{a}}
\newcommand{\DDc}{\DD_{c}}

\usepackage{float}
\usepackage{graphicx}
\usepackage{xcolor}
\usepackage{picins}
\colorlet{myblue}{black}
\colorlet{gold}{yellow!90!black!70!red}
\colorlet{darkgreen}{green!70!black}
\colorlet{lightred}{red!30!white}
\colorlet{lightblue}{blue!30!white}

\usepackage{rotating}
\usepackage{tikz}
\usepackage{tikz-cd}  
\usetikzlibrary{arrows,arrows.meta}
\tikzcdset{arrow style=tikz, diagrams={>=latex}}
\tikzset{>=latex}

\usepackage{caption}
\usepackage[labelfont=up,margin=5pt]{subcaption}
\usepackage{wrapfig}
\newcommand\mystretch{\baselinestretch}
\newcommand{\myfixwrapfig}{\textcolor{white}{~}\vspace{-\mystretch\baselineskip\relax}}
\usepackage[section]{placeins}

\usepackage{setspace} 
\usepackage[bookmarks=true,pdfencoding=auto,bookmarksdepth=3]{hyperref}
\usepackage{bookmark}
\hypersetup{colorlinks=true,citecolor=black,urlcolor=myblue,linkcolor=black}

\usepackage{amsmath}
\usepackage{amsfonts}
\usepackage{amssymb}
\usepackage{amsthm}
\usepackage{mathtools}
\usepackage{mathrsfs}
\mathtoolsset{mathic=true}

\def\co{\colon\thinspace\relax}

\newtheorem{theorem}{Theorem}[section]
\newtheorem*{theorem*}{Theorem}
\newtheorem{lemma}[theorem]{Lemma}

\newtheorem{convention}[theorem]{Convention}
\newtheorem{conjecture}[theorem]{Conjecture}

\newtheorem{corollary}[theorem]{Corollary}
\newtheorem{proposition}[theorem]{Proposition}
\AfterEndEnvironment{theorem}{\noindent\ignorespaces}
\theoremstyle{definition}
\newtheorem{definition}[theorem]{Definition}
\newtheorem{question}[theorem]{Question}

\newtheorem{example}[theorem]{Example}
\newtheorem{observation}[theorem]{Observation}
\newtheorem{remark}[theorem]{Remark}

\AfterEndEnvironment{definition}{\noindent\ignorespaces}

\usepackage[makeroom]{cancel}
\usepackage{enumitem}
\usepackage[normalem]{ulem} 

\newcommand{\vc}[1]{\vcenter{\hbox{#1}}}%
\newcommand{\mypic}[3]{%
	\newcommand{#3}{%
		\vc{%
			\includegraphics[page=#2]%
				{PSTricks/KhCurves-pics-#1-pics.pdf}%
			}%
	}%
}%

\mypic{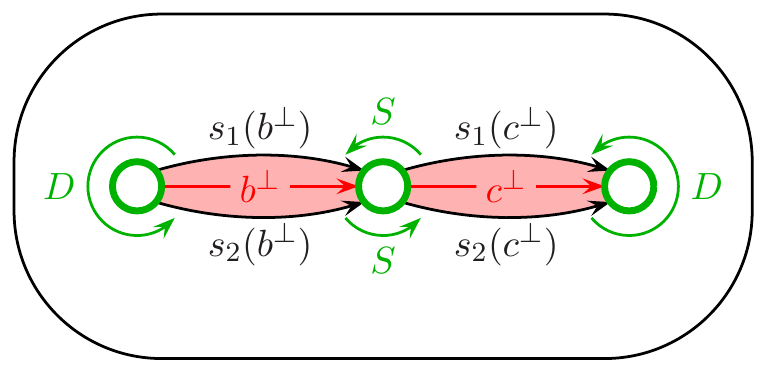}{1}{\ArcSystemForTwistedComplexes}
\mypic{Classification}{2}{\arcneighbourhood}
\mypic{Classification}{3}{\ArcSystemForTypeDStructures}
\mypic{Classification}{4}{\ProofPreloopToCCHomotopyHi}
\mypic{Classification}{5}{\ProofPreloopToCCHomotopyHii}
\mypic{Classification}{6}{\traintrackneighbourhoodofarc}
\mypic{Classification}{7}{\TraintrackMatrixDecompositionI}
\mypic{Classification}{8}{\TraintrackMatrixDecompositionIII}
\mypic{Classification}{9}{\CalculusForPreloopsMi}
\mypic{Classification}{10}{\CalculusForPreloopsMii}
\mypic{Classification}{11}{\CalculusForPreloopsMiiiA}
\mypic{Classification}{12}{\CalculusForPreloopsMiiiB}
\mypic{Classification}{13}{\CalculusForPreloopsMiiiC}
\mypic{Classification}{14}{\CalculusForPreloopsMiiiD}
\mypic{Classification}{15}{\CalculusForPreloopsMiiiE}
\mypic{Classification}{16}{\CalculusForPreloopsMiiiF}
\mypic{Classification}{17}{\MorCaseA}
\mypic{Classification}{18}{\MorCaseB}
\mypic{Classification}{19}{\MorCaseC}
\mypic{Classification}{20}{\MorCaseD}
\mypic{Classification}{21}{\bigonConventions}
\mypic{Classification}{22}{\BigonCounts}
\mypic{Classification}{23}{\UbigonConventions}
\mypic{Classification}{24}{\Ti}
\mypic{Classification}{25}{\Tii}
\mypic{Classification}{26}{\Tiii}
\mypic{Classification}{27}{\Tiva}
\mypic{Classification}{28}{\Tivb}
\mypic{Classification}{29}{\Tv}
\mypic{Classification}{30}{\PushPassingLoops}
\mypic{Classification}{31}{\Tia}
\mypic{Classification}{32}{\Tib}
\mypic{Classification}{33}{\traintrackneighbourhoodofarcSigns}
\mypic{Classification}{34}{\ExampleOverIntegers}
\mypic{Classification}{35}{\neighbourhoodofarcHRWconvention}

\mypic{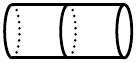}{1}{\tube}
\mypic{Cobordisms}{2}{\Torus}
\mypic{Cobordisms}{3}{\plane}
\mypic{Cobordisms}{4}{\planedot}

\mypic{Cobordisms}{5}{\planedotdot}
\mypic{Cobordisms}{6}{\Sphere}
\mypic{Cobordisms}{7}{\Spheredot}
\mypic{Cobordisms}{8}{\Spheredotdot}
\mypic{Cobordisms}{9}{\Spheredotdotdot}
\mypic{Cobordisms}{10}{\DiscT}
\mypic{Cobordisms}{11}{\DiscB}
\mypic{Cobordisms}{12}{\DiscR}
\mypic{Cobordisms}{13}{\DiscL}
\mypic{Cobordisms}{14}{\DiscTdot}
\mypic{Cobordisms}{15}{\DiscBdot}
\mypic{Cobordisms}{16}{\DiscRdot}
\mypic{Cobordisms}{17}{\DiscLdot}

\mypic{Cobordisms}{18}{\PairsOfPantsL}
\mypic{Cobordisms}{19}{\PairsOfPantsR}
\mypic{Cobordisms}{20}{\TuB}
\mypic{Cobordisms}{20,angle=180,origin=c}{\TuT}
\mypic{Cobordisms}{20,angle=-90,origin=c}{\TuL}
\mypic{Cobordisms}{20,angle=90,origin=c}{\TuR}
\mypic{Cobordisms}{21,angle=0,origin=c}{\TuDotA}
\mypic{Cobordisms}{21,angle=90,origin=c}{\TuDotB}
\mypic{Cobordisms}{21,angle=180,origin=c}{\TuDotC}
\mypic{Cobordisms}{21,angle=270,origin=c}{\TuDotD}
\mypic{Cobordisms}{22}{\TuNoDot}
	
\mypic{Cobordisms}{23}{\TorusRelI}
\mypic{Cobordisms}{24}{\TorusRelII}
\mypic{Cobordisms}{25}{\TorusRelIII}
\mypic{Cobordisms}{26}{\TorusRelIV}
\mypic{Cobordisms}{27}{\Star}
\mypic{Cobordisms}{28}{\StarH}
\mypic{Cobordisms}{29}{\StarDot}
\mypic{Cobordisms}{30}{\StarSheet}
\mypic{Cobordisms}{31}{\StarSheetH}
\mypic{Cobordisms}{32}{\StarTube}
\mypic{Cobordisms}{33}{\StarTubeTube}
\mypic{Cobordisms}{34}{\StarDotSameComp}
\mypic{Cobordisms}{35}{\StarNeck}
\mypic{Cobordisms}{36}{\StarNeckI}
\mypic{Cobordisms}{37}{\StarNeckII}
\mypic{Cobordisms}{38}{\StarNeckIII}
\mypic{Cobordisms}{39}{\StarNeckIp}
\mypic{Cobordisms}{40}{\StarNeckIIp}
\mypic{Cobordisms}{41}{\StarNeckIIIp}
\mypic{Cobordisms}{42}{\NoStarSheetH}
\mypic{Cobordisms}{43}{\NoStarSheetHRev}
\mypic{Cobordisms}{44}{\NoStarTube}
\mypic{Cobordisms}{45}{\StarDotRev}
\mypic{Cobordisms}{46}{\PairsOfPantsLAst}
\mypic{Cobordisms}{47}{\PairsOfPantsRAst}
\mypic{Cobordisms}{48}{\DiscRAst}
\mypic{Cobordisms}{49}{\PairsOfPantsGluedAst}
\mypic{Cobordisms}{50}{\planedotstar}

\mypic{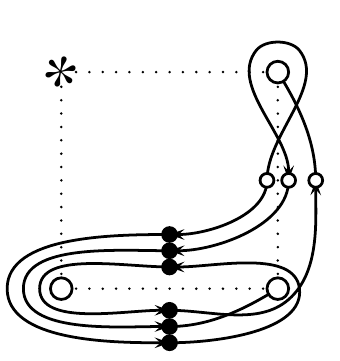}{1}{\pretzelarcinvariant}
\mypic{Curves}{2}{\pretzelarcinvariantMarkedFishtail}
\mypic{Curves}{3}{\pretzelarcinvariantWOfishtail}
\mypic{Curves}{4}{\pretzelarcinvariantII}
\mypic{Curves}{5}{\pretzelarcinvariantV}

\mypic{Curves}{6}{\PTImmersedCurveHFNonRatX}
\mypic{Curves}{7}{\PTImmersedCurveHFRatX}
\mypic{Curves}{8}{\PTImmersedCurveKhNonRatX}
\mypic{Curves}{9}{\PTImmersedCurveKhRatX}
\mypic{Curves}{10}{\EmbeddedCurvesX}
\mypic{Curves}{11}{\MutationCounterexampleCurve}
\mypic{Curves}{12}{\QuiverOnSurface}
\mypic{Curves}{13}{\QuiverOnSurfaceDual}
\mypic{Curves}{14}{\QuiverOnSurfaceDualWithQuiver}

\mypic{Curves}{15}{\PairingUnknotArcArcINTRO}
\mypic{Curves}{16}{\PairingUnknotLoopArcINTRO}
\mypic{Curves}{17}{\PairingTrefoilArcArcINTRO}
\mypic{Curves}{18}{\PairingTrefoilLoopArcINTRO}
\mypic{Curves}{19}{\PairingTrefoilArcINTRO}
\mypic{Curves}{20}{\PairingTrefoilLoopINTRO}
\mypic{Curves}{21}{\IdentificationOfParametrizations}

\mypic{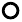}{1}{\DotC}
\mypic{Inline}{2}{\DotB}
\mypic{Inline}{3}{\DotCblue}
\mypic{Inline}{4}{\DotBblue}
\mypic{Inline}{5}{\DotCred}
\mypic{Inline}{6}{\DotBred}
\mypic{Inline}{1,scale=0.7}{\smallDotC}
\mypic{Inline}{2,scale=0.7}{\smallDotB}
\mypic{Inline}{3,scale=0.7}{\smallDotCblue}
\mypic{Inline}{4,scale=0.7}{\smallDotBblue}
\mypic{Inline}{5,scale=0.7}{\smallDotCred}
\mypic{Inline}{6,scale=0.7}{\smallDotBred}

\mypic{Inline}{7}{\Circle}
\mypic{Inline}{8}{\resolution}

\mypic{Inline}{9}{\Li}
\mypic{Inline}{10}{\Lo}
\mypic{Inline}{10,angle=270,origin=c}{\Lialt}
\mypic{Inline}{11}{\LiRed}
\mypic{Inline}{12}{\LoRed}
\mypic{Inline}{13}{\LiBlue}
\mypic{Inline}{14}{\LoBlue}
\mypic{Inline}{15}{\Ni}
\mypic{Inline}{16}{\No}
\mypic{Inline}{15,scale=0.7}{\smallNi}
\mypic{Inline}{16,scale=0.7}{\smallNo}
\mypic{Inline}{17}{\Lil}
\mypic{Inline}{18}{\Lol}
\mypic{Inline}{19}{\Nil}
\mypic{Inline}{20}{\Nol}
\mypic{Inline}{21}{\CrossingL}
\mypic{Inline}{22}{\CrossingR}
\mypic{Inline}{23}{\CrossingLDot}
\mypic{Inline}{24}{\CrossingRDot}
\mypic{Inline}{25}{\CrossingPos}
\mypic{Inline}{26}{\CrossingNeg}
\mypic{Inline}{27}{\CrossingPosMarkedi}
\mypic{Inline}{28}{\CrossingNegMarkedi}
\mypic{Inline}{29}{\CrossingPosMarkedii}
\mypic{Inline}{30}{\CrossingNegMarkedii}
\mypic{Inline}{31}{\CrossingPosMarkediii}
\mypic{Inline}{32}{\CrossingNegMarkediii}
\mypic{Inline}{33}{\CrossingPosMarkediv}
\mypic{Inline}{34}{\CrossingNegMarkediv}
\mypic{Inline}{25,angle=270,origin=c}{\CrossingPosR}
\mypic{Inline}{26,angle=270,origin=c}{\CrossingNegR}
\mypic{Inline}{27,angle=270,origin=c}{\CrossingPosMarkediR}
\mypic{Inline}{28,angle=270,origin=c}{\CrossingNegMarkediR}
\mypic{Inline}{29,angle=270,origin=c}{\CrossingPosMarkediiR}
\mypic{Inline}{30,angle=270,origin=c}{\CrossingNegMarkediiR}
\mypic{Inline}{31,angle=270,origin=c}{\CrossingPosMarkediiiR}
\mypic{Inline}{32,angle=270,origin=c}{\CrossingNegMarkediiiR}
\mypic{Inline}{33,angle=270,origin=c}{\CrossingPosMarkedivR}
\mypic{Inline}{34,angle=270,origin=c}{\CrossingNegMarkedivR}

\mypic{Inline}{35}{\LiO}
\mypic{Inline}{36}{\LiOl}
\mypic{Inline}{37}{\LilO}
\mypic{Inline}{38}{\LiOr}
\mypic{Inline}{39}{\LirO}
\mypic{Inline}{40}{\NiO}
\mypic{Inline}{41}{\NiOl}
\mypic{Inline}{42}{\NilO}
\mypic{Inline}{43}{\NiOr}
\mypic{Inline}{44}{\NirO}
\mypic{Inline}{45}{\LiDotL}
\mypic{Inline}{46}{\LiDotR}
\mypic{Inline}{47}{\LoDotT}
\mypic{Inline}{48}{\LoDotB}
\mypic{Inline}{49}{\NiDotL}
\mypic{Inline}{50}{\NiDotR}
\mypic{Inline}{51}{\NoDotT}
\mypic{Inline}{52}{\NoDotB}
\mypic{Inline}{53}{\TwistTwo}
\mypic{Inline}{53,angle=90,origin=c}{\TwistTwoUp}
\mypic{Inline}{54}{\TwistNegTwo}
\mypic{Inline}{55}{\TwistThree}
\mypic{Inline}{56}{\TwistNegThree}
\mypic{Inline}{57}{\TwistTwoDot}
\mypic{Inline}{58}{\TwistNegTwoDot}
\mypic{Inline}{59}{\TwistThreeDot}
\mypic{Inline}{60}{\TwistNegThreeDot}

\mypic{Inline}{61}{\CrossingLBasepoint}
\mypic{Inline}{61,angle=180,origin=c}{\CrossingLBasepointRotated}
\mypic{Inline}{62}{\CrossingRBasepoint}
\mypic{Inline}{62,angle=180,origin=c}{\CrossingRBasepointRotated}

\mypic{Inline}{63}{\LiT}
\mypic{Inline}{64}{\LoT}
\mypic{Inline}{65}{\NiT}
\mypic{Inline}{66}{\NoT}

\mypic{Inline}{67}{\singleB}
\mypic{Inline}{68}{\singleW}

\mypic{Inline}{69}{\ThreeTwistTangle}
\mypic{Inline}{70}{\ThreeTwistTangleBlue}

\mypic{Inline}{71}{\DotCarc}
\mypic{Inline}{72}{\DotBarc}
\mypic{Inline}{73}{\TrivialTwoTangle}
\mypic{Inline}{74}{\CircleDot}

\mypic{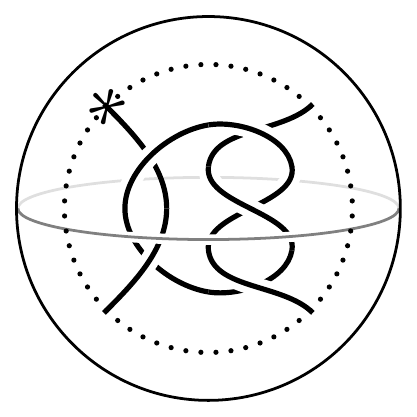}{1}{\ATangleInThreeBall}
\mypic{LargeTangles}{2}{\PretzeltangleUnoriented}
\mypic{LargeTangles}{3}{\PretzeltangleOriented}
\mypic{LargeTangles}{4}{\DehnTwistTangleDiagram}
\mypic{LargeTangles}{5}{\tanglepairingI}
\mypic{LargeTangles}{6}{\tanglepairingIalt}
\mypic{LargeTangles}{7}{\tanglepairingII}
\mypic{LargeTangles}{8}{\PairingStandardDiagramUnknot}
\mypic{LargeTangles}{9}{\PairingClosureDiagramUnknot}
\mypic{LargeTangles}{10}{\PairingStandardDiagramUnlink}
\mypic{LargeTangles}{11}{\PairingClosureDiagramUnlink}
\mypic{LargeTangles}{12}{\PairingStandardDiagramTrefoil}
\mypic{LargeTangles}{13}{\PairingClosureDiagramTrefoil}
\mypic{LargeTangles}{14}{\PairingStandardDiagramUnknotFromPT}
\mypic{LargeTangles}{15}{\PairingClosureDiagramUnknotFromPT}
\mypic{LargeTangles}{16}{\PairingStandardDiagramTorusknotFromPTa}
\mypic{LargeTangles}{17}{\PairingStandardDiagramTorusknotFromPTb}
\mypic{LargeTangles}{18}{\PairingStandardDiagramTorusknotFromPTc}
\mypic{LargeTangles}{19}{\ExampleInfinitelyTwistedKnot}
\mypic{LargeTangles}{20}{\ExampleInfinitelyTwistedKnotRedBlue}
\mypic{LargeTangles}{21}{\tangleIIItwistclosure}
\mypic{LargeTangles}{22}{\tangleIVtwistclosure}
\mypic{LargeTangles}{23}{\MutationTangle}
\mypic{LargeTangles}{24}{\MutationTangleX}
\mypic{LargeTangles}{25}{\MutationTangleY}
\mypic{LargeTangles}{26}{\MutationTangleZ}
\mypic{LargeTangles}{27}{\MutationCounterexampleI}
\mypic{LargeTangles}{28}{\MutationCounterexampleII}
\mypic{LargeTangles}{29}{\PairingStandardDiagramTrefoilUnoriented}
\mypic{LargeTangles}{30}{\PairingClosureDiagramTrefoilUnoriented}
\mypic{LargeTangles}{31}{\ATangleInThreeBallArcParametrization}
\mypic{LargeTangles}{32}{\ATangleInThreeBallWithBNrCurves}

\mypic{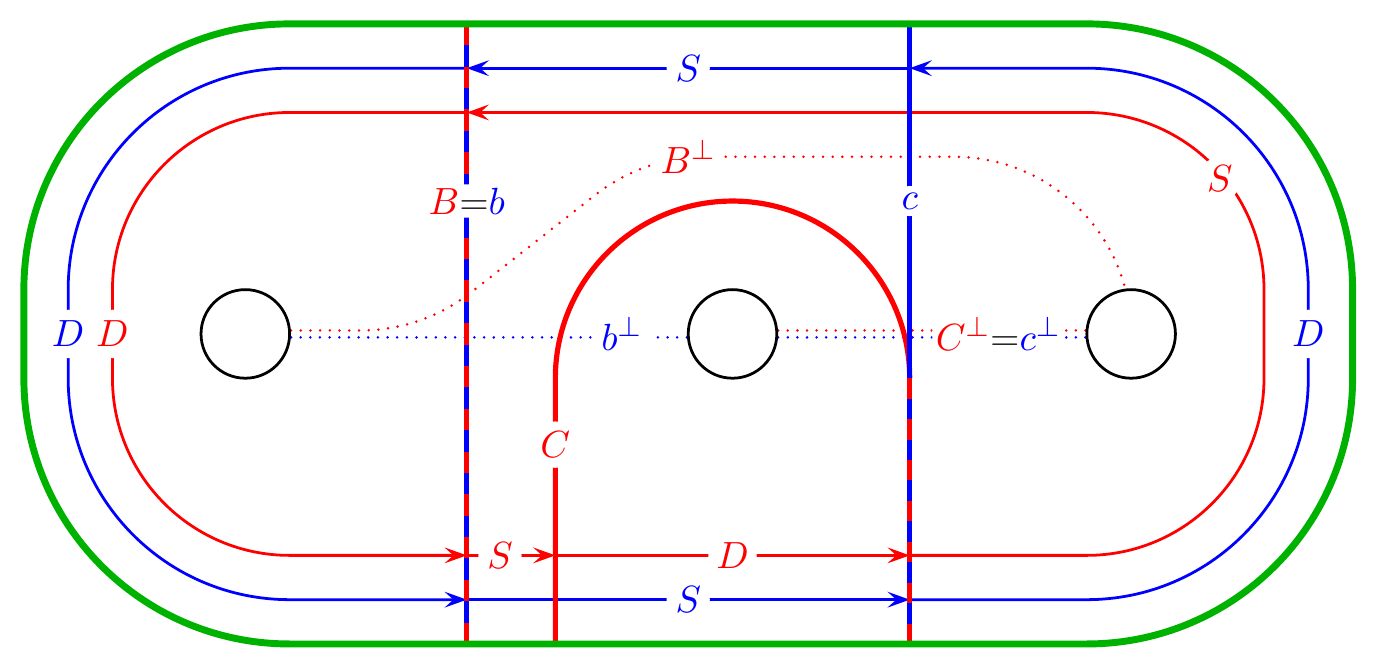}{1}{\MCGactionArcs}
\mypic{MCGaction}{2}{\MCGactionArcsBBHnD}
\mypic{MCGaction}{3}{\MCGactionArcsBBHnSS}
\mypic{MCGaction}{4}{\MCGactionArcsCCHnSS}
\mypic{MCGaction}{5}{\MCGactionArcsCCHnD}
\mypic{MCGaction}{6}{\MCGactionArcsBCHnS}
\mypic{MCGaction}{7}{\MCGactionArcsCBHnmDS}
\mypic{MCGaction}{8}{\MCGactionArcsCBHnmSD}
\mypic{MCGaction}{9}{\MCGactionArcsCBHDS}
\mypic{MCGaction}{10}{\MCGactionArcsCBHSD}
\mypic{MCGaction}{11}{\PillowcaseAsTorusQuotient}
\mypic{MCGaction}{12}{\PillowcaseAsPillowcase}
\mypic{MCGaction}{13}{\PillowcaseAsTangleBoundary}

\mypic{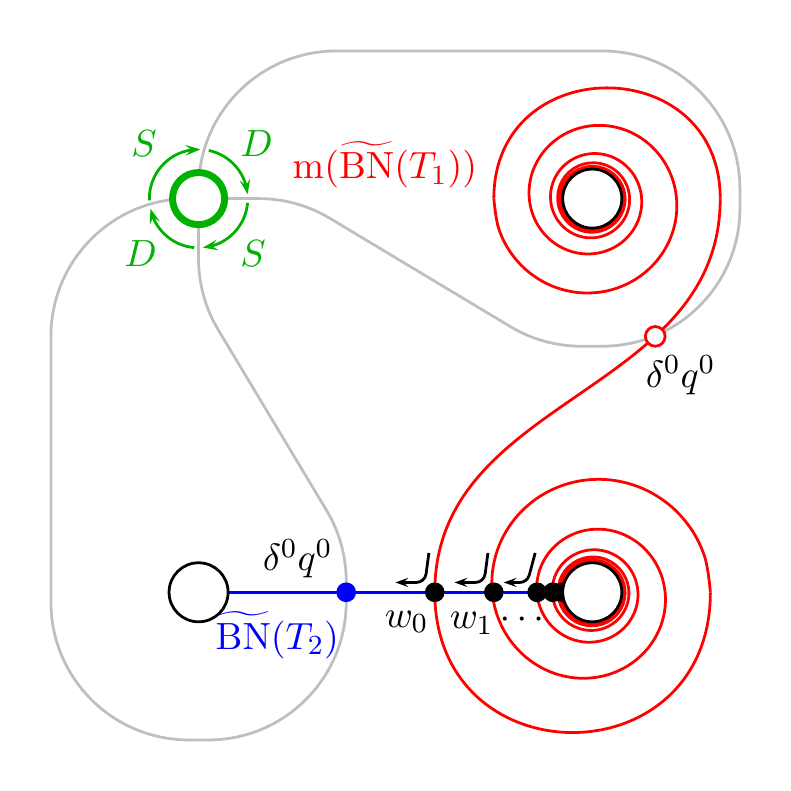}{1}{\PairingUnknotArcArc}
\mypic{Pairing}{2}{\PairingUnknotLoopArc}
\mypic{Pairing}{3}{\PairingUnlinkArcArc}
\mypic{Pairing}{4}{\PairingUnlinkLoopArc}
\mypic{Pairing}{5}{\PairingTrefoilArcArc}
\mypic{Pairing}{6}{\PairingTrefoilLoopArc}
\mypic{Pairing}{7}{\PairingTrefoilHeartsArc}
\mypic{Pairing}{8}{\PairingTorusknotFromPtArcArc}
\mypic{Pairing}{9}{\PairingTorusknotFromPtArcEight}

\newcommand{\ob}{\operatorname{ob}}
\DeclareMathOperator{\Mor}{Mor}
\DeclareMathOperator{\End}{End}
\DeclareMathOperator{\Mat}{Mat} 
\DeclareMathOperator{\Cx}{Cx} 
\DeclareMathOperator{\TypeD}{TypeD} 
\DeclareMathOperator{\Cxpre}{Cx^\text{\normalfont pre}}

\DeclareMathOperator{\Com}{Com}

\DeclareMathOperator{\loops}{\normalfont C}

\DeclareMathOperator{\Kob}{Kob}

\DeclareMathOperator{\Cob}{Cob}

\DeclareMathOperator{\BNTwistingCobCob}{\prescript{}{\Cob_{/l}}{\tau}^{\Cob_{/l}}}
\DeclareMathOperator{\BNTwistingTildeCobCob}{\prescript{}{\Cob_{/l}}{\widetilde{\tau}}^{\Cob_{/l}}}
\DeclareMathOperator{\BNTwistingBCob}{{}_{\BNAlgH}{\tau}^{\Cob_{/l}}}
\DeclareMathOperator{\BNTwisting}{\prescript{}{\BNAlgH}{\tau}^{\BNAlgH}}

\newcommand{\dual}[1]{\overline{#1}}
\newcommand{\revpre}{\operatorname{rev}^{pre}}
\newcommand{\rev}{\operatorname{rev}}
\newcommand{\emb}{\operatorname{emb}}
\newcommand{\embpre}{\operatorname{emb}^{pre}}

\newcommand{\QGrad}[1]{{\textcolor{violet}{#1}}}
\newcommand{\DeltaGrad}[1]{{\textcolor{darkgreen}{#1}}}
\newcommand{\HomGrad}[1]{{\textcolor{black}{#1}}}
\newcommand{\GGdqh}[4]{\prescript{\QGrad{#3}}{}{#1}^\DeltaGrad{#2}_\HomGrad{#4}}

\newcommand{\GGzqh}[4]{\prescript{\QGrad{#3}}{}{#1}_\HomGrad{#4}}
\newcommand{\GGdzh}[4]{#1^\DeltaGrad{#2}_\HomGrad{#4}}

\DeclareMathOperator{\lk}{lk}

\DeclareMathOperator{\HFT}{HFT}
\newcommand{\CFTminus}{\operatorname{CFT}^-}

\newcommand{\KhT}[1]{[\![ #1 ]\!]} 
\newcommand{\KhTl}[1]{[\![ #1 ]\!]_{/l}} 
\renewcommand{\L}[1]{L\!\left[ #1 \right]} 

\DeclareMathOperator{\Kh}{Kh}
\DeclareMathOperator{\CKh}{CKh}
\DeclareMathOperator{\CBN}{CBN}
\DeclareMathOperator{\BN}{BN}
\DeclareMathOperator{\fCBN}{{\it f}CBN}
\newcommand{\Khr}{\widetilde{\Kh}}
\newcommand{\CKhr}{\widetilde{\CKh}}
\newcommand{\CBNr}{\widetilde{\CBN}}
\newcommand{\BNr}{\widetilde{\BN}}
\newcommand{\fCBNr}{{\it f}\widetilde{\CBN}}

\newcommand{\Arc}{\BNr}
\newcommand{\BNra}{\BNr_{a}}
\newcommand{\BNrc}{\BNr_{c}}
\newcommand{\Eight}[1][]{L_{#1}}

\newcommand{\mirror}{\operatorname{m}} 
\newcommand{\bimoduleG}{\mathbf{G}} 
\newcommand{\refx}{\operatorname{ref}_x} 
\newcommand{\refy}{\operatorname{ref}_y} 
\newcommand{\rotz}{\operatorname{rot}_z} 
\newcommand{\refi}{\operatorname{ref}_i} 
\newcommand{\mutx}{\operatorname{mut}_x} 
\newcommand{\muty}{\operatorname{mut}_y} 
\newcommand{\mutz}{\operatorname{mut}_z} 
\newcommand{\muti}{\operatorname{mut}_i} 

\DeclareMathOperator{\HFK}{\widehat{HFK}}

\DeclareMathOperator{\HFhat}{\widehat{HF}}
\newcommand{\HF}{\operatorname{HF}}
\newcommand{\CF}{\operatorname{CF}}
\newcommand{\CFplus}{\operatorname{CF}^{+}}
\newcommand{\CFtimes}{\operatorname{CF}^{\times}}
\DeclareMathOperator{\Homology}{H_\ast}
\DeclareMathOperator{\HomologyZero}{H_0}

\DeclareMathOperator{\BNAlgH}{\mathcal{B}} 

\DeclareMathOperator{\gr}{gr}

\newcommand{\op}[1]{{#1}^{\text{op}}}

\DeclareMathOperator{\id}{id} 
\DeclareMathOperator{\Mod}{Mod}

\DeclareMathOperator{\Sym}{Sym}

\DeclareMathOperator{\GL}{GL}
\DeclareMathOperator{\im}{im}
\DeclareMathOperator{\rk}{rk}
\DeclareMathOperator{\coker}{coker}

\newcommand{\fs}{\operatorname{c}}

\renewcommand{\k}{\mathbf{k}}
\newcommand{\ii}{\mathbf{i}}

\newcommand{\F}{\mathbb{F}}

\newcommand{\fieldTwoElements}{\mathbb{F}}
\newcommand{\fieldTwoIsNotZero}{{{\mathbb{K}}}}
\newcommand{\field}{{\mathbf{k}}}

\newcommand{\FourPuncturedSphere}{S^2_{4,\ast}}
\newcommand{\PFPS}{(S^2\smallsetminus \mathrm{4~points})}
\newcommand{\RepVarSphere}{R(S^2,\mathrm{4~points})}

\newcommand{\object}{X}
\newcommand{\Z}{\mathbb{Z}}
\newcommand{\Q}{\mathbb{Q}}

\newcommand{\ignoreme}[1]{} 
\newcommand{\RxH}{{\mathcal{R}}}
\newcommand{\Rxy}{{\mathscr{R}}}
\newcommand{\RzT}{{\mathbf{R}}}

\newcommand{\bt}{\boxtimes}

\newcommand{\wrFuk}{\mathcal{W F}}

\newcommand{\Rcomm}{R} 
\newcommand{\gen}{\singleB}

\newcommand{\Lk}{\mathcal{L}} 
\newcommand{\Knot}{\mathcal{K}} 
\newcommand{\Diag}{\mathcal{D}} 
\newcommand{\x}{x} 
\newcommand{\y}{y} 
\newcommand{\xben}{\text{x}} 
\newcommand{\yben}{\text{y}} 
\newcommand{\xgen}{\mathbf{x}} 
\newcommand{\ygen}{\mathbf{y}} 

\newcommand{\DDinf}{{\DD_\infty}}
\newcommand{\Linf}{{L_\infty}}
\newcommand{\DDinfinf}{{\DD^\infty_\infty}}
\newcommand{\Linfinf}{{L{\substack{\infty \\ \infty}}}}

\newcommand{\SaddleBC}{S_{\smallDotC}}
\newcommand{\SaddleCB}{S_{\smallDotB}}
\newcommand{\DotcobB}{D_{\smallDotB}}
\newcommand{\DotcobC}{D_{\smallDotC}}

\begin{document}
\title{Immersed curves in Khovanov homology}

\author{Artem Kotelskiy}
\address{Department of Mathematics \\ Indiana University}
\email{artofkot@iu.edu}

\author{Liam Watson}
\address{Department of Mathematics \\ University of British Columbia}
\email{liam@math.ubc.ca}
\thanks{AK is supported by an AMS-Simons travel grant. LW is supported by an NSERC discovery/accelerator grant and was partially supported by funding from the Simons Foundation and the Centre de Recherches Math\'ematiques, through the Simons-CRM scholar-in-residence program.}

\author{Claudius Zibrowius}
\address{Department of Mathematics \\ University of British Columbia}
\email{claudius.zibrowius@posteo.net}

\begin{abstract}
We give a geometric interpretation of Bar-Natan's universal invariant for the class of tangles in the 3-ball with four ends: we associate with such 4-ended tangles $T$ multicurves $\BNr(T)$, that is, collections of immersed curves with local systems in the 4-punctured sphere. These multicurves are tangle invariants up to homotopy of the underlying curves and equivalence of the local systems. They satisfy a gluing theorem which recovers the reduced Bar-Natan homology of links in terms of wrapped Lagrangian Floer theory. Furthermore, we use $\BNr(T)$ to define two immersed curve invariants $\Khr(T)$ and $\Kh(T)$, which satisfy similar gluing theorems that recover reduced and unreduced Khovanov homology of links, respectively.  As a first application, we prove that Conway mutation preserves reduced Bar-Natan homology over the field with two elements and Rasmussen's $s$-invariant over any field. As a second application, we give a geometric interpretation of Rozansky's categorification of the two-stranded Jones--Wenzl projector. This allows us to define a module structure on reduced Bar-Natan and Khovanov homologies of infinitely twisted knots, generalizing a result by Benheddi. 
\end{abstract}

\maketitle
\renewcommand{\contentsname}{Contents}
\setcounter{tocdepth}{1}
\tableofcontents

\section{Introduction}\label{sec:intro}

\subsection{Bar-Natan's invariant for 4-ended tangles}

In~\cite{BarNatanKhT}, Bar-Natan defined an invariant 
$\KhTl{T}$ of tangles \(T\) in the 3-ball, which generalizes Khovanov homology of links in the 3-sphere~\cite{Khovanov}. 
Like Khovanov's link invariant, \(\KhTl{T}\) takes the form of a bigraded chain complex which is well-defined up to chain homotopy equivalence. 
However, while the link invariant is a chain complex over a field or a ring, \(\KhTl{T}\) is a chain complex over a certain cobordism category \(\Cob_{/l}\) which depends on the number of endpoints of \(T\). 
This category plays a central role in the construction of \(\KhTl{T}\). 
Its objects are crossingless tangles and its morphisms are cobordisms between those tangles. 
This feature makes \(\KhTl{T}\) amenable to cut-and-paste arguments, which
Bar-Natan exploited to effectively compute Khovanov homology~\cite{BN_fast}. However, this also means that for studying \(\KhTl{T}\), it is essential to obtain a good understanding of, 
first, the category \(\Cob_{/l}\) itself, and, second, the category of chain complexes over \(\Cob_{/l}\). 
This paper addresses these two problems for tangles with four ends. (For tangles with an arbitrary number of ends, we make  progress on the first problem. However, the second remains very much open.) For 4-ended tangles, the first question is effectively answered by the following result:

\begin{theorem}\label{thm_from_intro:BN_algebra=B}
The full subcategory \( \End_{/l}(\Ni\oplus\No) \) of \( \Cob_{/l} \) generated by the two objects \( \Ni \) and \( \No \) is isomorphic to the path algebra \( \BNAlgH \) of the quiver
\[ 
\begin{tikzcd}[row sep=2cm, column sep=1.5cm]
\DotB
\arrow[leftarrow,in=145, out=-145,looseness=5]{rl}[description]{D}
\arrow[leftarrow,bend left]{r}[description]{S}
&
\DotC
\arrow[leftarrow,bend left]{l}[description]{S}
\arrow[leftarrow,in=35, out=-35,looseness=5]{rl}[description]{D}
\end{tikzcd} 
 \]  
modulo the relations \( DS=0=SD \). 
\end{theorem}

The key to this result is a new basis for the morphism spaces of \(\Cob_{/l}\). This basis depends on a choice of distinguished tangle end, and so does the isomorphism between  \( \End_{/l}(\Ni\oplus\No) \) and \(\BNAlgH\)  in the theorem. 
An important feature of our basis is that one can easily understand the effect of choosing different distinguished tangle ends. As we will see below, this turns out to be crucial for understanding the effect of Conway mutation on Khovanov homology.
Moreover, this basis can be defined for tangles with arbitrarily many tangle ends, which goes a long way towards answering the first problem in general. Though Theorem~\ref{thm_from_intro:BN_algebra=B} only describes a full subcategory of \(\Cob_{/l}\) for 4-ended tangles, it is sufficient for studying the second problem, namely, understanding chain complexes over \(\Cob_{/l}\). 
By a procedure Bar-Natan calls \emph{delooping}, any chain complex over \( \Cob_{/l} \) is isomorphic to one whose underlying objects are crossingless tangles without closed components. 
As a result, when \( T \) is a 4-ended tangle, \( \KhTl{T} \) is isomorphic to a chain complex over \( \End_{/l}(\Ni\oplus\No) \); this chain complex is invariant up to chain homotopy. 
\begin{definition}
Let \( \DD(T) \) be the complex over the algebra $\BNAlgH$ which corresponds to \( \KhTl{T} \) under the isomorphism from Theorem~\ref{thm_from_intro:BN_algebra=B}.
\end{definition}
By construction, the chain homotopy type of \( \DD(T) \)  is a tangle invariant. 
However, it generally depends on the isomorphism between \( \End_{/l}(\Ni\oplus\No) \) and \( \BNAlgH \), or equivalently, on a choice of a distinguished tangle end. 
As such we view $ \DD(T)$
as an invariant of \emph{pointed} 4-ended tangles, that is, 4-ended tangles together with a choice of a distinguished tangle end. In tangle diagrams, we mark this tangle end by an asterisk, like so~\( \CrossingLDot \).  

\begin{example}\label{exa:intro:threetwist:DD}
	The chain complexes over $\BNAlgH$ corresponding to the trivial tangles are generated by a single element each and the differentials of these complexes vanish:
	\[
	\DD(\Lo)=[\DotB]
	\quad\qquad\quad
	\DD(\Li)=[\DotC]
	\]
	To give a slightly more interesting example: 
	\[ 
	\DD\Big(\ThreeTwistTangle\Big)=
	\Big[
	\begin{tikzcd}
	\DotC
	\arrow{r}{S}
	&
	\DotB
	\arrow{r}{D}
	&
	\DotB
	\arrow{r}{S^2}
	&
	\DotB
	\end{tikzcd}
	\Big]
	\]
	Note that the generators in all three examples carry an appropriate bigrading; see Example~\ref{exa:BNntwisttangles}.
\end{example}

\begin{wrapfigure}{r}{0.31\textwidth}
	\centering
	$\QuiverOnSurfaceDualWithQuiver$
	\vspace*{-10pt}
	\caption{The surface~$\FourPuncturedSphere$}\label{fig:quiverFuk:surface:dual}
\end{wrapfigure}

\subsection{Classification results}\label{subsec:intro:classification}
Bar-Natan's invariant \( \KhTl{T} \) is defined over integer coefficients and so is \( \DD(T) \). In this paper, we do not address the problem of understanding chain complexes over \(\Cob_{/l}\), now reduced to the problem of understanding chain complexes over \(\BNAlgH\), over integer coefficients. However, passing to field coefficients $\field$, we  give a complete classification of bigraded chain complexes over \(\BNAlgH\) and morphisms between them. 

Denote by $\FourPuncturedSphere$ the 4-punctured sphere in Figure~\ref{fig:quiverFuk:surface:dual}, together with the choices of a special puncture $\ast$, two grey parameterizing arcs separating the other three punctures, and an embedding of the quiver from Theorem~\ref{thm_from_intro:BN_algebra=B}. We also label the boundary chords on the special puncture by $\textcolor{darkgreen}{S}$ and $\textcolor{darkgreen}{D}$, according to how $\FourPuncturedSphere$ deformation retracts onto the quiver.

\begin{definition}\label{def:intro:multicurves}
	An immersed curve with local system on the 4-punctured sphere \(\FourPuncturedSphere\) is a pair \( (\gamma,X) \), where either
	\begin{enumerate}[label=(\roman*)]
		\item \label{enu:intro:multicurves1}  \(\gamma\) is an immersion of an oriented circle into \(\FourPuncturedSphere\) representing a non-trivial primitive element of its fundamental group, and \( X\in\GL_n(\field) \) for some positive integer \(n\); or
		\item \label{enu:intro:multicurves2} \(\gamma\) is an immersion of an unoriented interval into \(\FourPuncturedSphere\), which maps the two endpoints to the non-special punctures, and which defines a non-trivial element of the fundamental group of \(\FourPuncturedSphere\) relative to the three non-special punctures and \( X=\id\in\GL_n(\field) \) for some positive integer \(n\).
	\end{enumerate}
	In case~\ref{enu:intro:multicurves1}, we call \( (\gamma, X) \) \textbf{compact}, in case~\ref{enu:intro:multicurves2}, we call it \textbf{non-compact}. 
	We call \(\gamma\) the underlying curve and \( X \) its local system. 
	We consider curves up to homotopy of the underlying immersed curves and up to matrix similarity of their local systems.
	Geometrically, we may think of local systems in terms of $n$-dimensional vector bundles over circles or intervals up to isomorphism. 
	A \textbf{multicurve} is a finite set of such curves with the property that all
	underlying curves are pairwise non-homotopic as unoriented curves.
\end{definition}

\begin{theorem}[Classification Theorem]\label{thm_from_intro:classification}
Given any bigraded chain complex \(N\) over \(\BNAlgH\) with coefficients in a field \(\field\), we can define a (bigraded) multicurve \(\L{N}\) on the 4-punctured sphere \(\FourPuncturedSphere\) such that if \(N'\) is another such chain complex, \(\L{N}\) and \(\L{N'}\) agree iff \(N\) and \(N'\) are chain homotopy equivalent. Moreover, the homology of the morphism space of two chain complexes over \(\BNAlgH\) is equal to the wrapped Lagrangian Floer homology of the corresponding multicurves:
$$\Homology (\Mor(N,N'))\cong \HF (\L{N}, \L{N'})$$
\end{theorem}
We prove Theorem~\ref{thm_from_intro:classification} for a general class of parameterized surfaces and their corresponding algebras, which may be of independent interest.

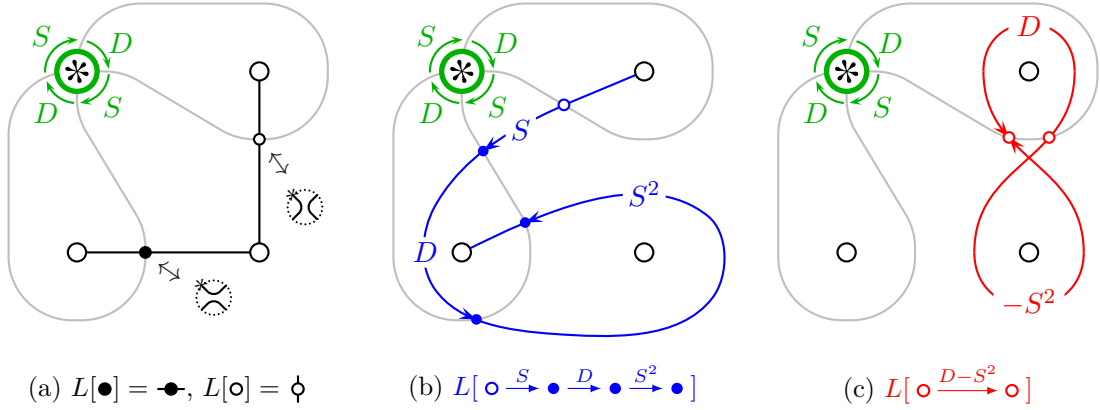
\begin{figure}[t]
	\centering
	\begin{subfigure}[t]{0.31\textwidth}
		\centering
		$\QuiverOnSurfaceDual$
		\vspace{4pt}
		\caption{$\L{\DotB}=\DotBarc$, $\L{\DotC}=\DotCarc$}\label{fig:exa:classification:curves:trivial}
	\end{subfigure}
	\begin{subfigure}[t]{0.31\textwidth}
		\centering
		$\PairingTrefoilArcINTRO$
		\caption{$\textcolor{blue}{
				L[
				\protect\begin{tikzcd}[nodes={inner sep=2pt}, column sep=13pt,ampersand replacement = \&]
				\protect\DotCblue
				\protect\arrow{r}{S}
				\protect\&
				\protect\DotBblue
				\protect\arrow{r}{D}
				\protect\&
				\protect\DotBblue
				\protect\arrow{r}{S^2}
				\protect\&
				\protect\DotBblue
				\protect\end{tikzcd}
				]}$}\label{fig:exa:classification:curves:arc}
	\end{subfigure}
	\begin{subfigure}[t]{0.31\textwidth}
		\centering
		$\PairingTrefoilLoopINTRO$
		\caption{$\textcolor{red}{
				L[
				\protect\begin{tikzcd}[nodes={inner sep=2pt}, column sep=23pt,ampersand replacement = \&]
				\protect\DotCred
				\protect\arrow{r}{D-S^2}
				\protect\&
				\protect\DotCred
				\protect\end{tikzcd}
				]}$}\label{fig:exa:classification:curves:loop}
		\bigskip
	\end{subfigure}
	\caption{The geometric interpretation of some chain complexes over the algebra $\BNAlgH$ illustrating the first part of Theorem~\ref{thm_from_intro:classification}
	}\label{fig:exa:classification:curves}
\end{figure}

\begin{example}\label{exa:intro:complexes_vs_curves}
	Under the correspondence from the first part of the Classification Theorem, the immersed curves associated with the chain complexes $\DD(\Lo)=[\DotB]$ and $\DD(\Li)=[\DotC]$ from Example~\ref{exa:intro:threetwist:DD} are the two non-compact curves  $\DotBarc$ and $\DotCarc$ in Figure~\ref{fig:exa:classification:curves:trivial}, intersecting the parameterizing arcs in the points $\DotB$ and $\DotC$, respectively. From the viewpoint of symplectic geometry, the Classification Theorem stems from the fact that the wrapped Fukaya subcategory $\wrFuk (\DotBarc, \DotCarc)$  of the 4-punctured sphere \(\FourPuncturedSphere\) is isomorphic to the algebra $\BNAlgH$ \cite[Theorem~7.6]{Bocklandt}. This, in turn, implies that the category of twisted complexes $\text{Tw} \wrFuk(\DotBarc,\DotCarc)$ is isomorphic to the category of chain complexes over $\BNAlgH$.
	
	More generally, the generators of a chain complex $N$ corresponding to some immersed curve $\L{N}$ with the 1-dimensional trivial local system are equal to the intersection points of $\L{N}$ with the two parameterizing arcs. Similarly, we can read off the differential of $N$ from $\L{N}$; see Figures~\ref{fig:exa:classification:curves:arc} and~\ref{fig:exa:classification:curves:loop}. Under the deformation retraction of $\FourPuncturedSphere$ onto the embedded quiver, each segment of the immersed curve $\L{N}$ is mapped to a path in the quiver. This path determines some power of $S$ or $D$ (an element of $\BNAlgH$), which defines the corresponding component of the differential.
	Thus, the blue non-compact curve in Figure~\ref{fig:exa:classification:curves:arc} represents the third chain complex from Example~\ref{exa:intro:threetwist:DD}.
	If the 1-dimensional local system of an immersed curve is non-trivial, one of the differentials in the corresponding chain complex is multiplied by that local system. So for example, the chain complex
	\begin{equation*}
	\Big[
	\begin{tikzcd}
	\DotC
	\arrow{r}{D-S^2}
	&
	\DotC
	\end{tikzcd}
	\Big]
	=
	\Big[
	\begin{tikzcd}
	\DotC
	\arrow[bend left]{r}{D}
	\arrow[bend right,swap]{r}{-S^2}
	&
	\DotC
	\end{tikzcd}
	\Big]
	\end{equation*}
	corresponds to the red compact curve in Figure~\ref{fig:exa:classification:curves:loop} carrying the local system $(-1)$.
	For immersed curves with general $n$-dimensional local systems, a similar method applies, except that each intersection point corresponds to $n$ generators.
\end{example}

\begin{example}\label{exa:intro_classification_part_II}
	Let us illustrate the second part of the Classification Theorem, where we equate the homology of the morphism space with the wrapped Lagrangian Floer homology.
	The dimension of the latter is usually equal to the minimal intersection number between the two immersed curves times the dimensions of the local systems.
	For example, the minimal intersection number between the red and the blue curve in Figure~\ref{fig:exa:classification:looparc} is three. So if the curves carry 1-dimensional local systems, as in Figures~\ref{fig:exa:classification:curves:arc} and~\ref{fig:exa:classification:curves:loop}, the homology of the morphism space between the two corresponding complexes is 3-dimensional: 
	\[
  	\Homology\left(
  		\Mor\left(
  		\textcolor{red}{[
  			\protect\begin{tikzcd}[nodes={inner sep=2pt}, column sep=23pt,ampersand replacement = \&]
  			\DotCred
  			\arrow{r}{D-S^2}
  			\&
  			\DotCred
  			\end{tikzcd}
  			]}
  			, 
			\textcolor{blue}{[
		  	\begin{tikzcd}[nodes={inner sep=2pt}, column sep=13pt,ampersand replacement = \&]
		  	\DotCblue
		  	\arrow{r}{S}
		  	\&
		  	\DotBblue
		  	\arrow{r}{D}
		  	\&
		  	\DotBblue
		  	\arrow{r}{S^2}
		  	\&
		  	\DotBblue
		  	\end{tikzcd}
		  	]}
    	\right)
    \right)
    =
    \field^3
    \]
	When computing the homology of two complexes which correspond to non-compact curves, we need to calculate this minimal intersection number after wrapping the first curve infinitely many times around the punctures at their ends. This is illustrated in Figure~\ref{fig:exa:classification:arcarc}, where we  compute the homology of the morphism space between $\DD(\Li)$ and $\DD(\ThreeTwistTangle)$:
	\[
	\Homology\left(
		\Mor\left(
			\textcolor{red}{[\DotCred]}
			,
			\textcolor{blue}{[
			\begin{tikzcd}[nodes={inner sep=2pt}, column sep=13pt,ampersand replacement = \&]
			\DotCblue
			\arrow{r}{S}
			\&
			\DotBblue
			\arrow{r}{D}
			\&
			\DotBblue
			\arrow{r}{S^2}
			\&
			\DotBblue
			\end{tikzcd}
			]}
		\right)
	\right) 
	\cong 
	\left[ \bigoplus^{\infty}_{i=1} \field \right]
	\oplus
	\field
	\]
	Finally, if two curves are parallel, the dimension of their wrapped Lagrangian Floer homology is the minimal intersection number plus a correction term which is determined by the local systems; see Theorem~\ref{thm:PairingFormula}. 
\end{example}

\begin{figure}[t]
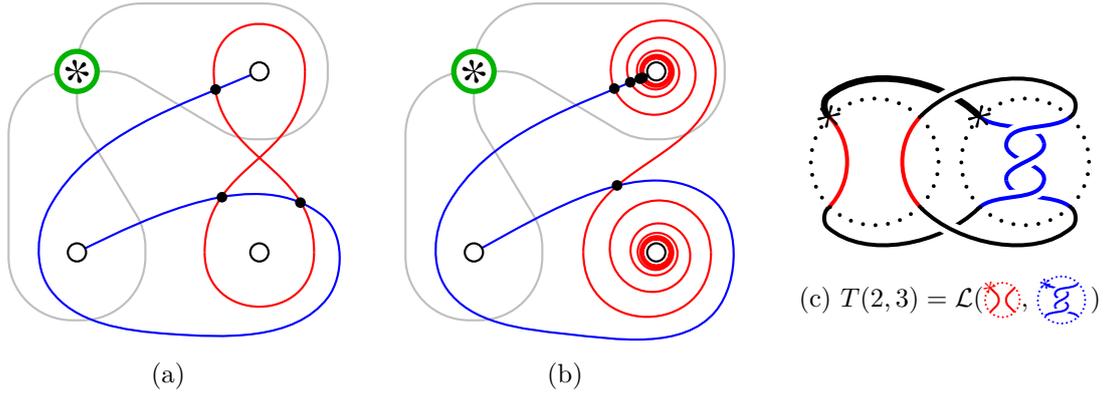

	\centering
	\begin{subfigure}[t]{0.32\textwidth}
		\centering
		$\PairingTrefoilLoopArcINTRO$
		\caption{}\label{fig:exa:classification:looparc}
	\end{subfigure}
	\begin{subfigure}[t]{0.32\textwidth}
		\centering
		$\PairingTrefoilArcArcINTRO$
		\caption{}\label{fig:exa:classification:arcarc}
		\bigskip
	\end{subfigure}
	\begin{subfigure}[t]{0.3\textwidth}
		\centering
		$\PairingStandardDiagramTrefoilUnoriented$
		\caption{$T(2,3)=\Lk(\protect\LiRed,\protect\ThreeTwistTangleBlue)$}
		\label{fig:intro:trefoil}
		\bigskip
	\end{subfigure}
	\caption{
		Figures~(a) and~(b) show the geometric interpretation of some morphism spaces from Example~\ref{exa:intro_classification_part_II} which illustrates the second part of Theorem~\ref{thm_from_intro:classification}. By Theorem~\ref{thm:intro:pairing}, these morphism spaces compute 
		the reduced Khovanov and Bar-Natan homologies of the trefoil knot, respectively, using the tangle splitting from Figure~(c). 
	}\label{fig:exa:classification}
\end{figure}
\subsection{Immersed curves for 4-ended tangles and gluing results}

By combining Theorems~\ref{thm_from_intro:BN_algebra=B} and \ref{thm_from_intro:classification}, we obtain the following three immersed curve invariants.

\begin{definition}	
Given a pointed 4-ended tangle \( T \), we associate three invariants 
\[
\begin{tikzcd}[nodes={inner sep=2pt}, column sep=25pt, row sep=0pt]
\BNr(T;\field)
=
L\Big[\DD(T;\field)
\Big]
\\
\Khr(T;\field)
=
L\Big[\DD(T;\field)
\arrow{r}{H \cdot  \id}
&
\DD(T;\field)\Big]
\\
\Kh(T;\field)
=
L\Big[\DD(T;\field)
\arrow{r}{H \cdot  \id}
&
\DD(T;\field) 
&
\DD(T;\field) 
\arrow{r}{H \cdot  \id} 
\arrow[swap]{l}{2 \cdot  \id}
&
\DD(T;\field)\Big]
\end{tikzcd}
\]
where \( \DD(T;\field)\coloneqq\DD(T)\otimes \field \). The latter two chain complexes are mapping cones of the maps indicated by the arrows, and the element \( H \) is equal to \( D-S^2 \in \BNAlgH \). 
\end{definition}

The notation in this definition is justified by the Geometric Pairing Theorem. For this, let \(\Lk=\Lk(T_1,T_2)\) be the link obtained by gluing two pointed 4-ended tangles \( T_1 \) and \( T_2 \) together as in Figure~\ref{fig:tanglepairingII}; see also Figure~\ref{fig:intro:trefoil}. 
Let \(\BNr(\Lk;\field)\), \(\Khr(\Lk;\field)\), and \(\Kh(\Lk;\field)\) denote the reduced Bar-Natan homology, the reduced Khovanov homology, and the unreduced Khovanov homology of \(\Lk\) over the field $\field$. In the first two cases, the reduction point of \(\Lk\) is equal to the identified distinguished tangle ends \(\ast\) of \(T_1\) and \(T_2\).

\begin{theorem}[Geometric Pairing Theorem]\label{thm:intro:pairing}
Up to appropriate grading shifts, there are bigraded isomorphisms
\begin{align*}
\BNr(\Lk;\field)&\cong\HF (\BNr(\mirror T_1;\field),\BNr(T_2;\field)) \\
\Khr(\Lk;\field)&\cong\HF (\Khr(\mirror T_1;\field),\BNr(T_2;\field)) \\
\Kh(\Lk;\field)&\cong\HF (\Kh(\mirror T_1;\field),\BNr(T_2;\field))
\end{align*}
where \( \mirror T_1 \) denotes the mirror of the tangle \(T_1\). The first is an isomorphism of \( \field[H] \)-modules, and the second and third are isomorphisms of \( \field \)-vector spaces.
\end{theorem}

The reduced Bar-Natan homology appears to be a natural object to consider; in Proposition~\ref{prop:alg_pairing_kh} we explain a close connection between $\BNr(\Lk;\field)$ and Bar-Natan's universal bracket. Furthermore, Theorem~\ref{theo:BNr_is_a_link_invariant} proves that  $\BNr(\Lk;\field)$ does not depend on the reduction basepoint and, in fact, carries a $\field[H_1,\ldots,H_\ell]$-module structure, where $\ell$ is the number of link components. This bears close resemblance to the structure on link Floer homology; see \cite{OS_Link_Floer}. 

\begin{example}
	Following Example~\ref{exa:intro:complexes_vs_curves}, 
	$\BNr(\Lo;\field)$ and $\BNr(\Li;\field)$ are the two non-compact curves $\DotBarc$ and $\DotCarc$ from Figure~\ref{fig:exa:classification:curves:trivial}, and the blue curve from Figure~\ref{fig:exa:classification:curves:arc} is equal to $\BNr\Big(\ThreeTwistTangle;\field\Big).$
	Therefore, according to the Geometric Pairing Theorem, the wrapped Lagrangian Floer homology between the two curves in Figure~\ref{fig:exa:classification:arcarc} computes the reduced Bar-Natan homology of the trefoil knot $\BNr(T(2,3))=\field[H]\oplus\field$, using the splitting from Figure~\ref{fig:intro:trefoil}. Here, the $H$-action on the generators corresponds to wrapping around the non-special puncture. 
	The figure-8 from Figure~\ref{fig:exa:classification:curves:loop} is equal to $\Khr(\Li)$, so the reduced Khovanov homology of the trefoil knot $\Khr(T(2,3))=\field^3$ is generated by the intersection points in Figure~\ref{fig:exa:classification:looparc}. 
\end{example}

For more examples of immersed curve invariants of tangles and illustrations of the Pairing Theorem, see Sections~\ref{sec:FigureEight} and~\ref{sec:Pairing}, respectively. 
In general, our invariants turn out to be almost as easy to compute as Bar-Natan's universal tangle invariant \( \KhTl{T} \); see the python module~\cite{KhT.py} written by Gurkeerat Chhina and the third author. Note that the curves defined above may not be the same for different choices of \( \field \), just like the dimension of the Khovanov homology of a fixed link may vary depending on the field of coefficients. We discuss some examples of this phenomenon in Section~\ref{subsec:field_dependence}.

\subsection{Naturality under the mapping class group action}\label{subsec:intro:naturality} 

Before moving on to applications of our tangle invariants, let us pause for a moment and recapitulate their construction. This is summarized in Figure~\ref{fig:summary:construction} in the case of the (most general) tangle invariant $\BNr(T)$. Broadly speaking, we associate with a 1-dimensional topological object in a 3-dimensional space an algebraic invariant, which we then reinterpret via a classification theorem in terms of a 1-dimensional object in a 2-dimensional space. This raises the obvious question, whether there exists a more intrinsic relationship between the two topological objects at the two ends of the construction. 

\begin{figure}[t]
	\centering
	\[ 
	\begin{tikzcd}[row sep=2cm, column sep=2.5cm]
	\substack{\text{pointed 4-ended}\\\text{tangle $T$}}
	\arrow[mapsto]{r}{\text{\protect\cite{BarNatanKhT}}}
	&
	\KhTl{T}
	\arrow[leftrightarrow]{r}{\substack{\End_{/l}(\smallNi\oplus\smallNo)\cong\BNAlgH\\\text{(Theorem~\ref{thm_from_intro:BN_algebra=B})}}}[swap]{\text{over any coefficients}}
	&
	\DD(T)
	\arrow[leftrightarrow]{r}{\substack{\text{Classification Theorem}\smallskip\\\text{(Theorem~\ref{thm_from_intro:classification})}}}[swap]{\text{over fields $\field$ only}}
	&
	\BNr(T;\field)
	\end{tikzcd} 
	\]  
	\caption{Summary of the construction of $\BNr(T;\field)$}
	\label{fig:summary:construction}
\end{figure}

The definition of tangles that Bar-Natan uses for the construction of $\KhTl{T}$ is purely combinatorial: tangles are tangle diagrams up to the usual Reidemeister moves. This is also the definition that we primarily use in this paper. However, if we want to interpret tangles topologically, we have some more flexibility. The following two definitions of 4-ended tangles are relevant:

\begin{definition}\label{def:4-ended_tangle}
	An {\bf unparameterized} pointed 4-ended tangle \( T \) in a 3-ball~\( D^3 \) is an embedding of two arcs and a finite (possibly empty) collection of circles
	\[ T\co\left(I \sqcup I \sqcup \bigsqcup S^1,\partial I \sqcup \partial I\right)\hookrightarrow \left(D^3,\partial D^3\right)\]
	considered up to isotopy, together with a choice of a distinguished tangle end 
	\(\ast\in \partial T\coloneqq T(\partial I \sqcup \partial I).\)
	A {\bf parameterized} pointed 4-ended tangle \( T \) consists of the data above together with a parameterization given by an oriented circle \( S^1 \) on the boundary \( \partial D^3 \) containing the four ends \(\partial T\). This is illustrated in Figure~\ref{fig:intro:CircleParametrization}. 
\end{definition}

Bar-Natan's definition of tangles is equivalent to the notion of parameterized tangles. To see this, note that the parameterization specifies an embedded disc in $D^3$, which is unique up to isotopy. The orientation of the parameterizing circle induces an orientation of the embedded disc. The image of a generic projection onto this disc determines a tangle diagram which is unique up to a sequence of Reidemeister moves.

We now turn to the topological objects at the other end of our construction. By definition, $\FourPuncturedSphere$ carries a parameterization by two arcs. However, the notion of immersed curves with local systems on $\FourPuncturedSphere$ does not require a choice of such a parameterization.
Also, note that up to this point, we have treated $\FourPuncturedSphere$ as an abstract surface. Given our topological definition of a 4-ended tangle $T$, there is already a 4-punctured sphere in our construction, namely the boundary of the 3-ball which surrounds the tangle minus the four tangle ends. 
Moreover, both this 4-punctured sphere, which we will refer to as the {tangle boundary}, and $\FourPuncturedSphere$ come with a distinguished puncture (the point carrying the mark $\ast$). In the construction of the immersed curve invariants, both of these punctures play important roles, the first in the correspondence between \(\KhTl{T}\) and \(\DD(T)\) and second in the correspondence between \(\DD(T)\) and $\BNr(T)$. This apparent similarity suggests that $\FourPuncturedSphere$ should be identified with the tangle boundary; Figure~\ref{fig:intro:NotionsOfTangles} illustrates how we do this. While we can always make such an identification, the question is whether our tangle invariants are natural with respect to this identification. In other words:

\begin{question}
	Are \(\BNr(T;\field)\), \(\Khr(T;\field)\) and \(\Kh(T;\field)\) invariants of \emph{unparameterized} pointed 4-ended tangles \(T\)?
\end{question}

We expect the answer to be positive for all fields $\field$. However, we only show this for the field of two elements, which we will denote throughout this paper by \(\fieldTwoElements\). We prove: 
\begin{theorem}
	Let \( T \) be a pointed 4-ended tangle and \( \rho \) an element of the mapping class group of the 4-punctured sphere which fixes the special end of \( T \). Then  
	\[ 
	\BNr(\rho(T);\fieldTwoElements)=\rho(\BNr(T;\fieldTwoElements)),
	~
	\Khr(\rho(T);\fieldTwoElements)=\rho(\Khr(T;\fieldTwoElements))
	~ 
	\text{and}
	~
	\Kh(\rho(T);\fieldTwoElements)=\rho(\Kh(T;\fieldTwoElements)).
	\]
\end{theorem}

\begin{figure}[t]
	\centering
	\begin{subfigure}[t]{0.32\textwidth}
		\centering
		$\ATangleInThreeBall$
		\caption{}\label{fig:intro:CircleParametrization}
	\end{subfigure}
	\begin{subfigure}[t]{0.32\textwidth}
		\centering
		$\ATangleInThreeBallArcParametrization$
		\caption{}\label{fig:intro:ArcParametrization}
	\end{subfigure}
	\begin{subfigure}[t]{0.32\textwidth}
		\centering
		$\ATangleInThreeBallWithBNrCurves$
		\caption{}\label{fig:intro:NoParametrization}
	\end{subfigure}
		\caption{An illustration of the identification of the tangle boundary with the 4-punctured sphere $\FourPuncturedSphere$ for the $(2,-3)$-pretzel tangle $T_{2,-3}$. 
		(a) shows the tangle with the usual parameterization of the tangle boundary by a circle. 
		In (b), the tangle boundary has been identified with $\FourPuncturedSphere$. 
		Finally, (c) shows the curve $\BNr(T_{2,-3};\fieldTwoElements)$ on the tangle boundary without any parametrization.
		}
	\label{fig:intro:NotionsOfTangles}
\end{figure}

\subsection{Conway mutation}

The fact that our immersed curve invariants satisfy a pairing theorem makes them ideal for studying how Khovanov and Bar-Natan homology of links
behave under tangle replacements such as Conway mutation. 
Our particular choice of basis for the morphism spaces of $\Cob_{/l}$, which underpins the isomorphism from Theorem~\ref{thm_from_intro:BN_algebra=B}, is also key to the following result:

\begin{theorem}\label{thm:intro:mutation:BNrTangles}
	Let \(T\) and \(T'\) be two pointed 4-ended tangles which are obtained from one another by Conway mutation. Then the underlying immersed curves of \(\BNr(T;\field)\) and \(\BNr(T';\field)\) agree, and so do their local systems up to multiplication by \(\pm1\). 
\end{theorem}
In particular, this implies:

\begin{theorem}\label{thm:intro:mutation:BNrLinks}
	Reduced Bar-Natan homology is mutation invariant over the field \(\fieldTwoElements\).
\end{theorem}

Some partial progress in the direction of Theorem~\ref{thm:intro:mutation:BNrLinks} was made by Guo \cite{Guo}. An analogous result for reduced Khovanov homology \(\Khr(\Lk;\fieldTwoElements)\) was proved by Bloom \cite{Bloom} and  Wehrli \cite{wehrli2010mutation}; note that Theorem \ref{thm:intro:mutation:BNrTangles} re-establishes their result via different methods. In the general case of arbitrary coefficients, Theorem~\ref{thm:intro:mutation:BNrLinks} is false. The first counterexample was provided by Wehrli~\cite{WehrliCounterexample}. Theorem~\ref{thm:intro:mutation:BNrTangles} now provides a very satisfying answer why this is the case; see Example~\ref{exa:mutation:counterexample}. However, it is still unknown if \emph{component preserving} mutation leaves the invariants unchanged. In the case of reduced Khovanov homology with rational coefficients, Lambert-Cole has verified this for certain infinite families of knots \cite{LambertCole}. In the same paper, he also provided a strategy of proving mutation invariance for Khovanov-Floer theories over \( \fieldTwoElements \), which was recently implemented by Saltz \cite{saltz2018mutation}.

The curves \( \Khr(T) \) and \( \Kh(T) \) are compact, but \( \BNr(T) \) is not. In fact, \( \BNr(T) \) contains \( 2^{n} \) non-compact components, where \(n\) is the number of closed components of \( T \). In particular, if \( T \) does not contain any closed component, then \( \BNr(T) \) contains exactly one non-compact curve and this component contains the information about Rasmussen's \( s \)-invariant~\cite{Rasmussen_slice_genus}. Since non-compact components do not have interesting local systems, this implies: 
\begin{theorem}
Rasmussen's \(s\)-invariant of knots is mutation invariant over any field \( \field \).
\end{theorem}

\subsection{Infinitely twisted knots}


Rozansky gives a definition of Khovanov homology of a link inside $S^2 \times S^1$ \cite{Kh_Roz_HH}. In this setting, a link is considered as a circular closure of a tangle $T$ and its Khovanov homology is defined using the Hochschild homology of a bimodule he associates with $T$. Rozansky also studies a certain universal projective resolution complex $\mathbf{P}_n$, which he approximates by complexes of torus braids with a large number of twists. This gives rise to a practical means of computing Khovanov homology in $S^2 \times S^1$ in terms of the limit of the Khovanov homologies of the ‘‘torus braid closures'' of $T$ within $S^3$.
In related work,  Rozansky  proves that \( \mathbf{P}_n \) categorifies the $n$-stranded Jones-Wenzl projector $P_n$ \cite{Rozansky2014}.

\begin{figure}[t]
\centering
\begin{subfigure}[b]{0.38\textwidth}
	\centering
	$\ExampleInfinitelyTwistedKnot$
	\caption{\( \Knot_\infty=\Lk(T_{-\infty},T_{2,-3}) \).}
	\label{fig:Kinf_intro}
\end{subfigure}
\begin{subfigure}[b]{0.3\textwidth}
	\centering
	\includegraphics[scale=0.45]{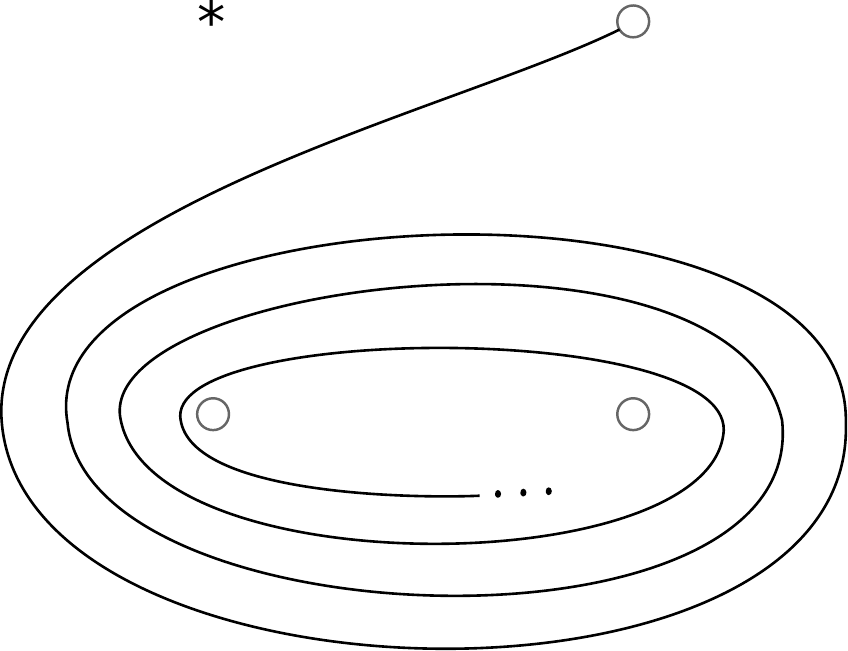}
	\caption{\( \Linf \)}
	\label{fig:inf_curve_1_intro}
\end{subfigure}
\begin{subfigure}[b]{0.3\textwidth}
	\centering
	\includegraphics[scale=0.45]{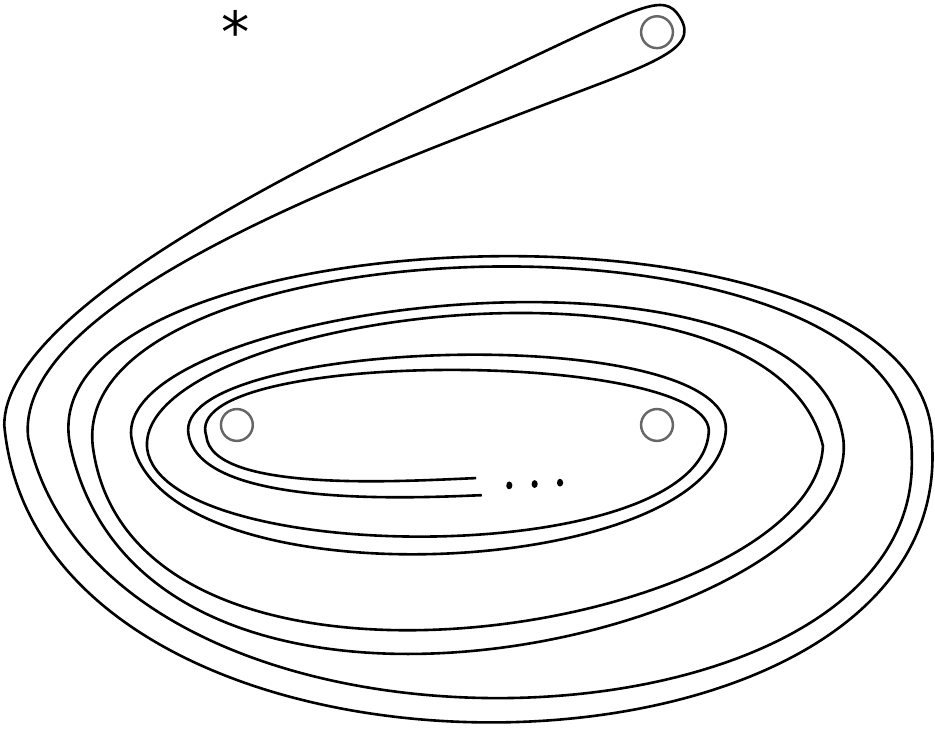}
	\caption{\( \Linfinf \)}
	\label{fig:inf_curve_2_intro}
\end{subfigure}
\caption{}
\label{fig:inf_curves_intro}
\end{figure}

Specializing to the case of 4-ended tangles---or in the terminology of \cite{Kh_Roz_HH}, $(2,2)$-tangles---let us consider the infinitely twisted tangle $T_{-\infty}$ in Figure~\ref{fig:Kinf_intro}. As a consequence of our geometric setup, we can now interpret \( \mathbf{P}_2= \KhTl{T_{-\infty}} \) as the non-compact curve \( \Linf \) in Figure~\ref{fig:inf_curve_1_intro}. 
We also consider the curve \( \Linfinf \) from Figure~\ref{fig:inf_curve_2_intro}, which corresponds to the mapping cone \(
\begin{tikzcd}[column sep=20pt]
\KhTl{T_{-\infty}}
\arrow{r}{H\cdot \id}
&
\KhTl{T_{-\infty}}
\end{tikzcd} 
\).
With these curves $\Linf$ and $\Linfinf$ in hand, we then define invariants of \( (-\infty,2) \)-twisted knots \( \Knot_\infty \). These are knots that are obtained as the closures of 4-ended tangles with $T_{-\infty}$ as illustrated in Figure~\ref{fig:Kinf_intro}. They are considered up to isotopy and up to moving twists \( \CrossingL \) and \( \CrossingR \) from the infinite part \( T_{-\infty} \) to the smooth part \( T \).
Following ideas from \cite{Kh_Roz_HH,EVZ}, we define
\[ \BNr(\Knot_\infty)=\HF (\Linf, \BNr(T)) \quad\text{and}\quad \Khr(\Knot_\infty)=\HF (\Linfinf,\BNr(T)) \] 
These two objects carry natural module structures over the rings \( \RxH=\fieldTwoElements[\xben,H]/(\xben H) \) and \( \RzT = \fieldTwoElements\langle z,T \rangle/ (z^2,T^2) \), which are subalgebras of the endomorphism spaces \( \HF(\Linf,\Linf)\) and \(\HF(\Linfinf,\Linfinf) \), respectively. Since the module structures are most easily seen on the algebraic level, we also define the above wrapped Lagrangian Floer homologies algebraically.

\begin{theorem} 
The modules \( \BNr(\Knot_\infty)_\RxH \) and \( \Khr(\Knot_\infty)_\RzT \) are invariants of infinitely twisted knots \( \Knot_\infty \).
\end{theorem} 
The algebra  \( \Rxy=\fieldTwoElements[\xben,\yben]/(\yben^2-\xben^3) \), which arises as a subalgebra of $\RzT$, appears in the work of Benheddi \cite{Benheddi}. In fact, for the infinitely twisted torus knot \( T(2,-\infty)=\Lk(T_{-\infty},\Ni) \), \( \Khr(T(2,-\infty))_\Rxy\) is equal to \(\Rxy_\Rxy \), which recovers Benheddi's result; see Example~\ref{ex:Uinf_invariants}. 

The invariant \(\BNr(\Knot_\infty)_\RxH \) can be easily computed from \(\DD(T)\) thanks to the following result:
\begin{theorem}\label{thm_intro:lift_to_D_str}
For any infinitely twisted knot \( \Knot_{\infty}=\Lk(T_{-\infty},T) \),
\[ 
\BNr(\Knot_{\infty})_\RxH \simeq \DD(T)^{\BNAlgH} \bt {}_{\BNAlgH}\bimoduleG^\RxH \bt \RxH
 \]
where the bimodule \( {}_{\BNAlgH}\bimoduleG^\RxH \) is equal to 
\[
\begin{tikzcd}
    \DotC\gen
    \arrow[in=-30,out=30,looseness=4]{rl}{\substack{(S,D,S | \xben) \\ (S,D,SS,D,S | \xben^2) \\ (S,D,SS,D,SS,D,S|\xben^3) \\ \cdots}}
    \arrow[rl, in=-150,out=150,looseness=4,"(D^k  | H^k)" left]
\end{tikzcd}
 \]
\end{theorem}

An infinitely twisted knot \( \Knot_\infty=\Lk(T_{-\infty},T) \) can be viewed as a sutured tangle \( T_s=T \) with a horizontal equatorial suture \( s \). In the presence of this suture, it is natural to consider \( T_s \) up to adding twists \( \CrossingR\) and \(\CrossingL \) on the top and on the bottom. As a result of the discussion above, we obtain invariants of sutured tangles \( \BNr(T_s)_\RxH \) and \( \Khr(T_s)_\Rxy \). As a vector space, the invariant  \( \Khr(mT_s) \) recovers \(
\displaystyle\mathop{\Kh}_{\longleftarrow}(T_s)\),
 as studied  by the second author \cite{Kappa}. In that setting, there is a finite dimensional graded vector space \( \varkappa(T_s) \) that is used as an invariant of the knot manifold arising as the double branched cover of the tangle, where the suture lifts to a distinguished (equivariant pair of) meridians. It turns out that \( \varkappa(T_s) \) agrees with the \( \xben \)-torsion of \( \Khr(mT_s)_\Rxy \); we calculate this for the \( (-2,3) \)-pretzel tangle in Example~\ref{exa:Pairing:TorusKnots}.

Since an infinitely twisted knot \( \Knot_\infty=\Lk(T_{-\infty},T) \) can also be viewed as a link obtained by a circular closure of $T$ inside $S^1\times S^2$, we expect a connection between our construction and the one in \cite{Kh_Roz_HH}. Our algebra $\RzT$ (respectively, $\RxH$) should arise as Hochschild homology of some version of the Khovanov arc algebra, with appropriate changes to account for the fact that we work with \emph{reduced} Khovanov homology $\Khr(\Knot_\infty)$ (respectively, \emph{deformed reduced} Bar-Natan homology $\BNr(\Knot_\infty)$),  whereas \cite{Kh_Roz_HH} describes the limit of the \emph{unreduced} Khovanov homology.  Our results also indicate that the invariants of links in $S^1\times S^2$ \cite{Kh_Roz_HH} and  $\#_k S^1\times S^2$ \cite{MMSW} should have a rich module structure.

\subsection{Relationship to other Khovanov type invariants of tangles}
The Jones polynomial, as defined by Kauffman's state sum formula, comes equipped with the ability to calculate in terms of local pictures, namely, by defining the appropriate skein module associated with a tangle. 
Bar-Natan's categorical framework achieves the same for Khovanov homology. However, his invariant $\KhTl{T}$ is only one of a number of Khovanov-type invariants associated with a tangle. The interaction between these different invariants is interesting in its own right. For instance, Khovanov constructed invariants of two-sided tangles that take the form of bimodules over certain arc algebras \cite{Khovanov2}, which appear again as endomorphism algebras in Bar-Natan's category \cite{Manion-thesis}. In another direction, Roberts proposed a framework for a bordered theory for Khovanov homology \cite{Roberts-A,Roberts-D}, which borrows nomenclature from and is defined in analogy with Lipshitz, Ozsváth, and Thurston's bordered Heegaard Floer theory~\cite{LOT-main}. Roberts defined candidates for type~D structures, type~A structures, and a box tensor product between them, which recovers the Khovanov homology of a knot from its decomposition into two tangles. Manion has reinterpreted Khovanov's algebraic framework for two-sided tangles in the bordered language and related it to Robert's bordered theories \cite{Manion-bordered}. 
Our algebra $\BNAlgH$ can be defined using the same bordered techniques and we expect that the corresponding bordered type~D structure agrees with our invariant $\DD(T)$; see Remarks~\ref{rmk:Kh_arc_algebra} and~\ref{rmk:bordered_constr_of_D_structure}. 

\subsection{Relationship to other immersed curve invariants}\label{subsec:intro:relationship_other_invariants}

%
%
%
%

In recent years, several other homology theories have been developed in low-dimensional topology. These include Heegaard Floer homology $\HFhat(Y)$ for closed 3-manifolds due to Ozsváth and Szabó~\cite{OSHF3mfds}, knot Floer homology \(\HFK(\Knot)\) discovered, independently, by the same authors and by Rasmussen~\cite{OS-hfk-0,Rasm_hfk}, and Kronheimer and Mrowka's knot homology \( I^\natural (\Knot) \) from instanton Floer theory~\cite{K-M_KHI}. The similarity between these invariants and Khovanov homology has been the subject of intense research for the past two decades. The first established relationship between Heegaard Floer and Khovanov theories was a spectral sequence from Khovanov homology $\Kh(\mirror\Knot)$ of (the mirror of) a knot $\Knot$ to the Heegaard Floer homology \(\HFhat(\Sigma(\Knot))\) of its double branched cover \cite{OS-spectral}. Subsequently, a spectral sequence from $\Khr(\mirror\Knot)$ to \( I^\natural (\Knot) \) was used to show that Khovanov homology detects the unknot~\cite{K-M_unknot_detector}. Previously, the existence of a spectral sequence from $\Khr(\mirror\Knot)$ to \(\HFK(\Knot) \) had been conjectured by Rasmussen~\cite{Rasm_comparison}; this conjecture has been resolved by Dowlin \cite{Dowlin}. Very recently, Kronheimer and Mrowka constructed a spectral sequence from reduced Bar-Natan homology \( \BNr(\Knot;\fieldTwoElements) \) to a  deformation of \( I^\natural (\Knot) \) \cite{KM_spectral_seq_from_BN}.  


By construction, instanton Floer theory counts solutions of certain PDEs and it is therefore inherently geometric. 
The same is true, at least morally, for Heegaard Floer theory. 
In the context of the Atiyah-Floer conjecture, these theories are expected to be isomorphic to the Lagrangian Floer homology of the corresponding local invariants for manifolds with boundary (and tangles embedded therein). So, in particular, these local invariants are expected be objects in the Fukaya categories of some appropriate moduli spaces. 
In the context of instanton knot Floer homology, for 4-ended tangles, this leads to representation theoretic constructions of  immersed curves \( R_\pi(T) \) and \( R^\natural_\pi(T) \) on the pillowcase due to Hedden, Herald and Kirk \cite{HHK1, HHK2}. The first author extended this construction algebraically~\cite{KotPillowcase}. 
In Heegaard Floer homology, Auroux \cite{Auroux1,Auroux2} and Lekili and Perutz \cite{PerutzLekili} suggested that the correct moduli space for the invariant of a compact oriented 3-manifold $M$
with $\partial M=\Sigma_g$ and a basepoint $z\in\Sigma_g$ should be $\Sym^g(\Sigma_g \smallsetminus z)$.
In particular, Auroux showed that bordered Heegaard Floer invariants  \cite{LOT-main} define objects in the partially wrapped Fukaya category of $\Sym^g(\Sigma_g \smallsetminus z)$ and that the pairing in
bordered Heegaard Floer homology is equivalent to taking the homology of morphism spaces; see also~\cite{LOT-mor}.
Conceptually, this construction is very useful, but for computations, one still has to resort to  algebraic machinery. 
The immersed curve invariant $\HFhat(M)$ due to Hanselman, Rasmussen and the second author addresses this issue in the case $g=1$, that is, for 3-manifolds with torus boundary~\cite{HRW}. There is also an immersed curve invariant $\HFT(T)$ of 4-ended tangles due to the third author~\cite{pqMod}, which achieves the same for link Floer homology using Zarev's bordered sutured Heegaard Floer homology~\cite{ZarevThesis}. 

The tangle invariants $\BNr(T)$, $\Khr(T)$, and $\Kh(T)$ fit into this growing family of immersed curve invariants. All of these invariants are objects in the (wrapped) Fukaya category of simple surfaces: the once-punctured torus (identified with the boundary of a 3-manifold) or the 4-punctured sphere (identified as a tangle boundary).  The multicurves \( R_\pi(T) \) and \( R^\natural_\pi(T) \) live on the pillowcase, a 2-dimensional sphere with four marked points, which, in this context, should be viewed as the space of traceless $SU(2)$-representations of the fundamental group of the 4-punctured sphere. As we show in Section~\ref{sec:pillowcase}, the pillowcase minus the four marked points can be naturally identified with the tangle boundary after choosing a basepoint. Thus, we may regard \( R_\pi(T) \) and \( R^\natural_\pi(T) \) as multicurves on the tangle boundary as well.

The geometric reinterpretation of both $\HFhat(M)$ and $\HFT(T)$ relies on classification results that are very similar to our Classification Theorem~\ref{thm_from_intro:classification}. In fact, the first part of this theorem, as well as the corresponding result in~\cite{pqMod}, ultimately appeal to the same simplification algorithm that was developed for $\HFhat(M)$ in~\cite{HRW}. This is closely related to a classification result of Haiden, Katzarkov and Kontsevich~\cite[Theorem~4.3]{HKK}, which says that twisted complexes over certain quiver algebras are classified by objects in the topological Fukaya category of corresponding marked surfaces. There is an algebraic appeal to representations of nets used to encode twisted complexes; by contrast our setup encodes the relevant algebraic structures geometrically. 
For the second part of Theorem~\ref{thm_from_intro:classification} (the classification of morphisms), we adapt the algebraic and geometric arguments developed for $\HFT(T)$ in~\cite{pqMod}. 

It is interesting to compare the local immersed curve invariants in simple examples in light of the aforementioned spectral sequences between their corresponding global counterparts. For instance, the double branched cover of the $\frac{p}{q}$-rational tangle $T_{p/q}$ is a (parametrized) solid torus, and its $\HFhat$-invariant is an embedded curve of slope $\frac{p}{q}$. Furthermore, after identifying the pillowcase with $\FourPuncturedSphere$, \( R_\pi(T_{p/q}) \) coincides with the arc \(\BNr(T_{p/q})\) and \( R^\natural_\pi(T_{p/q})\) agrees with the figure-8 curve \( \Khr(T_{p/q}; \fieldTwoElements)\). Finally, $\HFT(T_{p/q})$ is the embedded curve obtained by resolving the single point of self-intersection of $\Khr(T_{p/q})$. Even for non-rational tangles, such as the $(2,-3)$-pretzel tangle $T_{2,-3}$, there are remarkable similarities; see Section~\ref{sec:FigureEight}. One might expect the extra differentials on reduced Bar-Natan and Khovanov homology that induce the spectral sequences to Heegaard Floer and instanton Floer homology to admit a geometric interpretation in terms of smoothing certain points of self-intersection of the multicurves $\BNr(T)$ and $\Khr(T)$, though it seems hard to predict what form the general statement in each case might take. 
For instance, it is impossible to obtain $\HFT(T_{2,-3})$ by resolving self-intersection points of $\Khr(T_{2,-3})$, at least if $\Khr(T_{2,-3})$ is in minimal position; compare \cite[Section 7]{HRW-prop}. 
Also, \(\BNr(T;\field)\) depends on the field $\field$, whereas \(R_\pi(T)\) does not (though it may depend on the perturbation \(\pi\)). 
Nonetheless, the mere fact that the similarities between Khovanov, Heegaard Floer, and instanton Floer theories extend to a local level suggests that our immersed curve invariants \(\BNr(T)\), \(\Khr(T)\) and \(\Kh(T)\) will be able to offer new geometric insight.

\subsection{Origins of our immersed curve invariants}
This paper grew out of an attempt to simplify the construction of yet another immersed curve invariant, which was defined by Hedden, Herald, Hogancamp, and Kirk in~\cite{HHHK}. 
They showed that a certain quotient of Bar-Natan's cobordism category over \(\fieldTwoElements\) embeds into the Fukaya category of the pillowcase; see Remark~\ref{obs:HZero}. This allowed them to associate with any 4-ended tangle \(T\) a twisted complex \(L_T\) over this Fukaya category; the first author observed that \(L_T\) can be viewed as a multicurve on the pillowcase ~\cite{Artem-Notes}. 
Based on our computations, we expect that the curve \(L_T\) agrees with the curve \( \Khr(T;\fieldTwoElements)\).
Hedden, Herald, Hogancamp, and Kirk also obtained a pairing result for $L_T$. In terms of immersed curves, it says that
\[ \HF(L_T, \DotCarc)\cong \Khr( \Lk(T, \Li); \fieldTwoElements),\] where \( \DotCarc \) is the arc from Figure~\ref{fig:exa:classification:curves:trivial}.
Our key steps towards generalizing these results are the calculation of Bar-Natan's full universal cobordism category \(\Cob_{/l}\) (Theorem~\ref{thm_from_intro:BN_algebra=B}), together with the corresponding classification theorem (Theorem~\ref{thm_from_intro:classification}).


The question about the origins of our immersed curve invariants might also be asked on a more fundamental level. The combinatorial definition of Khovanov homology obfuscates what the right geometric framework should be. Nevertheless, Seidel and Smith discovered a symplectic version of Khovanov homology~\cite{Seidel-Smith}, which is now known to agree with Khovanov’s original theory by the work of Abouzaid and Smith~\cite{Abo_Smith_0, Abo_Smith}. Manolescu endowed this theory with a bigrading~\cite{Man_hilbert_schemes}, based on Bigelow's homological definition of the Jones polynomial \cite{Bigelow}.
Also, there is Witten's proposal to obtain Khovanov homology via gauge theory \cite{Witten_Kh} and an algebro-geometric interpretation of Khovanov homology due to Cautis and Kamnitzer \cite{Cau_Kam}. 
It is not clear if our immersed curves can be understood in terms of any of these theories. 
Interestingly, however, figure-8 curves seem to be a recurring feature. In our construction, they appear as the invariants \( \Khr(T) \) of rational tangles; see Example~\ref{exa:fig_eights_rational_tangles}.
They also appear in work of Bigelow \cite[Figure~2]{Bigelow} and Manolescu \cite[Figure~6]{Man_hilbert_schemes}. Finally, Droz and Wagner use grid diagrams to interpret the generators of the Khovanov chain complex in terms of intersection points of figure-8 curves and arcs; see \cite[Figure~2]{Droz_Wagner}. Their pictures look intriguingly similar to our pairing pictures. In particular, in the special case of 2-bridge knots, the intersection picture we get (see Figure~\ref{fig:exa:Pairing:Trefoil:Resolutions:EightArc}, for example) coincides with the Bigelow generators for reduced Khovanov homology~\cite[Section~7.5]{Man_hilbert_schemes}. 

\subsection*{Acknowledgements} This project has benefitted from inspiring conversations with a number of people, including: Antonio Alfieri, Sabin Cautis, Nate Dowlin, Jonathan Hanselman, Paul Kirk, Robert Lipshitz, Jake Rasmussen, Zoltán Szabó, and Ben Williams.

\section{Algebraic preliminaries}\label{sec:AlgStructFromGDCats}

We begin by fixing conventions for the algebraic structures appearing in this paper. We will generalize some tools from \cite[Section~1]{pqMod} for simplifying these algebraic structures over any commutative unital ring \( \Rcomm \), not just the field \( \fieldTwoElements \) with 2 elements. 

\subsection{Complexes over differential bigraded categories	}

	Most algebraic structures in this paper carry a bigrading. Specifically, we have a homological grading \( h \) and a quantum grading \( q \), both of which are integer gradings. Often it is convenient to combine these two gradings to a third grading: the \( \delta \)-grading is defined by 
	\( \delta=\tfrac{1}{2}q-h \)
	and takes values in \( \tfrac{1}{2}\mathbb{Z} \). A bigraded chain complex is a chain complex whose underlying chain module carries a bigrading and whose differential increases the homological grading by 1 and preserves quantum grading. Strictly speaking, such an object should be called a \emph{co}chain complex and the homological grading a \emph{co}homological grading. However, following the generally accepted terminology in Khovanov theory, we will omit the prefix ``co'' everywhere. Note that the \( \delta \)-grading decreases along the differential by 1. 

\begin{definition}\label{def:dgcat}
Let \( \Com \) be the category of free bigraded chain complexes over \( \Rcomm \) and (bigrading preserving) chain maps between them. A \textbf{differential bigraded category} \( \mathcal{C} \) is an enriched category over \( \Com \). \end{definition}To spell this out, the hom-objects are free bigraded \( \Rcomm \)-modules
\[ \Mor(A,B)=\bigoplus_{h,q\in\mathbb{Z}}\Mor(A,B;h,q)\]
endowed with a differential \[ 
\partial_{h,q}\co\Mor(A,B;h,q)
\rightarrow
\Mor(A,B;h+1,q)
 \]
 that is an \( \Rcomm \)-module homomorphisms usually denoted just by \( \partial \). This satisfies \( \partial^2=0 \) (which we will sometimes refer to as the $d^2$-relation) and the Leibniz rule 
\[ 
\partial(g\circ f)
=
\partial(g)\circ f+(-1)^{h(g)}\cdot g\circ\partial(f)
 \]
where \( h(g) \) denotes the homological grading of \( g \) and \( \circ \) denotes the composition in \( \mathcal{C} \), which is associative, unital, and respects the bigrading. Our sign convention here agrees with the one for differential graded categories from~\cite[Section~2.2]{Keller}. Note that the identity morphisms have vanishing bigrading
and lie in the kernel of~\( \partial \).

\begin{definition}\label{def:CatOfMatrices}
Given a differential bigraded category~\( \mathcal{C} \), we define another differential bigraded category \( \Mat(\mathcal{C}) \) as follows. Its objects are formal direct sums
\[ \bigoplus_{i\in I}h^{h_i}q^{q_i}\object_i, \]
where \( I \) is some finite index set and \( h^{h_i}q^{q_i}\object_i \) denotes the object \( \object_i\in\ob(\mathcal{C}) \) formally shifted by \( h_i \) in homological grading and \( q_i \) in quantum grading. 
The morphisms spaces are given by
\begin{equation*}\label{eq:morphism_spaces_in_Mat}
\Mor\bigg(\bigoplus_{i\in I}h^{h_i}q^{q_i}\object_i,\bigoplus_{j\in J}h^{h_j}q^{q_j}\object_j\bigg)\coloneqq \bigoplus_{(j,i)\in J\times I}h^{h_j-h_i}q^{q_j-q_i}\Mor(\object_i,\object_j)
\end{equation*}
as \( \Rcomm \)-modules. Composition in \( \Mat(\mathcal{C}) \) is the usual matrix multiplication, and the differential is given by
\[ \partial\Big(\big(f_{ji}\co h^{h_i}q^{q_i}\object_i\rightarrow h^{h_j}q^{q_j}\object_j\big)_{j,i}\Big)\coloneqq\big((-1)^{h_j}\cdot \partial(f_{ji})\co h^{h_i}q^{q_i}\object_i\rightarrow h^{h_j}q^{q_j}\object_j\big)_{j,i} \]
It is elementary to check that \( \Mat(\mathcal{C}) \) is a well-defined differential bigraded category.
\end{definition}


\begin{remark}
\label{rmk:gradings_of_morphisms}
	The notation \( h^{n}q^{m}\object \) for an object \( \object \) 
	in homological grading \( n \) and quantum grading \( m \) corresponds to Bar-Natan's notation \( \object[n]\{m\} \). To specify the \( \delta \)-grading of an object \( \object \), we write \( \delta^d\object \). As a compact alternative notation, we will sometimes use \( \GGdqh{X}{d}{m}{n} \) instead of \( h^{\HomGrad{n}} q^{\QGrad{m}}\delta^{\DeltaGrad{d}}X \). Often, one of the three gradings will be omitted in an attempt to declutter our notation. 
	The following is a concise summary of our grading conventions for \( \gr\in\{q,\delta, h\} \):
	\[ \gr\left(
	\begin{tikzcd}  
	x
	\arrow{r}{a}
	&
	y
	\end{tikzcd}
	\right)=\gr(y)-\gr(x)+\gr(a). \]
	This describes how to compute the grading of a basic morphism represented by a single arrow, labelled by some \( a\in \Mor(x,y) \), between two objects \( x \) and \( y \) in \( \mathcal{C} \) that are shifted into gradings \( \gr(x) \) and \( \gr(y) \), respectively, and
	 agrees with the conventions in \cite[Definition~1.4]{pqMod} and \cite[Section~6]{BarNatanKhT}.
\end{remark}

Given a differential bigraded category \( \mathcal{C} \), we define an auxiliary category \( \Cxpre(\mathcal{C}) \), \textbf{the category of pre-complexes}, which is an enriched category over the category of bigraded \( \Rcomm \)-modules and bigrading preserving morphisms between them. The objects are pairs \( (\object,d) \), where \( \object \) is an object of \( \mathcal{C} \) and \( d\in \Mor(\object,\object;1,0) \) is an endomorphism of \( \object \) which preserves quantum grading and increases homological grading by 1. 
As bigraded \( \Rcomm \)-modules, the hom-objects agree with those in~\( \mathcal{C} \): 
\(\Mor((\object,d),(\object',d'))=\Mor(\object,\object')\). We can define endomorphisms \( D \) on them by setting 
\begin{equation*}
 D(f)\coloneqq d'\circ f-(-1)^{h(f)} \cdot f \circ d+\partial(f)
 \end{equation*}
(again, our sign convention agrees with the one in~\cite[Section~2.2]{Keller}).
We would like \( D \) to be a differential so that \( \Cxpre(\mathcal{C}) \) is a differential bigraded category. 
It is easy to check that \( D \) satisfies the Leibniz rule. So \( D \) is a differential iff
\begin{align*}
    D^2(f)
    =
    &~(d'^2+\partial(d'))\circ f
-f\circ (d^2+\partial(d))
\end{align*}
vanishes. This is the case for the full subcategory \( \Cx(\mathcal{C}) \) of \( \Cxpre(\mathcal{C}) \) consisting of those objects \( (\object,d) \) for which
\begin{equation} 
d^2+\partial(d) \tag{\( \ast \)}\label{eqn:d2term}
\end{equation}
vanishes. 
\begin{definition}\label{def:CatOfComplexes}
We call \( \Cx(\mathcal{C}) \)  \textbf{the category of complexes over the differential bigraded category \( \mathcal{C} \)}.  By construction, \( \Cx(\mathcal{C}) \) is a differential bigraded category. 
\end{definition}

Note that we could also restrict to the full subcategory of those objects \( (\object,d) \) for which (\ref{eqn:d2term}) is equal to some central algebra element; see \cite[Section~1]{pqMod} for details.
	We can associate a directed graph with a category, where objects correspond to vertices and arrows to  morphisms. In the same way, we can also think of objects in \( \Mat(\mathcal{C}) \) and \( \Cx(\Mat(\mathcal{C})) \) as graphs.


\begin{example}\label{exa:HighBrowDefChainCxs}
	Let \( \mathcal{C} \) be the category with a single object \( \singleB \) in bigrading \( h^0q^0 \), \( \Mor(\singleB,\singleB)=\Rcomm \) and vanishing differential. Then, objects in \( \Mod^{\Rcomm}\coloneqq\Cx(\Mat(\mathcal{C})) \) correspond exactly to free bigraded chain complexes over \( \Rcomm \) together with a choice of an \( \Rcomm \)-linear basis. This can be seen by identifying the object \( \singleB \) with \( \Rcomm \), considered as an \( \Rcomm \)-module. The morphism spaces, however, consist of all \( \Rcomm \)-linear maps and not just chain maps. Moreover, by definition, the morphism spaces are themselves bigraded chain complexes over \( \Rcomm \). To recover the usual notions of chain maps between chain complexes, we need to pass to the underlying ordinary category of \( \Mod^{\Rcomm} \).
\end{example}

\begin{definition}\label{def:UnderlyingOrdinaryCat}\cite[Definition~3.4.5]{cathtpy}
	Given an enriched category \( \mathcal{C} \) over some monoidal category \( \mathcal{V} \), the \textbf{underlying ordinary category} \( \mathcal{C}_0 \) of \( \mathcal{C} \) has the same objects as \( \mathcal{C} \) and its hom-sets are defined by 
	\( \mathcal{C}_0(A,B)\coloneqq \Mor_\mathcal{V}(1_\mathcal{V},\mathcal{C}(A,B)). \)
\end{definition}
\begin{example}\label{exa:HighBrow_chain_map}
	Let \( \mathcal{C} \) be a differential bigraded category. The unit in \( \Com \) is the complex \( 0\rightarrow \Rcomm \rightarrow0 \), supported in bigrading \( h^0q^0 \), and the morphisms in \( \Com \) are (bigrading preserving) chain maps. Hence, the hom-sets of \( \mathcal{C}_0 \) become cycles in the bigrading \( h^0q^0 \): 
	\[ 
	\mathcal{C}_0(A,B)
	=
	\ker\big(\partial_{0,0}:\Mor(A,B;0,0)\rightarrow \Mor(A,B;1,0)\big)
	 \]
	We call elements of these hom-sets \textbf{chain maps}. If \( \mathcal{C} \) is the differential bigraded category \( \Mod^{\Rcomm} \) from Example~\ref{exa:HighBrowDefChainCxs}, this agrees with the usual notions of chain maps between chain complexes. In other words, \( \Com = (\Mod^{\Rcomm})_0 \). One reason for passing to the underlying category is that, if we do not, two objects could be isomorphic through morphisms that shift the bigrading. 
\end{example}

\begin{example}\label{exa:HighBrow_chain_homotopy}	
	Let \( \mathcal{C} \) be a differential bigraded category. Consider the enriched category \( \Homology(\mathcal{C}) \) over the category of bigraded \( \Rcomm \)-modules and bigrading preserving morphisms between them, obtained from \( \mathcal{C} \) by replacing the hom-objects by their homologies with respect to the differential \( \partial \). By passing to the underlying ordinary category, we pick out the morphisms in \( \Homology(\mathcal{C}) \) whose bigrading vanishes. Therefore, we denote this category by \( \HomologyZero(\mathcal{C}) \).
	The hom-sets in \( \HomologyZero(\mathcal{C}) \) are quotients of those in \( \mathcal{C}_0 \): \( \HomologyZero(\mathcal{C})(A,B)=\ker(\partial_{0,0})/\im(\partial_{-1,0}). \) 
	In the case \( \mathcal{C}=\Mod^{\Rcomm} \), elements of these hom-sets represent homotopy classes of chain maps. 
	Therefore, if \( \mathcal{C} \) is the category of complexes over any differential bigraded category, we call two chain maps \textbf{homotopic} if they define the same element in \( \HomologyZero(\mathcal{C}) \)  and we say that two objects in \( \mathcal{C} \) are \textbf{homotopy equivalent} (or \textbf{chain homotopic}) if they are isomorphic as objects in \( \HomologyZero(\mathcal{C}) \). Finally, two objects in \( \mathcal{C} \) are \textbf{isomorphic} if they are isomorphic as objects in \( \mathcal{C}_0 \).
\end{example}

In the following we will use a notion of differential bigraded unital algebra \( \mathcal{A} \) over a ring of idempotents \( \oplus^N_{j=1}\Rcomm=\mathcal{I}\subseteq\mathcal{A} \) taken from \cite[Definition~2.1.1]{LOT-bim}, with a caveat that we require all \( \mathcal{I} \)-modules to be free \( \Rcomm \)-modules, and the gradings, multiplication and differential satisfy the same conditions as in Definition~\ref{def:dgcat}. We will call these algebras \emph{differential bigraded \( \mathcal{I} \)-algebras} for short. Also, we will frequently use notation $(-).\iota, \ \iota\in  \mathcal{I}$; this means that $(-)$ is either an  \( \mathcal{I} \)-module, in which case $(-).\iota \coloneqq(-)\cdot\iota$, or $(-)$ is an element of some  \( \mathcal{I} \)-module, in which case $(-).\iota$ indicates that $(-)\cdot\iota=(-)$ and all other idempotents from \( \mathcal{I} \) annihilate $(-)$.

\begin{observation}\label{obs:algebras_vs_categories}
We can regard any differential bigraded \( \mathcal{I} \)-algebra as a differential bigraded category as follows. The objects are given by a fixed \( \Rcomm \)-basis \( \{\iota_j\}^N_{j=1} \) of idempotents of \( \mathcal{I} \), and we define \( \Mor(\iota_k,\iota_m)\coloneqq \iota_m.\mathcal{A}.\iota_k \) as the morphism space between the basis elements \( \iota_k \) and \( \iota_m \), with the obvious composition and differential. Analogously, any  differential bigraded category gives rise to a differential bigraded \( \mathcal{I} \)-algebra. In this paper we will sometimes move freely between these two points of view on the same algebraic structure, without changing notation.
\end{observation}

\begin{example}\label{exa:HighBrowDefTypeDoverI}
	Considering a differential bigraded \( \mathcal{I} \)-algebra \( \mathcal A \) as a differential bigraded category, 
	we call \( \Cx(\Mat(\mathcal{A})) \) \textbf{the category of complexes over the differential bigraded \( \mathcal{I} \)-algebra \( \mathcal{A} \)}.
Here is point of view for working with \( \Cx(\Mat(\mathcal{A})) \) in practice. Given some generators \( \{x_i.\iota_{k_i}\} \) of a complex \( (X,d) \),  write: 
	\[ d\vert_{x_i.\iota_{k_i}\rightarrow x_j.\iota_{k_j}}=d_{ji}\in\iota_{k_j}.\mathcal{A}.\iota_{k_i} \]
	Then the \( d^2\)-relation ensures that 
	\begin{equation*}
	\sum_{j}d_{kj}d_{ji}+(-1)^{h(x_k)}\partial(d_{ki})=0
	\end{equation*}
	for any \( i,k \). Moreover, if \( f\co (X,d)\rightarrow (X',d') \) is a morphism  whose homological grading vanishes (or is even), then \( D(f)=d'f-fd+\partial(f) \). In particular, if the differential \( \partial \) on \( \mathcal{A} \) vanishes, chain homomorphisms in the underlying ordinary category satisfy \( d'f=fd \), namely, the usual relation on chain maps. If the homological grading of \( f\co (X,d)\rightarrow (X',d') \) is odd, then \( D(f)=d'f+fd+\partial(f) \). 
	Again, if the differential \( \partial \) on \( \mathcal{A} \) vanishes, we recover the usual definition of chain homotopies.
\end{example}

\begin{convention}
Whenever a differential bigraded \( \mathcal I \)-algebra \( \mathcal A \) has only one idempotent (for example, \( \mathcal A=\Rcomm \) or the polynomial algebra \( \mathcal A=\Rcomm[H] \)), we will denote this idempotent by \( \gen \). 
\end{convention}

\begin{definition}
	Let  \( \mathcal{A} \) be a differential bigraded \( \mathcal{I} \)-algebra. A \textbf{type~D structure} is a pair \( (X,d) \), where \( X \) is a bigraded right \( \mathcal{I} \)-module, and 
	\( 
	d\co X \rightarrow X\otimes_{\mathcal{I}} \mathcal{A}
	 \)
	is a right \( \mathcal{I} \)-module homomorphism which increases the homological grading by 1, preserves the quantum grading, and satisfies
	\[ (\id_{X}\otimes\mu_1)\circ d+(\id_{X}\otimes \mu_2)\circ(d\otimes \id_\mathcal{A})\circ d=0 \]
	where \( \mu_1 \) and \( \mu_2 \) are the differential and the multiplication on \( \mathcal{A} \), respectively. If \( (X,d) \) and \( (X',d') \) are two type~D structures, a morphism \( f\co (X,d)\rightarrow (X',d') \) is a right \( \mathcal{I} \)-module homomorphism
	\( 
	f\co X \rightarrow X'\otimes_{\mathcal{I}} \mathcal{A}
	 \). The composition \( f\circ g \) of two morphisms \( g\co (X,d)\rightarrow (X',d') \) and \( f\co (X',d')\rightarrow (X'',d'') \) is defined by
	\[ (\id_{X''}\otimes \mu_2)\circ(f\otimes \id_\mathcal{A})\circ g. \]
	The space of all morphisms between two type~D structures is endowed with a differential
	\[ D(f)\coloneqq (\id_{X'}\otimes\mu_1)\circ f+(\id_{X'}\otimes \mu_2)\circ(d'\otimes \id_\mathcal{A})\circ f- (-1)^{h(f)}(\id_{X'}\otimes \mu_2)\circ(f\otimes \id_\mathcal{A})\circ d, \]
	turning the category of type~D structures over \( \mathcal A \) into a differential bigraded category, denoted by \( \TypeD(\mathcal A) \). As usual, the underlying ordinary category, consisting of type~D structures and bigrading preserving morphisms in the kernel of \( D \), is denoted by \( \TypeD(\mathcal A)_0 \).
\end{definition}

Generalizing Example~\ref{exa:HighBrowDefChainCxs} (where \( \Rcomm=\mathcal{I}=\mathcal{A} \)), complexes over differential bigraded \( \mathcal I \)-algebra \( \mathcal{A} \) and type~D structures over \( \mathcal{A} \) are essentially equivalent algebraic structures, except complexes come equipped with a choice of basis. (This is similar to the relationship between algebras and categories, which is explained in Observation~\ref{obs:algebras_vs_categories}.)

\begin{proposition}\label{prop:complexes_vs_Dstr}
	The underlying category \( \Cx(\Mat(\mathcal A))_0 \) of complexes over a differential bigraded \( \mathcal{I} \)-algebra \( \mathcal A \) is equivalent to the category \( \TypeD(\mathcal A)_0 \) of right type~D structures over \( \mathcal A \) and chain maps between them. Moreover, two chain maps between complexes are chain homotopic iff their corresponding morphisms of type~D structures are.
\end{proposition}

\begin{proof}[Sketch of proof]
	For the first part, it suffices to define an essentially surjective, fully faithful functor from \( \Cx(\Mat(\mathcal A))_0 \) to \( \TypeD(\mathcal A)_0 \). So, given a complex \( (X,d) \) over a differential bigraded \( \mathcal{I} \)-algebra \( \mathcal{A} \), we define the underlying \( \mathcal{I} \)-module of the corresponding type~D structure as the direct sum of copies \( R.x_i \) for each summand \( X_i \) of \( X \). The differential is defined by 
	\(d(x_i)=\sum_{j} x_j \otimes d_{ji}. \)
	Then, using the Koszul sign rule 
	\( (f\otimes g)(x\otimes y)=(-1)^{h(g)\cdot h(x)}f(x)\otimes g(y), \)  it is a routine check that the \( d^2 \)-relation is satisfied. 
	Similarly, we can define the homomorphism between type~D structures corresponding to a bigrading preserving morphism of complexes in the kernel of \( D \) and check that the notions of homotopy equivalence coincide.

	To see that this functor is essentially surjective, suppose \( N \) is a type~D structure over \( \mathcal A \). Then we define a complex \( (X(N),d) \) in \( \Cx(\Mat(\mathcal A)) \) as follows: first, we pick a basis \( \{\iota_j\}_{j} \) for \( \mathcal{I} \) and consider the decomposition \( N=\oplus_{j} N.\iota_j \). For each summand \( N.\iota_j \), we pick a basis \( \{x^k_j\} \). We set \( X(N)\coloneqq\bigoplus_{j,k}x^k_j.\iota_j \) and define \( d \) to be the matrix representing the differential of \( N \) in our chosen basis of \( N \). One can check that the \( d^2 \)-relation is satisfied, and the image of complex \( (X(N),d) \) under the functor from the previous paragraph is isomorphic to \( N \) in \( \TypeD(\mathcal A)_0 \).
\end{proof}

\begin{remark}\label{rmk:HighBrowDefTypeDoverI}
	These algebraic structures have appeared under different names in other contexts. What we call \emph{complexes over differential bigraded categories} first appeared (without the quantum grading) in \cite{Bondal_Kapranov} under the name of \emph{twisted complexes}.  Twisted complexes over differential graded categories were later promoted to \emph{twisted complexes over \( A_\infty \)-categories} \cite{Kontsevich_HMS}, the most general version of this algebraic structure. The concept of a \emph{complex over an ordinary category} was introduced by Bar-Natan \cite{BarNatanKhT}, which is a special case of a \emph{twisted complex over a differential graded category} where the differential is zero. Finally, \emph{type~D structures} were introduced by Lipshitz--Ozsv\'ath--Thurston \cite{LOT-main}, which, as shown in Proposition~\ref{prop:complexes_vs_Dstr}, can be regarded as twisted complexes over differential graded categories arising from differential graded \( \mathcal I \)-algebras.
Because our work builds on \cite{BarNatanKhT} and \cite{LOT-main}, we continue with the terminology \emph{complexes over differential bigraded \( \mathcal I \)-algebras} and \emph{type~D structures over differential bigraded \( \mathcal I \)-algebras}, which we use interchangably.
\end{remark}

Given a differential bigraded \( \mathcal{I} \)-algebra \( \mathcal A \), we will denote the category \( \Cx(\Mat(\mathcal A)) \) of complexes over \( \mathcal A \), or, equivalently, the category \( \TypeD(\mathcal A) \) of finite dimensional bigraded type~D structures over \( \mathcal A \), by \( \Mod^{\mathcal{A}} \). We will also sometimes indicate that \( N \in \ob(\Mod^{\mathcal{A}}) \) by adding a superscript \( N^{\mathcal A} \).
This notation differs slightly from \cite[Section~2.2.3]{LOT-bim}: we do not require type~D structures to be bounded, but we do require them to be finite dimensional. 
In Section~\ref{sec:InfinitelyStranded} 
we will have to work with infinite dimensional type~D structures, and so we define the following enlargement of \( \Mod^{\mathcal A} \).


\begin{definition}\label{def:infinite_D_structures}
Denote by \( \Mod_{\infty}^{\mathcal A} \) the enlargement of \( \Mod^{\mathcal A} \) consisting of type~D structures with possibly infinite dimensional underlying $\Rcomm$-module \( X=\bigoplus^{\infty}_{j=1} \Rcomm . \iota_j \).
The category \( \Mod_{\infty}^{\mathcal A} \), in a complete analogy with \( \Mod^{\mathcal A} \), is a differential bigraded category. 

\end{definition}


	Having a differential bigraded \( \mathcal I \)-algebra \( \mathcal{A} \), define the \textbf{opposite algebra} \( \op{\mathcal{A}} \) to be an algebra which differs from \( \mathcal{A} \) in multiplication: if \( ab=c \) in \( \mathcal{A} \), then \( ba=c \) in \( \op{\mathcal{A}} \). The bigrading and the differential in \( \op{\mathcal{A}} \) are the same as in  \( \mathcal{A} \). Given a right type~D structure \( N^\mathcal{A} \) over a differential bigraded \( \mathcal I \)-algebra \( \mathcal{A} \), \textbf{the dual type~D structure} \( \dual{N}^{\op{\mathcal{A}}} \) is the right type~D structure over the opposite algebra \( \op{\mathcal{A}} \), where the directions of arrows in \( N^\mathcal{A} \) are reversed, and the gradings of generators of \( N^\mathcal{A} \) are multiplied by \( -1 \). Equivalently, one can interpret the dual type~D structure as a left type~D structure \( ^{\mathcal{A}}\dual{N} \) over \( \mathcal{A} \).
In Sections~\ref{sec:MCGaction} and~\ref{sec:InfinitelyStranded} we will make use of the \emph{box tensor product} operation $\boxtimes$, as well as \emph{type~AD (DA, AA) structures}, sometimes also called \emph{type~AD (DA, AA) bimodules}. Because these sections restrict to $\fieldTwoElements$-coefficients, we need not pin down sign conventions for these algebraic structures. We refer the reader to the original source  for the definitions \cite[Section~2]{LOT-bim}.

\subsection{Cancellation and Cleaning-Up}\label{sec:cancellation _and_cleaning-up} We will make repeated use of two tools, which can also be found in \cite[Section~1]{pqMod}. We review these and, in particular, take care of signs.

\begin{wrapfigure}{r}{0.3333\textwidth}
	\centering
	\( \begin{tikzcd}[row sep=1.5cm, column sep=0.3cm]
	& (Z,\zeta) \arrow[bend left=12]{ld}{a} \arrow[bend left=12]{rd}{c} \\
	(Y_1,\varepsilon_1)\arrow[bend left=12]{rr}{f} \arrow[bend left=12]{ru}{b}
	& &
	(Y_2,\varepsilon_2) \arrow[bend left=12]{ll}{e}\arrow[bend left=12]{lu}{d}
	\end{tikzcd} \)
	\vspace*{-10pt}
	\caption{The object~\( (X,\delta) \).}\label{fig:AbstractCancellation}\vspace*{-5pt}
\end{wrapfigure}
\myfixwrapfig

\begin{lemma}[Cancellation Lemma]\label{lem:AbstractCancellation}
Let \((X,\delta)\) be an object of \(\Cx(\Mat(\mathcal{C}))\) for some differential bigraded category \(\mathcal{C}\) and suppose it has the form shown in Figure~\ref{fig:AbstractCancellation}, 
where \((Y_1,\varepsilon_1)\), \((Y_2,\varepsilon_2)\), \((Z,\zeta)\in\ob(\Cxpre(\Mat(\mathcal{C})))\) and \(f\) is an isomorphism with inverse \(g\). Then \((X,\delta)\) is chain homotopic to \((Z,\zeta-bgc)\).
\end{lemma}

\begin{proof}
First we check that \( (Z,\zeta-bgc) \) is indeed an object of \( \Cx(\Mat(\mathcal{C})) \):
\[(\zeta-bgc)^2+\partial(\zeta-bgc)=\big(\zeta\zeta+dc+ba+\partial(\zeta)\big)+b\big(\varepsilon_1g+g\varepsilon_2+\partial(g)+gcbg\big) c=0\]
The first term vanishes because we are assuming that \( (X,\delta) \) satisfies the \( d^2 \)-relation. The second term vanishes because $gcbg 
=
-g(f\varepsilon_1+\varepsilon_2f+\partial(f))g=
-\varepsilon_1g-g\varepsilon_2-\partial(g).$
Next, we consider the chain maps 
\[ 
F\co(Z,\zeta-bgc)\rightarrow(X,\delta)
\qquad\qquad 
G\co(X,\delta)\rightarrow(Z,\zeta-bgc)
 \]
defined in Figure~\ref{fig:ProofAbstractCancellationF} and \ref{fig:ProofAbstractCancellationG}, respectively. One easily checks that \( D(F)=0 \) and \( D(G)=0 \). Indeed, the only non-trivial terms we need to compute are
\begin{align*}
-(-gc)(\zeta-bgc)+\varepsilon_1(-gc)+a+\partial(-gc)=& ~g(c\zeta+\partial(c))-(\varepsilon_1g+\partial(g))c+a-gcbgc\\
=&~-g(fa+\varepsilon_2c)-(\varepsilon_1g+\partial(g))c+a-gcbgc\\
=&~-\left(g\varepsilon_2+\varepsilon_1g+\partial(g)+gcbg\right)c=0
\end{align*}
for the first identity; a similar calculation checks the second identity. 
Now, \( GF=1_{Z} \) and conversely, it is not hard to check that 
\( FG=1_{X}+D(H), \)
where \( H \) is the homotopy given by the dashed line in Figure~\ref{fig:ProofAbstractCancellationHomotopy}.
\end{proof}

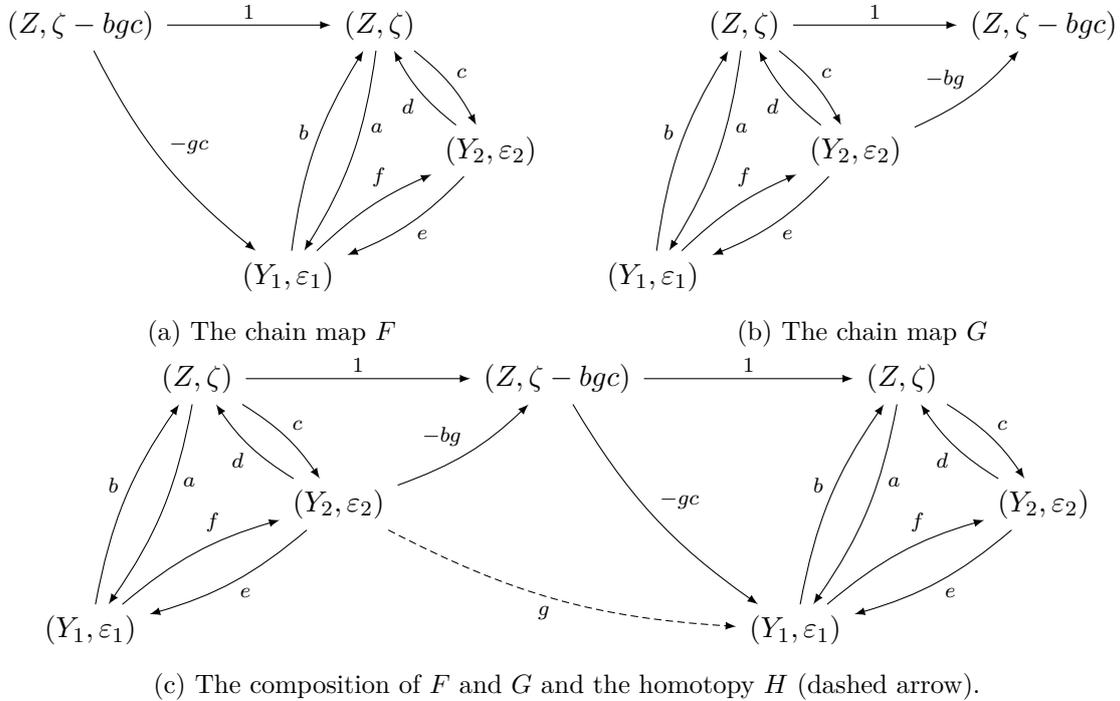
\begin{figure}[h]
	\centering
	\begin{subfigure}{0.48\textwidth}
		\centering
		\( \begin{tikzcd}[row sep=1cm, column sep=-0.2cm]
		(Z,\zeta-bgc)\arrow{rrr}{1}\arrow[bend right=12]{rrdd}{-gc} &~~~~~~~ && (Z,\zeta) \arrow[bend left=12,pos=0.3]{ldd}{a} \arrow[bend left=12]{rrd}{c}\\
		& &&&~& (Y_2,\varepsilon_2) \arrow[bend left=12]{llld}{e}\arrow[bend left=12]{llu}{d}\\
		& &(Y_1,\varepsilon_1) \arrow[bend left=10,pos=0.7]{rrru}{f} \arrow[bend left=12]{ruu}{b}
		\end{tikzcd} \)
		\caption{The chain map \( F \)}\label{fig:ProofAbstractCancellationF}
	\end{subfigure}
	\begin{subfigure}{0.48\textwidth}
		\centering
		\( \begin{tikzcd}[row sep=1cm, column sep=-0.2cm]
		& (Z,\zeta) \arrow[bend left=12,pos=0.3]{ldd}{a} \arrow[bend left=12]{rrd}{c}\arrow{rrrr}{1}&&&~~~~~ &(Z,\zeta-bgc)\\
		&&~& (Y_2,\varepsilon_2) \arrow[bend left=12]{llld}{e}\arrow[bend left=12]{llu}{d}\arrow[bend right=12]{rru}{-bg}\\
		(Y_1,\varepsilon_1) \arrow[bend left=10,pos=0.7]{rrru}{f} \arrow[bend left=12]{ruu}{b}
		\end{tikzcd} \)
		\caption{The chain map \( G \)}\label{fig:ProofAbstractCancellationG}
	\end{subfigure}
	\begin{subfigure}{0.95\textwidth}
		\centering
		\( \begin{tikzcd}[row sep=1cm, column sep=0cm]
		& (Z,\zeta) \arrow[bend left=12,pos=0.3]{ldd}{a} \arrow[bend left=12]{rrd}{c}\arrow{rrrr}{1}&&&~~~~~ &(Z,\zeta-bgc)
		\arrow{rrr}{1}\arrow[bend right=12]{rrdd}{-gc} &~~~~~~~ && (Z,\zeta) \arrow[bend left=12,pos=0.3]{ldd}{a} \arrow[bend left=12]{rrd}{c}
		\\
		&&~& (Y_2,\varepsilon_2) \arrow[bend left=12]{llld}{e}\arrow[bend left=12]{llu}{d}\arrow[bend right=12]{rru}{-bg}\arrow[dashed,swap, bend right=12]{rrrrd}{g}&&
		& &&&~& (Y_2,\varepsilon_2) \arrow[bend left=12]{llld}{e}\arrow[bend left=12]{llu}{d}
		\\
		(Y_1,\varepsilon_1) \arrow[bend left=12,pos=0.7]{rrru}{f} \arrow[bend left=12]{ruu}{b}&&&&&
		& &(Y_1,\varepsilon_1) \arrow[bend left=12,pos=0.7]{rrru}{f} \arrow[bend left=12]{ruu}{b}
		\end{tikzcd} \)
		\caption{The composition of \( F \) and \( G \) and the homotopy \( H \) (dashed arrow).}\label{fig:ProofAbstractCancellationHomotopy}
	\end{subfigure}
	\caption{Maps for the proof of Lemma~\ref{lem:AbstractCancellation}}\label{fig:ProofAbstractCancellation}
\end{figure}

\begin{lemma}[Clean-Up Lemma]\label{lem:AbstractCleanUp}
Let \((\object,d)\) be an object in \(\Cx(\mathcal{C})\) for some differential bigraded category \(\mathcal{C}\).
Then for any morphism \(h\in\Mor((\object,d),(\object,d))\) of bigrading \( h^0q^0 \) for which $h^2$ and $hD(h)$
vanish, \((\object,d)\) is isomorphic to \((\object,d+D(h))\).
\end{lemma}
\begin{proof}The identity \[(d+D(h))^2+\partial(d+D(h))=D(hD(h))-hD^2(h)=0\]
varifies that \( (\object,d+D(h)) \) is an object in \( \Cx(\Mat(\mathcal{C})) \).
Now consider the morphisms
\begin{center}
\( \begin{tikzcd}[row sep=1.5cm, column sep=0.6cm]
(\object,d) \arrow{rr}{1-h} && (\object,d+D(h))
\end{tikzcd} \)
\quad
\quad
\( \begin{tikzcd}[row sep=1.5cm, column sep=0.6cm]
(\object,d+D(h)) \arrow{rr}{1+h} && (\object,d)
\end{tikzcd} \)
\end{center}
and note that composing (in either order) gives \( 1-h^2=1 \). It remains to check that each morphism lies in the kernel of \( D \). To see this, notice that since \( D(h^2)=D(h)h+hD(h) \), the condition \( D(h)h=0 \) is equivalent to \( hD(h)=0 \) under the assumption that \( h^2=0 \). Therefore
\begin{align*}
    D(1-h)
    &=
    (d+D(h))(1-h)-(1-h)d+\partial(1-h)
    \\
    &=
    d+D(h)-dh-D(h)h-d+hd-\partial(h)
    \\
    &= (D(h)-dh+hd-\partial(h))- D(h)h  =0
\end{align*}
and $D(1+h)$ follows similarly.\end{proof}

\begin{remark}
Assuming that \( h^2=0 \) and \( \partial(h)=0 \), the second  condition of the Clean-Up Lemma is equivalent to \( hdh=0 \), since \( hD(h)=hdh-h^2d+h\partial(h) \). 
\end{remark}

\begin{definition}
	Given a complex \( (\object,d) \), we define 
	\( h^{n}q^{m}(\object,d)\coloneqq (h^nq^m\object,(-1)^n d). \)\end{definition}
	$h^{n}q^{m}(\object,d)$ is a well-defined complex since the sign of \( \partial(d) \) in the \( d^2 \)-relation~\eqref{eqn:d2term} changes by \( (-1)^n \) under a shift of \( \object \) in homological grading. In the special case \( (\object,d)\in \Mod^\Rcomm \) we obtain the usual conventions for shifting chain complexes; see \cite[Definition~2.1]{Keller}, for example. 

\begin{lemma}\label{lem:shifting_morphism_spaces}
	The map
	\begin{align*}
	\sigma: h^{n'-n}q^{m'-m}\Mor((X,d),(X',d'))
	&\rightarrow
	\Mor(h^nq^m(X,d),h^{n'}q^{m'}(X',d'))
	\\
	h^{n'-n}q^{m'-m} (f\co (X,d)\rightarrow(X',d'))
	&\mapsto
	(-1)^{n\cdot h(f)}\cdot f
	\end{align*}
	is an isomorphism of bigraded chain complexes.
\end{lemma}

\begin{proof} This is a routine check left to the reader. 
\end{proof}
\section{Khovanov-theoretic invariants of links}\label{sec:KhLinksReview}


In order to set notation and establish conventions, we give an overview of the variants of Khovanov homology used in this paper. In this section we will mainly work over  a unital commutative ring \( \Rcomm \), but when we wish to highlight results that require it, we will denote a general field by \( \field \). We will reserve particular notation in two cases: denote by \( \fieldTwoElements \) the field of two elements and let \( \fieldTwoIsNotZero \) denote any field of characteristic \( \neq 2 \).

\subsection{Khovanov homology}\label{subsec:Kh}
We briefly recall the definition of Khovanov homology \cite{Khovanov}, mainly following the notation from  \cite{BarNatanKh}. Given  an \( n \)-crossing diagram \( \Diag \) of a link \( \Lk \), the 0- and 1-resolution of a crossing is defined as follows: 
\[ \text{(0-resolution) }\No \leftarrow \CrossingR \rightarrow\Ni\text{ (1-resolution)}\] 
Enumerating the crossings of \( \Diag \), one associates with each vertex of an \( n \)-dimensional cube \( v\in\{0,1\}^n \) the diagram \( \Diag(v) \) obtained by taking the \( v_i \)-resolution of the \( i^\text{th} \) crossing for each \( i=1,\dots, n \). Each diagram \( \Diag(v) \) consists of several disjoint, possibly nested, circles. If a vertex \( w\in\{0,1\}^n \) is obtained from a vertex \( v\in\{0,1\}^n \) by changing a single coordinate \( v_k \) of \( v \) from 0 to 1, we draw an arrow from \( \Diag(v) \) to \( \Diag(w) \). This arrow is labelled by the 2-dimensional cobordism between \( \Diag(v) \) and \( \Diag(w) \), which is equal to the saddle cobordism in a neighbourhood of the \( i^\text{th} \) crossing and agrees with the product cobordism outside this neighbourhood.
As a result, we obtain an \( n \)-dimensional cube whose vertices are labelled by 1-dimensional manifolds and whose edges are labelled by cobordims. To this topological picture, we now apply a TQFT that associates with each circle a free \( \Rcomm \)-module \( V=\langle\x_{+},\x_{-} \rangle_\Rcomm \), and with a disjoint union of circles \( \Diag(v) \) the tensor product of the corresponding \( \Rcomm \)-modules \( V_{\Diag(v)}=V^{\otimes |\Diag(v)|} \), where \( |\Diag(v)| \) is the number of components in the diagram \( \Diag(v) \). The TQFT associates maps to cobordisms via the following formulas:
\begin{align*}
\text{(split map)}\quad\PairsOfPantsR &\Delta\co\begin{cases}
\x_{+} \mapsto \x_{+} \otimes \x_{-} +\x_{-} \otimes \x_{+} & \\
\x_{-} \mapsto \x_{-} \otimes \x_{-}  & \end{cases}\\
\text{(merge map)}\quad\PairsOfPantsL &m\co\begin{cases}
\x_{+} \otimes \x_{-} \mapsto \x_{-} \quad & \x_{+} \otimes \x_{+} \mapsto \x_{+}\\
\x_{-} \otimes \x_{+} \mapsto \x_{-} \quad & \x_{-} \otimes \x_{-} \mapsto 0 \end{cases}
\end{align*}
These edge maps commute on each 2-dimensional face of the cube. So, in order to obtain a differential, we need to add signs to these edge maps such that the number of signs on the boundary of every face is odd. Such edge assignments exist; for example, with the notation as above, we could assign an edge from \( \Diag(v) \) to \( \Diag(w) \) the sign \( (-1)^{v_1+\cdots + v_{k-1}} \). 
Thus, \( \bigoplus_{v\in\{0,1\}^n} V_{\Diag(v)} \) together with these signed edge maps is a chain complex, which we denote by \( \CKh(\Diag) \). 

This chain complex carries various gradings that depend on the orientation of the link. Let \( n_+ \) denote the number of positive crossings \( \CrossingPosR \) and let \( n_- \) denote the number of negative crossings \( \CrossingNegR \) so that \( n=n_+ +n_- \). Up to a shift, the homological grading on the complex is given by \( |v| = \sum v_i \) (the height of a vertex). Let \( \xgen=(\x_{+})^{\otimes k}\otimes (\x_{-})^{\otimes \ell} \) be a generator in \( V_{\Diag(v)} \). Then the homological grading, the quantum grading and the \( \delta \)-grading of \( \xgen \) are defined as follows:
\[ h(\xgen) \coloneqq \vert v\vert - n_- \qquad q(\xgen)\coloneqq k - \ell + \vert v\vert +n_+-2n_- \qquad  \delta(\xgen)\coloneqq\tfrac{1}{2}(k-\ell+n_+ -\vert v\vert)\]
The quantum grading on the basic $\Rcomm$-module \( V=\langle \x_{+},\x_{-}\rangle \) is \( q(\x_{+})=+1, q(\x_{-})=-1 \). Note that \( \delta=\tfrac{1}{2}q-h \); the complex \( \CKh(\mathcal{D}) \) is bigraded. The differential of \( \CKh(\mathcal{D}) \) increases the homological grading by one, preserves the quantum grading, and decreases the \( \delta \)-grading by one.    
\begin{theorem}[Khovanov \cite{Khovanov}]\label{thm:KhIsALinkInvariant}
The homology of \( \CKh(\Diag) \) is a bigraded link invariant. In particular, it is independent of the choice of diagram and the choice of edge assignment. 
\end{theorem}
In other words, \( \CKh(\Diag) \) is a link invariant up to homotopy equivalence or, equivalently, as an isomorphism class of objects in \( \HomologyZero(\Mod^\Rcomm) \). The resulting homological link invariant \( \Homology(\CKh(\Diag)) \) is called Khovanov homology, and is denoted by \( \Kh(\Lk) \). 

\subsection{Bar-Natan's deformation}\label{subsec:BN-deformation}
There is a similar construction, deforming Khovanov homology, due to Bar-Natan  \cite[Section~9.3]{BarNatanKhT}.  The cube of resolutions is constructed as before, except that the ground ring \( \Rcomm \) is replaced by \( \Rcomm[H]\) with \(\gr(H)=h^0 q^{-2} \), and the differential is adjusted as follows:
\begin{equation}\label{eq:merge-split_maps_for_BN}
\begin{split}
\text{(split map)}\quad\PairsOfPantsR &\Delta\co\begin{cases}
\x_{+} \mapsto \x_{+} \otimes \x_{-} +\x_{-} \otimes \x_{+} \textcolor{blue}{- H\cdot \x_{+} \otimes \x_{+} }& \\
\x_{-} \mapsto \x_{-} \otimes \x_{-} & 
\end{cases}\\
\text{(merge map)}\quad\PairsOfPantsL &m\co\begin{cases}
\x_{+} \otimes \x_{-} \mapsto \x_{-} \quad & \x_{+} \otimes \x_{+} \mapsto \x_{+}\\
\x_{-} \otimes \x_{+} \mapsto \x_{-} \quad & \x_{-} \otimes \x_{-} \mapsto \textcolor{blue}{ H\cdot \x_{-}}
\end{cases}
\end{split}
\end{equation}
We write \( \CBN(\Diag) \) for the resulting bigraded chain complex over \( \Rcomm[H] \). By construction, it satisfies \( \CBN(\Diag)|_{H=0}=\CKh(\Diag)\otimes R[H] \). Moreover, analogously to Theorem~\ref{thm:KhIsALinkInvariant}, its homotopy type is a link invariant. We denote its homology, the Bar-Natan homology of $\Lk$, by \( \BN(\Lk) \).
We will sometimes indicate the $\Rcomm[H]$-action by a subscript \( \BN(\Lk)_{\Rcomm[H]} \).

By \( \fCBN(\Lk) \) we denote the complex which has the same underlying \( R \)-module as \( \CKh(\Lk) \), but the differential is given by setting \( H=1 \) in the maps~(\ref{eq:merge-split_maps_for_BN}). This complex is filtered (thus the ``\( f \)'' in the notation) with respect to the quantum grading filtration \( \mathcal{F}_q=\bigoplus_{j\geq q} \fCBN^j(\Lk) \), and, analogously to \( \CKh(\Lk) \), the \emph{filtered} homotopy equivalence class of \( \fCBN(\Lk) \) is a link invariant. The following is a very useful structural property of Bar-Natan's deformation. It was originally proved over \( \fieldTwoElements \) in \cite{Turner}, and then extended to general coefficients in \cite[Proposition 2.4]{MTV}, while the main idea dates back to \cite{Lee}. By \( \vert\Lk\vert \) we denote the number of components of the link.

\begin{proposition}\label{prop:H=1_is_trivial}
The homology of \( \fCBN(\Lk) \) is isomorphic to a free \( \Rcomm \)-module of rank \( 2^{\vert\Lk\vert} \).
\end{proposition}

A slightly different point of view on \( \CBN(\Lk) \) is to regard this deformed Khovanov chain complex as a bigraded type~D structure over \( \Rcomm[H] \). Namely, the underlying \( \Rcomm \)-module is the same as for \( \CKh(\Lk) \), but the differential is given by the maps (\ref{eq:merge-split_maps_for_BN});  we denote the resulting type~D structure by \( \CKh(\Lk)^{\Rcomm[H]} \). This is a link invariant up to  homotopy equivalence or, equivalently, as an object in \( \HomologyZero (\Mod^{\Rcomm[H]}) \) (the proof is analogous to the ones for \( \Kh(\Lk) \) and \( \BN(\Lk) \)). The relation to Bar-Natan homology is given by
\( \CBN(\Lk)_{\Rcomm[H]}=\CKh(\Lk)^{\Rcomm[H]}\boxtimes {}_{\Rcomm[H]}{\Rcomm[H]}_{\Rcomm[H]}. \)
Strictly speaking, the box tensor product \( \boxtimes \) was defined in \cite{LOT-main} only over \( \fieldTwoElements \). Since we work with signs, let us spell out what we mean by this operation. As a vector space 
\( \CKh(\Lk)^{\Rcomm[H]}\boxtimes {}_{\Rcomm[H]}{\Rcomm[H]}\) is defined as \( \CKh(\Lk) \otimes \Rcomm[H] \).
The differential consists of arrows \(  \xgen \otimes H^i \rightarrow r\cdot \ygen \otimes H^{i+\ell} \) for each arrow \( \xgen \xrightarrow{r\cdot H^\ell} \ygen \) in \( \CKh^{R[H]} \) and \( i\geq 0 \) (\( r \) is a coefficient in \( R \)). The action by \( R[H] \) is defined by \( (\xgen \otimes H^i)\cdot H^\ell = \xgen \otimes H^{i+\ell} \). It is now clear that \( \CKh(\Lk)^{\Rcomm[H]}\boxtimes {}_{\Rcomm[H]}{\Rcomm[H]}_{\Rcomm[H]} \) is equal to \( \CBN(\Lk)_{\Rcomm[H]} \).

Summarizing, above we introduced three different structures, which are link invariants up to the relevant notion of homotopy: a bigraded chain complex \( \CBN(\Lk) \) with an \( \Rcomm[H] \)-action, an \( h \)-graded \( q \)-filtered chain complex \( \fCBN(\Lk) \) over \( R \), and a bigraded type~D structure \( \CKh(\Lk)^{\Rcomm[H]} \) over \( R[H] \). The relationships between them is as follows:
\begin{equation*}
\label{eq:D_str_vs_BN_homology}
\CBN(\Lk)_{\Rcomm[H]}=\CKh(\Lk)^{\Rcomm[H]}\boxtimes {}_{\Rcomm[H]}{\Rcomm[H]}_{\Rcomm[H]} \qquad \fCBN(\Lk) = \CKh(\Lk)^{\Rcomm[H]}|_{H=1}
\end{equation*}
The second equality says that one obtains \( \fCBN(\Lk) \) from \( \CKh(\Lk)^{\Rcomm[H]} \) by simply setting \( H=1 \) in the differential of the type~D structure. We note that, as with the \( \cdot\boxtimes R[H] \)-operation, the operation of obtaining filtered chain complexes from type~D structures by setting \( H=1 \) respects homotopies. This means that the homotopy equivalence class of the type~D structure \( \CKh(\Lk) ^{\Rcomm[H]} \) determines the filtered homotopy equivalence class of the filtered chain complex \( \fCBN(\Lk) \). This is because type~D structure homomorphisms and homotopies translate into filtered chain maps and homotopies. 

If \( \Rcomm \) is replaced by a field \( \field \), one can use the Cancellation Lemma~\ref{lem:AbstractCancellation} and the Clean-Up Lemma~\ref{lem:AbstractCleanUp} to prove that up to grading shifts any bigraded type~D structure \( D^{\field[H]} \) is homotopy equivalent to a direct sum of copies of \( [\gen] \) (a type~D structure with no actions), and copies of \( [\gen \xrightarrow{H^\ell} \gen], \ \ell\geq 1 \). Call such type~D structures fully cancelled and cleaned-up, and denote the fully cancelled and cleaned-up type~D structure \( \CKh(\Lk)^{\field[H]} \) by \( \Kh(\Lk)^{\field[H]} \). As a vector space \( \Kh(\Lk)^{\field[H]} \) is isomorphic to \( \Kh(\Lk;\field) \), but there are extra actions in \( \Kh(\Lk)^{\field[H]} \) picking up powers of \( H \). For example, in \( \Kh(T(2,3);\fieldTwoIsNotZero) \) (see the right of Figure~\ref{fig:exa:Pairing:Trefoil:unreduced}) there is an action \( v_1\xrightarrow{H^2}v_2 \). The following is an immediate corollary of Proposition~\ref{prop:H=1_is_trivial}; note that we say that \( \xgen \in \BN(\Lk)_{\field[H]} \) is \( H \)-torsion if \( \xgen \cdot H^\ell=0 \) for sufficiently high powers \( \ell \).

\begin{proposition}\label{prop:towers}
\( \Kh(\Lk)^{\field[H]} \) consists of exactly \( 2^{\vert\Lk\vert} \) free  summands \( [\gen] \), with possible additional summands \( [\gen \xrightarrow{H^\ell} \gen ], \ \ell\geq 1 \). It follows that \( \BN(\Lk)_{\field[H]} \) consists of \( 2^{\vert\Lk\vert} \) towers \( \field[H] \) plus possible  \( H \)-torsion.
\end{proposition}

\subsection{Reduced versions}
In this subsection it will be convenient to set \( \x_{+}=1 \) and \( \x_{-}=x \), so that the map \( m \) in Khovanov homology corresponds to multiplication in the ring \( R[\x]/(\x^2) \). In Bar-Natan homology, the appropriate ring is \( \Rcomm[H,\x]/(\x^2=H \x) \). The split map, the unit \( 1 \), and the counit \( \epsilon(\x)=1, \epsilon(1)=0 \) promote these rings to Frobenius algebras. 
\subsubsection{Reduced Khovanov homology}
Fix a basepoint \( p \) on the link \( \Lk \) away from the crossings, turning \( \Lk \) into a \textbf{pointed} link. Then every resolution \( \Diag(v) \) contains a marked circle. The construction above can be modified by assigning \( V^{\x}=\langle \x \rangle_\Rcomm \) instead of \( V=\langle 1,\x \rangle_\Rcomm \) to this circle in each resolution, resulting in a modified vector space \( V^{\x}_{\Diag(v)}\) assigned to each resolution. 
The split and merge maps are unchanged, except when the circle with the basepoint is involved in the saddle cobordism, we make the following adjustment:
\begin{align*}
\begin{split}
\Delta\co V^{x}\rightarrow V^{x}\otimes V
=
&
\begin{cases}
\textcolor{darkgreen}{\x} \mapsto \textcolor{darkgreen}{\x} \otimes \x
\end{cases}
\\
m\co V^{x}\otimes V\rightarrow V^{x}
=
&
\begin{cases}
\textcolor{darkgreen}{\x} \otimes 1 
\mapsto \textcolor{darkgreen}{\x},
\quad &
\textcolor{darkgreen}{\x} \otimes \x \mapsto 0.
\end{cases}
\end{split}
\end{align*}
Denote the resulting \( \Rcomm \)-module by \( \CKh^{\x}(\Diag,p)=\bigoplus_{v\in\{0,1\}^n} V^\x_{\Diag(v)} \) and its homology by \( \Kh^{\x}(\Lk,p) \); there is a dependence on the choice of a component containing the point \( p \). The invariance of \( \Kh^{\x}(\Lk,p) \) under Reidemeister moves is proved in the same way as invariance of \( \Kh(\Lk) \). To prove invariance under moving the basepoint along the component, one only needs to be able to move the basepoint through a crossing in the diagram:
\[ \CrossingLBasepointRotated\longleftrightarrow\CrossingLBasepoint\quad\text{and}\quad\CrossingRBasepointRotated\longleftrightarrow\CrossingRBasepoint\]
Observe that we can consider \( (\Diag,p) \) lying not on the plane \( \mathbb{R}^2 \), but on the sphere \( S^2 \). This adds an additional move to the diagrams that does not change the Khovanov complex, namely,  sliding an outer part of \( (\Diag,p) \) through infinity. As a result, fixing the basepoint on the left of the vertical strand, we can move the vertical strand to the right all the way through \( S^2 \), until it returns on the left of the basepoint. Throughout this process we only do Reidemeister moves and one move through the point at infinity. Thus, \( \CKh^{\x}(\Diag,p) \) is independent of the position of the basepoint on the link component.


Observe that \( \CKh^{\x}(\Diag,p) \) is a subcomplex of \( \CKh(\Diag) \), and this gives rise to a quotient complex 
\[ \CKh^{\x=0}(\Diag,p)=\CKh(\Diag)/\CKh^{\x}(\Diag,p).\]
Equivalently, we can use the \( \Rcomm \)-module \( V^{\x=0}:=V/V^{\x}=\langle 1 \rangle_\Rcomm \) on circles containing the basepoint. The split and merge maps are modified in the following way:
\begin{align*}
\begin{split}
\Delta\co V^{x=0}\rightarrow V^{x=0}\otimes V
=
&
\begin{cases}
\textcolor{darkgreen}{1} \mapsto \textcolor{darkgreen}{1} \otimes \x
\end{cases}
\\
m\co V^{x=0}\otimes V\rightarrow V^{x=0}
=
&
\begin{cases}
\textcolor{darkgreen}{1} \otimes 1 \mapsto \textcolor{darkgreen}{1},
\ &
\textcolor{darkgreen}{1} \otimes \x \mapsto 0
\end{cases}
\end{split}
\end{align*}
After identifying \( \x\in V^{\x} \) with \( 1\in V^{\x=0} \), these maps agree with those in the definition of \( \CKh^{\x}(\Diag,p) \). So \( \CKh^{\x=0}(\Diag,p) \) is identical to \( q^{2}h^{0}\CKh^{\x}(\Diag,p) \). We denote its homology by \( \Kh^{\x=0}(\Lk,p) \). To unify notation, we define the \textbf{reduced Khovanov homology of a link \( \Lk \)} by 
\[ \Khr(\Lk,p):=q^{-1}h^0\Kh^{\x=0}(\Lk,p)=q^{1}h^0\Kh^{\x}(\Lk,p).\]
The grading shifts are chosen such that \( \Khr \) of the unknot is supported in bigrading \( q^0h^{0} \).
The difference between the two definitions of \( \CKhr(\Lk,p) \) is that in one case, we view it as a subcomplex of \( \Kh(\Lk,p) \) and fix \( \x \) on the special circle. In the other case, we view \( \CKhr(\Lk,p) \) as a quotient complex and fix \( 1 \) on the special circle.

\subsubsection{Reduced Bar-Natan homology} \label{sec:BNr}
The reduced version of Bar-Natan homology is more complicated, which is ultimately a consequence of the fact that the polynomial \( \x^2 \) has one root, whereas the polynomial \( \x^2-H\x \) has two. Nevertheless, there is a well-defined reduced Bar-Natan homology; the construction is described below.


The starting point is that now the two reduced versions have different split/merge maps near the basepoint:
\begin{align}
\begin{split}\label{eq:merge-split_maps_reduced_BN_x}
\Delta\co V^{x}\rightarrow V^{x}\otimes V
=
&
\begin{cases}
\textcolor{darkgreen}{\x} \mapsto \textcolor{darkgreen}{\x} \otimes \x
\end{cases}
\\
m\co V^{x}\otimes V\rightarrow V^{x}
=
&
\begin{cases}
\textcolor{darkgreen}{\x}\otimes 1 \mapsto \textcolor{darkgreen}{\x},
\quad &
\textcolor{darkgreen}{\x} \otimes \x \mapsto H \textcolor{darkgreen}{\x}.
\end{cases}
\end{split}
\\
\begin{split}\label{eq:merge-split_maps_reduced_BN_x=0}
\Delta\co V^{x=0}\rightarrow V^{x=0}\otimes V
=
&
\begin{cases}
\textcolor{red}{1} \mapsto \textcolor{red}{1} \otimes \x - H \cdot \textcolor{red}{1} \otimes 1
\end{cases}
\\ 
m\co V^{x=0}\otimes V\rightarrow V^{x=0}
=
&
\begin{cases}
\textcolor{red}{1} \otimes 1 \mapsto \textcolor{red}{1},
\quad &
\textcolor{red}{1} \otimes \x \mapsto 0.
\end{cases}
\end{split}
\end{align}
For both theories, we can prove invariance under moving the basepoint via the sliding through infinity trick. Thus, at first glance, it seems that there are two different pointed link invariants: \( \BN^\x(\Lk,p) \) and \( \BN^{\x=0}(\Lk,p) \). It turns out that they are isomorphic up to a grading shift, which is proved in \cite[Theorem~4]{Wigderson} for \( \Rcomm=\fieldTwoElements \). Let us describe a shorter proof of this fact for any ring \( \Rcomm \).

First, there are two additional reduced versions, which one obtains by considering either \( V^{\x-H}=\langle \x-H \rangle_{\Rcomm[H]} \) or \( V^{x=H}=V/V^{\x-H} \) on the circles with basepoint. \( \CBN^{\x-H}(\Lk,p) \) is again a subcomplex of \( \CBN(\Lk) \), and we have the following modifications of the split/merge maps near the point \( p \):
\begin{align}
\begin{split}\label{eq:merge-split_maps_reduced_BN_x-H}
\Delta\co V^{x-H}\rightarrow V^{x-H}\otimes V
=
&
\begin{cases}
\textcolor{red}{(\x-H)} \mapsto \textcolor{red}{(\x-H)} \otimes \x - H \cdot \textcolor{red}{(\x-H)} \otimes 1
\end{cases}
\\
m\co V^{x-H}\otimes V\rightarrow V^{x-H}
=
&
\begin{cases}
\textcolor{red}{(\x-H)} \otimes 1 \mapsto\textcolor{red}{(\x-H)},
\quad &
\textcolor{red}{(\x-H)} \otimes \x \mapsto  0.
\end{cases}
\end{split}
\\
\begin{split}\label{eq:merge-split_maps_reduced_BN_x=H} 
\Delta\co V^{x=H}\rightarrow V^{x=H}\otimes V
=
&
\begin{cases}
\textcolor{darkgreen}{1} \mapsto \textcolor{darkgreen}{1}  \otimes x
\end{cases}
\\ 
m\co V^{x=H}\otimes V\rightarrow V^{x=H}
=
&
\begin{cases}
\textcolor{darkgreen}{1} \otimes 1 \mapsto \textcolor{darkgreen}{1},
\quad &
\textcolor{darkgreen}{1} \otimes \x \mapsto H \cdot \textcolor{darkgreen}{1}.
\end{cases}
\end{split}
\end{align}
By directly comparing the modified differentials, observe now that 
\begin{align*}
    \CBN^{\x}(\Lk,p) = q^{-2}h^0\CBN^{\x=H}(\Lk,p), \\
    \CBN^{\x-H}(\Lk,p) = q^{-2}h^0\CBN^{\x=0}(\Lk,p)
\end{align*}
as chain complexes. So we still have two candidates for the reduced Bar-Natan homology. 
\begin{lemma}\label{lem:OnlyOneReducedBNTheory}
$\CBN^{\x}(\Lk,p)$ and $\CBN^{\x-H}(\Lk,p)$ are isomorphic as bigraded chain complexes.
\end{lemma}

\begin{proof}
First, let us change the basis in the complex $\CBN^{\x}(\Lk,p)$: let us keep the variable $\textcolor{darkgreen}{\x}$ on the marked component, but on all other components let us use the basis $\langle y=x-H, 1 \rangle$. In this basis, the merge and split maps look as follows:
\begin{align}
\begin{split}\label{eq:merge-split_maps_reduced_BN_x-H_newbasis}
\text{(split) } &\Delta\co V\rightarrow V\otimes V = \begin{cases}
1 \mapsto 1 \otimes \y +\y \otimes 1 \textcolor{blue}{+ H\cdot 1 \otimes 1 }& \\
\y \mapsto \y \otimes \y & 
\end{cases}\\
\text{(merge) } &m\co V\otimes V\rightarrow V =\begin{cases}
1 \otimes \y \mapsto \y \quad & 1 \otimes 1 \mapsto 1\\
\y \otimes 1 \mapsto \y \quad & \y \otimes \y \mapsto \textcolor{blue}{ -H\cdot \y}
\end{cases} 
\end{split}
\\
\begin{split}
\text{(split near} \ast \mspace{-10mu}\Circle \text{) } &\Delta\co V^{x}\rightarrow V^{x}\otimes V
=
\begin{cases}
\textcolor{darkgreen}{\x} \mapsto \textcolor{darkgreen}{\x} \otimes \y+ H \cdot \textcolor{darkgreen}{\x} \otimes 1
\end{cases}
\\
\text{(merge near} \ast \mspace{-10mu}\Circle \text{) } &m\co V^{x}\otimes V\rightarrow V^{x}
=
\begin{cases}
\textcolor{darkgreen}{\x}\otimes 1 \mapsto \textcolor{darkgreen}{\x},
\quad &
\textcolor{darkgreen}{\x} \otimes \y \mapsto 0.
\end{cases}
\end{split}
\end{align}
We now compare the above maps with the maps of  $\CBN^{\x-H}(\Lk,p)$ (Formulas~\eqref{eq:merge-split_maps_for_BN} and~\eqref{eq:merge-split_maps_reduced_BN_x-H}). Sending $\textcolor{darkgreen}{\x} \mapsto \textcolor{red}{\x-H}, \y \mapsto \x, 1 \mapsto 1$, we see that the merge/split maps almost coincide; the only difference is that all the $H$-terms have different signs. To account for that, consider a splitting along the parity of the quantum grading: 
$$\CBN^{\x}=\CBN^{\x}_{q\equiv 1(4)} \oplus \CBN^{\x}_{q\equiv3(4)}, \qquad \CBN^{\x-H}=\CBN^{x-H}_{q\equiv3(4)} \oplus \CBN^{x-H}_{q\equiv3(4)}.$$
We now tweak the isomorphism as follows:
\begin{align*}
f\co\CBN^{\x}_{q\equiv 1(4)}  &\rightarrow \CBN^{\x-H}_{q\equiv 1(4)}   & f\co\CBN^{\x}_{q\equiv3(4)}  &\rightarrow \CBN^{\x-H}_{q\equiv3(4)}  \\
\textcolor{darkgreen}{\x} &\mapsto \textcolor{red}{\x-H} & \textcolor{darkgreen}{\x} &\mapsto \textcolor{red}{-(\x-H)} \\
\y &\mapsto \x & \y &\mapsto -\x \\
1 &\mapsto 1 & 1 &\mapsto -1 
\end{align*}
The above defines an isomorphism of chain complexes, since in both complexes the differentials restricted to $\CBN_{q\equiv 1(4)}$ (or $\CBN_{q\equiv3(4)}$) all have even powers of $H$, while the differentials between $\CBN_{q\equiv 1(4)}$ and $\CBN_{q\equiv3(4)}$ have odd powers of $H$.
\end{proof}
This proves that the two candidates for the reduced version of Bar-Natan homology are the same. So we can now define the \textbf{reduced Bar-Natan homology of a pointed link \( (\Lk,p) \)} as  \[ \BNr(\Lk,p):=q^{-1}h^0\BN^{\x=0}(\Lk,p)=q^{1}h^0\BN^{\x}(\Lk,p).\]
Later, in Theorem~\ref{theo:BNr_is_a_link_invariant},
we will prove that, as a bigraded $R$-module,  $\BNr(\Lk, p)$ does not depend on the basepoint $p$; moreover, we will endow $\BNr(\Lk)$ with an $R[H_1,\ldots,H_\ell]$-module structure, where $\ell$ is the number of link components. For us only one $H$-action will be important, and thus we will always work with pointed links, even though we will usually just write $\Lk$ instead of $(\Lk, p)$.

As in the unreduced case, there is a reduced type~D structure invariant \( \CKhr(\Lk)^{\Rcomm[H]} \). The proof that the two versions \( \CKh^x(\Lk)^{\Rcomm[H]}\) and \(\CKh^{x-H}(\Lk)^{\Rcomm[H]} \) coincide adapts: in Lemma~\ref{lem:OnlyOneReducedBNTheory} instead of the change of basis $\langle y=x-H, 1 \rangle$ one has to apply the Clean-Up Lemma~\ref{lem:AbstractCleanUp} for the morphisms $h\co x \xrightarrow{-H} 1$ on each non-marked circle. This will change the type~D structure \( \CKh^x(\Lk)^{\Rcomm[H]} \), preserving its chain homotopy type, to a type~D structure that is isomorphic to \( \CKh^{x-H}(\Lk)^{\Rcomm[H]} \). The invariance of \( \CKhr(\Lk)^{\Rcomm[H]} \) up to chain homotopy is proved just as for \( \BNr(\Lk,p) \). The relationship to reduced Bar-Natan homology is, as before, 
\begin{equation}
\label{eq:red_D_str_vs_BN_homology}
\CBNr(\Lk)_{\Rcomm[H]}=\CKhr(\Lk)^{\Rcomm[H]}\boxtimes {}_{\Rcomm[H]}{\Rcomm[H]}_{\Rcomm[H]}.
\end{equation}

The corresponding reduced filtered theory can be defined simply from the type~D structure invariant: \( \fCBNr(\Lk)\coloneqq \CKhr(\Lk)^{\Rcomm[H]}|_{H=1} \). The results of \cite{Turner} and \cite{MTV} hold in the reduced case too. The proof is a standard adaptation of the unreduced case.
\begin{proposition}\label{prop:towersReduced}
The homology of \( \fCBNr(\Lk) \) is isomorphic to a free \( \Rcomm \)-module of rank \( 2^{\vert\Lk\vert-1} \).
\end{proposition}

For $\Rcomm=\field$, we denote the fully cancelled and cleaned-up \( \CKhr(\Lk)^{\field[H]} \) by \( \Khr(\Lk)^{\field[H]} \), as in the unreduced case. As a vector space,  \( \Khr(\Lk)^{\field[H]} \) is isomorphic to \( \Khr(\Lk;\field) \). The following is an immediate consequence of Proposition~\ref{prop:towersReduced}.

\begin{proposition}\label{prop:towersReduced-field}
\( \Khr(\Lk)^{\field[H]} \) consists of exactly \( 2^{\vert\Lk\vert-1} \) free  summands \( [\gen] \) with possible additional summands 
\(
\begin{tikzcd}[nodes={inner sep=2pt}, column sep=13pt,ampersand replacement=\&]
[\gen
\arrow{r}{H^\ell} 
\&
\gen]
\end{tikzcd}
\), \(\ell\geq 1 \). It follows that \( \BNr(\Lk)_{\field[H]} \) consists of \( 2^{\vert\Lk\vert-1} \) towers \( \field[H] \) plus possible \( H \)-torsion.
\end{proposition}

\subsection{\texorpdfstring{The \(s\)-invariant}{The s-invariant}}

We now specialize to fields \( \Rcomm=\field \) and to knots \( \Lk=\Knot \) in order to discuss Rasmussen's \( s \)-invariant \( s^\field(\Knot) \). This was originally defined over \( \Q \) via the Lee spectral sequence \cite{Rasmussen_slice_genus}, and then extended to other fields via the Bar-Natan spectral sequence \( \Kh(\Knot) \rightrightarrows \Homology (\fCBN(\Knot))= \field \oplus \field \) \cite{MTV}. Over \( \Q \) the equivalence of the two definitions from Lee and Bar-Natan spectral sequences follows from \cite[Proposition 3.1]{MTV}. 

Write $C$ for the filtered complex $\fCBN(\Knot;\field)$ and $\{\mathcal{F}_q\}$ for the quantum filtration on $C$; recall that  \(
\mathcal{F}_q\coloneqq
\{
x\in C \mid q(x)\geq q
\}
\). Rasmussen's $s$-invariant is defined $s^\field\coloneqq s_{\min}^{\field}(\Knot)+1 = s_{\max}^{\field}(\Knot)-1$
where
\begin{align*}
s_{\min}^{\field}(\Knot) 
&\coloneqq 
\max
\{
q \mid 
\Homology(\mathcal{F}_q)
\longrightarrow
\Homology(C)=\field^2 
\text{ is surjective} 
\}
\\
s_{\max}^{\field}(\Knot)
&\coloneqq
\max
\{
q \mid
\Homology(\mathcal{F}_q)
\longrightarrow
\Homology(C)=\field^2 
\text{ is nonzero} 
\}
\end{align*}
These definitions, taken from \cite[Theorem~2.5]{LipSar}, are equivalent to those using the $s$-grading~\cite{Rasmussen_slice_genus}. 
We can also define a reduced $s$-invariant $\tilde{s}^\field(\Knot)$, by considering the filtration $\{\widetilde{\mathcal{F}}_q\}$ on $\widetilde{C}\coloneqq\fCBNr(\Knot;\field)=q^{+1}\fCBN^x(\Knot;\field)$, and setting
$$
\tilde{s}^{\field}(\Knot)
\coloneqq
\max
\{
q \mid
\Homology(\widetilde{\mathcal{F}}_q)
\longrightarrow
\Homology(\widetilde{C})=\field
\text{ is surjective/nonzero} 
\}
$$
\begin{proposition}
	For any knot \(\Knot\), \(\tilde{s}^{\field}(\Knot)=s^{\field}(\Knot)\).
\end{proposition}

\begin{proof}
	First, observe that, as a homologically graded chain complex, $C$ can be written as a direct sum of $C^x\coloneqq\fCBN^x(\Knot;\field)$ and \(C^{x-1}\coloneqq\fCBN^{x-1}(\Knot;\field)\). These correspond to the subcomplexes of $C$ generated by those elements which label the circle with the basepoint $p$ by $x$ and $x-1$, respectively. Define filtrations $\mathcal{F}^x_q\coloneqq  \mathcal{F}_q\cap C^{x}$ and $\mathcal{F}^{x-1}_q\coloneqq \mathcal{F}_q\cap C^{x-1}$ on $C^x$ and $C^{x-1}$. 
	Observe now that $(C^x,\mathcal{F}^x_q)$ and $(C^{x-1},\mathcal{F}^{x-1}_{q})$ are isomorphic as filtered chain complexes (the isomorphism is given by setting $H=1$ in the proof of Lemma~\ref{lem:OnlyOneReducedBNTheory}).
  Moreover,
	$$
	\mathcal{F}_{q+2}
	\subseteq 
	\mathcal{F}^{x-1}_{q}\oplus \mathcal{F}^x_{q}
	\subseteq
	\mathcal{F}_{q}
	$$
	which can be seen by considering the bases on the circle with the basepoint $p$.
	Together with the previous observation, this gives rise to the following diagram of homologically graded chain complexes:
	$$
	\begin{tikzcd}
		\mathcal{F}_{q+2}
		\arrow{d}{i_{q+2}}
		\arrow[hookrightarrow]{r}{}
		&
		\mathcal{F}^{x}_{q}\oplus \mathcal{F}^x_{q}
		\arrow{d}{i^x_q\oplus i^x_q}
		\arrow[hookrightarrow]{r}{}
		&
		\mathcal{F}_{q}
		\arrow{d}{i_q}
		\\
		C
		\arrow[phantom]{r}{\cong}
		&
		C^x\oplus C^x
		\arrow[phantom]{r}{\cong}
		&
		C
	\end{tikzcd}
	$$
	So, in particular,
	$$
	\Homology(i_{q+2}) \text{ is non-zero} 
	\quad\Longrightarrow\quad
	\Homology(i^x_{q}) \text{ is non-zero/surjective} 
	\quad\Longrightarrow\quad
	\Homology(i_{q}) \text{ is surjective} 
	$$
	Thus
	$$
	s^{\field}(\Knot)-1
	=
	s_{\max}^{\field}(\Knot)-2
	\leq
	\tilde{s}^{\field}(\Knot)-1
	\leq 
	s_{\min}^{\field}(\Knot)
	=
	s^{\field}(\Knot)-1
	$$
	proving the claim.
\end{proof}

The following reformulation of the \( s \)-invariant, first observed in \cite{Khovanov_Frob_Ext}, will be of most use:
\begin{proposition}\label{prop:s-inv-as-grading}
For a knot \( \Knot \), the generator of the single tower \( \field[H] \) in \( \BNr(\Knot)_{\field[H]} \), and the generator of the single free summand \( [\gen] \) in \( \Khr(\Knot)^{\field[H]} \), both have \( q \)-grading \( s^\field(\Knot) \) and \( h\)-grading $0$.
\end{proposition}
\begin{proof}
Observe, first, that the fact that the \( q \)-gradings of the two specified generators are equal follows immediately from Equation~\eqref{eq:red_D_str_vs_BN_homology}. 
Second, the relationship between type~D structures and filtered chain complexes implies that the type~D structure \( \Khr(\Knot) ^{\field[H]} \) determines the spectral sequence \( \Khr(\Knot)  \rightrightarrows \Homology (\fCBNr(\Knot)) = \field  \) in such a way that the actions 
\(
\begin{tikzcd}[nodes={inner sep=2pt}, column sep=13pt,ampersand replacement=\&]
\gen
\arrow{r}{H^{\ell}} 
\&
\gen
\end{tikzcd}
\)
translate into differentials appearing on the \( \ell \)'th page of the spectral sequence. This immediately implies: 
\[ 
\Khr(\Knot)^{\field[H]}  = q^{s^{\field}(\Knot)}h^0[\gen] \bigoplus_i 
\begin{tikzcd}[nodes={inner sep=2pt}, column sep=13pt,ampersand replacement=\&]
[\gen
\arrow{r}{H^{\ell_i}} 
\&
\gen]_i
\end{tikzcd} \qquad
\BNr(\Knot)_{\field[H]} =q^{s^{\field}(\Knot)}h^0 \field [H] \bigoplus_i \left[\field[H]/(H^{\ell_i}) \right]_i
\qedhere\]
\end{proof}

The analogue of the proposition above in the unreduced case implies:
\begin{align*}
\Kh(\Knot) ^{\field[H]}&=q^{s_\text{min}^{\field}(\Knot)}h^0[\gen]\oplus q^{s_\text{max}^{\field}(\Knot)} h^0[\gen] \bigoplus_i 
\begin{tikzcd}[nodes={inner sep=2pt}, column sep=13pt,ampersand replacement=\&]
[\gen
\arrow{r}{H^{\ell_i}} 
\&
\gen]_i
\end{tikzcd} \\
\BN(\Knot)_{\field[H]}&=q^{s_\text{min}^{\field}(\Knot)}h^0 \field [H]\oplus q^{s_\text{max}^{\field}(\Knot)} h^0 \field [H] \bigoplus_i \left[\field[H]/(H^{\ell_i}) \right]_i
\end{align*}

We highlight that the collection of invariants \( \{s^\field\} \) ranging over all possible choices of coefficient field \( \field \) is an interesting collection of invariants that deserves further study. In particular, Seed has observed that there exist examples of knots \( \Knot \) (eg \( \Knot 14n19265 \)) for which \( s^\fieldTwoElements(\Knot)\ne s^\Q(\Knot) \) \cite[Remark 6.1]{LipSar}; see also Section~\ref{subsec:field_dependence}.

\subsection{Exact sequences} We finish with a discussion of two exact sequences, one over \( \Rcomm \) and the other over {\( \fieldTwoIsNotZero \), which connect the three main homological invariants discussed so far: \( \BNr(\Lk) \), \( \Khr(\Lk) \), and \( \Kh(\Lk) \).
\begin{proposition}\label{prop:kh_mapping_cones}
There are two short exact sequences
\begin{gather*}
\begin{tikzcd}[ampersand replacement=\&]
0
\arrow{r}{}
\&
q^{-2}h^{0} \CBNr(\Lk)
\arrow{r}{H}
\&
\CBNr(\Lk)
\arrow{r}{}
\&
\CKhr(\Lk)
\arrow{r}{}
\&
0
\end{tikzcd}
\\
\begin{tikzcd}[ampersand replacement=\&]
0
\arrow{r}{}
\&
q^{-4}h^{0}\CBNr(\Lk;\fieldTwoIsNotZero) 
\arrow{r}{H^2}
\&
\CBNr(\Lk;\fieldTwoIsNotZero)
\arrow{r}{}
\&
q^{-1}h^{0}\CKh(\Lk;\fieldTwoIsNotZero)
\arrow{r}{}
\&
0
\end{tikzcd}
\end{gather*}
resulting in the mapping cone formulas
\begin{align*}
\CKhr(\Lk)&\simeq
\Big[\begin{tikzcd}[ampersand replacement=\&]
q^{-2}h^{-1}\CBNr(\Lk) 
\arrow{r}{H}
\&
\CBNr(\Lk)
\end{tikzcd}\Big]\\
\CKh(\Lk;\fieldTwoIsNotZero)&\simeq 
\Big[\begin{tikzcd}[ampersand replacement=\&]
q^{-3}h^{-1}\CBNr(\Lk;\fieldTwoIsNotZero)
\arrow{r}{H^2}
\&
q^{1}h^{0}\CBNr(\Lk;\fieldTwoIsNotZero)
\end{tikzcd}\Big]
\end{align*}
\end{proposition}

\begin{proof}
We start with the first long exact sequence. Observe that there is an isomorphism of chain complexes
\[  \CBNr(\Lk) / \left( H \cdot q^{-2}h^{0}\CBNr(\Lk) \right) \cong \CBNr_{H=0}(\Lk),\] 
where \( \CBNr_{H=0}(\Lk) \) is the reduced Bar-Natan complex over the quotient Frobenius algebra \( \Rcomm[H,\x]/(\x^2=H \x, H=0)=\Rcomm[\x]/(\x^2) \), which is by definition equal to \( \CKhr(\Lk) \).
For the second long exact sequence, we observe in the same fashion that there is an isomorphism of chain complexes
\[ \CBNr(\Lk;\fieldTwoIsNotZero) /  \left( H^2 \cdot q^{-4}h^{0}\CBNr(\Lk;\fieldTwoIsNotZero) \right) \cong \CBNr_{H^2=0}(\Lk;\fieldTwoIsNotZero),\] 
where \( \CBNr_{H^2=0}(\Lk;\fieldTwoIsNotZero) \) is the reduced Bar-Natan complex over the quotient Frobenius algebra \( \fieldTwoIsNotZero[H,\x]/(\x^2=H \x, H^2=0)=\langle 1,H,\x,H\x \rangle_{\Rcomm} \). 

Observe that there is a convenient change of basis in this Frobenius algebra: \( \y=\x-\frac{H}{2} \) (here we use that \( 2\neq0 \) in \( \fieldTwoIsNotZero \)). With respect to this basis, the merge map is multiplication in \( \fieldTwoIsNotZero[H,\y]/(\y^2=0 , H^2=0)=\langle 1,H,\y,H\y \rangle_{\Rcomm} \), and the split map is the  co-multiplication \( 1 \mapsto 1 \otimes \y + \y \otimes 1, \  \y \mapsto \y \otimes \y \).  Motivated by this change of basis, we find a special basis for \( \CBNr_{H^2=0}(\Lk;\fieldTwoIsNotZero) \), with respect to which the complex is isomorphic to the unreduced complex \( \CKh(\Lk;\fieldTwoIsNotZero) \). We will work with the \( q^{-1}h^{0}\CBNr_{H^2=0}=\CBN^\x_{H^2=0} \) version of the reduced Bar-Natan complex, and change its basis by \( \y=\x-\frac{H}{2} \). Applying the TQFT in each full resolution, the circles with basepoint \( \ast \mspace{-7mu}\Circle \) are mapped to \(  \langle \x\rangle_{\fieldTwoIsNotZero[H]/(H^2)} \cong \langle \y+\frac{H}{2} \rangle_{\fieldTwoIsNotZero[H]/(H^2)} \), and the circles without the basepoint \( \Circle \) are mapped to \(  \langle 1, \x \rangle_{\fieldTwoIsNotZero[H]/(H^2)} \cong \langle 1, \y+ \frac{H}{2} \rangle_{\fieldTwoIsNotZero[H]/(H^2)} \cong \langle 1, \y \rangle_{\fieldTwoIsNotZero[H]/(H^2)}  \). We now stop working over \( \fieldTwoIsNotZero[H]/(H^2) \) and pick a basis over \( \fieldTwoIsNotZero \): circles \( \ast \mspace{-7mu}\Circle \) are mapped to \( \langle \y+\frac{H}{2}, \y\frac{H}{2} \rangle_\fieldTwoIsNotZero \) (note that \( \y\frac{H}{2}= \frac{H}{2} (\y+\frac{H}{2}) \)), and circles \( \Circle \) are mapped to \( \langle 1, \y\rangle_\fieldTwoIsNotZero \). 
Finally, in the basis above, the isomorphism 
\[ q^2 h^0 \CBN^{\x(=\y+\frac{H}{2})}_{H^2=0} (\Lk;\fieldTwoIsNotZero) \cong \CKh(\Lk;\fieldTwoIsNotZero)\]
is established via the map \( \y+\frac{H}{2} \mapsto \x_+, \y\frac{H}{2} \mapsto \x_- \) on \( \ast \mspace{-7mu}\Circle \) and \( y\mapsto \x_-, 1\mapsto \x_+ \) on $\Circle$.
\end{proof}

\begin{corollary}
Over \( R=\field \) we have 
\begin{equation} \label{mcformula:full_kh}
\CKh(\Lk;\field) \simeq
    \left[ 
    \begin{tikzcd}[ampersand replacement=\& ]
        q^{-3}h^{-1}\CBNr(\Lk;\field) \arrow{r}{H} \& q^{-1}h^{0}\CBNr(\Lk;\field) \\
        q^{-1}h^{-1}\CBNr(\Lk;\field)  \arrow{ru}{2} \arrow{r}{H} \&  q^{1}h^{0} \CBNr(\Lk;\field) 
    \end{tikzcd}
    \right]
\end{equation}
\end{corollary}
\begin{proof}
Over  \( \fieldTwoIsNotZero\)  this follows from the second part of Proposition~\ref{prop:kh_mapping_cones}, after cancelling the  \(\times 2\)  differential by Lemma~\ref{lem:AbstractCancellation}. Over  \(\fieldTwoElements\)  (and other fields of characteristic 2) the right hand side splits into two  \(\CKhr(\Lk;\fieldTwoElements)\), thanks to the first part of Proposition~\ref{prop:kh_mapping_cones}. The statement now follows from the well-known fact that  \(\CKh(\Lk;\fieldTwoElements)\)  also splits into two summands, both isomorphic to  \(\CKhr(\Lk;\fieldTwoElements)\) \cite{Shum_torsion}.
\end{proof}
\begin{remark}
It is possible that the formula (\ref{mcformula:full_kh}) is true over $R$. After passing to homology in the rows of the right hand side complex in (\ref{mcformula:full_kh}), we obtain two copies of reduced Khovanov homology, connected by a differential which is 0 over  \(\fieldTwoElements \). We think that this differential actually agrees with (or is homotopic to) the connecting homomorphism of the short exact sequence 
\[ 0\rightarrow q^{-1}h^{0}\CKhr(\Lk) \rightarrow \CKh(\Lk) \rightarrow q^{+1}h^{0}\CKhr(\Lk) \rightarrow 0\]
but we were not able to find an appropriate change of basis.
\end{remark}

\section{Bar-Natan's universal cobordism category and tangle invariants}\label{sec:tangle_invariants}
We review Bar-Natan's cobordism category for Khovanov homology of tangles. Our conventions aim to stay as close to those of~\cite{BarNatanKh,BarNatanKhT,BarNatanBurgosSoto} as possible.



\begin{definition}\label{def:Cob}\cite[Definition~3.1, Section 4.1.2 and Section~11.3]{BarNatanKhT}
Given a finite, non-empty set \( B \) of points on the boundary of a disc \( D^2 \), let \( \Cob:=\Cob^3(B) \) be the following category:
\begin{itemize}
    \item The objects of \( \Cob \) are crossingless tangle diagrams \( T \) whose boundary \( \partial T \) coincides with \( B \). 

    \item A morphism \( T_0\rightarrow T_1 \) between two objects \( T_0 \) and \( T_1 \) of \( \Cob \) is a formal \( \Rcomm \)-linear combination of cobordisms from \( T_0 \) to \( T_1 \). By cobordism, we mean an orientable (possibly disconnected) surface with boundary, together with an identification of the boundary with
    \[ (T_0\times\{0\})\cup (B\times I)\cup (T_1\times\{1\}),
    \]
    which can be thought of as the subset of the boundary of a solid cylinder:
    \[ 
    (D^2\times\{0\})\cup (S^1\times I)\cup (D^2\times\{1\})=\partial(D^2\times I).
    \]
    We consider cobordisms up to 
    boundary preserving homeomorphism.
    \item The identity morphisms are given by product cobordisms; composition of morphisms in \( \Cob \) is given by concatenation.
\end{itemize}
We define the quantum grading of a cobordism \( C \) representing a morphism in \( \Cob \) by:
\[ q(C):=\chi(C)-\tfrac{1}{2}\vert B\vert\]
This turns \( \Cob \) into a graded category. 
\end{definition}
\begin{remark}
  Bar-Natan allows \( B \) to be the empty set. Since our focus lies on reduced theory, we will ignore this case. Moreover, he considers cobordisms which are embedded into the cylinder \( (D^2\times I,B\times I) \). The definition above corresponds to the {\em abstract} cobordisms from~\cite[Section~11.3]{BarNatanKhT}; see \cite[Section~7.5]{Naot} for an explanation of why the abstract cobordism theory is equivalent to the embedded cobordism theory.  Finally, the quantum grading of a cobordism is what Bar-Natan calls the degree of a cobordism. Note that the quantum grading of any product cobordism vanishes. 
\end{remark}
\begin{definition}\label{def:Cobl}
The category \( \Cob_{/l} \) is defined as a quotient of \( \Cob \) by imposing the following local relations on the morphisms of \( \Cob \): 
\begin{description}
    \item[\( S \)-relation] Whenever a cobordism contains a component which is a sphere, the cobordism is set equal to 0; in short: \[ \Sphere=0\]
    \item[\( T \)-relation] Whenever a cobordism contains a component which is a torus, the cobordism is equal to twice the same cobordism but with this torus component removed; in short: \[ \Torus=2\]
    \item[\( 4Tu \)-relation] Given a cobordism \( C \), let us consider four embedded open discs \( D_1 \) through \( D_4 \) on \( C \). 
    For \( i\neq j \), let \( C_{ij} \) denote the cobordism obtains from \( C \) by removing the discs \( D_i \) and \( D_j \) and replacing them by a tube with the same boundary. Then \[ C_{12}+C_{34}=C_{13}+C_{24}\] 
    or in pictures: 
    \[ \TuT+\TuB=\TuR +\TuL\]
\end{description}
The three relations are homogeneous with respect to the quantum grading on \( \Cob \), so the quantum grading descends to \( \Cob_{/l} \). 
\end{definition}


\begin{observation}
	The relations \( 4Tu \) and \( S \) imply the relation \( T \), unless the torus in the \( T \)-relation is the only component of the cobordism. This is because we can regard the torus as a sphere with a 1-handle and apply the \( 4Tu \)-relation:
	\[ \TorusRelI=\TorusRelIII+\TorusRelIV-\TorusRelII\]
	The third term is equal to 0 by the \( S \)-relation. Each of the first two terms is equal to the cobordism on the left without the torus. See also \cite{Tanaka}.
\end{observation}

\begin{definition}\label{def:Kob}
Let us endow \( \Cob \) and \( \Cob_{/l} \) with a homological grading which is constantly equal to 0. We then define \( \Kob \) and \( \Kob_{/l} \) as \( \Cx(\Mat(\Cob)) \) and \( \Cx(\Mat(\Cob_{/l})) \), respectively.
\end{definition}

Given an oriented tangle diagram \( T \), Bar-Natan defines a chain complex \( \KhT{T} \in \Kob \) as follows.
Let \( \mathcal{D} \) be an oriented diagram for a tangle \( T \). Let \( n_+ \) be the number of all positive crossings \( \CrossingPosR \), \( n_- \) the number of all negative crossings \( \CrossingNegR \) and \( n=n_++n_- \) the total number of crossings in the diagram \( \mathcal{D} \). The 0- and 1-resolution of a crossing is defined by
    \[ \text{(0-resolution) }\No \leftarrow \CrossingR \rightarrow\Ni\text{ (1-resolution)}\] 
which is independent of the orientation of \( T \). Fixing an enumeration of all crossings of \( \mathcal{D} \), one may then associate with a vertex of an \( n \)-dimensional cube \( v\in\{0,1\}^n \) an object \( \mathcal{D}(v) \) in \( \Cob \) obtained by taking the \( v_i \)-resolution of the \( i^\text{th} \) crossing for each \( i=1,\dots, n \). If a vertex \( w\in\{0,1\}^n \) is obtained from a vertex \( v\in\{0,1\}^n \) by changing a single coordinate \( v_k \) of \( v \) from 0 to 1, we draw an arrow from \( \mathcal{D}(v) \) to \( \mathcal{D}(w) \). This arrow is labelled by the cobordism which outside of a neighbourhood of the  \( i^\text{th} \) crossing agrees with the product cobordism, and in the neighbourhood of the \( i^\text{th} \) crossing is equal to the saddle cobordism, which we denote by
\[ 
    \begin{tikzcd}  
    \No
    \arrow{r}{\Nol}
    &
    \Ni
    \end{tikzcd}
\] 
As in the algebraic definition of Khovanov homology in Section~\ref{subsec:Kh}, we add signs to these edge maps such that the number of signs on the boundary of every 2-dimensional face of the cube is odd. 
Denote by \( \vert v\vert := \sum v_i \) the height of the vertex in the cube, which up to a shift is the homological grading on the complex. Namely, the gradings are defined as follows: 
\[
    h(\mathcal{D}(v))\coloneqq \vert v\vert - n_- \qquad
    q(\mathcal{D}(v))\coloneqq\vert v\vert +n_+-2n_- \qquad
    \delta(\mathcal{D}(v))\coloneqq\tfrac{1}{2}(n_+-\vert v\vert)
\]
Note that the quantum grading of the morphism on each arrow is equal to \( -1 \), since the quantum grading of a saddle cobordism is \( -1 \), the quantum grading of the product cobordism is 0 and the quantum grading is additive under gluing. Hence, the differential preserves quantum grading. Moreover, the homological grading increases along the differential by 1. So this defines an object in \( \Kob \) which we will denote by \( \KhT{T} \). We denote its induced object in \( \Kob_{/l} \) by \( \KhTl{T} \). This notation is justified by the following theorem.

\begin{theorem}\cite[Theorem 1, Section 4.2]{BarNatanKhT}
\(\KhTl{T}\) is an invariant of the oriented tangle \(T\) up to grading preserving chain homotopy.
\end{theorem}
In the terminology of Section~\ref{sec:AlgStructFromGDCats}, \(\KhTl{T}\) as an object in \( \HomologyZero (\Kob_{/l})=\HomologyZero ( \Cx (\Mat(\Cob_{/l}))) \) is an oriented tangle invariant.

Like the Khovanov type invariants for links with multiple components, the invariants for tangles depend on an orientation of the components. However, reversing the orientation of a component only affects the absolute bigrading. Let us make this more precise: 

\begin{definition}\label{def:reversingOneComponent}
Given an oriented tangle \( T \) and a component \( t \) of \( T \), let \( n_+(t) \) be the number of positive crossings that involve \( t \) and a \emph{different} component of \( T \). Similarly, define \( n_-(t) \) using  negative crossings. Then define the \textbf{linking number of \( t \) in \( T \)} as 
\[ \lk_T(t):=\tfrac{1}{2}n_+(t)-\tfrac{1}{2}n_-(t).\]
\end{definition}

\begin{proposition}\label{prop:reversingOneComponent}
Let \(T'\) be the oriented tangle obtained from an oriented tangle \(T\) by reversing the orientation on one of its components \(t\). Then \( \KhTl{T'}=h^{-2\lk_T(t)}q^{-6\lk_T(t)}\delta^{-\lk_T(t)}\KhTl{T} \).
\end{proposition}
\begin{proof}
This follows on inspecting the effect of reversing the orientation of a tangle component on \( n_+ \) and \( n_- \).
\end{proof}

\begin{corollary}\label{cor:reversingOneComponent}
	Let \(T'\) be the oriented tangle obtained from an oriented tangle \(T\) by reversing the orientation of all components. Then \( \KhTl{T'}=\KhTl{T} \).
\end{corollary}
\begin{proof}
	This can be seen by applying the previous proposition step by step to each component. Then each crossing involving different components contributes exactly twice, but with opposite signs. Alternatively, we can simply use the fact that the sign of a crossing does not change if we change the orientation of both strands. 
\end{proof}

\subsection{Bases of morphism spaces}\label{subsec:tangle_invariants:dotted_cobordisms}
Our next goal is to obtain a better understanding of the morphism spaces in \( \Cob_{/l} \). We do this by fixing certain bases for the morphism spaces, and in order to do this, it turns out to be useful to fix a special component for every cobordism in \( \Cob_{/l} \). To make these choices consistent, we distinguish one of the points in the set \( B \) of tangle ends (eg \( \Lo \)), and say that the special component of any cobordism between two crossingless tangles is the one containing this distinguished point of \( B \). (This is where we are using the fact that \( B \) is non-empty.) Note that this is compatible with composition of cobordisms. In local pictures of cobordisms, we will decorate the special component by an asterisk \( \ast \). Then we can introduce the following notation.

\begin{definition}\label{def:NotationDotsAndH}
We introduce a formal variable \( H, \ \gr(H)=q^{-2}h^0, \)  which corresponds to multiplying a cobordism by \( -1 \) and adding a single 1-handle to the special component:
\[ H\cdot \Star:=-\StarH\]
We also allow to decorate components of a cobordism by dots  \( \bullet \), which can freely move around on their components. Such a dot denotes the sum of the cobordism obtained by joining the component with the dot to the special component and minus the cobordism with an additional 1-handle on the special component:
\[ \StarDot:=\StarTube-\StarSheetH=\StarTube+H\cdot\StarSheet\]
Note that this implies that a cobordism with a dot in the special component is equal to 0:
\[ \StarDotSameComp=0\]
We adjust the formula for computing the quantum grading of a cobordism \( C \) with dots \( \bullet \) as follows:
\begin{equation}\label{eqn:quantum_grading:with_dots}
q(H^i \cdot C):=\chi(C)-2i-2\#\{\text{dots \( \bullet \)}\}-\tfrac{1}{2}|B|
\end{equation}
\end{definition}

\begin{proposition}\label{prop:CobRH}
With the above notation, we have the following relations.
\begin{description}
    \item[\(S_\bullet\)-relation] Whenever a cobordism contains a component which is a sphere with a single dot \(\bullet\), it is equal to the same cobordism but with this component removed: \[ \Spheredot=1\]
    \item[\(H\)-trading relation] Whenever two dots occupy the same component of a cobordism, we can remove one of them at the expense of formally multiplying the morphism by \(H\): \[ \planedotdot=H\cdot\planedot\]
    \item[Neck-cutting relation] Given a cobordism \(C\) containing a compressing disc \(D\), consider the cobordism \(C'\) obtained from \(C\) by doing surgery along \(D\), which contains two embedded discs \(D_1\) and \(D_2\) as mementos from the surgery. Then the morphism represented by \(C\) is equal to the formal sum of three morphisms, of which the first two are obtained from \(C'\) by placing a dot in \(D_1\) and \(D_2\), respectively, and the third is \(C'\) formally multiplied by \((-H)\). In pictures: \[ \tube=\DiscLdot\DiscR+\DiscL\DiscRdot-H\cdot\DiscL\DiscR\] 
\end{description}
\end{proposition}
\begin{proof}
The \( S_\bullet \)-relation is obvious. The \( H \)-trading relation can be seen by applying the definition to one of the dots and observing that one of the two resulting terms contains a dot on the special component, so it is zero. Finally, the neck-cutting relation follows essentially from the \( 4Tu \)-relation:
\begin{align*}
    \StarNeck 
    &
    ~~\begin{alignedat}[t]{3}%
  \,=\phantom{+\,H\cdot }
  &\StarNeckI~+
  &
  &\StarNeckII~-
  &
  &\StarNeckIII\\
  +\,H\cdot 
  &\StarNeckIIIp~+
  &\,H\cdot
  &\StarNeckIIIp~-
  &\,2H\cdot
  &\StarNeckIIIp\\
  \,=\phantom{+\,H\cdot }
  &\StarNeckIp~+
  &
  &\StarNeckIIp~-
  &\,H\cdot
  &\StarNeckIIIp\\
  \end{alignedat}
\end{align*}
\end{proof}

\begin{remark}\label{obs:HZero}
The category of dotted cobordisms has been studied previously. However, to the best of the authors' knowledge, the topological interpretation of dots from Definition~\ref{def:NotationDotsAndH} is new. For example, Naot interprets a single dot either as \( \tfrac{1}{2} \) the cobordism with a 1-handle attached near the dot (when working over \( \Q \)) or in terms of a tube connecting the component with the dot to a special marked component (when working over \( \Z \))~\cite[Sections~3 and~5]{Naot}. Bar-Natan defines \( \Cob_\bullet \) as the category of dotted cobordisms without any topological interpretation of dots~\cite[Section~11.2]{BarNatanKhT}. He then describes how to obtain a tangle invariant over the category \( \Cob_{\bullet/l} \) obtained from \( \Cob_\bullet \) by imposing the \( S \)-, \( S_\bullet \)-, \( H \)-trading and neck-cutting relations above, albeit with \( H=0 \). In fact, if we take the quotient of \( \Cob_{/l} \) by \( H=0 \), we obtain a quotient of Bar-Natan's category \( \Cob_{\bullet/l} \)
\[
\Cob_{/l} / \left( H=0 \right) \ \ = \ \ \Cob_{\bullet/l} / \left(\planedotstar =0\right)
\] 
since the proof of Proposition~\ref{prop:CobRH} shows that the \( 4Tu \)- and the neck-cutting relations are equivalent. The latter quotient category was used in \cite{HHHK}: they showed that the algebra $\BNAlgH|_{H=0}=\End_{\bullet/l}(\Li,\Lo) / \left(\planedotstar =0\right)$  (denoted by $\mathcal A_{\mathcal D}$ in their paper) embeds in the Fukaya category of the pillowcase; this is used to construct the 4-ended tangle invariant $L_T$, which is a twisted complex built out of two Lagrangians in the pillowcase.
\end{remark}

\begin{definition}
We say that a cobordism is \textbf{simple} if it does not contain any closed component and every open component is contractible and contains at most one dot.
\end{definition}

\begin{observation}
Up to adding or removing dots on the non-special components of cobordisms, there is exactly one simple cobordism between two crossingless tangles \( T,T'\in\ob(\Cob) \). So if \( k \) is the number of components in this cobordism, there are \( 2^{k-1} \) simple cobordisms between \( T \) and \( T' \). Moreover, using the neck-cutting relation, we can rewrite any cobordism as an \( R[H] \)-linear combination of cobordisms in which all components are contractible, but might have one or multiple dots on them. The \( H \)-trading relation allows us to reduce the number of dots until we arrive at an \( R[H] \)-linear combination of simple cobordisms. In other words,  simple cobordisms generate the morphism spaces over \( R[H] \). 
\end{observation}
In fact, simple cobordisms form bases for the morphism spaces in \( \Cob_{/l} \):
\begin{proposition}[{\cite[Proposition~4.4]{Naot}}]\label{prop:lin independence of cobordisms}
For any two objects \(T\) and \(T'\) of \(\Cob_{/l}\), the morphism space \(\Mor_{\Cob_{/l}}(T,T')\) is freely generated over \(\Rcomm[H]\) by simple cobordisms.
\end{proposition}

\begin{definition}
We call a complex in \( \Kob_{/l} \) \textbf{simple} if there is no generator with a closed component. 
\end{definition}


\begin{proposition}\label{prop:simplifyingComplexes}
Every complex in \( \Kob_{/l} \) is chain isomorphic to a simple one.
\end{proposition}

\begin{observation}[Delooping]\label{obs:delooping}

If an object of \( \Cob_{/l} \), ie a crossingless tangle, contains a closed component, it is isomorphic in \( \Mat(\Cob_{/l}) \) to the direct sum of the same object but without this closed component. There are actually two pairs of chain isomorphisms that one can choose, but they are homotopic in \( \Cob_{/l} \):
\[
    \begin{tikzcd}[row sep=0.3cm, column sep=0.5cm]
    &
    \delta^{-\frac{1}{2}}q^{-1}\emptyset
    \arrow{dr}{\DiscRdot-H\cdot\DiscR}
    \\
    \delta^{0}q^0\Circle
    \arrow{ur}{\DiscL}
    \arrow[swap]{dr}{\DiscLdot}
    &&
    \delta^{0}q^0\Circle
    \\
    &
    \delta^{+\frac{1}{2}}q^{+1}\emptyset
    \arrow[swap]{ur}{\DiscR}
    \end{tikzcd}
    \qquad\text{ and }\qquad
    \begin{tikzcd}[row sep=0.3cm, column sep=0.6cm]
    &
    \delta^{-\frac{1}{2}}q^{-1}\emptyset
    \arrow{dr}{\DiscRdot}
    \\
    \delta^{0}q^0\Circle
    \arrow{ur}{\DiscL}
    \arrow[swap]{dr}{\DiscLdot-H\cdot\DiscL}
    &&
    \delta^{0}q^0\Circle
    \\
    &
    \delta^{+\frac{1}{2}}q^{+1}\emptyset
    \arrow[swap]{ur}{\DiscR}
    \end{tikzcd}
\]
To interpret these pictures correctly, one should extend all cobordisms via the identity to the rest of the crossingless tangle. Note that it is crucial that the rest of the tangle is non-empty and, in particular, that there is a special component. Otherwise, the symbols would be meaningless. Bar-Natan calls this procedure \textbf{delooping}~\cite[Lemma~3.1]{BarNatanBurgosSoto}. Delooping  is used to replace complexes over \( \Cob_{/l} \) by chain isomorphic ones whose underlying objects do not contain any closed components. 
\end{observation}

\begin{proof}[Proof of Proposition~\ref{prop:simplifyingComplexes}]
This follows from delooping by induction on the number of closed components.
\end{proof}

\subsection{The invariant for 4-ended tangles}\label{subsec:tangle_invariants:Four-ended_invariant}

In the case of 4-ended tangles \( T \), Proposition~\ref{prop:simplifyingComplexes} implies that the tangle invariant \( \KhTl{T} \) can be regarded as a complex over the full subcategory \(\End_{/l}(\Ni\oplus\No)\) of \(\Cob_{/l}\) generated by the two objects \(\Ni\) and \(\No\). 

\begin{definition}\label{def:BNAlgH}
Let \( \BNAlgH \) be the path algebra over \( \Rcomm \) of the quiver
\[
    \begin{tikzcd}[row sep=2cm, column sep=1.5cm]
    \DotB
    \arrow[in=145, out=-145,looseness=4]{rl}[description]{\DotcobB}
    \arrow[bend left]{r}[description]{\SaddleCB}
    &
    \DotC
    \arrow[bend left]{l}[description]{\SaddleBC}
    \arrow[in=35, out=-35,looseness=4]{rl}[description]{\DotcobC}
    \end{tikzcd}
\]
modulo the relations 
\[
\DotcobC \SaddleCB=0= \SaddleCB \DotcobB
\quad\text{and}\quad
\DotcobB \SaddleBC=0= \SaddleBC \DotcobC,
\]
and with gradings 
\[\gr(\DotcobB)=\gr(\DotcobC)=q^{-2}h^0
\quad\text{and}\quad
\gr(\SaddleCB)=\gr(\SaddleBC)=q^{-1}h^0.
\]
\end{definition}

\begin{remark}
We often abuse notation and write \( D \) for either \( \DotcobB \) or \( \DotcobC \) 
and \( S \) for \( \SaddleBC \) or \( \SaddleCB \). 
With this is mind, the relations above can be summarized as
\[
DS=0=SD.
\]
Moreover, note that we multiply paths from right to left, like morphisms in a category. For example, \( \DotcobC \SaddleCB=0 \) because of the relations and \( \SaddleCB \DotcobC=0 \) because the algebra elements are incomposable.
\end{remark}

\begin{theorem}\label{thm:OmegaFullyFaithful}
Let \( B=\{1,2,3,4\} \). For each choice \( i=1,2,3,4 \) of distinguished tangle end in \( B \), there exists an isomorphism 
\[ \omega_i\co\BNAlgH\rightarrow \End_{/l}(\Ni\oplus\No)\]
which on objects is defined by 
\(
\DotB\mapsto\No, 
\DotC\mapsto\Ni,
\)
and sends the two morphisms \( S \) to the saddles cobordisms, and the morphisms \( D \) to the dot cobordisms, that is, the identity cobordisms with a single dot in the non-special component. \end{theorem}
For example, $\omega_1$ sends
\(
\SaddleCB \mapsto \Lol, \SaddleBC \mapsto \Lil,   \DotcobB \mapsto  \LoDotB, \DotcobC \mapsto \LiDotR.
\) 

\begin{proof}[Proof of Theorem \ref{thm:OmegaFullyFaithful}]
The composition of any dot cobordism with a saddle cobordism vanishes, so the maps \( \omega_i \) are well-defined. Moreover, by the definition of the dot cobordism, the preimage of \( H \) under \( \omega_i \) is equal to \( D-SS \). Now note that over \( \Rcomm[H] \), the morphism spaces in \( \BNAlgH \) are generated by the identity morphisms, \( D \), and \( S \), and by Proposition~\ref{prop:lin independence of cobordisms}, the same is true for the identity, dot and saddle cobordisms in \( \End_{/l}(\Ni\oplus\No) \).
\end{proof}


\begin{remark}\label{rmk:Kh_arc_algebra}
There is a different way to arrive at the algebra $\BNAlgH$. Let  \( \widetilde{\mathcal{H}}^2 \) be the reduced version of  Khovanov's arc algebra \( \mathcal{H}^2 \); see, for example, \cite[Section 2.1]{Artem-Notes} where it is denoted by $B_r$. It turns out that $\widetilde{\mathcal{H}}^2 = \BNAlgH|_{H=0}$. In order to obtain the full algebra  $\BNAlgH$ where $H\neq 0$, one can change the ground ring from $\Rcomm$ to $\Rcomm[H]$, and deform the multiplication in \( \widetilde{\mathcal{H}}^2 \) using deformed merge-split maps (\ref{eq:merge-split_maps_for_BN}). The resulting deformed algebra \( \widetilde{\mathcal{H}}^2_\text{def} \) is equal to \( \BNAlgH \), which follows from Proposition~\ref{prop:alg_pairing_kh}.
\end{remark}

A  map of algebras \( f\co\mathcal A_1 \rightarrow \mathcal A_2 \) that preserves grading may be used to turn any complex \( (X=\bigoplus_i X_i,d) \) over \( {\mathcal A_1} \) into a complex over \( {\mathcal A_2} \) by substituting the objects \( X_i \) and the differential \( d \) according to the map \( f \); see for example~\cite[Definition~1.15]{pqMod}. Armed with this observation, we are now, finally, ready to introduce the main tangle invariant of the paper.

\begin{definition}\label{def:MainTangleInvariantTypeD}
Recall that \(\Mod^{\BNAlgH}\) is the category of bigraded complexes over \(\BNAlgH\). The maps \( \omega_i \) induce bigraded isomorphisms of categories
\[ \Omega_i\co\Mod^{\BNAlgH}\rightarrow \Mod^{\End_{/l}(\Ni\oplus\No)}\]
Given a \emph{pointed} 4-ended tangle \( T \), we define a complex over \( \BNAlgH \) via the preimage:
\[
\DD(T)\coloneqq \Omega_i^{-1}(\KhTl{T})
\]
\end{definition}

Note that, according to the conventions above, the pointed tangle \( T \) has a distinguished end \( i \) marked by \( * \). This pins down the choice of index \( i\in\{1,2,3,4 \}\).

\begin{observation}\label{obs:mutation}
Note that the four functors \( \omega_i \) agree on powers of \( S \). Moreover, the \( 4Tu \)-relation implies that
\begin{align*}
    \StarDot&=\NoStarTube-\NoStarSheetH\\
    &=\NoStarSheetHRev-\NoStarTube=-\StarDotRev
\end{align*}
So the images of the morphisms \( D \) differ by at most a minus sign. More specifically, \( \omega_i(D)=\pm \omega_j(D) \) with ``\( - \)'' iff the ends \( i \) and \( j \) lie on different components of the cobordisms in \( \omega_i(D) \). Thus, the same is true for odd powers of \( D \) and the functors \( \omega_i \) agree on even powers of \( D \). So, given the complex \( \DD(T) \) for a pointed tangle \( T \), this observation allows us to calculate the complex \( \DD(T') \) for the pointed tangle \( T' \) obtained from \( T' \) by changing the distinguished tangle end. We will return to this in Section~\ref{sec:mutation}, where we discuss mutation.
\end{observation}

 Let \( \revpre\co\op{(\BNAlgH)}\leftrightarrow {\BNAlgH}\) be the isomorphism of algebras given by 
    \[
    \op{\DotcobB} \leftrightarrow \DotcobB \qquad \op{\SaddleCB} \leftrightarrow \SaddleBC \qquad \op{\SaddleBC} \leftrightarrow \SaddleCB \qquad \op{\DotcobC} \leftrightarrow \DotcobC \]   This arises  from the fact that the inverse of a cobordism is itself a cobordism. Denote by \[\rev\co\Mod^{\op{(\BNAlgH)}}\leftrightarrow \Mod^{\BNAlgH} \] the induced functor for complexes.


\begin{definition}\label{def:mirrorsTangleAndTypeD}
    Given a 4-ended tangle \( T \), let \( \mirror T \) be the {\bf mirror} of \( T \) obtained from \( T \) by reversing all over- and under-crossings: \( \CrossingR\longleftrightarrow\CrossingL \).
    
    Given any type~D structure \(C^{\BNAlgH}\) (that is, a complex over \( \BNAlgH \)), define its mirror as
    \[
    \mirror(C)^{\BNAlgH} \coloneqq \rev\Big(\dual{C}^{\op{\BNAlgH}}\Big)
    \]  
    More explicitly, \(\mirror(C)^{\BNAlgH}\) is obtained from \(C^{\BNAlgH}\) by reversing the direction of all arrows, exchanging \( \SaddleBC \) and \( \SaddleCB \) and reversing the gradings of the generators. Notice that because both dualization and \( \rev \) change the algebra to its opposite, their composition does not change the algebra.
\end{definition}

\begin{proposition}\label{prop:mirrorsAndTypeDstructures}
The equalities \(\KhTl{\mirror T}=\mirror(\KhTl{T})\) and \(\DD(\mirror T)=\mirror(\DD(T))\) hold for any pointed 4-ended tangle \(T\).
\end{proposition}
\begin{proof}
    This follows from the definitions.
\end{proof}

\begin{example}\label{exa:BNntwisttangles}
Let us do some computations for the \( n \)-twist rational tangle. For \( n=1 \), we have
\[ 
    \KhTl{\CrossingPosMarkediiR}=\left[
    \begin{tikzcd}  
    \GGdqh{\Lo}{\frac{1}{2}}{1}{0}
    =
    h^0\delta^{\frac{1}{2}}q^1\Lo
    \arrow{r}{\Lol}
    &
    h^{1}\delta^0q^2\Li
    =
    \GGdqh{\Li}{0}{2}{1}
    \end{tikzcd}\right]
\] 
where \( \Lol \) denotes the saddle cobordism. 
By induction, for \( n\geq 1 \) we can show that
\[   \KhTl{\underbrace{\CrossingPosMarkediiR\cdots\CrossingPosR}_{n}}\!=\!
    \left[
    \begin{tikzcd}
    \GGdzh{\Lo}{\frac{n}{2}}{n}{0}
    \arrow{r}{\Lol}
      &
    \GGdzh{\Li}{\frac{n-1}{2}}{n+1}{1}
    \arrow{r}{\LiDotR}
    &
    \GGdzh{\Li}{\frac{n-1}{2}}{n+3}{2}
    \arrow{r}{\LiT}
    &
    \GGdzh{\Li}{\frac{n-1}{2}}{n+5}{3}
    \arrow{r}{\LiDotR}
    &
    \cdots
    \arrow{r}{}
    &
    \GGdzh{\Li}{\frac{n-1}{2}}{3n-1}{n}
    \end{tikzcd}\right]
\]
where \( \LiDotR \) denotes the identity cobordism \( \Li \) with a single dot in the non-special component, and \( \LiT=\LiDotR-H\cdot \Li \) is the composition of two saddle cobordisms. Therefore:
\[ 
\DD\Big(\underbrace{\CrossingPosMarkediiR\cdots\CrossingPosR}_{n}\Big)=
    \left[
    \begin{tikzcd}
    \GGdzh{\DotB}{\frac{n}{2}}{n}{0}
    \arrow{r}{S}
      &
    \GGdzh{\DotC}{\frac{n-1}{2}}{n+1}{1}
    \arrow{r}{D}
    &
    \GGdzh{\DotC}{\frac{n-1}{2}}{n+3}{2}
    \arrow{r}{SS}
    &
    \GGdzh{\DotC}{\frac{n-1}{2}}{n+5}{3}
    \arrow{r}{D}
    &
    \cdots
    \arrow{r}{}
    &
    \GGdzh{\DotC}{\frac{n-1}{2}}{3n-1}{n}
    \end{tikzcd}\right]
\]
If we move the distinguished tangle end to the bottom left corner, the complex does not change at all, since the distinguished component of the cobordism corresponding \( \DotcobC \) does not change, and there are no (odd) powers of \( \DotcobB \) in this complex. If we move the distinguished tangle end to one of the other two corners, we do see a change of signs; however, the complex above is chain isomorphic to any complex obtained by replacing some of the arrow labels by their negatives, because changes of basis \( \DotB \leftrightarrow (-\DotB), \ \DotC \leftrightarrow (-\DotC) \) can be used to ‘‘move'' signs along the complex \( \DD(T) \). So we conclude that the invariant of the \( n \)-twist tangle is independent of the choice of distinguished tangle end. 

The invariant for the \( -n \)-twist tangle, consisting of \( n \) negative crossings, can be easily obtained from the above, using Proposition~\ref{def:mirrorsTangleAndTypeD}:
\[ 
\DD\Big(\underbrace{\CrossingNegMarkediiR\cdots\CrossingNegR}_{n}\Big)=
    \left[
    \begin{tikzcd}
    \GGdzh{\DotC}{-\frac{n-1}{2}}{-3n+1}{-n}
    \arrow{r}{}
    &
    \cdots
    \arrow{r}{D}
    &
    \GGdzh{\DotC}{-\frac{n-1}{2}}{-n-5}{-3}
    \arrow{r}{SS}
    &
    \GGdzh{\DotC}{-\frac{n-1}{2}}{-n-3}{-2}
    \arrow{r}{D}
    &
    \GGdzh{\DotC}{-\frac{n-1}{2}}{-n-1}{-1}
    \arrow{r}{S}
    &
    \GGdzh{\DotB}{-\frac{n}{2}}{-n}{0}
    \end{tikzcd}\right]
\]
Again, the choice of basepoint is irrelevant for this example.
\end{example}

\begin{figure}[th]
  \begin{subfigure}{0.95\textwidth}
      \[ 
      \begin{tikzcd}[column sep=0pt,row sep=0pt]  
      \PretzeltangleOriented
      &
      \GGdzh{\Lo}{-1}{-8}{-3}
      \arrow{rr}{\LoT}
      &\phantom{XXXX}&
      \GGdzh{\Lo}{-1}{-6}{-2}
      \arrow{rr}{\LoDotB}
      &\phantom{XXXX}&
      \GGdzh{\Lo}{-1}{-4}{-1}
      \arrow{rr}{\Lol}
      &\phantom{XXXX}&
      \GGdzh{\Li}{-\frac{3}{2}}{-3}{0}
      \\
      \GGdzh{\Li}{0}{-4}{-2}
      \arrow[swap]{dd}{\Lil}
      &
      \GGdzh{\Li}{-1}{-12}{-5}
      \arrow{rr}{2\LiDotR-H\cdot\Li}
      \arrow[swap]{dd}{\Lil}
      &&
      \GGdzh{\Li}{-1}{-10}{-4}
      \arrow{dd}{-\Lil}
      &&
      \xcancel{\GGdzh{\Li}{-1}{-8}{-3}}
      \arrow{rr}{\LirO}
      \arrow{dd}{\Lil}
      &&
      \xcancel{\GGdzh{\LiO}{-\frac{3}{2}}{-7}{-2}}
      \arrow[swap]{dd}{-\LiOl}
      \\
      \parbox[c][20pt][c]{1pt}{}
      \\
      \GGdzh{\Lo}{-\frac{1}{2}}{-3}{-1}
      \arrow[swap]{dd}{\LoDotB}
      &
      \GGdzh{\Lo}{-\frac{3}{2}}{-11}{-4}
      \arrow{rr}{\LoT}
      \arrow[swap]{dd}{\LoDotB}
      \arrow[dashed]{rruu}{\Lol}
      &&
      \GGdzh{\Lo}{-\frac{3}{2}}{-9}{-3}
      \arrow{rr}{\LoDotB}
      \arrow[swap]{dd}{-\LoDotB}
      &&
      \GGdzh{\Lo}{-\frac{3}{2}}{-7}{-2}
      \arrow{rr}{\Lol}
      \arrow{dd}{\LoDotB}
      \arrow[dashed,swap]{ddll}{-\id}
      &&
      \xcancel{\GGdzh{\Li}{-2}{-6}{-1}}
      \\
      \parbox[c][20pt][c]{1pt}{}
      \\
      \GGdzh{\Lo}{-\frac{1}{2}}{-1}{0}
      &
      \GGdzh{\Lo}{-\frac{3}{2}}{-9}{-3}
      \arrow{rr}{\LoT}
      &&
      \GGdzh{\Lo}{-\frac{3}{2}}{-7}{-2}
      \arrow{rr}{\LoDotB}
      &&
      \GGdzh{\Lo}{-\frac{3}{2}}{-5}{-1}
      \arrow{rr}{\Lol}
      &&
      \GGdzh{\Li}{-2}{-4}{0}
      \end{tikzcd}
      \]
  \caption{The computation of \( \KhTl{T_{2,-3}} \).  The \( (2,-3) \)-pretzel tangle is shown in the top left corner. The topmost row and leftmost column show the complex for the right and left hand side of the tangle, respectively. The lower right double-complex is the result of pairing the complexes for these two rational tangles.}\label{fig:2m3pt:BNBracket:Raw}
  \end{subfigure}
  \begin{subfigure}{0.95\textwidth}
    \[ 
    \left[
    \begin{tikzcd}[column sep=40pt,row sep=20pt]  
    \GGdzh{\Li}{-1}{-12}{-5}
      \arrow{r}{\LiDotR}
      \arrow{d}{\Lil}
      &
      \GGdzh{\Li}{-1}{-10}{-4}
      \arrow{r}{\Lil}
      &
      \GGdzh{\Lo}{-\frac{3}{2}}{-9}{-3}
      \arrow{r}{\LoDotB}
      &
      \GGdzh{\Lo}{-\frac{3}{2}}{-7}{-2}
      \\
      \GGdzh{\Lo}{-\frac{3}{2}}{-11}{-4}
      \arrow{r}{\LoDotB}
      &
      \GGdzh{\Lo}{-\frac{3}{2}}{-9}{-3}
      \arrow{r}{\LiT}
      &
      \GGdzh{\Lo}{-\frac{3}{2}}{-7}{-2}
      \arrow{r}{\LoDotB}
      &
      \GGdzh{\Lo}{-\frac{3}{2}}{-5}{-1}
      \arrow{r}{\Lol}
      &
      \GGdzh{\Li}{-2}{-4}{0}
      \end{tikzcd}
      \right]
      \]
    \caption{The complex \( \KhTl{T_{2,-3}} \) after doing the cancellations and generalized base changes indicated by the dashed arrows in figure (a)}\label{fig:2m3pt:BNBracket:Simplified}
  \end{subfigure}
  \begin{subfigure}{0.95\textwidth}
    \[ 
    \left[
    \begin{tikzcd}[column sep=40pt,row sep=10pt]  
    \GGdzh{\DotC}{-1}{-12}{-5}
      \arrow{r}{D}
      \arrow{d}{S}
      &
      \GGdzh{\DotC}{-1}{-10}{-4}
      \arrow{r}{S}
      &
      \GGdzh{\DotB}{-\frac{3}{2}}{-9}{-3}
      \arrow{r}{D}
      &
      \GGdzh{\DotB}{-\frac{3}{2}}{-7}{-2}
      \\
      \GGdzh{\DotB}{-\frac{3}{2}}{-11}{-4}
      \arrow{r}{D}
      &
      \GGdzh{\DotB}{-\frac{3}{2}}{-9}{-3}
      \arrow{r}{SS}
      &
      \GGdzh{\DotB}{-\frac{3}{2}}{-7}{-2}
      \arrow{r}{D}
      &
      \GGdzh{\DotB}{-\frac{3}{2}}{-5}{-1}
      \arrow{r}{S}
      &
      \GGdzh{\DotC}{-2}{-4}{0}
      \end{tikzcd}
      \right]
      \]
    \caption{The complex \( \DD(T_{2,-3}) \).}
    \label{fig:2m3pt:BN}
  \end{subfigure}
  \caption{Computation of the Bar-Natan complex for the \( (2,-3) \)-pretzel tangle}\label{fig:2m3ptBNComplexComputationRH}
\end{figure}
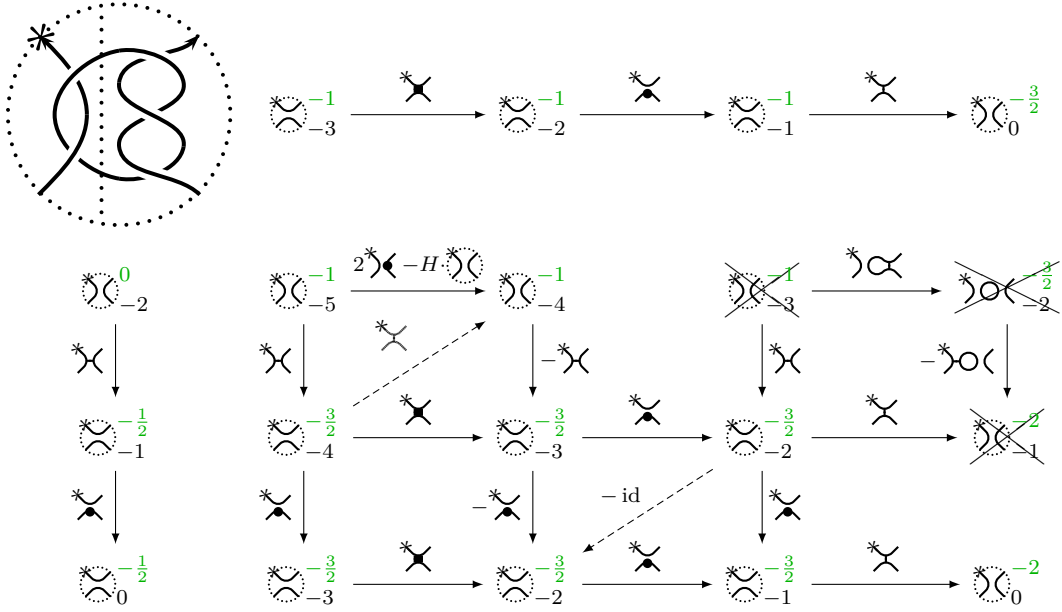
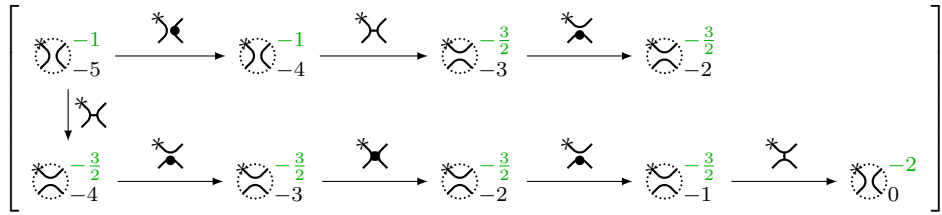
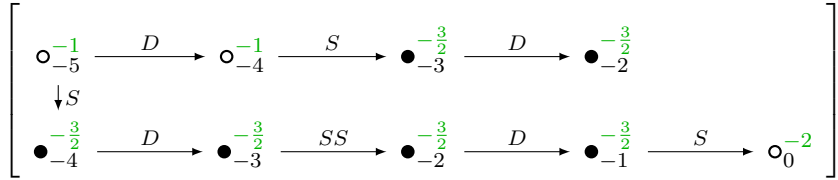

\begin{example}\label{exa:2m3ptBNComplexComputation}
In Figure~\ref{fig:2m3pt:BNBracket:Raw}, we compute \( \KhTl{T_{2,-3}} \) for the \( (2,-3) \)-pretzel tangle \( T_{2,-3} \), by splitting it into two rational tangles along the dotted line in the tangle diagram in the upper left corner of the figure. The complex for the 2-twist rational tangle is shown in the first column; notice the change in gradings compared to Example~\ref{exa:BNntwisttangles}, which is due to the different orientation. The complex for the \( -3 \)-twist rational tangle is shown in the first row of the figure. \( \KhTl{T_{2,-3}} \) is the double-complex obtained by gluing the objects together and extending the morphisms in each row and column by the identity. Note that we need to compute the horizontal morphisms very carefully, since for some of them, the special component changes. For example, the second morphisms of the first row of the double-complex is not equal to the dot cobordism, but 0. 
We simplify this complex by delooping the circle in the top right corner and cancelling the two resulting identity components of the differential. We also apply Lemma~\ref{lem:AbstractCleanUp} by setting the morphism \( h \) equal to the dashed arrows. As a result, we obtain the complex in Figure~\ref{fig:2m3pt:BNBracket:Simplified}. This complex now lies in the image of the functor \( \Omega_1 \), and its preimage \( \DD(T_{2,-3}) \) is shown in Figure\ref{fig:2m3pt:BN}. 
\end{example}
\begin{remark}\label{rmk:bordered_constr_of_D_structure}
There is another construction of a tangle invariant like \(\DD(T)\), using algebraic techniques of Khovanov \cite{Khovanov2} and bordered techniques of Manion \cite{Manion-bordered}. Instead of the algebra \( \BNAlgH \), one can use the deformed reduced Khovanov arc algebra \( \widetilde{\mathcal{H}}^2_\text{def} \) from Remark~\ref{rmk:Kh_arc_algebra}. To obtain a pointed 4-ended tangle invariant in the form of a type~D structure over \( \widetilde{\mathcal{H}}^2_\text{def} \), one starts with Manion's definition of the type~D structure \( \widehat{D}(T)^{\mathcal{H}^2}\) over the arc algebra; see \cite[Definition 3.1.3]{Manion-bordered}. Next, one can adapt it to the reduced case and obtains \( \widehat{D}(T)^{\widetilde{\mathcal{H}}^2} \). Finally, by deforming the differential of \( \widehat{D}(T)^{\widetilde{\mathcal{H}}^2} \) using the merge-split maps (\ref{eq:merge-split_maps_for_BN}) from page~\pageref{eq:merge-split_maps_for_BN}, one obtains a type~D structure \( \widehat{D}(T)^{\widetilde{\mathcal{H}}^2_\text{def}} \) which should be homotopy equivalent to \( \DD(T)^{\BNAlgH} \). 

\end{remark}


\begin{definition}\label{def:tanglepairing}
Given two pointed 4-ended tangles \( T_1 \) and \( T_2 \), let \( \Lk(T_1,T_2) \) be the link obtained by gluing \( T_1 \) to \( T_2 \) as shown in Figure~\ref{fig:tanglepairingII}. Equivalently, \( \Lk(T_1,T_2) \) is obtained by gluing \reflectbox{\( T_1 \)} to \( T_2 \) as shown in Figure~\ref{fig:tanglepairingI}, where \reflectbox{\( T_1 \)} is the tangle obtained from \( T_1 \) by rotating it along the vertical axis (along with the parameterization of the boundary). Note that \( \Lk(T_1,T_2)=\Lk(T_2,T_1) \), which can be seen by rotating the link in Figure~\ref{fig:tanglepairingI} along the vertical axis by \( \pi \).
\end{definition}

\begin{figure}[t]
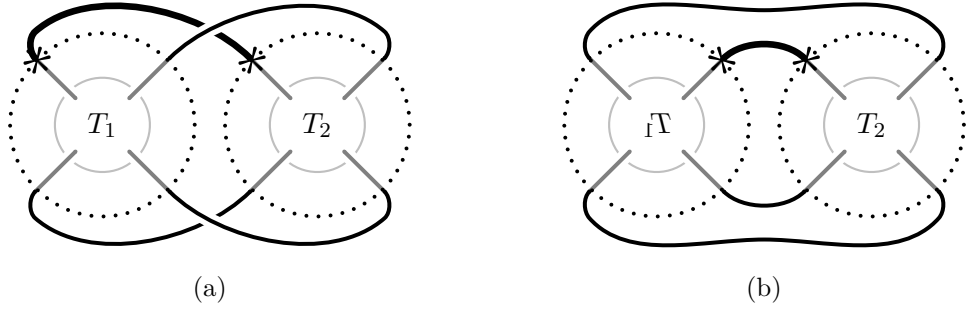

\centering
\begin{subfigure}{0.45\textwidth}
\centering
$\tanglepairingII$
\caption{}\label{fig:tanglepairingII}
\end{subfigure}
\begin{subfigure}{0.45\textwidth}
\centering
$\tanglepairingIalt$
\caption{}\label{fig:tanglepairingI}
\end{subfigure}
\caption{Two diagrams for the pointed link \( \Lk=\Lk(T_1,T_2) \) obtained by pairing two tangles \( T_1 \) and \( T_2 \)}
\end{figure}

Bar-Natan expresses various versions of Khovanov homology of a link \( \Lk \) in terms of morphism spaces from the zero object $\varnothing$ to \( \KhTl{\Lk} \) \cite[Section 9]{BarNatanKhT}. In case of tangles, instead of $\varnothing$ one can use different complexes as test objects as in \cite[Proposition~4.2.3]{Manion-thesis}. The proof of the following general result is based on the same ideas.

\begin{proposition}\label{prop:alg_pairing_kh}
Given two pointed 4-ended tangles \(T_1\) and \(T_2\), let \(\Lk=\Lk(T_1,T_2)\). Then
\[ 
h^0q^{-1}\delta^{-\frac{1}{2}}\CBNr(\Lk) 
\cong 
\Mor(\mirror(\DD(T_1)),\DD(T_2))
\]
as chain complexes of \( \Rcomm[H] \)-modules. 
\end{proposition}

\begin{proof}
	The right-hand side is equal to \( \Mor(\mirror(\KhTl{T_1}),\KhTl{T_2}) \). This follows from Theorem~\ref{thm:OmegaFullyFaithful} and both statements of Proposition~\ref{prop:mirrorsAndTypeDstructures}. So if we take the definition of \( \CBNr(\Lk) \) as \( h^0q^{1}\delta^{\frac{1}{2}}\CBN^{x-H}(\Lk) \), it suffices to show that
	\[ 
	\CBN^{x-H}(\Lk) 
	\cong
	\Mor(\mirror(\KhTl{T_1}),\KhTl{T_2}).
	\]	
	Let us first identify the underlying \( \Rcomm[H] \)-modules on both sides. Fix diagrams \( \mathcal{D}_1 \) and \( \mathcal{D}_2 \) for \( T_1 \) and \( T_2 \), respectively, and let \( \mathcal{D} \) be the corresponding link diagram for \( \Lk \) from Figure~\ref{fig:tanglepairingI}. 
	Choose an enumeration of the crossings of \( \Lk \) such that the indices of those crossings in \( T_1 \) are smaller than those in \( T_2 \). 
	Say, \( T_i \) has \( n_i \) crossings for \( i=1,2 \).
	Then the vertices of the cube of resolutions of \( \Lk \) are indexed by vectors \( v=v_1\oplus v_2\in\{0,1\}^{n_1}\oplus \{0,1\}^{n_2} \). We claim that 
	\[ 
	V^{x-H}_{\mathcal{D}(v)}
	\cong
	\Mor(\mathcal{D}_1(v_1),\mathcal{D}_2(v_2)).
	\]
	To see this, recall from Proposition~\ref{prop:lin independence of cobordisms} that the right-hand side is generated freely over \( \Rcomm[H] \) by simple cobordisms from \( \mathcal{D}_1(v_1) \) to \( \mathcal{D}_2(v_2) \). The particular way we connect $T_1$ to $T_2$ in Figure~\ref{fig:tanglepairingII} is behind the following key observation: each boundary component of a simple cobordism from \( \mathcal{D}_1(v_1) \) to \( \mathcal{D}_2(v_2) \) corresponds to a closed component of \( \mathcal{D}(v) \), where  \( \mathcal{D}(v) \) is a diagram obtained from Figure~\ref{fig:tanglepairingI} by resolving crossings according to \( v \). Moreover, each component of a simple cobordism, except the one containing the basepoint, may carry at most one dot; similarly, each component of \( \mathcal{D}(v) \) corresponds to two generators \( 1 \) and \( x \), except when it contains the basepoint. The special component of a cobordism may not carry any dot and the corresponding tensor factor \( V^{x-H} \) is one-dimensional. Thus, we obtain a one-to-one correspondence between generators of both sides: 
	\begin{equation}\label{eqn:algebraic_pairing:dictionary}
	V\ni 1=x_+\leftrightarrow \DiscR\qquad V\ni x=x_-\leftrightarrow \DiscRdot\qquad V^{x-H}\ni (x-H) \leftrightarrow \DiscRAst
	\end{equation}
	Let us verify that this one-to-one correspondence preserves the  bigrading. The homological grading vanishes on both sides. The contribution of a non-special component of a simple cobordism to the quantum grading is equal to \( -1 \) or \( +1 \), depending on whether it contains a dot or not. Then the remaining terms in \eqref{eqn:quantum_grading:with_dots} are the Euler characteristic of the special component minus \( \tfrac{1}{2}|B| \), which is \( 1-\tfrac{1}{2}\cdot 4=-1 \). Correspondingly, the quantum grading of \( x \) and \( 1\in V \) are \( -1 \) and \( +1 \), respectively, and the generator of \( V^{x-H}=\langle x-H\rangle \) has quantum grading \( -1 \).  
	So we conclude that if \( n^+_{i} \) and \( n^-_{i} \) are the number of positive and negative crossings in \( \mathcal{D}_i \), respectively, \[ h^{|v|-n_-}q^{|v|+n_+-2n_-}V^{x-H}_{\mathcal{D}(v)}\]
	is bigraded isomorphic to 
	\[ 
	\Mor(v_1,v_2)\coloneqq\Mor(\mirror(h^{|v_1|-n^-_1}q^{|v_1|+n^+_1-2n^-_1}\mathcal{D}_1(v_1)),h^{|v_2|-n^-_2}q^{|v_2|+n^+_2-2n^-_2}\mathcal{D}_2(v_2)).
	\]
	Notice that the gradings of the first factor count with minus sign in the grading of the morphism space elements (see Remark~\ref{rmk:gradings_of_morphisms}), but this minus sign is cancelled due to the fact that we consider the mirror of the first factor.  Since the \( H \)-action is formal on both sides, this is an isomorphism of \( \Rcomm[H] \)-modules. 
\begin{figure}[t]
	\centering
	\begin{equation*}
	\begin{split}
	\PairsOfPantsRAst &\Delta^{x-H}\co V^{x-H}\rightarrow V^{x-H}\otimes V
	=
	\begin{cases}
	\DiscRAst \mapsto \PairsOfPantsGluedAst\!\!\!\!=\DiscRAst \otimes \DiscRdot - H \cdot \DiscRAst \otimes \DiscR
	\end{cases}
	\\
	\PairsOfPantsLAst &m^{x-H}\co V^{x-H}\otimes V\rightarrow V^{x-H}
	=
	\begin{cases}
	\DiscRAst \otimes \DiscR \mapsto\DiscRAst,  & \\
	\DiscRAst \otimes \DiscRdot \mapsto  0. & 
	\end{cases}
	\end{split}
	\end{equation*}
	\begin{equation*}
	\begin{split}
	\PairsOfPantsR &\Delta\co V\rightarrow V\otimes V
	=
	\begin{cases}
	\DiscR \mapsto \DiscR \otimes \DiscRdot +\DiscRdot \otimes \DiscR - H\cdot \DiscR \otimes \DiscR & \\
	\DiscRdot \mapsto \DiscRdot \otimes \DiscRdot & 
	\end{cases}\\
	\PairsOfPantsL &m\co V\otimes V\rightarrow V
	=
	\begin{cases}
	\DiscR \otimes \DiscRdot \mapsto \DiscRdot \quad & 
	\DiscR \otimes \DiscR \mapsto \DiscR\\
	\DiscRdot \otimes \DiscR \mapsto \DiscRdot \quad & 
	\DiscRdot \otimes \DiscRdot \mapsto H\cdot \DiscRdot
	\end{cases}
	\end{split}
	\end{equation*}
	\caption{The calculation of compositions of merge and split maps with simple cobordisms. The first row illustrates that capping-off a pair of pants on one side results in a cobordism which can be simplified using the neck-cutting relation. The calculation for the other three maps is similar. }\label{fig:Composition:Pants:Simple:Cobordisms}
\end{figure}

The differential on the right-hand side is defined as the pre- and post-composition with the differentials of \( \mirror(\KhTl{T_1}) \) and \( \KhTl{T_2} \). 
	So if \( f\in \Mor(v_1,v_2) \), then
	\begin{equation}\label{eqn:algebraic_pairing:decomposition}
	 D(f)\in 
	\Big(\bigoplus \Mor(v'_1,v_2)\Big)
	\oplus
	\Big(\bigoplus \Mor(v_1,v'_2)\Big)
	\end{equation}
	where the direct sums are over all vectors \( v'_1 \) and \( v'_2 \) that can be obtained from \( v_1 \) and \( v_2 \), respectively, by replacing a single entry 0 by 1. Let us consider the components of the second direct sum first. These are equal to the composition of \( f \) with the differential 
  \( d^{v_2}_{v'_2} \)
  of \( \KhTl{T_2} \), ie the composition of the cobordism \( f \) with pairs of pants. This computation is done in Figure~\ref{fig:Composition:Pants:Simple:Cobordisms}.  If we compare the result with the definition of the merge and split maps (\ref{eq:merge-split_maps_for_BN}) and (\ref{eq:merge-split_maps_reduced_BN_x-H}) of \( \CBN^{x-H}(\Lk) \) via the dictionary \eqref{eqn:algebraic_pairing:dictionary}, we see that the two differentials agree up to the choices of edge assignments. 
	
	Let us now consider the components in the first direct summand of 
	\eqref{eqn:algebraic_pairing:decomposition}. 
	They are equal to \( -(-1)^{|v|-n_-} f\circ\mirror(d^{v_1}_{v'_1}) \), where \( \mirror(d^{v_1}_{v'_1}) \) is the component of the differential of \( \mirror(\KhTl{T_1}) \) from \( v'_1 \) to \( v_1 \). The sign comes from the definition of the differential \( D \) in Definition~\ref{def:CatOfComplexes} together with the fact that the homological grading of \( f \) is equal to 
	\((|v_2|-n_2^-)-(n_1-|v_1|-n_1^+)=|v|-n_-.\)
	Note that if \( d^{v_1}_{v'_1} \) is a merge cobordism, \( \mirror(d^{v_1}_{v'_1}) \) is a split cobordism, and vice versa. Thus, by contravariance of pre-composition \( -\circ\mirror(d^{v_1}_{v'_1}) \), \( f\circ\mirror(d^{v_1}_{v'_1}) \) corresponds to a merge map iff \( d^{v_1}_{v'_1} \) does. So these edge maps also agree with those in \( \CBN^{x-H}(\Lk) \) up to the choices of edge assignments.
	
	Finally, it suffices to check that the edge assignment on the right-hand side is such that all faces of the cube have an odd number of signs. For faces in which one of the two components stays constant, this is obvious. For the remaining faces, the edge assignment is 
	\[ 
	\begin{tikzcd}[column sep=40pt,row sep=12pt]  
	(v_1,v_2)
	\arrow{r}{}
	\arrow[swap]{d}{-(-1)^{|v_1|+|v_2|-n_-}}
	&
	(v_1,w_2)
	\arrow{d}{-(-1)^{|v_1|+|w_2|-n_-}=-(-1)^{|v_1|+|v_2|+1-n_-}}
	\\
	(w_1,v_2)
	\arrow{r}{}
	&
	(w_1,w_2)
	\end{tikzcd}
	\]
	where the horizontal maps have the same sign.
\end{proof}

\begin{remark}\label{rem:LinkInvariantFrom2EndedTangles}
	We may regard any pointed link \((\Lk,p)\) as a 2-ended tangle \(T_{\Lk,p}\) by cutting the link open at the basepoint \(p\). Then, by delooping, the complex \(\KhTl{T_{\Lk,p}}\) is chain isomorphic to a complex over $\Rcomm[H]$, built solely out of copies of the trivial 2-ended tangle \(\TrivialTwoTangle\). In fact, it agrees with \(\CKhr(\Lk,p)^{\Rcomm[H]}\). This can be seen by composing the delooping isomorphisms with pairs of pants, which computes precisely the split and merge maps of the type D structure \(\CKh^x(\Lk,p)^{\Rcomm[H]}\) from the Equations~\eqref{eq:merge-split_maps_for_BN} and~\eqref{eq:merge-split_maps_reduced_BN_x} on pages \pageref{eq:merge-split_maps_for_BN} and~\pageref{eq:merge-split_maps_reduced_BN_x}, except that $H$ is replaced by $(-H)$. By the same argument as in the proof of Lemma~\ref{lem:OnlyOneReducedBNTheory}, this complex is chain isomorphic to \(\CKh^x(\Lk,p)^{\Rcomm[H]}\).
\end{remark}

\begin{observation}
Here we consider the dependence on the basepoint of reduced Bar-Natan homology of the trefoil union the unknot. 
First, we need two auxiliary computations: reduced and unreduced Khovanov homologies of the trefoil, as type~D structures over $\Q[H]$, are 
\[\Khr(T(2,3))^{\Q[H]}=[\gen]\oplus[\gen \xrightarrow{H} \gen] \qquad 
\Kh(T(2,3))^{\Q[H]}=[\gen]\oplus[\gen]\oplus[\gen \xrightarrow{H^2} \gen]\]
On the level of vector spaces, these follow from computations in Example~\ref{exa:Pairing:Trefoil}. On the level of type~D structures, the additional actions picking up $H$ and $H^2$ follow from Proposition~\ref{prop:towersReduced-field} and Proposition~\ref{prop:towers}, together with the fact that for knots the free summands $[\gen]$ are contained in homological grading zero (see Proposition~\ref{prop:s-inv-as-grading}). On the one hand, if we reduce the trefoil component, $\Khr^{\Q[H]} \boxtimes \Q[H] \simeq \CBNr$  implies
\[\BNr(^\ast T(2,3) \cup \Circle ; \Q)= \BNr(T(2,3);\Q) \oplus \BNr(T(2,3);\Q) = \Q[H]\oplus \Q[H] \oplus \Q[H]/(H) \oplus \Q[H]/(H).\]
On the other hand, reducing the unknot, $ \Kh^{\Q[H]} \boxtimes \Q[H] \simeq \CBN$  implies 
\[  \BNr(T(2,3) \cup \ast \mspace{-7mu}\Circle ; \Q)= \BN(T(2,3);\Q) = \Q[H]\oplus \Q[H] \oplus \Q[H]/(H^2). \]
Note that the two reduced Bar-Natan homologies above are isomorphic as (bigraded) vector spaces, but are not isomorphic as $\Q[H]$-modules.
\end{observation}
The above observation leads to the following result:
\begin{theorem}\label{theo:BNr_is_a_link_invariant}
Reduced Bar-Natan homology $\BNr(\Lk,p)$, as a bigraded $R$-module, does not depend on the basepoint $p$. Moreover, if $\ell$ is the number of link components, there is an $R[H_1,\ldots,H_\ell]$-module structure on reduced Bar-Natan homology $\BNr(\Lk)$.
\end{theorem}
\begin{proof}
In Section~\ref{sec:BNr} we showed that $\BNr(\Lk,p)$ only depends on the component of the link that contains the basepoint $p$. So it is enough to prove that moving the basepoint between the link components does not change $\BNr(\Lk,p)$, as a bigraded $R$-module. For that, fix two components $\Lk_1, \Lk_2$ of $\Lk$, and fix a decomposition $\Lk(\Ni,T)=\Lk$, such that the left and right components of $\Ni$ belong to $\Lk_1$ and $\Lk_2$, respectively. One can easily obtain such a decomposition by isotoping the link appropriately. 
Then, by Proposition~\ref{prop:alg_pairing_kh}, 
\begin{align*}
\BNr(\Lk,p\in \Lk_1) 
&\cong  \BNr(\Lk (\Li, {}^\ast T)) 
\cong  h^0 q^1\Homology( \Mor(\DD(\Li),\DD( {}^\ast T))) \quad\text{and}
\\ 
\BNr(\Lk,p\in \Lk_2)
&\cong \BNr(\Lk(\Lialt, T^\ast)) 
\cong h^0 q^1\Homology( \Mor(\DD(\Lialt),\DD( T ^\ast))).
\end{align*}
As bigraded $R$-modules, both are equal to 
$h^0 q^1\Homology(\Mor(\KhTl{\Ni},\KhTl{T}))$ by Theorem~\ref{thm:OmegaFullyFaithful}, which proves the first claim.

Now, to define the $H_i$-action on $\BNr(\Lk)$ one simply needs to take the version $\BNr(\Lk)=\BNr(\Lk, p \in \Lk_i)$, putting basepoint on the $i^\text{th}$ component. The fact that the different $H_i$ and $H_j$ actions commute follows from the comparison argument above: in the version $\BNr(\Lk)=h^0 q^1\Homology(\Mor(\KhTl{\Ni},\KhTl{T}))$ both actions are defined as adding a genus on the marked components (see Definition~\ref{def:NotationDotsAndH}). These operations adding genus clearly commute, resulting in the desired $R[H_1,\ldots,H_\ell]$-module structure on $\BNr(\Lk)$.
\end{proof}

\section{Classification results}\label{sec:classification}

In this section, we classify type~D structures over Bar-Natan's algebra \( \BNAlgH \) in terms of immersed curves on the 4-punctured sphere and we classify morphism spaces between any two such type~D structures in terms of wrapped Lagrangian Floer theory. This will follow from results that are more general (and do not pertain to Khovanov homology): we prove these two classification results for the categories of type~D structures over more general algebras associated with surfaces carrying certain extra structure. 

\subsection{Oriented surfaces with arc systems}\label{subsec:classification:setup}

    Let \( \Sigma \) be an oriented surface with boundary and without closed components and let \( A \) be a set of pairwise disjoint embedded arcs on \( \Sigma \) with boundary on \( \partial \Sigma \).     
    For each arc \( a\in A \), choose a closed neighbourhood \( N(a) \) of \( a \) such that \( N(a)\cap N(a')=\emptyset \) for any distinct arcs \( a,a'\in A \). 
    Let \( s_1(a) \) and \( s_2(a) \) be the closures of the two components of \( \partial N(a)\smallsetminus\partial\Sigma \) in \( \Sigma \) such that \( s_1(a) \) lies to the left of the oriented arc \( a \) and \( s_2(a) \) to its right; we call \( s_1(a) \) and \( s_2(a) \) the two \textbf{sides} of \( a \). 
    We also fix a foliation \( \mathcal{F}_a=I\times I \) of \( N(a) \) such that \( s_1(a) \), \( s_2(a) \) and \( a \) are leaves of \( \mathcal{F}_a \).
    Furthermore, we call the closures of the connected components of \( \Sigma\smallsetminus \bigcup_{a\in A}N(a) \) \textbf{faces} and denote the set of all faces by \( F(\Sigma,A) \).
    We denote the number of all sides that lie on the boundary of some face \(f\) in \( F(\Sigma,A) \) by \( n_f \). 
	We call~\( A \) (together with a choice of fixed~\( N(a) \) and foliation~\( \mathcal{F}_a \) as above) an \textbf{arc system} if 
	all faces \( f\in F(\Sigma,A) \) are annuli bounding exactly one component of \( \partial \Sigma \). We call those boundary components of \( \Sigma \) \textbf{inner} boundary components, denoted by \( \partial_i\Sigma \); we call all other boundary components, namely, those that are endpoints of arcs, \textbf{outer}. 

	Surfaces with arc systems correspond to marked surfaces with arc systems from~\cite[Definition~4.1]{pqMod}. Note that in this paper we will usually draw the surfaces \( \Sigma \) (or parts thereof) such that their normal vector fields, determined by the right-hand rule, point out of the plane of projection. This is opposite to the conventions used in~\cite{pqMod}.
	 

\begin{example}\label{exa:ThreePuncturedDisc}
Figure~\ref{fig:ExampleMarkedSurface} shows an arc system \( A:=\{b,c\} \) on the 3-punctured disc \( \Sigma \) and an arc system \( A^\perp:=\{b^\perp,c^\perp\} \) on \( -\Sigma \). The inner boundary components are drawn in black, outer ones in \textcolor{darkgreen}{green}. The oriented arcs in \( A \) are drawn in \textcolor{red}{red}, their neighbourhoods are shaded in light red (\textcolor{lightred}{\( \blacksquare \)}), bounded by solid arrows representing the sides of the arcs.
\end{example}

\begin{remark}\label{rem:dualarcsystem}
	The two arc systems from the previous example are dual to each other: given an arc system on a surface \( \Sigma \), replace each arc \( a \) by a perpendicular one connecting the inner boundary components of the two faces of \( \Sigma \) with respect to the original arc system on either side of \( a \). One can do this in such a way that all new arcs are pairwise disjoint. We now observe that outer boundary components become inner ones and vice versa, and moreover, each new face has exactly one inner boundary component, so the new arc system is well-defined. Finally, we change the orientation of the underlying surface \( \Sigma \).
\end{remark}

\begin{figure}[t]
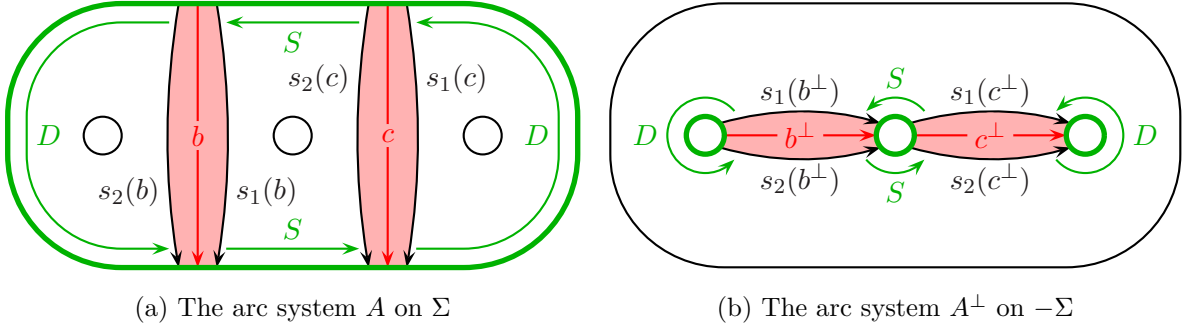

	\centering
	\begin{subfigure}[b]{0.49\textwidth}
	    \centering
		$\ArcSystemForTypeDStructures$
		\caption{The arc system \( A \) on \( \Sigma \)}\label{fig:ArcSystemForTypeDStructures}
	\end{subfigure}
	\begin{subfigure}[b]{0.49\textwidth}
		\centering
		$\ArcSystemForTwistedComplexes$
		\caption{The arc system \( A^\perp \) on \( -\Sigma \)}\label{fig:ArcSystemForTwistedComplexes}
	\end{subfigure}
	\caption{An arc system on the 3-punctured disc and its dual. The normal vector field in (a) points out of the projection plane, and into the plane in (b).}\label{fig:ExampleMarkedSurface}
\end{figure}

\begin{wrapfigure}{r}{0.3333\textwidth}
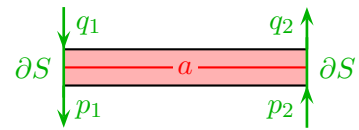

	\centering
	$\arcneighbourhood$
	\caption{A typical neighbourhood of an arc \( a \) }\label{fig:TypicalNeighbourhoodOfArc}\bigskip
\end{wrapfigure}
\myfixwrapfig

	Given an arc system \( A \) on a surface \( \Sigma \), consider the graph \( Q(\Sigma,A) \) whose vertices are obtained by contracting the closed neighbourhood of each arc in \( A \) to a single point and whose edges are deformation retracts of each face to its outer boundary component. By choosing the induced boundary orientation on \( \partial f \), the 1-cells inherit an orientation from \( \Sigma \), which turns \( Q(\Sigma,A) \) into a quiver. 
	 A typical neighbourhood of the image of an arc \( a\in A \) is shown in Figure~\ref{fig:TypicalNeighbourhoodOfArc}. Each arc \( a \) is the starting vertex of two arrows and the ending vertex of two arrows; note that these arrows need not be distinct.
	We associate with the pair \( (\Sigma,A) \) the path algebras with relations:
	\[ \mathcal{A}=\mathcal{A}(\Sigma,A):=\fieldTwoElements Q(\Sigma,A)/\mathcal{R}\quad\text{and}\quad\mathcal{A}^\perp=\mathcal{A}^\perp(\Sigma,A):=\fieldTwoElements Q(\Sigma,A)/\mathcal{R}^\perp\]
	where 
	\[ \mathcal{R}=\{p_1q_1=0=q_2p_2 \mid \text{ arcs }a\in A\}
	\quad\text{and}\quad
	\mathcal{R}^\perp=\{p_1p_2=0=q_2q_1 \mid \text{ arcs }a\in A\}\]
	Note that we follow the convention to read algebra elements from right to left.
	Every arrow in the quiver corresponds to an algebra element which we call an \textbf{elementary algebra element}. 
	For each arc \( a\in A \), denote the idempotent corresponding to~\( a \) by~\( \iota_a \) and let \( \mathcal{I} \) be the ring of all idempotents.
	We can consider both \( \mathcal{A} \) and \( \mathcal{A}^\perp \) as categories, where the underlying objects are given by the arcs in \( A \) in each case. Then for any \( a, b\in  A \)
	\[ \Mor_{\mathcal{A}}(a, b):=\iota_b.\mathcal{A}.\iota_a\qquad\qquad \Mor_{\mathcal{A}^\perp}(a, b):=\iota_b.\mathcal{A}^\perp.\iota_a\]
	and composition is given by algebra multiplication. 

\begin{remark}\label{rem:comparisonToKontsevich}
    The quivers obtained from an arc system \( A \) and its dual \( A^\perp \) on a surface \( \Sigma \) are the same. Moreover, we have
    \( \mathcal{A}(-\Sigma,A^\perp)=\mathcal{A}^\perp(\Sigma,A) \).
    The relations \( \mathcal{R}^\perp \) in the definition of \( \mathcal{A}^\perp \) resemble those in~\cite{HKK}. In fact the algebra \( \mathcal{A} \) is Koszul dual to the algebra \( \mathcal{A}^\perp \) in the sense of \cite[Definition 8.5]{LOT-mor}. 
    For the once-punctured torus, in particular, the corresponding algebras are opposite to each other; compare \cite{HRW}. In the sequel we will use \( \mathcal{A} \).  
\end{remark}
\begin{definition}
    Given an arc system \( A \) on a surface \( \Sigma \), we define a \( \mathbb{Q}^{\leq0} \)-grading on \( \mathcal{A} \) by setting \( q(p)=-\frac{2}{n_f} \) for any elementary algebra elements \( p \) of a face \(f\). This grading is referred to as the \textbf{quantum grading}; a \( \delta \)-grading is obtained  by setting \( \delta(p)=\tfrac{1}{2}q(p) \). 
\end{definition}

    Let \( \Mod^\mathcal{A} \) denote the category of type~D structures over \( \mathcal{A} \) whose differentials decrease the \( \delta \)-grading by 1 and preserve the quantum grading \( q \).

\begin{example}\label{exa:ThreePuncturedDiscQuiver}
    The quiver for the arc systems in Example~\ref{exa:ThreePuncturedDisc} is 
    \[
		\begin{tikzcd}[row sep=2cm, column sep=1.5cm]
	    b
		\arrow[leftarrow,in=145, out=-145,looseness=5]{rl}[description]{D}
		\arrow[leftarrow,bend left]{r}[description]{S}
		&
		c
		\arrow[leftarrow,bend left]{l}[description]{S}
		\arrow[leftarrow,in=35, out=-35,looseness=5]{rl}[description]{D}
		\end{tikzcd}
  \]
    which agrees with the quiver from Definition~\ref{def:BNAlgH}. 
    In fact, the algebras \( \mathcal{A}(\Sigma,A) \) and \( \BNAlgH \) are equal as bigraded algebras. Thus, \( \Mod^{\mathcal{A}} \) agrees with \( \Mod^{\BNAlgH} \).
\end{example}

\subsection{\texorpdfstring{An algebraic reinterpretation of type~D structures over \( \mathcal{A}(\Sigma,A) \)}{An algebraic reinterpretation of  type~D structures over A(Σ,A)}}\label{subsec:classification:precurves_algebra}
We will make use of an alternate description of the algebra \( \mathcal{A}(\Sigma,A) \) as a subalgebra of a larger algebra that is built out of simpler pieces.
    
    Let \( A \) be an arc system on an oriented surface \( \Sigma \).  For each face \( f\in F(\Sigma,A) \), let \( A_f \) be the free path algebra of the cyclic quiver with \( n_f \) vertices. Each side \( s \) of the face \(f\) corresponds to a vertex and thus to the constant path \( \iota_s \) at that vertex. 
    We denote the vector space generated by the idempotents of \( \mathcal{A}_f \) by \( \mathcal{I}_f \) and the (non-unital) subalgebra generated by paths of non-zero length by \( \mathcal{A}^+_f \).
    These fit into a short exact sequence:
	\[ 
	\begin{tikzcd}[row sep=10pt]
	0
	\arrow{r}{}
	&
	\mathcal{A}^+_f
	\arrow{r}{}
	&
	\mathcal{A}_f
	\arrow{r}{}
	&
	\mathcal{I}_f
	\arrow{r}{}
	&
	0
	\end{tikzcd}
	\]
	By taking the direct sum over all faces \( f\in F(S,M,A) \), we obtain the short exact sequence
	\[ 
	\begin{tikzcd}[row sep=10pt]
	0
	\arrow{r}{}
	&
	\overline{\mathcal{A}}^+
	\arrow{r}{}
	&
	\overline{\mathcal{A}}
	\arrow{r}{}
	&
	\overline{\mathcal{I}}
	\arrow{r}{}
	&
	0
	\end{tikzcd}
	\]
	where $\overline{\mathcal{A}}^+:=\bigoplus_f\mathcal{A}^+_f$, $\overline{\mathcal{A}}:=\bigoplus_f\mathcal{A}_f$, and $\overline{\mathcal{I}}:=\bigoplus_f\mathcal{I}_f$. We can identify \( \mathcal{I} \) with the subring of \( \overline{\mathcal{I}} \) given by \( \{\iota_{s_1(a)}+\iota_{s_2(a)} \mid a\in A\} \). Moreover, \( \mathcal{A} \) agrees with the subalgebra of \( \overline{\mathcal{A}} \) whose underlying vector space is given by \( \mathcal{I}\oplus\overline{\mathcal{A}}^+ \).
    For each face \( f\in F(\Sigma,A) \), let \( p_f \) be the sum of all length 1 paths in \( \mathcal{A}_f \), 
\( U_f:=p_f^{n_f} \)
 and \( U \) the sum of \( U_f \) over all faces \(f\). Our \textbf{standard basis} of \( \mathcal{A}^+_f.\iota_s \) is given by \( \{p_f^n.\iota_s\}_{n\geq1} \). We call \( n \) the \textbf{length of a basis element} \( p_f^n.\iota_s \). We denote the shortest element of the standard basis from a side \( s \) of a face to a side \( t \) of the same face by \( p^s_t \). 

\begin{example}\label{exa:ThreePuncturedDiscQuiver:Decomposed}
    For the quiver from Example~\ref{exa:ThreePuncturedDiscQuiver}, the definition above corresponds to a decomposition into three cyclic quivers:
    \[
    \begin{tikzcd}[row sep=2cm, column sep=1.5cm]
    s_2(b)
    \arrow[leftarrow,in=145, out=-145,looseness=5]{rl}[description]{D}
    &
    s_1(b)
    \arrow[leftarrow,bend left]{r}[description]{S}
    &
    s_2(c)
    \arrow[leftarrow,bend left]{l}[description]{S}
    &
    s_1(c)
    \arrow[leftarrow,in=35, out=-35,looseness=5]{rl}[description]{D}
    \end{tikzcd}
    \]
    Then, the element \( H\in\BNAlgH \) corresponds to \[ U=\iota_{s_2(b)}.D.\iota_{s_2(b)}+\iota_{s_1(b)}.SS.\iota_{s_1(b)}+\iota_{s_2(c)}.SS.\iota_{s_2(c)}+\iota_{s_1(c)}.D.\iota_{s_1(c)}.\]
\end{example}
    
The inclusions \( \mathcal{I}\hookrightarrow\overline{\mathcal{I}} \) and \( \mathcal{A}\hookrightarrow\overline{\mathcal{A}} \) induce a functor between the additive enlargements \( \Mat\mathcal{A} \) of \( \mathcal{A} \) and \( \Mat\overline{\mathcal{A}} \) of \( \overline{\mathcal{A}} \); compare \cite[section~4.2]{pqMod}. Moreover, in order to recover a given object in \( \Mat\mathcal{A} \) from its image under this functor, one only needs to remember how the generators in  idempotents \( \iota_{s_1(a)} \) and \( \iota_{s_2(a)} \) are matched up for each \( a\in A \). The same strategy appears in \cite[Section 3.4]{HRW}, but the surface decomposition used there is different.

\begin{definition}\label{def:SplittingCatsForEquivalenceFunctors}
    Let \( \Mat_i\overline{\mathcal{A}} \) be the category enriched over bigraded vector spaces
    \begin{itemize}
        \item whose objects are pairs \( (C, \{P_a\}_{a\in A}) \), where \( C \) is a bigraded right module over \( \overline{\mathcal{I}} \) and 
        \[ P_a\co C.\iota_{s_1(a)}\rightarrow C.\iota_{s_2(a)}\] 
        is a bigrading preserving vector space isomorphism for every arc \( a \); 
        \item whose morphism spaces \( \Mor((C, \{P_a\}_{a\in A}),(C', \{P'_a\}_{a\in A})) \) are defined as
		\[ \{\varphi\in\Mor_{\Mat\overline{\mathcal{A}}}(C,C') \mid \forall a\in A\co (\iota_{s_2(a)}.\varphi^\times.\iota_{s_2(a)})\circ P_a=P'_a\circ(\iota_{s_1(a)}.\varphi^\times.\iota_{s_1(a)})\},
		\]
		where \( \varphi^\times \) denotes the restriction of \( \varphi \) to its identity components; and 
		\item composition in \( \Mat_i\overline{\mathcal{A}} \) is inherited from \( \Mat\overline{\mathcal{A}} \). 
    \end{itemize} 
\end{definition}  

The categories   \(\Mat\mathcal{A}\) and \(\Mat_i\overline{\mathcal{A}}\) are equivalent \cite[Lemma~4.14]{pqMod}.



\begin{definition}\label{def:precurvesAlgebraic}
    Let us write \( \Mod_i^{\overline{\mathcal{A}}}=\Cx(\Mat_i\overline{\mathcal{A}}) \) for the category of complexes over \( \Mat_i\overline{\mathcal{A}} \). More explicitly, 
    \begin{itemize}
        \item objects are triples \( (C, \{P_a\}_{a\in A},\partial) \), where \( C \) is an object in \( \Mat\overline{\mathcal{I}} \), 
        \[ P_a\co C.\iota_{s_1(a)}\rightarrow C.\iota_{s_2(a)}\] 
        is a bigrading preserving isomorphism for every arc \( a \) and
        \[ \partial\co C\longrightarrow C\]
        defines an endomorphism of \( (C, \{P_a\}_{a\in A}) \) which decreases \( \delta \)-grading by 1, preserves quantum grading and satisfies \( \partial^2=0 \). 
        \item Morphism spaces are morphism spaces of the underlying objects in \( \Mat_i\overline{\mathcal{A}} \). 
        The differentials on these morphism spaces are given as usual by pre- and post-composition with the differentials. 
        \item Finally, \( \Mod_i^{\overline{\mathcal{A}}} \) inherits its composition from \( \Mat_i\overline{\mathcal{A}} \). 
    \end{itemize}
    We call an object in \( \Mod_i^{\overline{\mathcal{A}}} \) a \textbf{precurve}. If there is no risk of ambiguity, we often denote a precurve by \( C \) instead of \( (C, \{P_a\}_{a\in A},\partial) \). We call a precurve \( (C, \{P_a\}_{a\in A},\partial) \) \textbf{fully cancelled} if its differential does not contain any identity component, ie \( \partial^\times=0 \).
\end{definition}
    
Just as \(\Mat\mathcal{A}\) and \(\Mat_i\overline{\mathcal{A}}\) are equivalent, \( \Mod^{\mathcal{A}} \) and \( \Mod_i^{\overline{\mathcal{A}}} \) are equivalent as differential bigraded categories. Moreover, every bigraded type~D structure over \( \mathcal{A} \) is chain homotopic to a fully cancelled one~\cite[Lemma~4.17]{pqMod}. The same holds for precurves. 
    



\subsection{A geometric interpretation of precurves}\label{subsec:classification:precurves_geometry}


A fully cancelled precurve \( C \) may be described geometrically on the surface \( \Sigma \) with arc system \( A \). This corresponds to the train tracks introduced in \cite{HRW}; though the interpretation given here follows \cite{pqMod} more closely. For each side \( s \), we represent the fixed basis \( \{e^s_{i}\}_{i=1,\dots,n} \) of \( C.\iota_s \) by pairwise distinct dots \( \bullet \) on \( s \), which we order according to the orientation of \( s \). We label the \( i^\text{th} \) dot on \( s \) by \( (s,i) \).
Next, since we are assuming that the precurve is fully cancelled, every component
	\[
	\begin{tikzcd}
	e^s_i
	\arrow{r}{p_f^n.\iota_s}
	&
	e^{s'}_{i'}
	\end{tikzcd}
	\]
of the differential lies in \( \mathcal{A}_f^+ \) for some face \(f\).  Hence, we can represent it by an immersed oriented interval connecting the dot \( \bullet(s,i) \) to \( \bullet(s',i') \) on the face \(f\) which is homotopic to the path \( p_f^n.\iota_s \) in the quiver for the face \(f\). We call such an immersed interval a \textbf{two-sided \(f\)-join}. We then connect the dot for each generator of \( C \) which is neither the start nor end of a differential by an immersed curve on the face \(f\) to the inner boundary of \(f\). We call those intervals \textbf{one-sided \(f\)-joins}. 
We define the \textbf{length of an \(f\)-join} to be the length of the corresponding algebra element if the \(f\)-join is two-sided and 0 if it is one-sided. 

\begin{figure}[t]
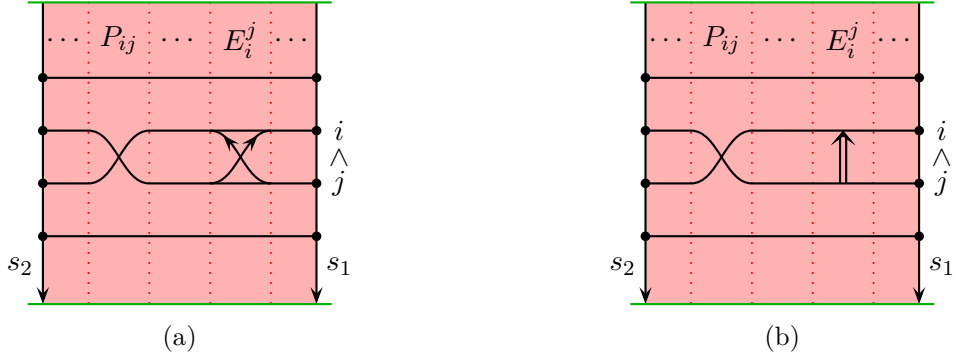

	\centering
	\begin{subfigure}[b]{0.49\textwidth}
		\centering
		$\traintrackneighbourhoodofarc$
		\caption{}\label{fig:traintrackmoves:KWZ}
	\end{subfigure}
	\begin{subfigure}[b]{0.49\textwidth}
		\centering
		$\neighbourhoodofarcHRWconvention$
		\caption{}\label{fig:traintrackmoves:HRW}
	\end{subfigure}
	\caption{Two equivalent views of a precurve in a neighbourhood of an arc: the conventions used in this paper are shown in (a) and the conventions from \cite{HRW} are shown in (b).}\label{fig:traintrackmoves}
\end{figure} 

Finally, each arc \( a\in A \) is labelled by a matrix \( P_a \), and we can also encode these matrices geometrically: consider a decomposition $P_a=P_a^{l_a}\dots P_a^1$
where, for each \( k=1,\dots,l_a \) and some \( i=i(k) \) and \( j=j(k) \) with \( j\neq i \), \( P_a^k \) is either equal to the elementary matrix \( E^{j}_{i}=(\delta_{i'j'}+\delta_{ii'}\delta_{jj'})_{i'j'} \) or the matrix \( P_{ij} \) corresponding to the permutation \( (i,j) \). 
We divide the neighbourhood \( N(a) \) of \( a \) along leaves of \( \mathcal{F}_a \) into \( l_a \)~segments, ordered from \( s_1(a) \) (left) to \( s_2(a) \) (right), and label the \( k^\text{th} \) segment by the matrix \( P_a^k \). Finally, each matrix \( P_a^k \) may be represented graphically as shown in Figure~\ref{fig:traintrackmoves:KWZ}. We call the crossing arcs for \( P_a^k=P_{ij} \) \textbf{crossings} and the pair of arrows for \( P_a^k=E^{j}_{i} \) \textbf{switches} or \textbf{crossover arrows}. This corresponds to the double-arrow notation from \cite[Figure 23]{HRW}, as shown in Figure~\ref{fig:traintrackmoves:HRW}. The advantage of the graphical notation used here is that it extends more easily to the setting where we replace $\fieldTwoElements$ by an arbitrary field; see Section \ref{subsec:classification:signs}.

	The point of the graphical notation for the matrices \( P_a \) is that we can recover the entry in the \( j^\text{th} \) column and \( i^\text{th} \) row of each \( P_a \) by counting (up to homotopy and modulo 2) all paths in the restriction of the precurve to \( N(a) \)
	\begin{enumerate}
		\item which start at \( \bullet(s_1(a),j) \) and end at \( \bullet(s_2(a),i) \),
		\item which are transverse to the foliation \( \mathcal{F}_a \) and 
		\item whose orientation agree with each crossover arrow they traverse. 
	\end{enumerate}
	Since
	\[ P_a^{-1}=(P_a^{l_a}\cdots P_a^1)^{-1}=(P_a^1)^{-1}\cdots(P_a^{l_a})^{-1}=P_a^1\cdots P_a^{l_a},\] 
	we can similarly read off the entry in the \( j^\text{th} \) column and \( i^\text{th} \) row of each \( P_a^{-1} \) by following paths from \( \bullet(s_2(a),j) \) to \( \bullet(s_1(a),i) \) which also satisfy the other two conditions above. 

The decomposition of \( P_a \) into elementary matrices is not unique; neither is the graphical representation of \( P_a \) in terms of crossings and crossover arrows. The relations between the elementary matrices \( E^j_i \) and \( P_{ij} \) give rise to certain moves for crossings and crossover arrows, which do not change the matrix \( P_a \). These moves are given in \cite[Figure 27]{HRW}; see also~\cite[Figure 26]{pqMod} (one such is illustrated in the example of Figure \ref{fig:simply-faced}). Rather than repeat these moves here, the reader can deduce them from Figure \ref{fig:traintrackmovesSigns} (which deals, more generally, with coefficients over an arbitrary field $\field$), by specializing the coefficient system in that figure to $\F$. 
Up to these moves and isotopies of the immersed curves, the above geometric interpretation of a precurve is unique and also uniquely defines a precurve. 

\begin{definition}\label{def:simplyfaced}
    A precurve is \textbf{simply-faced} if it is fully cancelled and every dot on each side lies on exactly one \(f\)-join.
\end{definition}

\newsavebox{\smlmat}
\savebox{\smlmat}{$\left(\begin{smallmatrix}1&0\\1&1\end{smallmatrix}\right)\left(\begin{smallmatrix}0&1\\1&0\end{smallmatrix}\right)=\left(\begin{smallmatrix}1&1\\0&1\end{smallmatrix}\right)\left(\begin{smallmatrix}1&0\\1&1\end{smallmatrix}\right)$}

\labellist \tiny
\pinlabel $D$ at 40 170  \pinlabel $D$ at 230 170 
\pinlabel $S$ at 120 165  \pinlabel $S$ at 120 272
\pinlabel {\bf M3a} at 52 104 \pinlabel {\bf M3b} at 137 107
\endlabellist
\begin{figure}[b]
\includegraphics[scale=0.75]{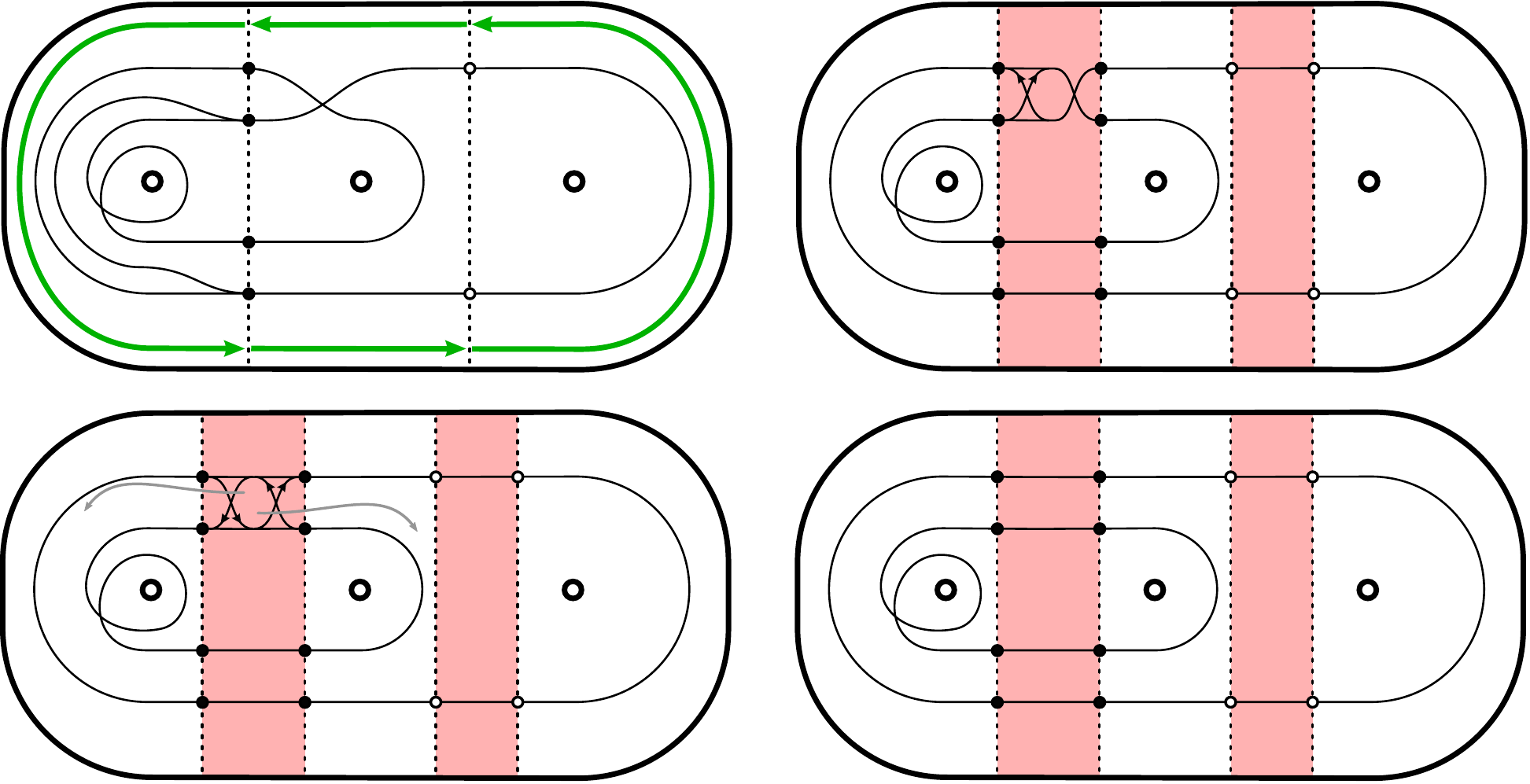}
\caption{Four views of one type~D structure over $\BNAlgH$:  the top left describes a type~D structure graphically, immersed in the 3-punctured disk as a train track (compare \cite{HRW}). This same type~D structure is expressed at top right as a simply-faced precurve, and given an equivalent definition at lower left; this equivalence corresponds to the choice of decomposition of $P_a$ according to~\usebox{\smlmat}. 
Finally, by applying Lemma \ref{lem:CalculusForPreloops} as indicated, we may cancel these switches, replacing the precurve with a pair of curves shown at lower right.}\label{fig:simply-faced}
\end{figure}

Said another way, a simply-faced precurve, when interpreted as a train track, has all of the complicated switching data confined to the arc neighbourhoods; see Figure \ref{fig:simply-faced}. Precurves in this form are more easily managed, and we may assume that precurves are simply-faced owing to the following:  


\begin{proposition}\label{prop:PreloopToCC}
	Every fully cancelled precurve is chain isomorphic to a simply-faced precurve.
\end{proposition}

The proof of this result follows very closely that of \cite[Proposition~4.16]{pqMod} (see also \cite[Proposition~23]{HRW}) so we omit it. Towards simplifying precurves, we introduce the following moves (illustrated in Figure \ref{fig:GraphicCalculus}):

\begin{figure}[t]
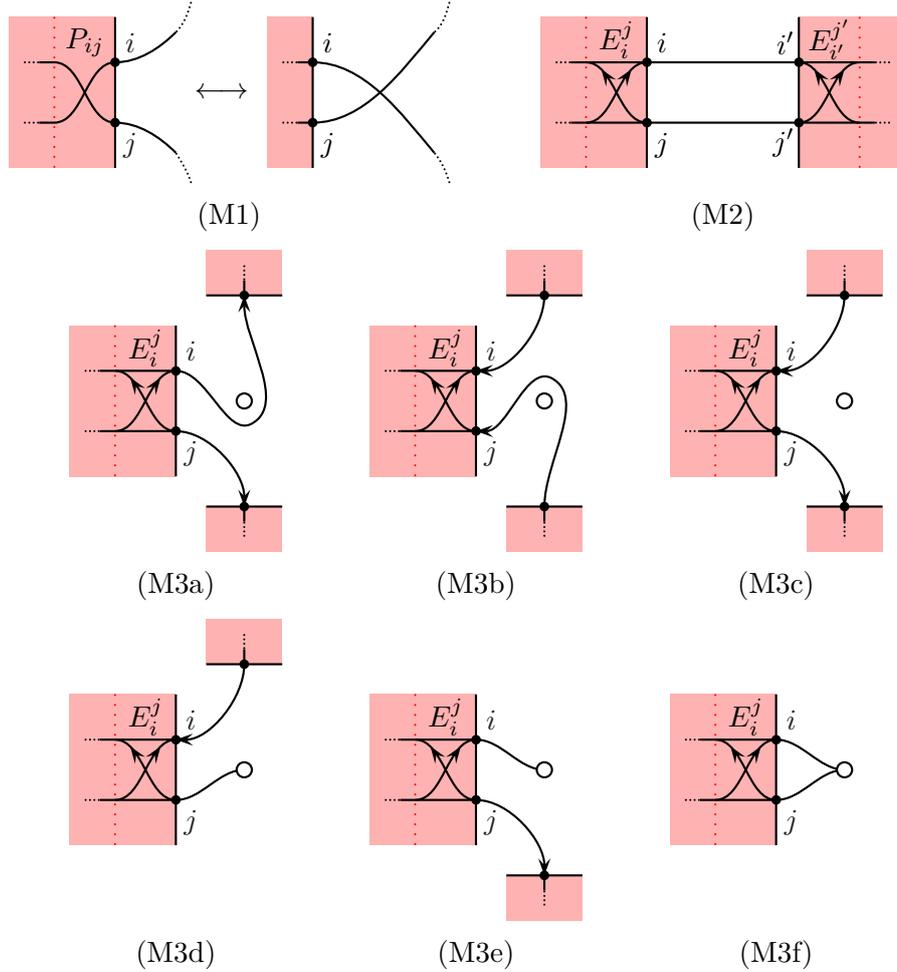

	\centering
	\begin{subfigure}{0.4\textwidth}\centering
		$\CalculusForPreloopsMi$\\
		(M1)
	\end{subfigure}
	\begin{subfigure}{0.4\textwidth}\centering
		$\CalculusForPreloopsMii$\\
		(M2)
	\end{subfigure}
	\\
	\begin{subfigure}{0.24\textwidth}\centering
		$\CalculusForPreloopsMiiiA$\\
		(M3a)
	\end{subfigure}
	\begin{subfigure}{0.24\textwidth}\centering
		$\CalculusForPreloopsMiiiB$\\
		(M3b)
	\end{subfigure}
	\begin{subfigure}{0.24\textwidth}\centering
		$\CalculusForPreloopsMiiiC$\\
		(M3c)
	\end{subfigure}
	\\
	\begin{subfigure}{0.24\textwidth}\centering
		$\CalculusForPreloopsMiiiD$\\
		(M3d)
	\end{subfigure}
	\begin{subfigure}{0.24\textwidth}\centering
		$\CalculusForPreloopsMiiiE$\\
		(M3e)
	\end{subfigure}
	\begin{subfigure}{0.24\textwidth}\centering
		$\CalculusForPreloopsMiiiF$\\
		(M3f)
	\end{subfigure}
	\caption{The graphical calculus for precurves: (M1) allows us to move a crossing into or out of and arc neighbourhood; (M2) allows us to cancel pairs of crossover arrows (or, equivalently, move crossover arrows along a parallel set of arcs in come face); and the (M3a)--(M3f) describe configurations of arrows that can be removed. Note that in general, this says that arrows traveling clockwise around a puncture can be removed.}\label{fig:GraphicCalculus}
\end{figure}

\begin{description}
		\item[(M1)] Multiply \(P_a\) on the right/left by \(P_{ij}\), depending on whether \(s=s_1(a)\) or \(s=s_2(a)\), and switch the endpoints of the two  two-sided \(f\)-joins ending in  \(\bullet(s,i)\) and \(\bullet(s,j)\). 
		\item[(M2)] Suppose \(\bullet(s,i)\) and \(\bullet(s,j)\) have the same bigrading. Suppose further that \(\bullet(s,i)\) and \(\bullet(s,j)\) lie on two (distinct) two-sided \(f\)-joins which are homotopic to each other relative to the sides of \(f\). Say the other endpoints of the two \(f\)-joins are \(\bullet(s',i')\) and \(\bullet(s',j')\), respectively. Then multiply \(P_a\) on the right/left by \(E^j_{i}\), depending on whether \(s=s_1(a)\) or \(s=s_2(a)\), and multiply \(P_{a'}\) on the right/left by \(E^{j'}_{i'}\), depending on whether \(s'=s_1(a')\) or \(s'=s_2(a')\).  
		\item[(M3)] Suppose \(\bullet(s,i)\) and \(\bullet(s,j)\) have the same bigrading. Suppose further that \(\bullet(s,i)\) and \(\bullet(s,j)\) lie on two (distinct) \(f\)-joins which diverge, ie one of the following applies: 
		\begin{enumerate}[label=(\alph*)]
		    \item Both \(f\)-joins are two-sided and start at \(\bullet(s,i)\) and \(\bullet(s,j)\), respectively, and the first \(f\)-join is strictly longer than the second. 
		    \item Both \(f\)-joins are two-sided and end at \(\bullet(s,i)\) and \(\bullet(s,j)\), respectively, and the second \(f\)-join is strictly longer than the first.  
		    \item Both \(f\)-joins are two-sided and the first ends at \(\bullet(s,i)\) and the second starts at \(\bullet(s,j)\). 
		    \item The first \(f\)-join is two-sided and ends at \(\bullet(s,i)\) and the second \(f\)-join is one-sided. 
		    \item The second \(f\)-join is two-sided and starts at \(\bullet(s,j)\) and the first \(f\)-join is one-sided. 
		    \item Both \(f\)-joins are one-sided. 
		\end{enumerate}
		Then multiply \(P_a\) on the right/left by \(E^j_{i}\) depending on whether \(s=s_1(a)\) or \(s=s_2(a)\). 
	\end{description}

These moves are sufficient to relate any two chain isomorphic precurves, according to the following lemma, which corresponds to \cite[Lemma~4.26]{pqMod}.

\begin{lemma}\label{lem:CalculusForPreloops}
	Let \((C, \{P_a\}_{a\in A},\partial)\) be a simply-faced precurve on a surface \(\Sigma\) with arc system~\(A\). For \(a\in A\), let \(s\) be a side of \(a\) and \(f\) the face adjacent to \(s\). Let \(\bullet(s,i)\) and \(\bullet(s,j)\) be two distinct dots on \(s\). Then \((C, \{P_a\}_{a\in A},\partial)\) is chain isomorphic to the precurve obtained by one of the moves (M1), (M2), or (M3).\qed
\end{lemma}



Lemma \ref{lem:CalculusForPreloops}, ultimately, allows us to simplify an arbitrary simply-faced precurve considerably. To this end, our main object of interest is:

\begin{definition}\label{def:curves}
	Let \( \Sigma \) be a surface with an arc system \( A \). A \textbf{curve} on \( (\Sigma,A) \) is a pair \( (\gamma, X) \), where either
	\begin{enumerate}
		\item \( \gamma \) is an immersion of an oriented circle into \( \Sigma \), representing a non-trivial primitive element of \( \pi_1(\Sigma) \) and 
		\( X\in \GL_n(\fieldTwoElements) \) for some \( n \); or \label{def:curves:loops}
		\item \( \gamma \) is an immersion of an interval \( (I,\partial I) \) into \( (\Sigma,\partial_i\Sigma) \), defining a non-trivial element of \( \pi_1(\Sigma,\partial_i \Sigma) \) and 
		\( X=\id\in\GL_n(\fieldTwoElements) \) for some positive integer \( n \) \label{def:curves:paths}
	\end{enumerate}
	satisfying the following properties:
	\begin{itemize}
		\item \( \gamma \) restricted to each component of the preimage of each face is an embedding and 
		\item \( \gamma \) restricted to each component of the preimage of the neighbourhood \( N(a) \) of each arc \( a \) is an embedding, intersecting each leaf of \( \mathcal{F}_a \) exactly once.
	\end{itemize}
	\end{definition}
In case \eqref{def:curves:loops}, we call the curve \textbf{compact}, in case \eqref{def:curves:paths} \textbf{non-compact}. We call 
	\( X \) the \textbf{local system} of a curve \( (\gamma,X) \), which for  non-compact curves only records some positive integer \( n \). We consider curves up to homotopy through curves. Furthermore, we consider the local systems of compact curves up to matrix similarity, and also up to transposing-and-inverting the matrix together with inverting the orientation of \( \gamma \). As in~\cite[Definition~4.28]{pqMod}, we define a \textbf{bigraded curve} in terms of a consistent bigrading on the intersection points of the curve with the arcs in \( A \). 
	A \textbf{multicurve} is a finite set of bigraded curves such that all underlying curves are pairwise non-homotopic as unoriented curves. We denote the set of all such multicurves by \( \loops(\Sigma,A) \).

	
	Given an arc system \( A \) on a surface \( \Sigma \), we define a map
	\[ \Pi_i\co\loops(\Sigma,A)\rightarrow \ob\Mod_i^{\overline{\mathcal{A}}}/\text{(chain homotopy)}\]
	as follows: consider a single curve \( (\gamma,X) \) and choose a small immersed tubular neighbourhood of \( \gamma \). Replacing \( \gamma \) by \( \dim X \) parallel copies of \( \gamma \) in this neighbourhood, in each face \( f\in F(\Sigma,A) \) the joins of~\( \Pi_i(\gamma,X) \) are given by the intersection of these  \( \dim X \) parallel copies with~\(f\). Then pick an intersection point \( x \) of an arc \( a \) with \( \gamma \). Let the matrix \( P_a \) be the diagonal block matrix with blocks of dimension \( \dim X \) such that all blocks are equal to the identity matrix with the exception of the one corresponding to the intersection point \( x \). We define this block to be equal to \( X \) if \( \gamma \) goes through \( x \) from the left of \( a \) to its right (ie from \( s_1(a) \) to \( s_2(a) \)), and set it equal to \( (X^t)^{-1} \) otherwise. On all other arcs, we choose the identity matrix. 
	Finally, we extend \( \Pi_i \) to multicurves by taking unions on both sides. 
	
	Note that the definition of \( \Pi_i \) is independent of the choice of \( a \) and \( x \) up to homotopy in \( \Mod_i^{\overline{\mathcal{A}}} \), which can be seen by repeatedly applying (M2).  Similarly, we see that conjugation of the local system \( X \) of a loop \( (\gamma,X) \) does not change the homotopy type of the image under \( \Pi_i \).
	
	Since the categories of precurves and type~D structures over \( \mathcal{A} \) are equivalent, \( \Pi_i \) induces a map
	\[ \Pi\co \loops(\Sigma,A)\rightarrow\ob\Mod^{\mathcal{A}}/\text{(chain homotopy)}.\]
  
  \begin{remark}Restricting for a moment to 1-dimensional (trivial) local systems, the generators of $\Pi(\gamma)$ are in one-to-one correspondence with intersections between $\gamma$ and the arcs of $A$; note that the arc system $A$ forms a \emph{skeleton} in the sense that $\Sigma$ deformation retracts to $A \cup \{\text{outer boundary of }\Sigma\}$. We point out that the strategy of obtaining generators of a twisted complex $\Pi(\gamma)$ from intersections of $\gamma$ with a skeleton works in higher dimensions: in \cite{CDGG} the authors prove that an object in the wrapped Fukaya category of a Weinstein manifold, defined by a Lagrangian submanifold, is quasi-isomorphic to a twisted complex generated by the intersections of the Lagrangian with the skeleton.\end{remark}

Our main observation is that precurves can be simplified to curves: 

\begin{theorem}\label{thm:EverythingIsLoopTypeUpToLocalSystems}
	Any bigraded type~D structure over \( \mathcal{A} \) is homotopic to a type~D structure in the image of \(\Pi\). 
\end{theorem}

\begin{proof}
This is essentially proved in \cite{HRW}; see \cite{pqMod} for a setup closer to the one used here. The main technical result underpinning this is an algorithm that allows us to remove all of the crossover arrows with the exception of crossover arrows connecting parallel curves \cite[Proposition 29]{HRW}; Figure \ref{fig:simply-faced} shows a particular, very simple example in the case when the surface in question is a 4-punctured sphere. A new feature that arises in the present setting is that precurves may wind around the inner boundary components more than once, so in particular, they may return to the same side. Note, however, that the quantum grading ensures that those ends cannot be connected by crossings or crossover arrows, so they can be treated as if they were on different sides. 
\end{proof}

We will return to this result, and the key aspects of its proof, in Section \ref{subsec:classification:signs}, for the more general setting in which $\F$ is replaced by an arbitrary field $\field$. 


\subsection{Classification of morphisms between simply-faced precurves}\label{subsec:classification:morphisms}

The classification of morphisms is very similar to~\cite[Section~4.5]{pqMod}. We first homotope precurves into some standard pairing position. However, since it is slightly harder to keep track of intersection points when curves can wrap around inner boundary components multiple or even infinitely many times, the setup in this subsection is more elaborate than in~\cite{pqMod}.

    For each face \(f\) of a surface \( \Sigma \) with arc system \( A \), we identify \( f\smallsetminus\partial_i\Sigma \) with \( S^1\times[0,\infty) \), parametrized by polar coordinates via
    \begin{align*}
        \varphi_f\co \mathbb{R}\times[0,\infty)&\rightarrow S^1\times [0,\infty)\\
        (t_1,t_2)&\mapsto\left(\exp(\tfrac{2\pi i}{n_f}\cdot t_1),t_2\right)
    \end{align*}
    We do this in such a way that for a suitably choosen \( \varepsilon>0 \), the interval \( [s-2\varepsilon, s+2\varepsilon]\times{0} \) lies on the \( s^\text{th} \) side of the face \(f\). 
    \begin{definition}\label{def:MorSpaces}
    We say a bigraded precurve \( C \) on a surface \(\Sigma\) with arc system \(A\) is in a \textbf{first pairing position} if the following holds:
	\begin{enumerate}
	    \item For each face \(f\), the intersection points of \( C \) with the \( s^\text{th} \) side of \(f\) lies in an open \( \varepsilon \)-neighbourhood of the point \( \varphi_f(s+\varepsilon) \). \label{enu:pairingposSides}
	    \item On\label{enu:pairingposArcs} the neighbourhood \( N(a) \) of each arc \( a\in A \), the precurve stays in a small tubular neighbourhood of a straight line between the points chosen in~(\ref{enu:pairingposSides}).
	    \item The\label{enu:pairingposOuterJoinsI}  two-sided \(f\)-joins of \( C \) of absolute length \( a \) starting at the \( s^\text{th} \) side of \(f\) lie in a small tubular neighbourhood of the path 
	    \begin{align*}
	        \vartheta^{+}_f(s,a)\co [0,2a]&\rightarrow S^1\times [0,\infty)\\
	        t&\mapsto
	        \begin{cases}
	        \varphi_f(s+\varepsilon+t,t)& 0\leq t\leq a\\
	        \varphi_f(s+\varepsilon+a,2a-t)& a\leq t\leq 2a
	        \end{cases}
	    \end{align*}
	    The red curve in Figure~\ref{fig:MorphismsBetweenFJoinsTwoTwoIllustrated} shows a typical example of a path \(\vartheta^{+}_f(s,a)\).
	    \item The\label{enu:pairingposInnerJoinsI} one-sided \(f\)-joins of \( C \) starting at the \( s^\text{th} \) side of \(f\) lie in a small tubular neighbourhood of the wrapped ray
	    \begin{align*}
	        \vartheta^{+}_f(s,\infty)\co[0,\infty)&\rightarrow S^1\times [0,\infty)\\
	        t&\mapsto\varphi_f(s+\varepsilon+t,t)
	    \end{align*}
	    The red curve in Figure~\ref{fig:MorphismsBetweenFJoinsOneTwoIllustrated} shows a typical example of a ray \(\vartheta^{+}_f(s,\infty)\).
	\end{enumerate}
	We say a bigraded curve \( C \) on a surface \(\Sigma\) with arc system \(A\) is in a \textbf{second pairing position} if the above holds after replacing \( \varepsilon \) by \( -\varepsilon \) in conditions (\ref{enu:pairingposSides})--(\ref{enu:pairingposOuterJoinsI}), renaming the paths in condition (\ref{enu:pairingposOuterJoinsI})  \( \vartheta^{-}_f(s,a) \) and replacing condition (\ref{enu:pairingposInnerJoinsI}) by the following condition:
	\begin{enumerate}
	    \item[(4')] The\label{enu:pairingposInnerJoinsII} one-sided \(f\)-joins of \( C \) starting at the \( s^\text{th} \) side of \(f\) lie in a small tubular neighbourhood of the straight ray 
	    \begin{align*}
	        \vartheta^{-}_f(s,\infty)\co[0,\infty)&\rightarrow S^1\times [0,\infty)\\
	        t&\mapsto\varphi_f(s-\varepsilon,t)
	    \end{align*}
	\end{enumerate}
\end{definition}
For examples of curves in a second pairing position, see the blue curves in Figures~\ref{fig:MorphismsBetweenFJoinsTwoTwoIllustrated} and~\ref{fig:MorphismsBetweenFJoinsTwoOneIllustrated}, which show a path \( \vartheta^{-}_f(s'-a',a') \) and a ray \( \vartheta^{-}_f(s',\infty) \), respectively.
	We say a pair \( (C,C') \) of two bigraded curves \( C=(C, \{P_a\}_{a\in A},\partial) \) and \( C'=(C', \{P'_a\}_{a\in A},\partial') \) is in \textbf{pairing position}, if \( C \) is in a first pairing position, \( C' \) is in a second pairing position, the crossover arrows and crossings for \( P_a \) and \( P'_a \) lie in a small neighbourhood of the second side \( s_2(a) \) for each arc \( a \) and the two precurves intersect minimally with respect to these two positions. The following two definitions correspond to \cite[Definitions~4.35 and~4.37]{pqMod}.

\begin{figure}[t]
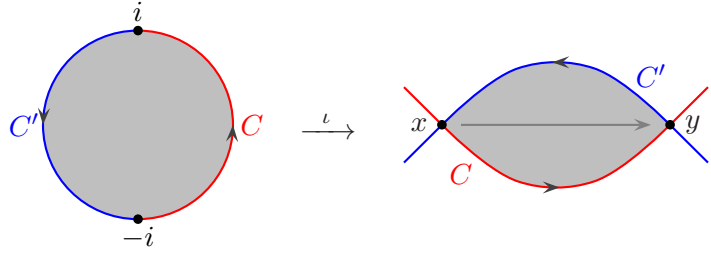
\centering
$\bigonConventions$
\caption{Our orientation conventions for bigons: \( \iota \) maps the unit disc in \( \mathbb{C} \) to \( \Sigma \). In both cases, the normal vector, determined by the right-hand rule, points out of the projection plane. The arrows on the boundary of the disc and bigon indicate the induced boundary orientations. The generator \( x \) is an upper generator (so it lies in the neighbourhood of an arc) and \( y \) is a lower generator (so it lies in some face).  }\label{fig:LagrangianHFConventions}
\end{figure}

\begin{definition}\label{def:LagrangianFH}
With a pair of two simply-faced precurves \( C \) and \( C' \) in pairing position, we associate a chain complex \( \CF(C,C') \) as follows:
as a vector space over \( \fieldTwoElements \), \( \CF(C,C') \) decomposes into two summands \( \CFtimes(C,C') \) and \( \CFplus(C,C') \). The former is generated by the intersection points between the curves in the neighbourhoods of the arcs, the latter by those on faces. We call the former \textbf{upper} and the latter \textbf{lower} intersection points/generators of \( \CF(C,C') \). 
The differential on \( \CF(C,C') \) is a map 
\[ d\co\CFtimes(C,C')\rightarrow\CFplus(C,C')\]
defined by counting bigons connecting upper intersection points to lower ones. More precisely, a bigon from an upper intersection point \( x \) to a lower one \( y \) is an orientation-preserving embedding
\[ \iota\co D^2\hookrightarrow \Sigma\]
satisfying the following properties:
\begin{itemize}
\item the restriction of \( \iota \) to the non-negative real part of \( \partial D^2 \) is a path from \( x=\iota(-i) \) to \( y=\iota(+i) \) on the first precurve \( C \) such that the orientation is opposite to the orientation of any crossover arrows in \( C \);
\item the restriction of \( \iota \) to the non-positive real part of \( \partial D^2 \) is a path from \( y \) to \( x \) on the second precurve \( C' \) such that the orientation is opposite to the orientation of any crossover arrows in \( C' \);
\item \( x \) and \( y \) are convex corners of the image of \( \iota \). 
\end{itemize}
See Figure~\ref{fig:LagrangianHFConventions} for an illustration of the above conventions.
If \( \mathcal{M}(x,y) \) denotes the set of such bigons up to reparameterization, the differential \(d\) is defined by
\[ d(x)=\sum \#\mathcal{M}(x,y)~y.\]
By construction, it only connects upper generators to lower ones. We denote the homology of \( \CF(C,C') \) by \( \HF(C,C') \) and call it the \textbf{wrapped Lagrangian Floer homology} of \( C \) and \( C' \).
\end{definition}

\begin{definition}\label{def:resolution}
With an intersection point \( x \) between two simply-faced precurves \( C \) and \( C' \) in pairing position, we may associate a morphism of precurves \( \varphi(x) \) from \( C \) to \( C' \) as follows. Consider the union of all paths \( \gamma\co [0,1]\rightarrow C\cup C' \) (considered up to homotopy) satisfying the following conditions: 
\begin{enumerate}
\item the restriction of \( \gamma \) to \( [0,\frac{1}{2}] \) is a path on \( C \) from a dot \( \bullet(s,i) \) to \( x \) which does not meet any other dot on \( C \) and which follows the orientation of any crossover arrows,
\item the restriction of \( \gamma \) to \( [\frac{1}{2},1] \) is a path on \( C' \) from \( x \) to a dot \( \bullet(s',i') \) which does not meet any other dot on \( C' \) and which follows the orientation of any crossover arrows,
\item \( \gamma \) turns right at the intersection point \( x \): $\resolution$.
\end{enumerate}
If such a path can be chosen to lie entirely inside an arc neighbourhood, we label it by the corresponding idempotent. Otherwise, the path can be chosen to lie entirely inside a face \(f\).  In this case, if the path winds around the inner boundary component of \(f\) in anticlockwise direction, it corresponds to an element of \( \mathcal{A}_f \), which we choose as its label. If the path winds in clockwise direction, we ignore it. Finally, counting the remaining paths modulo 2, we may regard them as a morphism \( \varphi(x) \) from \( C \) to \( C' \), which we call the \textbf{resolution of \( x \)}. 
\end{definition}

The following lemmas are proved in the same way as~\cite[Lemmas~4.38, 4.39 and 4.41]{pqMod}.

\begin{lemma}\label{lem:ResolutionWellDefined}
The resolution \(\varphi(x)\) is a well-defined morphism of precurves from \(C\) to \(C'\).\qed
\end{lemma}

\begin{lemma}
The resolution \(\varphi(x)\) of any generator \(x\) of \(\CF(C,C')\) is homogeneous with respect to the bigrading.\qed
\end{lemma}

\begin{definition}
We endow \( \CF(C,C') \) with a \textbf{bigrading} by defining the bigrading of any intersection point to be the bigrading of its resolution. 
\end{definition}

\begin{lemma}
The differential \(d\) on \(\CF(C,C')\) decreases the \(\delta\)-grading by 1 and preserves the quantum grading.\qed
\end{lemma}

Given two fully cancelled precurves \( C \) and \( C' \), let us write \( (\Mor^+(C,C'),D^+) \) for the subcomplex of \( (\Mor(C,C'),D) \) generated by those morphisms which do not contain an identity component. Similarly, let \( \Mor^\times(C,C') \) be the vector space of all morphisms that consist of identity components only. Then \( (\Mor(C,C'),D) \) is naturally chain isomorphic to the cone of the morphism 
\[ 
 \beta\co
 (\Mor^\times(C,C'),0)
 \hookrightarrow
 (\Mor(C,C'),D)
 \xrightarrow{~D~}
 (\Mor(C,C'),D)
 \twoheadrightarrow
 (\Mor^+(C,C'),D^+),
\]
where the first and third maps are induced by the inclusion \( \mathcal{A}^\times\hookrightarrow\mathcal{A} \) and surjection \( \mathcal{A}\twoheadrightarrow\mathcal{A}^+ \), respectively. Since \( D^+D=0 \), \( \beta \) induces a homomorphism \( \beta_\ast \) on homology. Naturally, this homomorphism agrees with the boundary map of the long exact sequence induced by the short exact sequence of chain complexes
\begin{equation*}
0
\rightarrow
(\Mor^+(C,C'),D^+)
\hookrightarrow
(\Mor(C,C'),D)
\twoheadrightarrow
(\Mor^\times(C,C'),0)
\rightarrow
0.
\end{equation*}
So in summary, we obtain
\begin{align*}
\Homology(\Mor(C,C'),D)
&\cong  
\Homology\left[
(\Mor^\times(C,C'),0)
\xrightarrow{~\beta~}
(\Mor^+(C,C'),D^+)
\right]
\\
&\cong 
\Homology\left[
\Mor^\times(C,C')
\xrightarrow{~\beta_\ast~}
\Homology(\Mor^+(C,C'),D^+)
\right].
\end{align*}
We are now ready to state and prove the central result of this subsection; compare~\cite[Theorem~4.43]{pqMod}:

\begin{theorem}\label{thm:PairingMorLagrangianFH}
Given a pair of simply-faced precurves \(C\) and \(C'\) in pairing position on a surface \(\Sigma\) with arc system \(A\), the chain complex \(\CF(C,C')\) and the mapping cone of the map
\[
\beta_\ast\co \Mor^\times(C,C')
\longrightarrow
\Homology(\Mor^+(C,C'),D^+)
\]
are bigraded chain isomorphic via the correspondence between generators of \(\CF(C,C')\) and their resolutions. In particular, 
\(\HF(C,C')\cong \Homology(\Mor(C,C'),D).\)
\end{theorem}

\begin{figure}[b]
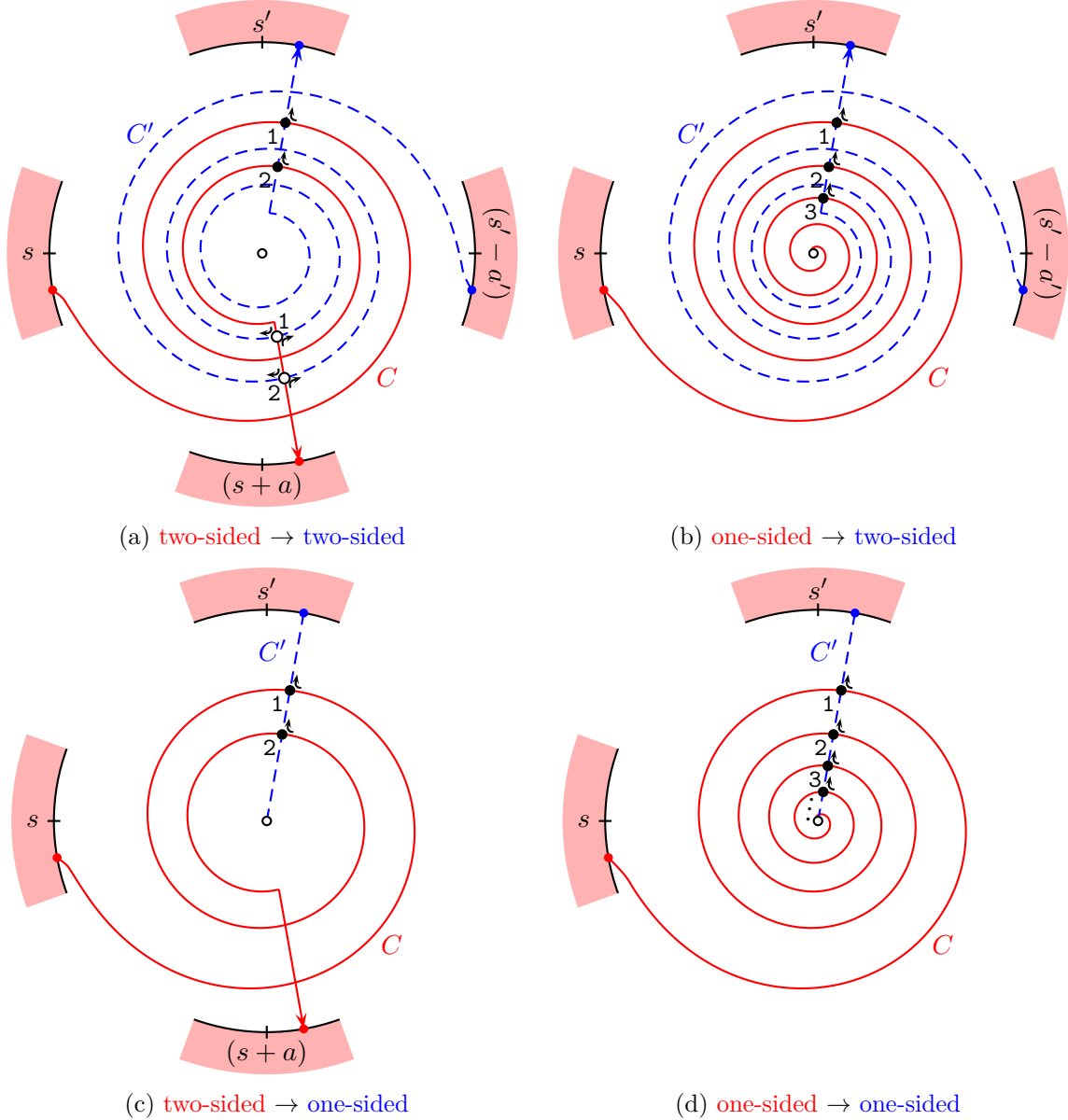

	\centering
	\begin{subfigure}{0.48\textwidth}
    	\centering
		$\MorCaseA$
		\caption{\textcolor{red}{two-sided} \( \rightarrow \) \textcolor{blue}{two-sided}}\label{fig:MorphismsBetweenFJoinsTwoTwoIllustrated}
	\end{subfigure}
	\begin{subfigure}{0.48\textwidth}
		\centering
		$\MorCaseB$
		\caption{\textcolor{red}{one-sided} \( \rightarrow \)  \textcolor{blue}{two-sided}}\label{fig:MorphismsBetweenFJoinsOneTwoIllustrated}
	\end{subfigure}
	\medskip
	\\
	\begin{subfigure}{0.48\textwidth}
    	\centering
		$\MorCaseC$
		\caption{\textcolor{red}{two-sided} \( \rightarrow \) \textcolor{blue}{one-sided}}\label{fig:MorphismsBetweenFJoinsTwoOneIllustrated}
	\end{subfigure}
	\begin{subfigure}{0.48\textwidth}
		\centering
		$\MorCaseD$
		\caption{\textcolor{red}{one-sided} \( \rightarrow \) \textcolor{blue}{one-sided}}\label{fig:MorphismsBetweenFJoinsOneOneIllustrated}
	\end{subfigure}
	\caption{An illustration of the four different cases in the local computation of the homology of \( \Mor^+(C,C') \) from the proof of Theorem~\ref{thm:PairingMorLagrangianFH}; see also Figure~\ref{fig:MorphismsBetweenFJoins} for a graphical representation for the same four cases. The integer labels of the intersection points indicate in which direction the indices \( l \) from the proof increase on the straight line segments of the \( f \)-joins, see Remark~\ref{rem:LagHFindexing}. Depending on \( s \), \( s' \), \( a \), and \( a' \), the actual indices for each straight line segment might be shifted by some integer.} \label{fig:MorphismsBetweenFJoinsIllustrated}
\end{figure}

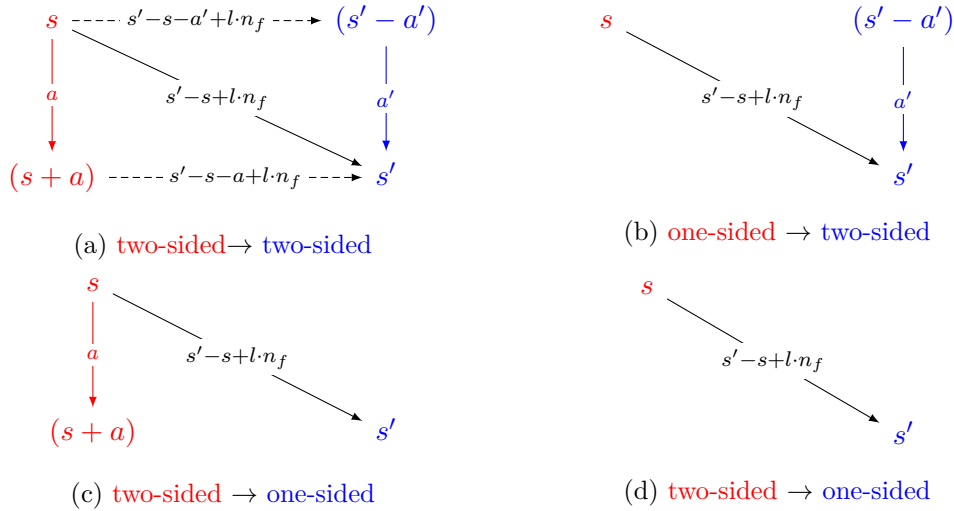
\begin{figure}[b]
	\centering
	\begin{subfigure}{0.45\textwidth}
		\[ 
		\begin{tikzcd}[column sep=80pt,row sep=40pt]
		\textcolor{red}{s}
		\arrow[red]{d}[description]{a}
		\arrow[dashed]{r}[description]{s'-s-a'+l\cdot n_f}
		\arrow{rd}[description]{s'-s+l\cdot n_f}
		&
		\textcolor{blue}{(s'-a')}
		\arrow[blue]{d}[description]{a'}
		\\
		\textcolor{red}{(s+a)}
		\arrow[dashed]{r}[description]{s'-s-a+l\cdot n_f}
		&
		\textcolor{blue}{s'}
		\end{tikzcd}
		\]
		\caption{\textcolor{red}{two-sided}\( \rightarrow \)  \textcolor{blue}{two-sided}}\label{fig:MorphismsBetweenFJoinsOuterOuter}
	\end{subfigure}
	\begin{subfigure}{0.45\textwidth}
		\[ 
		\begin{tikzcd}[column sep=80pt,row sep=40pt]
		\textcolor{red}{s}
		\arrow{rd}[description]{s'-s+l\cdot n_f}
		&
		\textcolor{blue}{(s'-a')}
		\arrow[blue]{d}[description]{a'}
		\\
		&
		\textcolor{blue}{s'}
		\end{tikzcd}
		\]
		\caption{\textcolor{red}{one-sided} \( \rightarrow \)  \textcolor{blue}{two-sided}}\label{fig:MorphismsBetweenFJoinsInnerOuter}
	\end{subfigure}
	\\
	\begin{subfigure}{0.45\textwidth}
		\[ 
		\begin{tikzcd}[column sep=80pt,row sep=40pt]
		\textcolor{red}{s}
		\arrow[red]{d}[description]{a}
		\arrow{rd}[description]{s'-s+l\cdot n_f}
		\\
		\textcolor{red}{(s+a)}
		&
		\textcolor{blue}{s'}
		\end{tikzcd}
		\]
		\caption{\textcolor{red}{two-sided} \(  \rightarrow \) \textcolor{blue}{one-sided}}\label{fig:MorphismsBetweenFJoinsOuterInner}
	\end{subfigure}
	\begin{subfigure}{0.45\textwidth}
		\[ 
		\begin{tikzcd}[column sep=80pt,row sep=40pt]
		\textcolor{red}{s}
		\arrow{rd}[description]{s'-s+l\cdot n_f}
		\\
		&
		\textcolor{blue}{s'}
		\end{tikzcd}
		\]
		\caption{\textcolor{red}{two-sided} \( \rightarrow \) \textcolor{blue}{one-sided}}\label{fig:MorphismsBetweenFJoinsInnerInner}
	\end{subfigure}
	\caption{A graphical representation of morphisms between two \(f\)-joins, illustrating the four different cases in the local computation of the homology of \( \Mor^+(C,C') \) from the proof of Theorem~\ref{thm:PairingMorLagrangianFH}. Generators are labelled by their sides (modulo \( n_f \)) and morphisms by their lengths.}\label{fig:MorphismsBetweenFJoins}
\end{figure}

\begin{proof}
	The identification of \( \Mor^\times(C,C') \) with upper intersection points between \( C \) and \( C' \) works just as in the proof of~\cite[Theorem~4.43]{pqMod}.
	The identification of \( \Homology(\Mor^+(C,C'),D^+) \) with lower intersection points requires more work. As in~\cite{pqMod}, we can focus on a single \(f\)-join of \( C \) and a single \(f\)-join of \( C' \) for the same face \(f\). Then there are some \( s,s',a,a'\in\mathbb{Z} \), such that the \(f\)-join of \( C \) lies in the neighbourhood of \( \alpha\in\{\vartheta^{+}_f(s,a),\vartheta^{+}_f(s,\infty)\} \) and the \(f\)-join of \( C' \) lies in the neighbourhood of \( \alpha'\in\{\vartheta^{-}_f(s'-a',a'),\vartheta^{-}_f(s',\infty)\} \). 
	We want to show that the resolutions of intersection points of these two \(f\)-joins form a basis of the homology of the morphism space between the type~D structures over \( \mathcal{A}_f^+ \) associated with the two \(f\)-joins. 
	For this, we distinguish four different cases, depending on whether the \(f\)-joins are one- or two-sided. For simplicity, we will identify the \(f\)-joins with \( \alpha \) and \( \alpha' \), respectively. 
    
    \begin{description}
    	\item[Case (a)] Suppose \( (\alpha,\alpha')=(\vartheta^{+}_f(s,a),\vartheta^{-}_f(s'-a',a')) \). This case is illustrated in Figure~\ref{fig:MorphismsBetweenFJoinsTwoTwoIllustrated}. The first segments of the two paths, ie those that wind around the inner boundary components of \( (\Sigma,A) \), are disjoint and so are the second (straight) path segments. Hence, the only intersection points between \( \alpha \) and \( \alpha' \) lie on the path segments \( \varphi_f(\{s+\varepsilon+a\}\times[0,a]) \) and \( \varphi_f(\{s'-\varepsilon\}\times[0,a']) \). Let us discuss these two types of intersection points separately: 
    \begin{itemize}
            \item[$\DotC$]
            The set of intersection points on \( \varphi_f(\{s+\varepsilon+a\}\times[0,a]) \) is given by 
            \[ \{\varphi_f(s'-a'-\varepsilon+t,t) \mid t\in[0,\min(a,a')]\co s+\varepsilon+a=s'-a'-\varepsilon+t+l\cdot n_f \text{ for some }l\in\mathbb{Z}\}.\]
            Thus, we can index these intersection points by integers \( l\in\mathbb{Z} \) satisfying 
            \[ 0\leq \overbrace{s-s'+a+a'+2\varepsilon-l\cdot n_f}^t\leq\min(a,a'),\]
            which is equivalent to 
            \[ 0\leq s-s'+a+a'-l\cdot n_f<\min(a,a').\]
            Let us rewrite this once more and express it in terms of the following three conditions:
            \[ 0\geq s'-s-a-a'+l\cdot n_f,\quad s'-s-a'+l\cdot n_f>0\quad\text{ and }\quad s'-s-a+l\cdot n_f>0.\]
            Let us now turn to Figure~\ref{fig:MorphismsBetweenFJoinsOuterOuter}, which shows some morphisms between the type~D structures given by the two \( f \)-joins. It is straightforward to check that the resolution of an intersection point \( \DotC \) from above is given by the horizontal (dashed) arrows. Moreover, the last two conditions above are satisfied iff the horizontal arrows define a morphism over \( \mathcal{A}_f^+ \), ie if their lengths are strictly positive; also, such a  morphism is not in the image of \( D^+ \) iff the length \( s'-s-a-a'+l\cdot n_f \) of a diagonal arrow from \( (s+a) \) to \( (s'-a') \) is less than or equal to 0, ie iff the first condition is satisfied. Note that if the length of this diagonal is equal to 0, the corresponding horizontal morphism lies in the image of~\( D \).
            
            \item[$\DotB$]
            The set of intersection points on \( \varphi_f(\{s'-\varepsilon\}\times[0,a']) \) is given by 
            \[ \{\varphi_f(s+\varepsilon+t,t) \mid t\in[0,\min(a,a')]\co s'-\varepsilon=s+\varepsilon+t-l\cdot n_f \text{ for some }l\in\mathbb{Z}\}.\]
            So, in the same way as above, we can index these intersection points by integers \( l\in\mathbb{Z} \) satisfying 
            \[ 0\leq s'-s-2\varepsilon+l\cdot n_f\leq\min(a,a'),\] 
            or equivalently, 
            \[ 0< s'-s+l\cdot n_f\leq\min(a,a').\]
            The resolution of such an intersection point is given by the diagonal arrow in Figure~\ref{fig:MorphismsBetweenFJoinsOuterOuter}. Moreover, such an arrow defines a morphism over \( \mathcal{A}^+_f \) iff its length \( s'-s+l\cdot n_f \) is strictly positive, ie iff the first inequality is satisfied; moreover, by considering either of the two horizontal arrows of length \( s'-s-a'+l\cdot n_f \) from \( s \) to \( (s'-a') \) and \( s'-s-a+l\cdot n_f \) from \( (s+a) \) to \( s' \), we see that such a morphism is not in the image of \( D^+ \) iff the lengths of both of these null-homotopies are less than or equal to 0. This is equivalent to the second inequality above. Again, note that if the length of one of these arrows is equal to 0, the morphism lies in the image of~\( D \).
        \end{itemize}
        We now conclude with the observation that any morphism between the two \( f \)-joins which lies in the kernel of \( D^+ \) is a linear combination of pairs of horizontal arrows or diagonal arrows as considered above.
    \item[Case (b)] Suppose \( (\alpha,\alpha')=(\vartheta^{+}_f(s,\infty),\vartheta^{-}_f(s'-a',a')) \). This case is illustrated in Figures~\ref{fig:MorphismsBetweenFJoinsOneTwoIllustrated} and~\ref{fig:MorphismsBetweenFJoinsInnerOuter}. 
    The intersection points in this case lie on the line \( \varphi_f(\{s'-\varepsilon\}\times[0,a']) \) and are given by 
    \[ \{\varphi_f(s+\varepsilon+t,t) \mid t\in[0,a']\co s'-\varepsilon=s+\varepsilon+t-l\cdot n_f \text{ for some }l\in\mathbb{Z}\}.\]
    Thus, we can index these intersection points by integers \( l\in\mathbb{Z} \) satisfying 
    \[ 0\leq s'-s-2\varepsilon+l\cdot n_f\leq a',\] 
    or equivalently, 
    \[ 0< s'-s+l\cdot n_f \leq a'.\]
    We can now argue as for the intersection points \( \DotB \) in case (a). 
    
    \item[Case (c)] Suppose \( (\alpha,\alpha')=(\vartheta^{+}_f(s,a),\vartheta^{-}_f(s',\infty)) \), see Figures~\ref{fig:MorphismsBetweenFJoinsTwoOneIllustrated} and~\ref{fig:MorphismsBetweenFJoinsOuterInner}
    The intersection points in this case lie on the line \( \varphi_f(\{s'-\varepsilon\}\times[0,a]) \) and are given by 
    \[ \{\varphi_f(s+\varepsilon+t,t) \mid t\in[0,a]\co s'-\varepsilon=s+\varepsilon+t-l\cdot n_f \text{ for some }l\in\mathbb{Z}\}.\]
    Thus, we can index these intersection points by integers \( l\in\mathbb{Z} \) satisfying 
    \[ 0\leq s'-s-2\varepsilon+l\cdot n_f\leq a,\] 
    or equivalently, 
    \[ 0< s'-s+l\cdot n_f \leq a.\]
    Again, we can now argue as for the intersection points \( \DotB \) in case (a). 
    
    \item[Case (d)] Suppose \( (\alpha,\alpha')=(\vartheta^{+}_f(s,\infty),\vartheta^{-}_f(s',\infty)) \).
    This is the simplest of the four cases, see Figures~\ref{fig:MorphismsBetweenFJoinsOneOneIllustrated} and~\ref{fig:MorphismsBetweenFJoinsInnerInner}. The intersection points lie on the line \( \varphi_f(\{s'-\varepsilon\}\times[0,\infty)) \) and are given by 
    \[ \{\varphi_f(s+\varepsilon+t,t) \mid t\in[0,\infty)\co s'-\varepsilon=s+\varepsilon+t-l\cdot n_f \text{ for some }l\in\mathbb{Z}\}.\]
    Thus, we can index these intersection points by integers \( l\in\mathbb{Z} \) satisfying 
    \[ 0\leq s'-s-2\varepsilon+l\cdot n_f,\] 
    or equivalently, 
    \[ 0< s-s'+l\cdot n_f.\]
    Again, we can now argue as before. 
\end{description}

\begin{figure}[t]
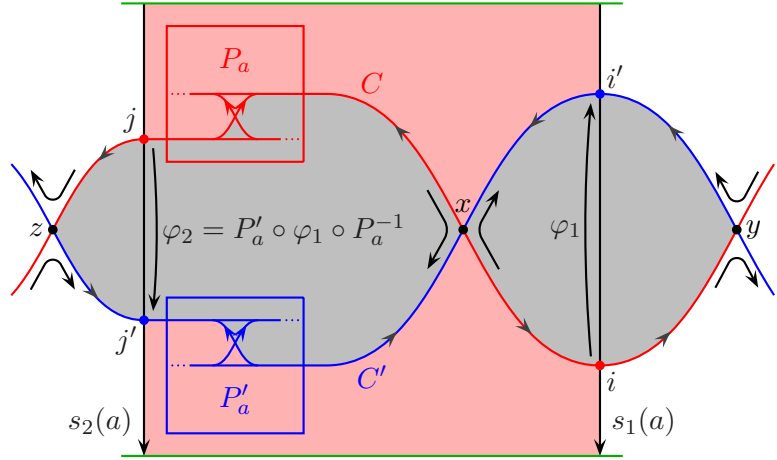

	\centering
	$\BigonCounts$
	\caption{Illustration for the identification of bigons with the map \( \beta_\ast \) in the proof of Theorem~\ref{thm:PairingMorLagrangianFH}.}\label{fig:BigonCounts}
\end{figure}


Finally, we need to identify the map \( \beta_\ast \) with the differential on \( \CF(C,C') \). This works just as in~\cite[proof of Theorem~4.43]{pqMod}, noting that \cite[Figure~33]{pqMod} needs to be replaced by Figure~\ref{fig:BigonCounts} due to the difference in orientation conventions and that a generator of \( \Homology(\Mor^+(C,C')) \) is the end of a bigon iff there is a null-homotopy via \( D \), but not \( D^+ \).
\end{proof}

\begin{remark}\label{rem:LagHFindexing}
Theorem~\ref{thm:PairingMorLagrangianFH} provides a \emph{canonical} isomorphism between \( \HF(C,C')\) and \( \Homology (\Mor (C,C'), D)\), and thus the \( \Rcomm[U] \)-action on \( \Homology (\Mor (C,C'), D) \) gives a well-defined \( \Rcomm[U]\)-action on wrapped Lagrangian Floer homology \(\HF(C,C') \).

We pause to consider the indexing of the intersection points on the straight line segments \( \varphi_f(\{s+\varepsilon\}\times[0,\infty)) \) and \( \varphi_f(\{s'-\varepsilon\}\times[0,\infty)) \) in the proof above.  The indices of the intersection points on the straight line segment for the first \(f\)-join increase from the inner boundary of \(f\) to the outside; the intersection points on the straight line segment for the second \(f\)-join are indexed in the opposite direction. Thus, the \( U \)-action on \( \Homology(\Mor^+(C,C')) \) can be translated into a \( U \)-action on \( \CFplus(C,C') \) as follows: given an intersection point \( x\in\CFplus(C,C') \), \( x\cdot U \) is the next intersection point on the same straight line segment as \( x \) (if it exists) or 0 (otherwise). 
  
  One can interpret this \( U \)-action geometrically in terms of counting punctured bigons \( \iota_U \) with both convex and concave corners, see Figure~\ref{fig:LagrangianHFConventionsUaction}. One can even extend this action to \( \CF(C,C') \), but without more work this is not  canonical.
\end{remark}

\begin{figure}[t]
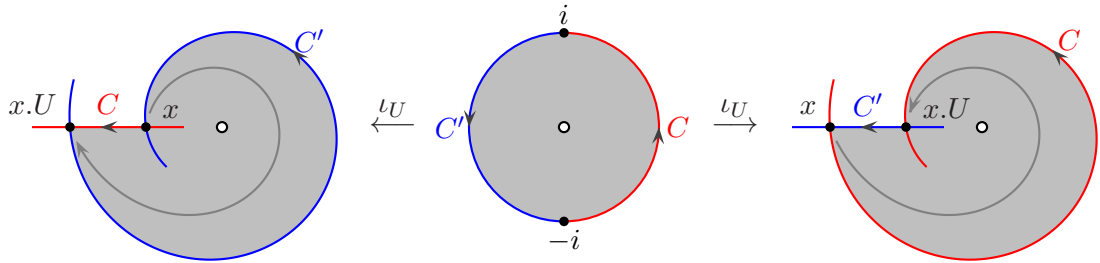
\centering
	$\UbigonConventions$
	\caption{A geometric interpretation of the \( U \)-action on upper intersection points from Remark~\ref{rem:LagHFindexing}.}\label{fig:LagrangianHFConventionsUaction}
\end{figure}

\subsection{A formula for computing the homology of morphism spaces between curves}\label{subsec:classification:morphisms-formula}

If \( C \) and \( C' \) are two precurves, then by Theorem~\ref{thm:EverythingIsLoopTypeUpToLocalSystems}, there are two multicurves \( L \) and \( L' \) such that \( C \) and \( \Pi_i(L) \) as well as \( C' \) and \( \Pi_i(L') \) are homotopic, respectively. Thus, 
\( \Mor(C,C') \) and \( \Mor(L,L'):=\Mor(\Pi_i(L),\Pi_i(L')) \) are chain homotopic. If we put \( \Pi_i(L) \) and \( \Pi_i(L') \) into pairing position, Theorem~\ref{thm:PairingMorLagrangianFH} says that the homology of \( \Mor(L,L') \) is graded isomorphic to \( \HF(L,L'):=\HF(\Pi_i(L),\Pi_i(L')) \). But as in~\cite{pqMod}, we can actually compute \( \HF(L,L') \) without putting the curves into pairing position. For simplicity, let us assume that both curves are compact; the other cases are treated similarly. 

\begin{definition}
	Let \( (L,L') \) be a pair of bigraded curves \( L=(\gamma,X) \) and \( L'=(\gamma',X') \) on an oriented surface \( \Sigma \) with arc system \( A \) and let \( n:=\dim X \) and \( n':=\dim X' \). Assume that \( \gamma \) and \( \gamma' \) intersect minimally. Let \( \fieldTwoElements\langle\gamma\cap\gamma'\rangle \) denote the vector space over \( \fieldTwoElements \) spanned by intersection points between \( \gamma \) and \( \gamma' \). 
	Each intersection point can be \( (\delta, q)-\)bigraded in exactly the same way as intersection points in \( \CF(C,C') \). 
	If \( \gamma \) and \( \gamma' \) are parallel, let \( \delta(\gamma,\gamma') \) be the unique real number one needs to add to the \( \delta \)-grading of each intersection point of \( \gamma \) with arcs in \( A \) such that \( \gamma \) and \( \gamma' \) agree as \( \delta \)-graded curves. Similarly, define \( q(\gamma,\gamma') \) using the quantum grading of \( \gamma \) and \( \gamma' \). 
	For any non-negative integer \( m \), let \( V_{\delta}^q(m) \) be an \( m \)-dimensional vector space in \( \delta \)-grading \( \delta \) and quantum grading \( q \).
\end{definition}

\begin{theorem}\label{thm:PairingFormula}
	With the notation from above, \(\HF(L,L')\) is graded isomorphic to 
	\begin{equation*}\label{eqn:MorSpacesNonparallel}
	V^0_0(n\cdot n')\otimes \fieldTwoElements\langle\gamma\cap\gamma'\rangle,
	\end{equation*}
	unless \(\gamma\) and \(\gamma'\) are parallel. If they are parallel, let us assume without loss of generality that their orientations agree. Then, \(\HF(L,L')\) is graded isomorphic to 
	\begin{equation}\label{eqn:MorSpacesParallel}
	\Big(V^0_0(n\cdot n')\otimes\fieldTwoElements\langle\gamma\cap\gamma'\rangle\Big)
	\oplus
	\Big(\!\big(V^0_0(1)\oplus V^0_{1}(1)\big)\otimes V^{q(\gamma,\gamma')}_{\delta(\gamma,\gamma')}\left(\dim\left(\ker\left((X^{-1})^t\otimes X'-\id\right)\right)\right)\!\!\Big).
	\end{equation}
\end{theorem}

\begin{proof}
	This is identical to the proof of~\cite[Theorem~4.45]{pqMod}. 
\end{proof}


\begin{observation}\label{obs:LargeHAction}
    Consider the cone of the map
    \[
    \begin{tikzcd}
    \Pi(L)
    \arrow{r}{U^N\cdot\id}
    &
    \Pi(L)
    \end{tikzcd}
    \]
    for some curve \(L=(\gamma,X)\) and some integer \(N>0\). 
    This cone is a type~D structure over \( \mathcal{A} \) and hence, it is represented by some curve. For arbitrary \(N\), one essentially has to go through the arrow-pushing algorithm to find this curve. However, if \(N\) is large enough, such that \(\gamma\) wraps around each inner boundary component of \((\Sigma,A)\) strictly less than \(N\) times, this is very simple. Indeed, for such large \(N\), the map \(U^N\cdot\id\) is null-homotopic if \(\gamma\) is compact. So in this case, the cone is represented by two copies of \(L\) with appropriate grading shifts. If \(\gamma\) is non-compact, we can still find null-homotopies for all components of \(U^N\cdot\id\) except at the two endpoints of \(\gamma\). So in this case, the cone of \(U^n\cdot\id\) is represented by the compact curve obtained by taking two copies of \(\gamma\) and joining their ends by a path that wraps \(N\) times around the two inner boundary components of \((\Sigma,A)\) at which \(\gamma\) ends. Note that the local system on this compact curve is trivial, since the local systems of non-compact curves are trivial. 
\end{observation}

We are now ready to finish the classification of the objects of \( \Mod^{\overline{\mathcal{A}}}_i \) and \( \Mod^{\mathcal{A}} \). 

\begin{theorem}\label{thm:CompleteClassification}
	Let \(L=\{(\gamma_i,X_i)\}_{i\in I}\) and \(L'=\{(\gamma'_{i'},X'_{i'})\}_{i'\in I'}\) be two multicurves. Then \(\Pi_i(L)\) is homotopic to~\(\Pi_i(L')\) iff there is a bijection \(\iota\co I\rightarrow I'\) such that \(\gamma_i\)~is homotopic to~\(\gamma'_{\iota(i)}\) and \(X_i\)~is similar to~\(X'_{\iota(i)}\). The same holds for \(\Pi\) instead of \(\Pi_i\).
\end{theorem}

\begin{proof}
    It suffices to show the statement for \( \Pi_i \). 
     For compact curves, this follows from the same arguments as~\cite[Theorem~4.46]{pqMod}, which rely on different growth properties of the dimensions of the two summands in \eqref{eqn:MorSpacesParallel} from Theorem~\ref{thm:PairingFormula}~\cite[Theorem~4.45]{pqMod} under pairing with particular test curves. These arguments do not work for non-compact curves. However, we can reduce the general case to the compact case using Observation~\ref{obs:LargeHAction}. Indeed, when passing to the mapping cone of the map \( U^N\cdot \id \) for sufficiently large \( N \), any two non-homotopic curves stay non-homotopic, since those curves coming from compact curves wrap around any inner boundary component strictly less than \( N \) times and those coming from non-compact ones wrap \( N \) times around exactly two inner boundary components and otherwise stay parallel to the original curves. So we can in fact recover the original curves \( L \) and \( L' \) from the respective curves of their cones. 
\end{proof}

\subsection{Signs}\label{subsec:classification:signs}

All classification results in this section remain true if we replace the field \( \fieldTwoElements \) by an arbitrary field \( \field \). The proofs are essentially the same; we itemize the key modifications. The most important of these modifications will be adapting the arrow sliding algorithm \cite{HRW}, which plays a crucial role in the proof of Theorem \ref{thm:EverythingIsLoopTypeUpToLocalSystems}, to work with signs. This will require slightly more complicated train tracks (in the terminology of \cite{HRW}) and result in precurves that (in their geometric guise) have additional decorations. 

First, we need to change the definition of \( U_f \) from \( U_f:=p_f^{n_f} \) to \( U_f:=-(-p_f)^{n_f} \) so that in Example~\ref{exa:ThreePuncturedDiscQuiver:Decomposed}, we can identify \( H=D-SS \) with the sum \( U \) of \( U_f \) over all faces \( f \). Our precurves need to be adapted as follows: first, we label each two-sided \( f \)-join corresponding to a 
\begin{wrapfigure}{r}{0.3\textwidth}
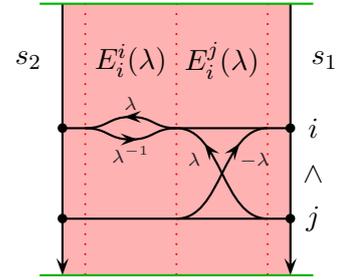

	\centering
	$\traintrackneighbourhoodofarcSigns$
	\caption{A passing loop (left) and a pair of labelled crossover arrows (right).}\label{fig:traintrackneighbourhoodofarcSigns}
\end{wrapfigure}
component 
\[
\begin{tikzcd}
e^s_i
\arrow{r}{\lambda p_f^n.\iota_s}
&
e^{s'}_{i'}
\end{tikzcd}
\hspace{0.3333\textwidth}
\]
by the coefficient \( \lambda\in\field\smallsetminus\{0\} \). Furthermore, we replace the elementary matrices \( E^j_i \) over \( \fieldTwoElements \) by the matrices \( E^j_i(\lambda) \) over \( \field \) defined by
\[
E^{j}_{i}(\lambda):=
\begin{cases}
(\delta_{i'j'}+\lambda\delta_{ii'}\delta_{jj'})_{i'j'} & \text{if \( j\neq i \)}
\\
(\delta_{i'j'}+(\lambda-1)\delta_{ii'}\delta_{jj'})_{i'j'} & \text{if \( j=i \)}
\end{cases}
\hspace{0.3333\textwidth}
\]
for some \( \lambda\in\field\smallsetminus\{0\} \). Note that for \( \field=\fieldTwoElements \) and \( i\neq j \), \( E^j_i=E^j_i(1) \). Moreover,
\[
\Big(E^{j}_{i}(\lambda)\Big)^{-1}=
\begin{cases}
E^{j}_{i}(-\lambda) & \text{if \( j\neq i \)}
\\
E^{j}_{i}(\lambda^{-1}) & \text{if \( j=i \)}
\end{cases}
\hspace{0.3333\textwidth}
\]
Thus, following the basic principles of traintracks that give rise to our graphical notation, with the understanding that we multiply all labels along paths through the precurve, we represent \( E^{j}_{i}(\lambda) \) by labelled crossover arrows (or, labelled switches) and \textbf{passing loops}. The convention we will use for this is shown in Figure~\ref{fig:traintrackneighbourhoodofarcSigns}. Observe that this graphical representation of precurves is unique up to isotopy of the immersed curves and the moves from Figure~\ref{fig:traintrackmovesSigns}. Finally, in Definition~\ref{def:simplyfaced}, we add the condition that every two-sided \( f \)-join in a simply-faced precurve be labelled by the identity in \( \field \).

\begin{figure}
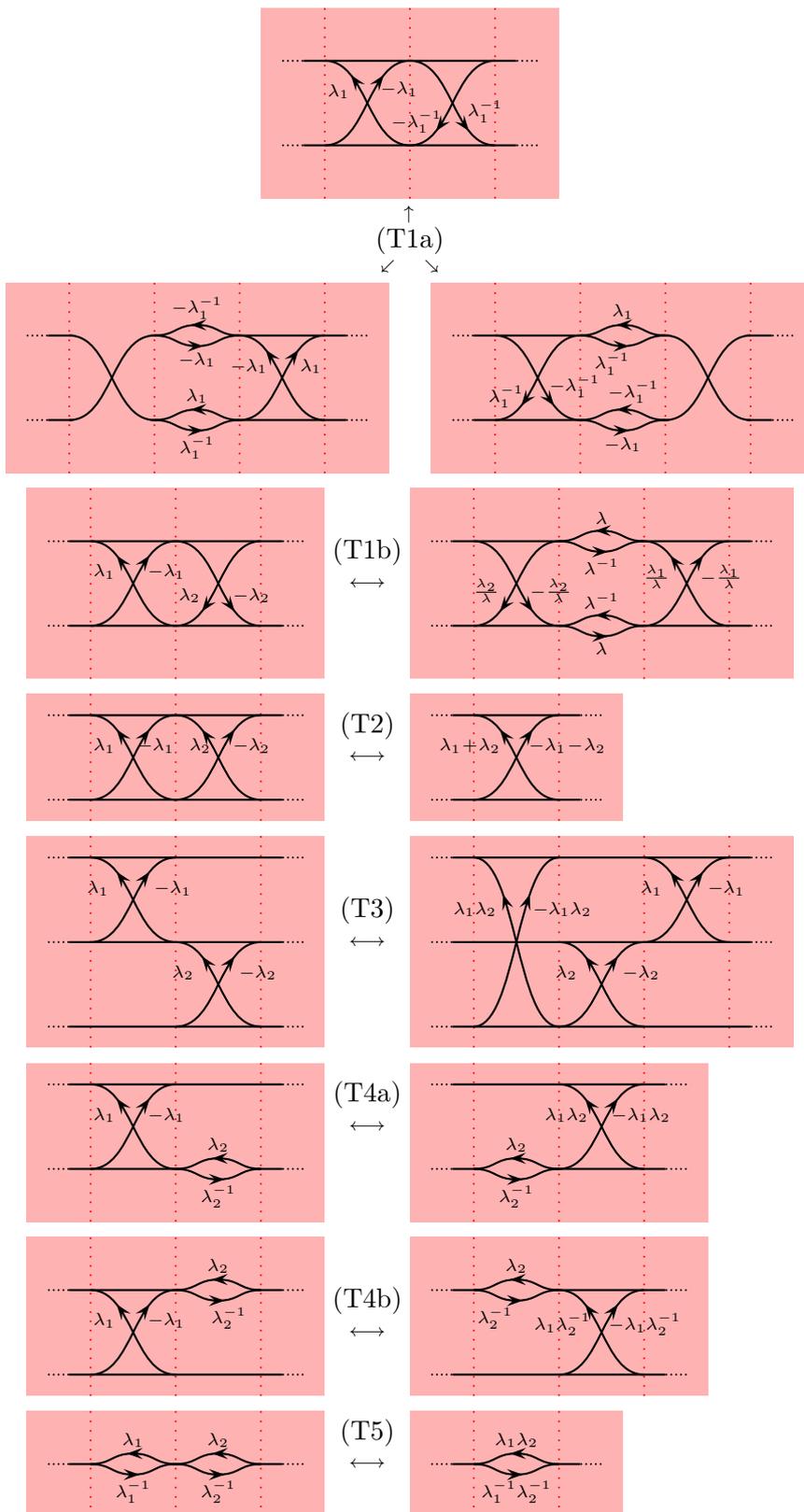

    \centering
    $\Tia$\\
    $\Tib$\\
    $\Tii$\\
    $\Tiii$\\
    $\Tiva$\\
    $\Tivb$\\
    $\Tv$
    \caption{Train-track moves over any field \( \field \). In move (T1b), \( \lambda=1+\lambda_1\lambda_2 \), which we assume to be non-zero.}
    \label{fig:traintrackmovesSigns}
\end{figure}


The proof of Proposition~\ref{prop:PreloopToCC} works over any field \( \field \) up to two modifications: first, we need to take into account that two-sided \( f \)-joins are now labelled by non-zero elements in \( \field \). Hence, the morphisms \( h_1 \) and \( h_2 \) also need to be multiplied by a suitable coefficient. Second, the algorithm only gives us a fully cancelled pre-curve in which every dot lies on exactly one \( f \)-join; some two-sided \( f \)-joins might still be labelled by some non-identity element of \( \field \). We can fix this using the following observation:

\begin{observation}
    Let \((C, \{P_a\}_{a\in A},\partial)\) be a precurve on a surface \(\Sigma\) with arc system~\(A\) and \(\lambda\in\field\). For \(a\in A\), let \(s\) be a side of \(a\), \(f\) the face adjacent to \(s\) and \(\bullet(s,i)\) a dot on \(s\). Then \((C, \{P_a\}_{a\in A},\partial)\) is chain isomorphic to the precurve obtained by multiplying the labels of all two-sided \(f\)-joins starting at \(\bullet(s,i)\) by \(\lambda\), multiplying the labels of all two-sided \(f\)-joins ending at \(\bullet(s,i)\) by \(\lambda^{-1}\) and multiplying \(P_a\) on the right by \(E^i_i(\lambda)\) if \(s=s_1(a)\) or on the left by \(E^i_i(\lambda^{-1})\) if \(s=s_2(a)\).
\end{observation}

\begin{figure}[t]
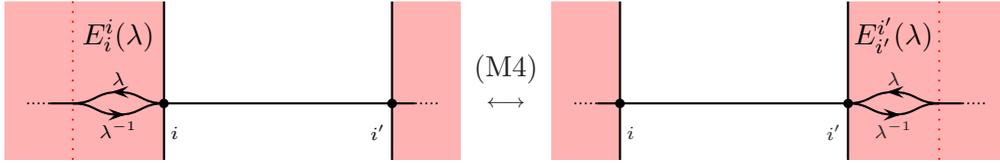

    \centering
    $\PushPassingLoops$
    \caption{Pushing passing loops along \( f \)-joins}
    \label{fig:pushing_passing_loops}
\end{figure}

Lemma~\ref{lem:CalculusForPreloops} also remains true after replacing the elementary matrices \( E^j_i \) by \( E^j_i(\lambda) \) for any \( \lambda\in\field \) and the matrix \( E^{j'}_{i'} \) in move (M2) by \( E^{j'}_{i'}(-\lambda) \).
There is a fourth move (M4) which allows us to push passing loops along \( f \)-joins; see Figure~\ref{fig:pushing_passing_loops}. Notice that, taking this move together with those collected in Figure \ref{fig:traintrackmovesSigns}, we can move passing loops to anywhere that is convenient (up to adjusting the coefficients) and, effectively, ignore passing loops when we simplify. 

\labellist \tiny
\endlabellist
\begin{figure}[t]
\includegraphics[scale=0.75]{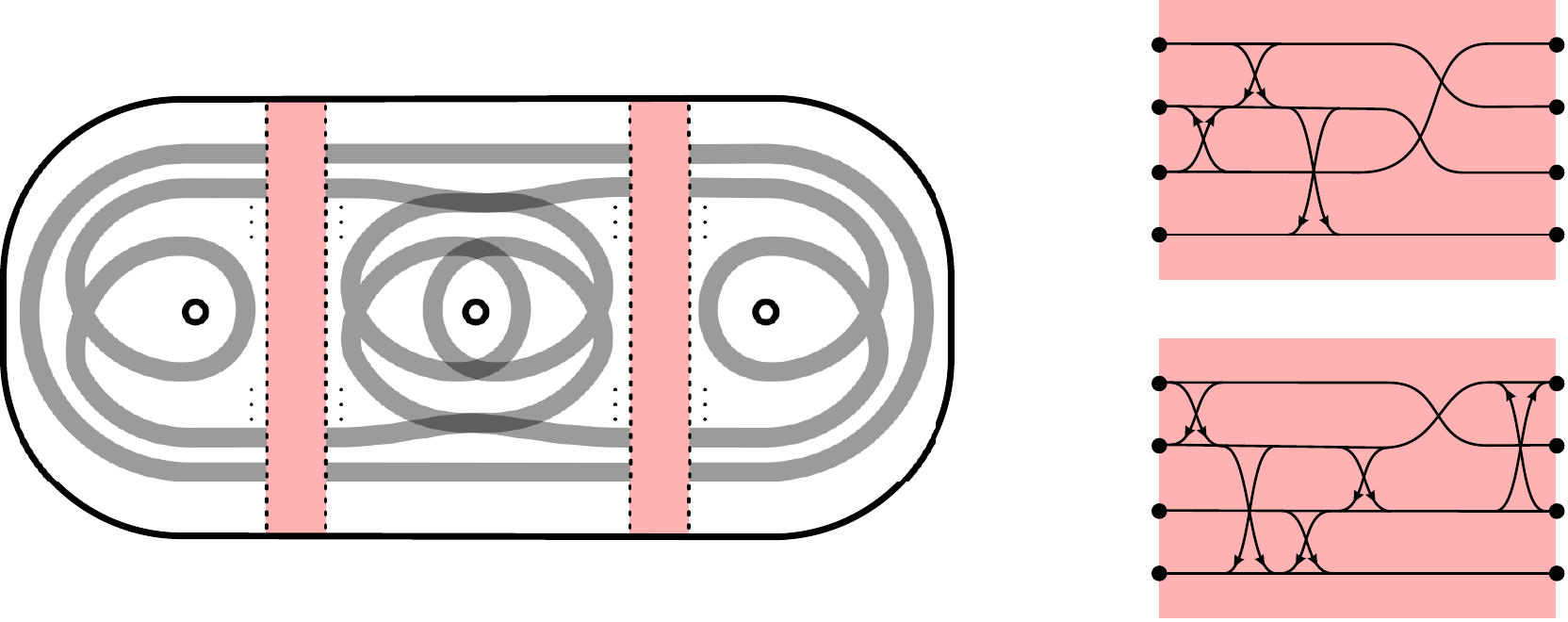}
\caption{Summarizing the simplification process, when $\field = \fieldTwoElements$: on the left, the precurve is arranged so that strands are ordered from outside to inside according to how much they wrap in each face. On the right, two equivalent arc-regions are shown. The lower of the two illustrates the desired form, which is guaranteed by \cite[Lemma 31]{HRW}.}\label{fig:taxonomy}
\end{figure}

This last observation allows us to enhance the arrow-pushing algorithm from \cite{HRW} to work with signs,  so that the proof of Theorem~\ref{thm:EverythingIsLoopTypeUpToLocalSystems} still works. For the purpose of the algorithm, we can simply ignore passing loops, which will ultimately be absorbed into the local systems. We take a moment to explain what needs to be carefully checked in order to deal with labelled switches. Central to the algorithm is \cite[Lemma 31]{HRW}, which says that crossover arrows can be arranged in a particular way in each of the arc neighbourhoods. For concreteness, consider the marked surface associated with $\BNAlgH$. As indicated in Figure \ref{fig:taxonomy}, we can choose our simply faced precurves so that strands are ordered from most-to-least wrapping (compare \cite[Section 3.7,  particularly Figure 32]{HRW}). Note that making such a choice is a simple matter of multiplication by a permutation matrix. 

With this choice in hand, the version of \cite[Lemma 31]{HRW} that is of relevance to us has each arc neighbourhood expressed as follows: a permutation among the strands, on the left of which are only downward pointing switches and on the right of which are only upward pointing switches. (There are, of course, passing loops as well, but we can assume that these are collected to one side of the arc region and ignore them, owing to the observation above.) 
The existence of such a decomposition follows from a version of Gaussian elimination in which the row permutation is performed after eliminating all entries below the pivot
elements; this is sometimes called an $LPU$ decomposition. However, in order to show that the algorithm terminates, one needs to keep track of individual arrows (and certain weights associated with them, see~\cite[Section~3.7]{HRW}). Let us discuss how this works in some more detail. 

\labellist
\pinlabel $=$ at 137 118 \pinlabel $=$ at 137 30
 \tiny
\pinlabel $\lambda_1$ at 28 132  \pinlabel $-\lambda_1$ at 51 132
\pinlabel $\lambda_2$ at 75 110  \pinlabel $-\lambda_2$ at 98 110
\pinlabel $\frac{\lambda_2}{\lambda}$ at 174 108  \pinlabel $-\frac{\lambda_2}{\lambda}$ at 197 108
\pinlabel ${\lambda_1\!\lambda}$ at 209 128  \pinlabel $-\lambda_1\!\lambda$ at 232 128
\pinlabel $\lambda$ at 248 149  \pinlabel $\frac{1}{\lambda}$ at 248 128
\pinlabel $\lambda$ at 248 92  \pinlabel $\frac{1}{\lambda}$ at 248 112
\pinlabel $\lambda_1$ at 28 43 \pinlabel $-\lambda_1$ at 51 43
\pinlabel $\frac{1}{\lambda_1}$ at 174 24 \pinlabel $-\frac{1}{\lambda_1}$ at 197 24
\pinlabel $-\lambda_1$ at 209 43 \pinlabel $\lambda_1$ at 232 43
\pinlabel $\lambda_1$ at 248 62 \pinlabel $\frac{1}{\lambda_1}$ at 248 40
\pinlabel $-\lambda_1$ at 248 3  \pinlabel $-\frac{1}{\lambda_1}$ at 248 24
\endlabellist
\begin{figure}[t]
\includegraphics[scale=0.75]{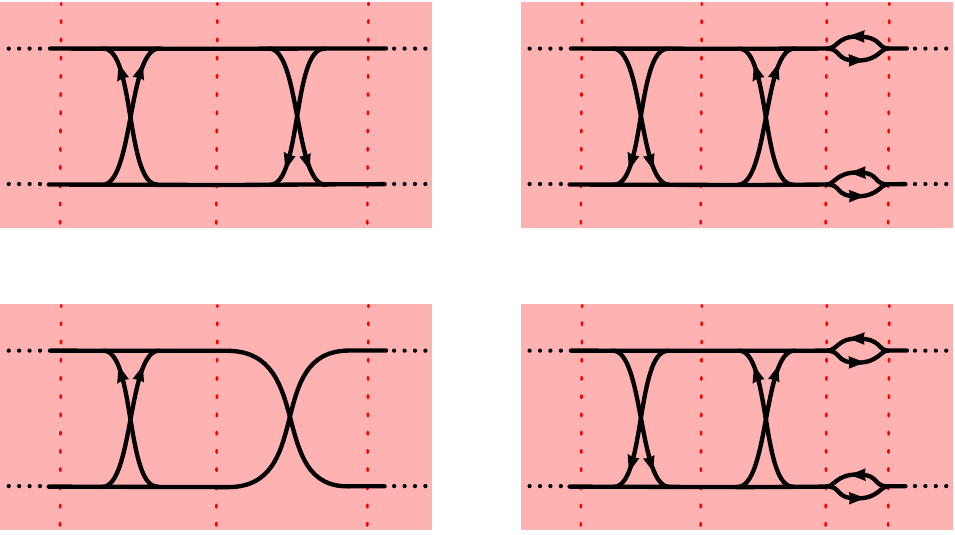}
\caption{The top row shows two equivalent local pictures of train tracks representing equal matrices, where $\lambda_1,\lambda_2\in\field$, in the generic case where $\lambda=1+\lambda_1\lambda_2$ does not vanish (the non-generic case where $\lambda_2=-\lambda_1^{-1}$ is analogous to the usual setup with mod 2 coefficients). This is an application of move (T1b) from Figure \ref{fig:traintrackmovesSigns}. The lower row shows equivalent local pictures of train tracks used to move a crossover arrow past a transposition, using the move (T1a).}\label{fig:casework}
\end{figure}

After the initial reordering of the strands, we can push all crossings to the middle of the arc neighbourhood, potentially with switches pointing in any direction on either side. Suppose there are upward pointing arrows to the left of the crossings. Let us focus on the {\em innermost} of these, that is, the upward switch closest to the crossings; see Figure \ref{fig:taxonomy} for a specific example. Then, up to changing coefficients according to the move (T3), it is possible to slide this switch past all of the downward pointing switches it meets (which adds more arrows, but they point in the correct direction), with one exception. This exceptional case, where an upward pointing crossover arrow meets a downward pointing one between the same strands, is shown in Figure \ref{fig:casework}. The generic case when $\lambda_2\ne-\lambda_1^{-1}$ is handled by (T1b): crossover arrows with generic labels in $\field$ slide freely past one another up to the addition of new passing loops (these can be pushed away using (T4) moves). The non-generic case is handled as in the case where $\field = \fieldTwoElements$; this uses (T1a) when working with signs.

Once the upward pointing switch reaches the crossings, it either freely slides to the other side, with a possible change in its length but not a change to its direction, or it meets a crossing as on the bottom left of Figure~\ref{fig:casework}. As illustrated in the same figure, we may use (T1a) to produce two switches and push them to the left and right of the remaining crossings such that they point in the correct directions. 

We can now repeat this procedure until there are no more upward pointing switches to the left of the crossings. We can deal similarly with the downward pointing switches to the right of the crossings. 
Now, having adapted \cite[Lemma 31]{HRW} to work with signs, \cite[Proposition 29]{HRW} goes through as before, and hence Theorem~\ref{thm:EverythingIsLoopTypeUpToLocalSystems} does not depend of the choice of $\field$.



There remains one final adjustment: the bigon count in the definition of the differential of the wrapped Lagrangian Floer homology of two pre-curves needs to take the labelling of the crossover arrows and passing loops into account. So with the notation as in Definition~\ref{def:LagrangianFH}, we define
\[ d(x)=\sum_{y}\Big(\sum_{\iota\in \mathcal{M}(x,y)}\lambda(\iota)\Big)y,\]
where \( \lambda(\iota) \) is the product of all labels on the boundary of the bigon \( \iota \). Similarly, we adapt the definition of the resolution of an intersection point between two simply-faced precurves in pairing position. With these changes, the proofs of Theorems~\ref{thm:PairingMorLagrangianFH}, \ref{thm:PairingFormula} and~\ref{thm:CompleteClassification} for arbitrary fields \( \field \) go through. 

\begin{wrapfigure}{r}{0.3333\textwidth}
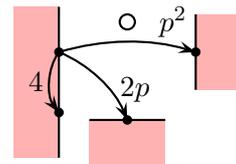

	\centering
	$\ExampleOverIntegers$
	\caption{A precurve over \( \mathbb{Z} \)}\label{fig:ExampleOverIntegers}\medskip
  \label{fig:block_near_arc}
\end{wrapfigure}
\myfixwrapfig
\begin{remark}\label{rem:classification_over_Z}
    A classification of precurves with coefficients in \( \mathbb{Z} \) seems to be considerable harder than over fields. There are two reasons for this: first, we can no longer fully cancel all precurves and, second, we may not be able to simplify the precurves with respect to faces (Proposition~\ref{prop:PreloopToCC}). Figure~\ref{fig:ExampleOverIntegers} illustrates this: there, we have four generators which could be part of a precurve. Over \( \Q \), the arrow labelled by 4 could be cancelled. However, over \( \mathbb{Z} \), this is not possible. Furthermore, over \( \Q \), the arrow \( p^2 \) could be cancelled using the shorter arrow labelled \( 2p \). Again, this is not possible over \( \mathbb{Z} \).
\end{remark}

\section{Immersed curve invariants}\label{sec:FigureEight}
We are now equipped to define various curve invariants of pointed 4-ended tangles. This section produces the main three curve invariants for the paper: \( \BNr(T) \), \( \Khr(T) \), and \( \Kh(T) \).
\subsection{Arc type curve}
The classification results from Section~\ref{sec:classification} allow us to interpret Bar-Natan's invariant \( \DD(T) \) of a pointed 4-ended tangle in terms of immersed curves on the 3-punctured disc from Figure~\ref{fig:ExampleMarkedSurface}. We usually identify this surface with the 4-punctured sphere \(\FourPuncturedSphere\) from Section~\ref{subsec:intro:classification} such that the arcs $b$ and $c$ correspond to the arcs $\DotBarc$ and $\DotCarc$, respectively. 

\begin{definition}\label{def:ImmersedCurveInvariantspqMod}
    Given a pointed 4-ended tangle \( T \), let \( \Arc(T) \) denote the multicurve associated with \( \DD(T) \). We call \( \Arc(T) \) the \textbf{arc type invariant of \( T \)}. 
\end{definition}
Note that up to homotopy,  \( \DD(T) \) and \( \Arc(T) \) contain the same amount of information, so \( \Arc(T) \) simply offers an interpretation of the algebraic invariant \( \DD(T) \) as a geometric invariant living on the parameterized boundary of the tangle.  
In fact, as we will show in Theorem~\ref{thm:MCGaction}, over \(\Rcomm=\fieldTwoElements \) adding twists to the ends of a tangle \( T \) corresponds to adding Dehn twists to the curves \( \Arc(T) \). So one may actually ignore the parameterization given by the arc system without loosing any information about \( \Arc(T;\fieldTwoElements ) \), or equivalently \( \DD(T;\fieldTwoElements ) \), except for the absolute bigradings. 

\begin{example}[\(\Arc \) for trivial and rational tangles]\label{ex:curves_for_rational_tangles}
The invariant \(\Arc(\Li) \) is a single arc \(\DotCarc\) from Figure~\ref{fig:exa:classification:curves:trivial}, which can be obtained by pushing the tangle component without the reduction point \(*\) to the boundary of \( D^3 \). More generally, by computations from Example~\ref{exa:BNntwisttangles}, the invariants of \( n\)-twist tangles are obtained from the arc \(\DotCarc\) by applying the corresponding Dehn twists. In other words, the arc type invariant \( \BNr(T_n) \) is obtained by pushing the tangle component without the reduction point \(*\) to the boundary of \( D^3 \).  Examples of immersed curves for
\[
		\DD(T_{n})=\DD\left( \left.\substack{\CrossingPosMarkedi \\ \raisebox{5pt}{\vdots} \\ \CrossingPos} \right\}n \right) =
    \left[
    \begin{tikzcd}[nodes={inner sep=2pt},column sep=14pt]
    \GGzqh{\DotC}{\frac{n}{2}}{n}{0}
    \arrow{r}{S}
      &
    \GGzqh{\DotB}{\frac{n-1}{2}}{n+1}{1}
    \arrow{r}{D}
    &
    \GGzqh{\DotB}{\frac{n-1}{2}}{n+3}{2}
    \arrow{r}{SS}
    &
    \GGzqh{\DotB}{\frac{n-1}{2}}{n+5}{3}
    \arrow{r}{D}
    &
    \cdots
    \arrow{r}{}
    &
    \GGzqh{\DotB}{\frac{n-1}{2}}{3n-1}{n}
    \end{tikzcd}\right]
\]
\[ 
    \DD(T_{-n})=\DD\left( \left.\substack{\CrossingNegMarkedi \\ \raisebox{5pt}{\vdots} \\ \CrossingNeg} \right\}n \right)=
    \left[
    \begin{tikzcd}[nodes={inner sep=2pt},column sep=14pt]
    \GGzqh{\DotB}{-\frac{n-1}{2}}{-3n+1}{-n}
    \arrow{r}{}
    &
    \cdots
    \arrow{r}{D}
    &
    \GGzqh{\DotC}{-\frac{n-1}{2}}{-n-5}{-3}
    \arrow{r}{SS}
    &
    \GGzqh{\DotB}{-\frac{n-1}{2}}{-n-3}{-2}
    \arrow{r}{D}
    &
    \GGzqh{\DotB}{-\frac{n-1}{2}}{-n-1}{-1}
    \arrow{r}{S}
    &
    \GGzqh{\DotC}{-\frac{n}{2}}{-n}{0}
    \end{tikzcd}\right]
\]
are shown in Figure~\ref{fig:rational_curves}.
By Theorem~\ref{thm:MCGaction} (see above), the arc type invariant \( \BNr(T;\fieldTwoElements) \) of \emph{any} rational tangle is obtained by pushing the tangle component without the reduction point \(*\) to the boundary of \( D^3 \). We expect that this is true over a general field \( \field \) as well.

Thomson described an algorithm for computing Bar-Natan's tangle invariant $\KhTl{T}$ for any rational tangle $T$ over the dotted cobordism category with the additional relation $H=0$~\cite{ThompsonRationalTangles}. He proves that one can describe these invariants in terms of zig-zag graphs and that the bigradings on these invariants can be understood explicitly in terms of an action by the braid group on three strands. Our interpretation of Bar-Natan's invariant in terms of immersed curves offers a satisfying explanation for these observations.
\end{example}

\begin{figure}[t]
  \centering
  \begin{subfigure}[t]{0.45\textwidth}
  \centering
  \includegraphics[scale=0.6]{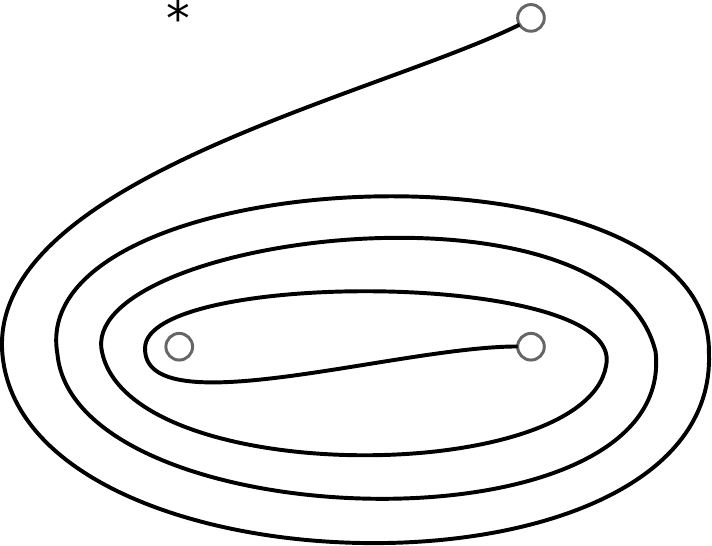}
  \caption{\( \BNr(T_8) \)}
  \end{subfigure}
  \begin{subfigure}[t]{0.45\textwidth}
  \centering
  \includegraphics[scale=0.6]{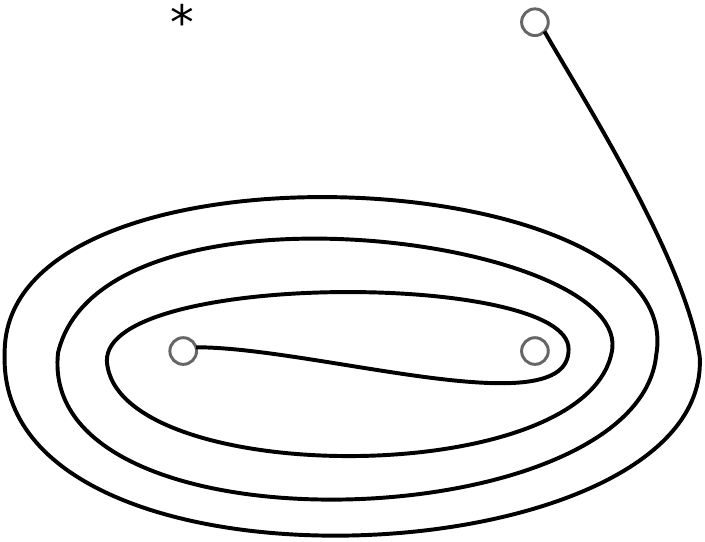}
  \caption{\( \BNr(T_{-7}) \)}
  \end{subfigure}
  \caption{Arc type curves associated to rational tangles.}\label{fig:rational_curves}
\end{figure}

\begin{example}[\(\Arc \) for some 2-stranded pretzel tangles]\label{exa:ArcTwoStrandedPretzel}
Consider the invariants associated with \( (2,-n) \)-pretzel tangles \( T_{2,-n} \); these tangles include the already familiar tangle \( T_{2,-3} \) from the top-left of Figure~\ref{fig:2m3ptBNComplexComputationRH}. Figure~\ref{fig:ArcTypePretzel} shows the curves $\BNr(T_{2,-n} )$ for some integers $n$; it includes the generators and arrows of the complexes $\DD(T_{2,-n} )$ to facilitate the translation between the algebraic object \(\DD(T)\) and the geometric object \(\BNr(T)\). 
\begin{itemize}
    \item[\( n=2 \):] The tangle \( T_{2,-2} \) has a single closed component. Choosing the same orientation of the tangle ends as in Example~\ref{exa:2m3ptBNComplexComputation}, the complex \( \DD(T_{2,-2}) \) is equal to
    \[ 
		\begin{tikzcd}[column sep=30pt,row sep=20pt]  
		\GGdzh{\DotC}{-\frac{1}{2}}{-9}{-4}
    	\arrow{r}{S}
    	&
    	\GGdzh{\DotB}{-1}{-8}{-3}
    	\arrow{r}{D}
    	&
     	\GGdzh{\DotB}{-1}{-6}{-2}
     	&
     	\GGdzh{\DotB}{-1}{-6}{-2}
    	\arrow{r}{D}
    	&
    	\GGdzh{\DotB}{-1}{-4}{-1}
    	\arrow{r}{S}
    	&
     	\GGdzh{\DotC}{-\frac{3}{2}}{-3}{0}
    	\end{tikzcd}
    \]
    So \( \Arc(T_{2,-2}) \) has two components, which coincide with the arc type invariants of the two 2-twist rational tangles, see Figure~\ref{fig:ArcTypePretzelII}.
    
    \item[\( n=3 \):] We have already computed \( \DD(T_{2,-3}) \) of the \( (2,-3) \)-pretzel tangle in Example~\ref{exa:2m3ptBNComplexComputation}, see in particular Figure~\ref{fig:2m3pt:BN}.
    Figure~\ref{fig:ArcTypePretzelIII} shows the corresponding immersed curve \( \BNr(T_{2,-3}) \).  
    \item[\( n=5 \):] The \( (2,-5) \)-pretzel tangle \( T_{2,-5} \) is interesting, because it illustrates that \( \Arc(T) \) may contain closed components. \( \Arc(T_{2,-5}) \) consists of two components, of which one agrees with the immersed curve from Figure~\ref{fig:ArcTypePretzelIII} and the other with the one from Figure~\ref{fig:ArcTypePretzelV}. 
\end{itemize}
\end{example}

\begin{figure}[t]
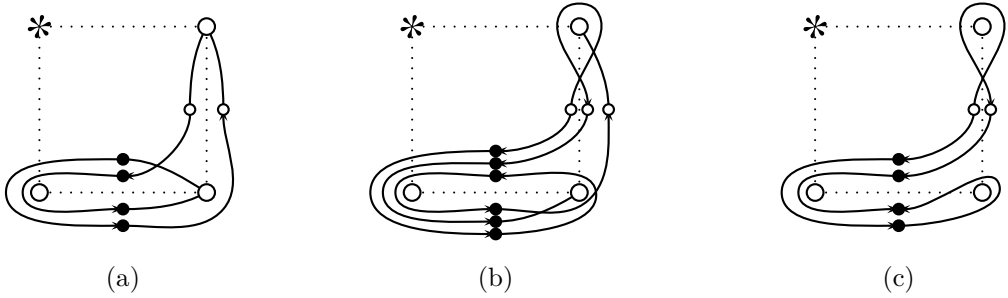

    \centering
    \begin{subfigure}{0.3\textwidth}
    \centering
    $\pretzelarcinvariantII$
    \caption{}\label{fig:ArcTypePretzelII}
    \end{subfigure}
    \begin{subfigure}{0.3\textwidth}
    \centering
    $\pretzelarcinvariant$
    \caption{}\label{fig:ArcTypePretzelIII}
    \end{subfigure}
    \begin{subfigure}{0.35\textwidth}
    \centering
    $\pretzelarcinvariantV$
    \caption{}\label{fig:ArcTypePretzelV}
    \end{subfigure}
    \caption{The arc invariants of some two-stranded pretzel tangles from Example~\ref{exa:ArcTwoStrandedPretzel}}
    \label{fig:ArcTypePretzel}
\end{figure}

\begin{example}[\(\BNr\) for connected sums \( (\Lk,p) \# \Li \)]\label{example:non-prime-tangles}

By Remark~\ref{rem:LinkInvariantFrom2EndedTangles}, the type D structure \(\CKhr(\Lk,p)^{\Rcomm[H]}\) associated with a pointed link \((\Lk,p)\) is chain isomorphic to the complex \(\KhTl{T_{\Lk,p}}\) associated with the 2-ended tangle \(T_{\Lk,p}\) obtained by cutting the link open at the basepoint \(p\). 
By adding a trivial vertical tangle strand to \(T_{\Lk,p}\), we obtain a 4-ended tangle which we denote by \((\Lk,p)\#\Li\). As the notation suggests, we may think of this as a \emph{pointed connected sum} of the trivial tangle $\Li$ with the link \((\Lk,p)\) at the basepoints. By construction, the complexes \(\KhTl{T_{\Lk,p}}\) and \(\KhTl{(\Lk,p)\#\Li}\) agree, except that the objects in the former are the trivial 2-ended tangles \(\TrivialTwoTangle\) and the objects of the latter are the trivial 4-ended tangles $\Li$. Now, the complex \(\KhTl{(\Lk,p)\#\Li}\) is essentially the same as \(\DD((\Lk,p)\#\Li)\). Consequently, \(\DD((\Lk,p)\#\Li)\) and \(\CKhr(\Lk,p)^{\Rcomm[H]}\) are also essentially the same. More precisely, 
\[ \DD((\Lk,p) \# \Li)^{\BNAlgH} \simeq \emb (\CKhr(\Lk,p)^{\Rcomm[H]}) \] 
where \( \emb\co\Mod^{\Rcomm[H]}\rightarrow\Mod^{\BNAlgH} \) is the functor induced by the algebra inclusion
\[ 
\embpre\co\Rcomm[H] \rightarrow \BNAlgH   \] sending $1\mapsto \DotC$ and $H \mapsto \DotcobC -S\SaddleBC$.
For an explicit example, see Section~\ref{sec:mutation}, in particular Figures~\ref{fig:kh-field-dependence} and~\ref{fig:curve-field-dependence}. Similar statements are true for pointed connected sums with the trivial tangle \( \Lo \).
\end{example}

\subsection{Figure-8 type curve}\label{sec:FigureEightCurve}
As we will see in Proposition~\ref{prop:NumberOfNonCompactComponentsInArcInvariant}, the arc type invariant \( \BNr(T) \) always  has at least one component which is non-compact. In this section, we define an immersed curve invariant which only has compact components. 

\begin{definition}\label{def:FigureEightCurve}
	Given a pointed 4-ended tangle \( T \), we denote the mapping cone 
	\[
	\Big[\begin{tikzcd}
	q^{-1}h^{-1}\delta^{\frac{1}{2}} \DD(T)
	\arrow{r}{H\cdot \id}
	&
	q^{+1}h^0\delta^{\frac{1}{2}} \DD(T)
	\end{tikzcd}\Big]
	\]
	by \( \DD_1(T) \). This is a well-defined bigraded complex, since \( \gr(H)=q^{-2}h^0 \). We denote the immersed curve corresponding to \( \DD_1(T) \) by \( \Khr(T) \). We call it the \textbf{figure-8 type invariant} of the tangle.
\end{definition}

\begin{wrapfigure}{r}{0.13\textwidth}
  \vspace{0.2cm}
  \centering
  \includegraphics[scale=0.3]{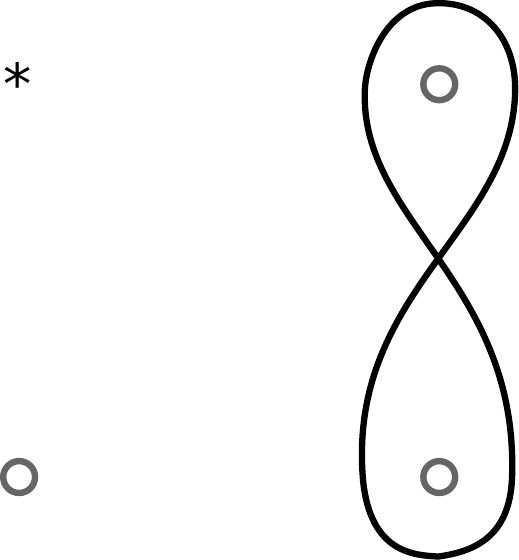}
  \caption{}
  \label{fig:eight}
  \bigskip
\end{wrapfigure}
\myfixwrapfig

\begin{example}[\( \Khr \) for trivial and rational tangles]\label{exa:fig_eights_rational_tangles}
Justifying the choice of the name, the figure-8 invariant  \( \Khr(\Li) \) is obtained by replacing the arc type invariant \( \BNr(\Li)=\DotCarc \) by the figure-8 curve in Figure~\ref{fig:eight}. By Theorem~\ref{thm:MCGaction}, the mod 2 invariants \( \Khr(T;\fieldTwoElements) \) of rational tangles are obtained from this curve by applying the corresponding Dehn twists, and thus the figure-8 invariant of any rational tangle is obtained by turning its arc type invariant (see Example~\ref{ex:curves_for_rational_tangles}) into a figure-8. 
\end{example}

\begin{figure}[t]
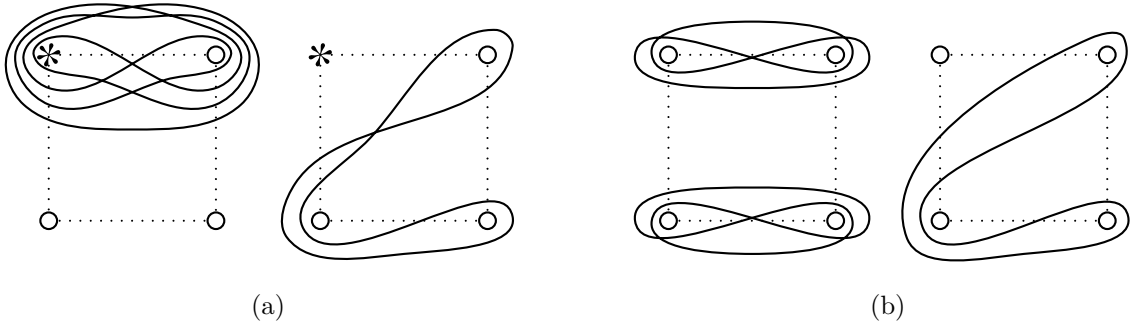

	\centering
	\begin{subfigure}{0.48\textwidth}
		\centering
		$\PTImmersedCurveKhNonRatX$
		$\PTImmersedCurveKhRatX$
		\caption{}\label{fig:2m3ptImmersedCurveKhX}
	\end{subfigure}
	\quad
	\begin{subfigure}{0.48\textwidth}
		\centering
		$\PTImmersedCurveHFNonRatX$
		$\PTImmersedCurveHFRatX$
		\caption{}\label{fig:2m3ptImmersedCurveHFX}
	\end{subfigure}	
	\caption{The immersed curves from Khovanov homology (a) and Heegaard Floer theory (b) for the \( (2,-3) \)-pretzel tangle.}\label{fig:2m3ptImmersedCurves}
\end{figure}

\begin{figure}
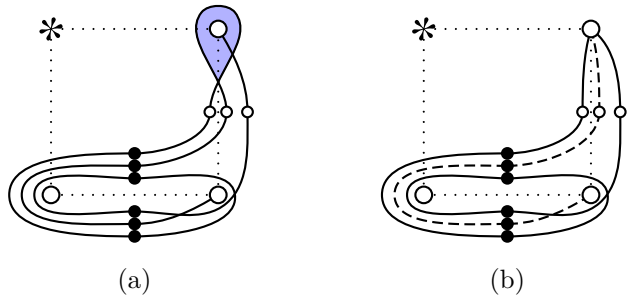

    \centering
    \begin{subfigure}{0.3\textwidth}
    \centering
    $\pretzelarcinvariantMarkedFishtail$
    \caption{}
    \label{fig:pretzelarcinvariantWOfishtail:a}
    \end{subfigure}
    \begin{subfigure}{0.3\textwidth}
    \centering
    $\pretzelarcinvariantWOfishtail$
    \caption{}
    \end{subfigure}
    \caption{The arc invariant for the \( (2,-3) \)-pretzel tangle before (a) and after (b) resolving the single teardrop marked blue at the top right corner. In subfigure (b), the rational tangle component is dashed.}
    \label{fig:pretzelarcinvariantWOfishtail}
\end{figure}

\begin{example}[\( \Khr \) for some 2-stranded pretzel tangles]\label{exa:8curveTwoStrandedPretzel} 
We compute the figure-8 invariants for the same \( (2,-n) \)-pretzel tangles \( T_{2,-n} \) as in Example~\ref{exa:ArcTwoStrandedPretzel} above. 
\begin{itemize}
    \item[\( n=2 \):] Each component of the arc type invariant of the \( (2,-2) \)-pretzel tangle is an embedded arc, so by the previous example, its corresponding figure-8 invariant is obtained by replacing each of them by a figure-8 curve. 
    \item[\( n=3 \):] When passing to the figure-8 invariant, \( \Arc(T_{2,-3}) \) decomposes into two components, which are shown in Figure~\ref{fig:2m3ptImmersedCurveKhX}. In Figure~\ref{fig:pretzelarcinvariantWOfishtail}, we offer a geometric interpretation for this phenomenon: first we resolve the teardrop marked blue in subfigure (a) to obtain the two components in subfigure (b). One of those components (the dashed one) is the curve of a two-twist rational tangle; the other one starts and ends at the top right puncture and winds around the top two punctures twice. Now observe that  \( \Khr(T_{2,-3}) \) is obtained by replacing each of those two components by a figure-8 curve. 
    
    \item[\( n=5 \):] Of course, the mapping cone of \( H\cdot \id \) on the non-compact component of \( \Arc(T_{2,-5}) \) is equal to the same two components as for the \( (2,-3) \)-pretzel tangle. The compact component simply splits into two copies of itself. Thus \( \Khr(T_{2,-5}) \) consists of a curve on the left of Figure~\ref{fig:pretzelarcinvariantWOfishtail:a}, together with three identical curves, which agree with the invariant of a two-twist rational tangle from the right of Figure~\ref{fig:pretzelarcinvariantWOfishtail:a}. 
\end{itemize}
\end{example}

\begin{remark}
It is interesting to compare the figure-8 invariants to the immersed curves \( \HFT(T) \) from~\cite{pqSym,pqMod}, coming from Heegaard Floer homology: 
\begin{itemize}
    \item \( \HFT(T) \) for a rational tangle \( T \) is a single embedded curve which can be obtained from the figure-8 invariant by resolving the self-intersection. 
    \item Up to grading shifts, we have
    \(\Khr(T_{2,-2})=\Khr(T_{2})\oplus\Khr(T_{-2})\). The invariant \(\HFT(T_{2,-2}) \) also contains two rational components, namely \(\HFT(T_2)\) and \(\HFT(T_{-2})\), which can be obtained from \(\Khr(T_{2,-2})\) by resolving the self-intersections of \(\Khr(T_{2})\) and \(\Khr(T_{-2})\) as above. However, \(\HFT(T_{2,-2})\) contains two more (irrational) components, which have no corresponding counterparts in \(\Khr(T_{2,-2})\). 
    \item \( \HFT(T_{2,-3}) \), shown in Figure~\ref{fig:2m3ptImmersedCurveHFX}, has three components, while \(\Khr(T_{2,-3})\) only has two. However, both invariants have a component which, on its own, is the invariant of the same 2-twist rational tangle, namely 
\( \TwistTwoUp\relax\).
    \item The multicurve \(\HFT(T_{2,-5})\) is obtained from \(\HFT(T_{2,-3})\) by adding two more copies of its rational component. The same is true for \(\Khr(T_{2,-5})\) and \(\Khr(T_{2,-3})\). 
    A similar behaviour can also be observed for the two immersed curve invariants of other 2-stranded pretzel tangles.
\end{itemize}
\end{remark}

In all examples above (in fact, for all examples that we know the invariants for), \( \Khr(T) \) and \( \HFT(T) \) contain a component which agrees with the respective invariant of a rational tangle. As the \( (2,-2) \)-pretzel tangle illustrates, there might be multiple such components corresponding to different rational tangles. However, if a tangle has no closed component, we have so far only observed such components corresponding to the same rational tangles.
\begin{question} Is this a general phenomenon? In other words, given a tangle \( T \) without closed components, is there a unique rational tangle \( T_R \) such that the immersed curve invariant (\(\Khr\) from Khovanov / \( \HFT \) from Heegaard Floer) of \( T \) contains a component which is the invariant of \( T_R \)?
\end{question}

Of course, in the light of Dowlin's spectral sequence from Khovanov homology to knot Floer homology \cite{Dowlin}, one might ask if a similar general relationship between the Khovanov and Heegaard Floer tangle invariants can be made precise. However, the comparison above also raises the question if there might be an arc type invariant for tangle Floer homology. The answer to this is yes: such an arc type invariant can be easily defined as the image of the generalised peculiar module \( \CFTminus(T) \) defined in~\cite{pqMod} under a suitable quotient functor. Alternatively, it should be possible to define such an arc invariant using a suitable generalization of Ozsváth and Szabó's algebraic tangle invariants for \( (1, 3) \)-tangles. This leaves off with the questions of whether a gluing theorem works in this setting.  

If an immersed curve \( L \) contains at most one teardrop and no closed components (as is the case for \( \Khr(T)(T_{2,-3}) \)) then the curve corresponding to the mapping cone of \( H\cdot \id \) on \( L \) is obtained by resolving this teardrop and replacing every resulting non-compact component by an immersed figure-8 curve. This follows from the classification of morphisms in terms of intersection points from Section~\ref{sec:classification}. \begin{question}\label{que:ArcToFigureEights}How does this generalize to arbitrary immersed curves?
\end{question}

\subsection{A whole family of compact curves}
\begin{definition}\label{def:Kh_curve}
Let \( \Kh(T) \) be a curve invariant associated with the type~D structure
\[ 
\left[ 
    \begin{tikzcd}[ampersand replacement=\&]
        q^{-2}h^{-1}\DD(T) \arrow{r}{H \cdot \id} \& q^{0}h^{0}\DD(T) \\
        q^{0}h^{-1}\DD(T)  \arrow{ru}{2 \cdot \id} \arrow{r}{H \cdot \id} \& q^{2}h^{0} \DD(T) 
    \end{tikzcd}
    \right]
\]
\end{definition}
Observe that over \( \fieldTwoElements \) one has \( \Kh(T;\fieldTwoElements) = q^{-1}h^{0}\Khr(T;\fieldTwoElements) \oplus q^{+1}h^{0}\Khr(T;\fieldTwoElements) \), due to the vanishing of the diagonal \( (2 \cdot \id) \) arrows in the definition of \( \Kh(T;\fieldTwoElements) \). Over \( \fieldTwoIsNotZero \) (where $\fieldTwoIsNotZero$ is any field with characteristic not equal to 2) the curve \( \Kh(T;\fieldTwoIsNotZero) \)  corresponds to a type~D structure \[ \DD_2(T;\fieldTwoIsNotZero) \coloneqq [q^{-2}h^{-1}\delta^{1-\frac{2}{2}} \DD(T) \xrightarrow{H^2\cdot \id} q^{+2}h^{0}\delta^{\frac{2}{2}} \DD(T)] \] which is obtained by cancellation of \( (2 \cdot \id) \) arrows (using Lemma~\ref{lem:AbstractCancellation}) in the definition of \( \Kh(T;\fieldTwoIsNotZero) \). This dependence on the field is illustrated in the case of the trivial tangle \( \Kh(\Li) \) in Observation~\ref{obs:2_torsion_in_Kh}.

\begin{definition}\label{def:whole_family_of_curves}
	Given a pointed 4-ended tangle \( T \), we denote the mapping cone 
	\[
	\Big[\begin{tikzcd}
	q^{-n}h^{-1}\delta^{1-\frac{n}{2}} \DD(T)
	\arrow{r}{H^n\cdot \id}
	&
	q^{+n}h^{0}\delta^{\frac{n}{2}} \DD(T)
	\end{tikzcd}\Big]
	\]
	by \( \DD_n(T) \). Note that for \( n=1 \), \( \DD_n(T)=\DD_1(T) \) from Definition~\ref{def:FigureEightCurve}. Again, this is a well-defined bigraded complex, since \( \gr(H)=q^{-2}h^0 \). We denote the immersed curve corresponding to \( \DD_n(T) \) by \( \Eight[n](T) \). We call it the \textbf{\( \mathbf{n}^\text{th} \) figure-8 type invariant} of the tangle.  
\end{definition}

\begin{remark}\label{rmk:Kh_and_Khr_in_terms_of_Ln}
Note that the family \( \Eight[n](T) \) generalizes the previous compact curves we introduced:
\( \Khr(T)=\Eight[1](T)\), \(\Kh(T;\fieldTwoIsNotZero)=\Eight[2](T;\fieldTwoIsNotZero)\), \(\Kh(T;\fieldTwoElements)=q^{-1}h^{0}\Eight[1](T) \oplus q^{+1}h^{0}\Eight[1](T)\).
\end{remark}

\begin{wrapfigure}{r}{0.13\textwidth}
  \vspace{0.2cm}
  \centering
  \includegraphics[scale=0.3]{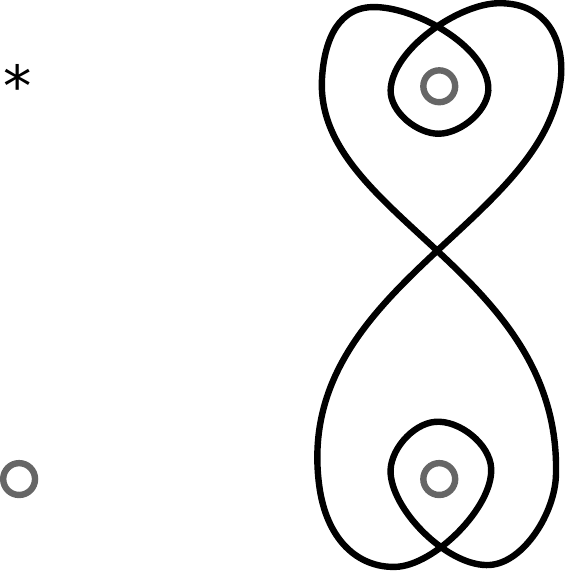}
  \caption{}
  \label{fig:heart}
  \bigskip
\end{wrapfigure}
\myfixwrapfig

\begin{example}[\( L_n \) for trivial tangles and rational tangles]
The \( n^\text{th} \) figure-8 invariant of the tangle \( \Li \) is obtained by replacing the arc type invariant \( \BNr(\Li)=\DotCarc \) by a figure-8 curve which winds around each of the two punctures \( n \) times; for example, \( \Eight[2](\Li) \) is the curve in Figure~\ref{fig:heart}. Again, the same is true for \( \Eight[n](T;\fieldTwoElements) \) for any rational tangle by the naturality of the mapping class group. Interestingly, for the \( (2,-3) \)-pretzel tangle, the algorithm from Example~\ref{exa:fig_eights_rational_tangles} of turning an arc into generalized figure-8 works for \( n>1 \), since the curve \( \Arc(T_{2,-3}) \) winds around any puncture at most once, so in particular strictly less than \( n \) times. In fact, for large \( N \), we can recover \( \Arc(T) \) from \( \Eight[N](T) \) by Observation~\ref{obs:LargeHAction}. 
\end{example}


\subsection{Properties of curves}
\begin{proposition}\label{prop:NumberOfNonCompactComponentsInArcInvariant}
Let \(T\) be a tangle with exactly \(l\) closed components. Then the curve \( \Arc(T) \) has exactly \(2^{l}\) non-compact components. Moreover, at each of the two non-special punctures which are connected by the tangle \(T\), there are exactly \(2^l\) ends.
\end{proposition}
\begin{proof}
    Let us enumerate the tangle ends from 1 to 4 in anticlockwise direction, starting at the special puncture $\ast$. Suppose that the tangle \( T \) connects the tangle ends 2 and 4. (The other two cases can be handled similarly.) For \( i\in\{2,3,4\} \), let \( n_i \) be the number of ends of \( \Arc(T) \) at the puncture \( i \). For \( i,j\in\{2,3,4\} \) with \( i\neq j \), let \( \Lk_{i,j} \) be the link obtained by pairing the tangle \( T \) with some rational tangle connecting the tangle ends \( i \) and \( j \). Then \( \Lk_{2,3} \) and \( \Lk_{3,4} \) have \( l+1 \) components, whereas \( \Lk_{2,4} \) has \( l+2 \) components. Hence, by 
    Proposition~\ref{prop:towersReduced}, \( \BNr(\Lk_{2,3}) \) and \( \BNr(\Lk_{3,4}) \) have \( 2^l \) \( \Rcomm[H] \)-towers each, whereas \( \BNr(\Lk_{2,4}) \) has \( 2^{l+1} \) \( \Rcomm[H] \)-towers. By the geometric pairing Theorem~\ref{thm:pairing}, and the fact that each tower comes from from the wrapping behavior of non-compact curves, we obtain the linear system
    \begin{align*}
        n_2+n_3 &= 2^l
        &
        n_2+n_4 &= 2^{l+1}
        &
        n_3+n_4 &= 2^l
    \end{align*}
    which is equivalent to \( n_2=n_4=2^l \) and \( n_3=0 \).
\end{proof}

\begin{question}
Does each of the non-compact components of \(\Arc(T)\) end on distinct punctures?
\end{question}

\begin{proposition}\label{prop:compactness_eights}
    The multicurve \( \Eight[n](T) \) is compact for any \( n>0 \).
\end{proposition}
\begin{proof}
 This follows from the Pairing Theorem (Theorem~\ref{thm:pairing}), more precisely, pairing (\ref{pairing:coneH}): the wrapped Lagrangian Floer homology of \( \Eight[n](T) \) with any of the arcs corresponding to trivial tangles is \( [\BNr(\Lk) \xrightarrow{H^n}\BNr(\Lk)] \) for some link $\Lk$. This is finite dimensional, so \( \Eight[n](T) \) cannot contain any non-compact components. 
\end{proof}

\begin{proposition}\label{prop:noembeddedcurves}
For any pointed 4-ended tangle \( T \), no component of \( \Arc(T) \),  \(\Khr(T)\), \(\Kh(T)\) or more generally \( \Eight[n](T) \) for \( n>0 \) is embedded and compact. In fact, no compact embedded curve can be equipped with a quantum grading.
\end{proposition}

\begin{wrapfigure}{r}{0.15\textwidth}
	\centering
	$\EmbeddedCurvesX$
	\vspace*{-7pt}
	\caption{}\label{fig:EmbeddedCurvesEight}
\end{wrapfigure}
\myfixwrapfig

\begin{proof}
It suffices to show the second statement. Note that the quantum grading is independent of the field $\field$ of coefficients. So let us assume without loss of generality that $\field=\fieldTwoElements$. 
Then, any embedded compact curve can be mapped to one of the three curves in Figure~\ref{fig:EmbeddedCurvesEight} by iteratively applying bimodules for suitable Dehn twists from Section~\ref{sec:MCGaction}. So it suffices to see that none of these three curves admits a quantum grading. This is elementary.
\end{proof}
\section{Pairing theorem}\label{sec:Pairing}

We can now combine results from previous sections, namely Theorem~\ref{thm:PairingMorLagrangianFH} and Proposition~\ref{prop:alg_pairing_kh}, to show how to interpret Khovanov and Bar-Natan homology of a link in \( S^3 \) in terms of the wrapped Lagrangian Floer theory in the 4-punctured sphere.

\subsection{\texorpdfstring{$\BNr(\Lk)$, $\Khr(\Lk)$ and $\Kh(\Lk)$  viewed as wrapped Lagrangian Floer homology}{BN(L) and Kh(L) viewed as wrapped Lagrangian Floer homology}}
Given a multicurve \( L \) in $\FourPuncturedSphere$, let \( \mirror(L) \) denote the multicurve defined as follows: the underlying oriented curves of \( \mirror(L) \) are obtained as the image of the underlying oriented curves of \( L \) under the involution which rotates the Figure~\ref{fig:ArcSystemForTypeDStructures} about the middle horizontal line by $180^\circ$. The bigrading of \( \mirror(L) \) is equal to the reversed bigrading of \( L \). Finally, the local systems of \( L \) and \( \mirror(L) \) are obtained by inversion and transposition.

\begin{proposition}\label{prop:mirrorsAndCurves}
For any multicurve \(L\) we have \(\Pi(\mirror(L))=\mirror(\Pi(L))\). Moreover, if \(T\) is a pointed 4-ended tangle then \(\mirror(\Arc(T))=\Arc(\mirror T )\) and, similarly, \(\mirror(\Eight[i](T))=h^{1}\Eight[i](\mirror T )\) for any positive integer \(i\).
\end{proposition}
\begin{proof}
    The first part follows directly from Definition~\ref{def:mirrorsTangleAndTypeD} together with the correspondence from Section~\ref{subsec:classification:precurves_algebra} between complexes and precurves. In particular, note that the order of crossover arrows on the arcs \( b \) and \( c \) stays the same under mirroring, but the orientation of those crossover arrows is reversed. Since the underlying curves of \( \mirror(L) \) pass through the two arcs in the same direction as the underlying curves of \( L \), this corresponds precisely to inverting and transposing the local systems. (This is where the difference to the analogous Proposition~5.4 in~\cite{pqMod} lies.)
    The second part follows from the first in conjunction with Proposition~\ref{prop:mirrorsAndTypeDstructures}. 
\end{proof}

\begin{theorem}[Pairing Theorem]\label{thm:pairing}
Let \(T_1\) and \(T_2\) be two pointed 4-ended tangles, and \(\Lk\) be the link \(\Lk(T_1,T_2)\) from Definition~\ref{def:tanglepairing}. Then over any field \(\field\) we have:
\begin{align}
    \BNr(\Lk) 
    &\cong q^{1}h^{0}\HF(\mirror(\BNr(T_1)),\BNr(T_2)) \label{pairing:BNr}\tag{P1}
    \\
    \Khr(\Lk) 
    &\cong \HF(\mirror(\BNr(T_1)),\Khr(T_2)) \cong \HF(\mirror(\Khr(T_1)),\BNr(T_2)) \label{pairing:Khr}\tag{P2}
    \\
    \Kh(\Lk)
    & \cong \HF(\mirror(\BNr(T_1)),\Kh(T_2)) \cong \HF(\mirror(\Kh(T_1)),\BNr(T_2)) \label{pairing:Kh}\tag{P3}
    \\
    q^{-1}h^{-1}\Khr(\Lk)  \oplus  q^{+1}h^{0}\Khr(\Lk)
    &
    \cong 
    \HF(\mirror(\Khr(T_1)),\Khr(T_2))
    \label{pairing:KhrKhr}\tag{P4}
\end{align}
More generally, if \(i\geq 1\) and
\(
\coker H^i \coloneqq \Homology
\Big[\begin{tikzcd}
q^{-i-1}h^{-1}\CBNr(\Lk)
\arrow{r}{H^i\cdot\id}
&
q^{i-1}h^{0}\CBNr(\Lk)
\end{tikzcd}\Big],
\)
then over any field \(\field\),
\begin{align}
\coker H^i 
&\cong  
\HF(\mirror(\BNr(T_1)),\Eight[i](T_2)) 
\cong
\HF(\mirror(\Eight[i](T_1)),\BNr(T_2))
\label{pairing:coneH}\tag{P5}
\end{align}
Finally, if \(i,j\geq1\), \(\mathfrak{M}=\max(i,j)\) and \(\mathfrak{m}=\min(i,j)\), then over any field \(\field\),
\begin{align}
q^{-\mathfrak{M}} h^{-1}\coker H^{\mathfrak{m}} 
\oplus
q^{\mathfrak{M}} h^{0}\coker H^{\mathfrak{m}} 
&
\cong 
\HF(\mirror(\Eight[i](T_1)),\Eight[j](T_2))
\label{pairing:ConeHConeH}\tag{P6}
\end{align}
\end{theorem}

\begin{proof}
\eqref{pairing:BNr} follows from 
\[ \HF(\mirror(\BNr(T_1)),\BNr(T_2)) \overset{(a)}{\cong} \Homology\bigg(\Mor\Big(\mirror(\DD(T_1)),\DD(T_2)\Big)\bigg) \overset{(b)}{\cong} q^{-1}h^{0}\BNr(\Lk)\]
where (a) is Theorem~\ref{thm:PairingMorLagrangianFH}, and (b) is Proposition~\ref{prop:alg_pairing_kh}. This is the main pairing result, which implies all the others, as we explain below.

\eqref{pairing:Khr} follows from the more general statement \eqref{pairing:coneH} with \( i=1 \), because \( \Khr(T)=\Eight[1](T_1) \) and \(\coker H^1=\Khr(\Lk)\) by Proposition~\ref{prop:kh_mapping_cones}. 

\eqref{pairing:Kh} over \( \fieldTwoIsNotZero \) follows from \eqref{pairing:coneH} with \( i=2 \), because \( \Kh(T;\fieldTwoIsNotZero)=\Eight[2](T_1;\fieldTwoIsNotZero) \) and \(\Kh(\Lk;\fieldTwoIsNotZero) = \coker H^2\) by Proposition~\ref{prop:kh_mapping_cones}.

\eqref{pairing:Kh} over \( \fieldTwoElements \) follows from the fact that \( \Kh(T;\fieldTwoElements) = q^{-1}h^{0}\Khr(T;\fieldTwoElements) \oplus q^{+1}h^{0}\Khr(T;\fieldTwoElements) \). Thus we have
\begin{align*}
&\HF(\mirror(\Kh(T_1;\fieldTwoElements)),\BNr(T_2;\fieldTwoElements))\\ 
& \cong q^{-1}h^{0}\HF(\mirror(\Khr(T_1;\fieldTwoElements)),\BNr(T_2;\fieldTwoElements))  \oplus  q^{+1}h^{0}\HF(\mirror(\Khr(T_1;\fieldTwoElements)),\BNr(T_2;\fieldTwoElements)) \\
& \overset{(a)}{\cong}  q^{-1}h^{0} \Khr(\Lk;\fieldTwoElements) \oplus q^{+1}h^{0}  \Khr(\Lk;\fieldTwoElements) \overset{(b)}{\cong} \Kh(\Lk;\fieldTwoElements)
\end{align*}
where \( (a) \) follows from (\ref{pairing:BNr}), and \( (b) \) is a well known fact: Khovanov homology over \( \fieldTwoElements \) splits into two summands, each of which is isomorphic to reduced Khovanov homology.

\eqref{pairing:KhrKhr} follows from \eqref{pairing:ConeHConeH} with \( i=1=j \).

\eqref{pairing:coneH} can be proved as follows: by Theorem~\ref{thm:PairingMorLagrangianFH}, 
\[ 
\HF(\mirror(\Eight[i](T_1)),\Arc(T_2)) 
\cong
\Homology
\bigg[
\Mor\Big(\!\!
\begin{tikzcd}
q^{-i}h^{0}\mirror(\DD(T_1))
\arrow{r}{H^i}
&
q^{+i}h^{1}\mirror(\DD(T_1))
\end{tikzcd}\!\!,
\DD(T_2)
\Big)
\bigg]
\]
By the definition of morphism spaces as bigraded chain complexes, this is equal to the homology of the cone
\begin{align*}
\Big[ \Mor\Big(q^{+i}h^{1}\mirror(\DD(T_1)),\DD(T_2)\Big)
&\rightarrow
\Mor\Big(q^{-i}h^{0}\mirror(\DD(T_1)),\DD(T_2)\Big) \Big]
\\
f
&\mapsto
-(-1)^{h(f)}\cdot f\cdot H^i
\end{align*}
By Lemma~\ref{lem:shifting_morphism_spaces}, this is isomorphic to the mapping cone 
\[ 
\Big[
\begin{tikzcd}
q^{-i}h^{-1}\Mor\Big(\mirror(\DD(T_1)),\DD(T_2)\Big)
\arrow{r}{H^i}
&
q^{i}h^{0}\Mor\Big(\mirror(\DD(T_1)),\DD(T_2)\Big)
\end{tikzcd}\Big]
\]
and applying (\ref{pairing:BNr}) gives
\[  
\Big[
\begin{tikzcd}
q^{-i-1}h^{-1}\CBNr(\Lk) 
\arrow{r}{H^i}
&
q^{i-1}h^{0}\CBNr(\Lk) 
\end{tikzcd}\Big]
=
\coker H^i
\]
The second isomorphism of \eqref{pairing:coneH} is proved analogously.

\eqref{pairing:ConeHConeH} is proved using a similar argument. Again, by
Theorem~\ref{thm:PairingMorLagrangianFH} and the definition of morphism spaces, we have
\begin{align*}
&~
\HF(\mirror(\Eight[i](T_1)),\Eight[j](T_2)) 
\\
\cong
&~
\Homology\bigg[\Mor\Big(
\!\!
\begin{tikzcd}[ampersand replacement=\&]
q^{-i}h^{0}\mirror(\DD(T_1))
\arrow{r}{H^i}
\&
q^{+i}h^{1}\mirror(\DD(T_1))
\end{tikzcd}\!\!,
\!\!
\begin{tikzcd}[ampersand replacement=\&]
q^{-j}h^{-1}\DD(T_2)
\arrow{r}{H^j}
\&
q^{+j}h^{0}\DD(T_2)
\end{tikzcd}\!\!
\Big)\bigg] 
\\
\cong
&~
\Homology
\left[
\begin{tikzcd}[column sep=-3cm, row sep = 1cm, ampersand replacement=\&] 
\&  
\Mor\Big(q^{+i}h^{1}\mirror(\DD(T_1)),q^{-j}h^{-1}\DD(T_2)\Big) 
\arrow[dl,"f\mapsto f\cdot H^j" description] 
\arrow[rd,"f\mapsto -(-1)^{h(f)}\cdot f\cdot H^i" description]
\\
\Mor\Big(q^{+i}h^{1}\mirror(\DD(T_1)),q^{+j}h^{0}\DD(T_2)\Big)
\arrow[dr,"f\mapsto -(-1)^{h(f)}\cdot f\cdot H^i" description] 
\&\&
\Mor\Big(q^{-i}h^{0}\mirror(\DD(T_1)),q^{-j}h^{-1}\DD(T_2)\Big) 
\arrow[ld,"f\mapsto f\cdot H^j" description]
\\
\&
\Mor\Big(q^{-i}h^{0}\mirror(\DD(T_1)),q^{+j}h^{0}\DD(T_2)\Big)
\end{tikzcd}
\right]
\end{align*}
By Lemma~\ref{lem:shifting_morphism_spaces} and (\ref{pairing:BNr}), this is isomorphic to
\begin{align*}
&~ 
q^{-1}h^{0}\Homology\left[
\begin{tikzcd}[column sep=0pt,ampersand replacement=\&]   
\&
q^{-i-j}h^{-2}\CBNr(\Lk) \arrow[ld,"-H^j" description] 
\arrow[rd,"H^i" description]
\\
q^{j-i}h^{-1}\CBNr(\Lk)
\arrow[rd,"H^i" description] 
\arrow[rr, dashed, "H^{i-j}"]   
\&\&
q^{i-j}h^{-1}\CBNr(\Lk) 
\arrow[ld,"H^j" description]
\\
\&
q^{i+j}h^{0}\CBNr(\Lk)
\end{tikzcd}
\right]
\end{align*}
(but without the dashed arrow). Note that this complex is isomorphic to the complex obtained by switching \( i \) and \( j \). This symmetry allows us to assume without loss of generality that \( i\geq j \). Then we apply the Clean-Up Lemma~\ref{lem:AbstractCleanUp} to the endomorphisms given by the dashed arrow above and obtain
\[ 
q^{-1}h^{0}\Homology \left[
\begin{tikzcd}[column sep=0pt,ampersand replacement=\&]   
\&
q^{-i-j}h^{-2}\CBNr(\Lk) \arrow[ld,"H^j" description] 
\\
q^{j-i}h^{-1}\CBNr(\Lk)
\&\&
q^{i-j}h^{-1}\CBNr(\Lk) 
\arrow[ld,"H^j" description]
\\
\&
q^{i+j}h^{0}\CBNr(\Lk)
\end{tikzcd}
\right]
\]
which is equal to 
\( q^{-i}h^{-1}\coker H^j
\oplus
q^{+i}h^{0}\coker H^j \).
\end{proof}

\subsection{Basic examples}

In the remainder of this section, we illustrate the Pairing Theorem~\ref{thm:pairing} above in a number of examples. Let us start with three very basic ones: the unknot, the 2-component unlink and the right-handed trefoil.

\begin{figure}[ph]
    \centering
    \begin{subfigure}{\textwidth}
		\[ \PairingClosureDiagramUnknot=\PairingStandardDiagramUnknot\]
		\caption{A splitting of the unknot \( \Circle=\Lk\left(\LiRed,\LoBlue\right) \)}\label{fig:exa:Pairing:Unknot:Diagram}
	\end{subfigure}
	\begin{subfigure}{\textwidth}
	\centering
    $\PairingUnknotArcArc$
	$\vc{
    	\begin{tikzpicture}[scale=.5]
	    \draw[lightgray,step=1] (-0.5,-3.5) grid (2.5,1.5);
        \draw[line width=1pt,->] (-0.5,-3) -- (3,-3) node[right] {\( h \)};
        \draw[line width=1pt,->] (0,-3.5) -- (0,2) node[left,above] {\( q \)};
        \draw (0.5,-3) node[below] {\scriptsize \( 0 \)};
        \draw (1.5,-3) node[below] {\scriptsize \( 1 \)};
        \draw (0,-2.5) node[left] {\scriptsize \( -6 \)};
        \draw (0,-1.5) node[left] {\scriptsize \( -4 \)};
        \draw (0,-0.5) node[left] {\scriptsize \( -2 \)};
        \draw (0,0.5) node[left] {\scriptsize \( 0 \)};
        \draw(0.5, 0.5) node {\( w_0 \)};
        \draw(0.5, -0.5) node {\( w_1 \)};
        \draw(0.5, -1.5) node {\( w_2 \)};
        \draw(0.5, -2.5) node {\( \raisebox{0.2cm}{\vdots} \)};
    \end{tikzpicture}}$
    \caption{\( q^1 h^0 \HF \left( \textcolor{red}{\mirror(\BNr(\LiRed))},\textcolor{blue}{\BNr(\LoBlue)} \right) \cong  \BNr\left( \Circle \right) \)}\label{fig:exa:Pairing:Unknot:ArcArc}
    \end{subfigure}
    \begin{subfigure}{\textwidth}
    \centering
    $\PairingUnknotLoopArc$
    $\vc{
    	\begin{tikzpicture}[scale=.5]
	    \draw[lightgray,step=1] (-0.5,-0.5) grid (2.5,2.5);
	    \draw[line width=1pt,->] (-0.5,0) -- (3,0) node[right] {\( h \)};
	    \draw[line width=1pt,->] (0,-0.5) -- (0,3) node[above] {\( q \)};
        \draw (0.5,0) node[below] {\scriptsize \( 0 \)};
        \draw (1.5,0) node[below] {\scriptsize \( 1 \)};
        \draw (0,0.5) node[left] {\scriptsize \( 0 \)};
        \draw (0,1.5) node[left] {\scriptsize \( 2 \)};
        \draw(0.5, 0.5) node {\( w_0 \)};
    \end{tikzpicture}}$
    \caption{\( \HF \left( \textcolor{red}{\mirror(\Khr(\LiRed))},\textcolor{blue}{\BNr(\LoBlue)}  \right) \cong \Khr\left( \Circle  \right) \)}\label{fig:exa:Pairing:Unknot:EightArc}
    \end{subfigure}
    \caption{Illustration for Example~\ref{exa:Pairing:Unknot}}\label{fig:exa:Pairing:Unknot}
\end{figure}
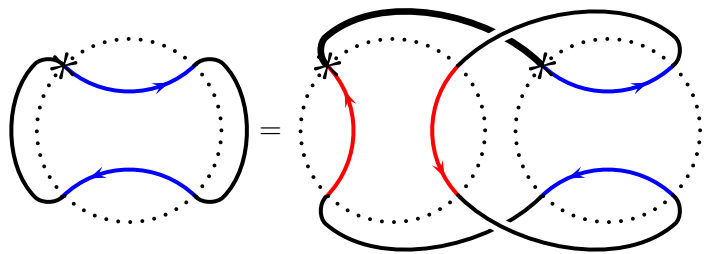
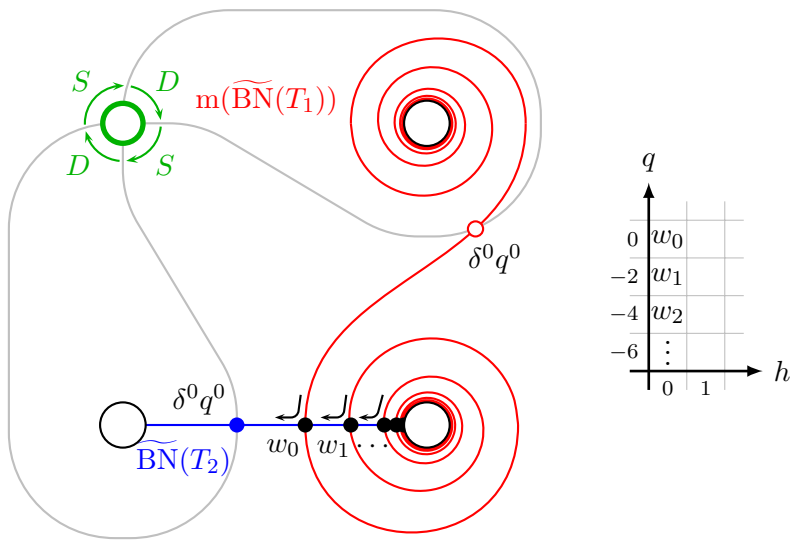
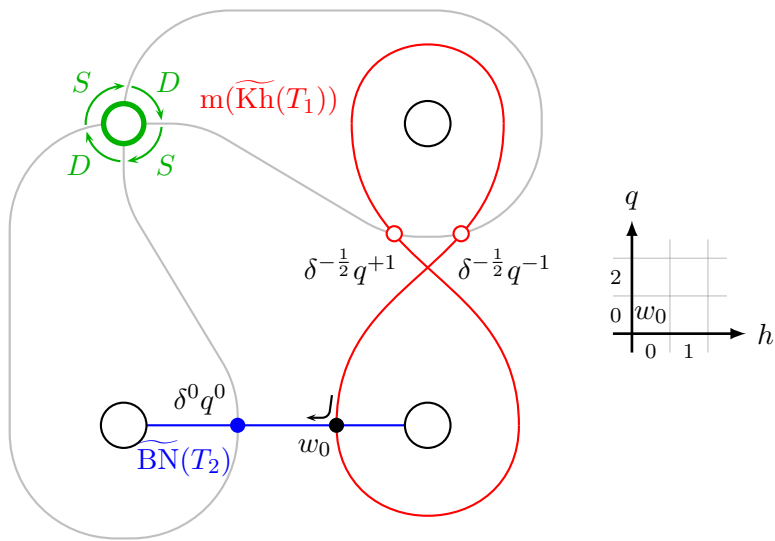

\begin{figure}[ph]
    \centering
    \begin{subfigure}{\textwidth}
		\[ \PairingClosureDiagramUnlink=\PairingStandardDiagramUnlink\]
		\caption{A splitting of the two-component unlink \( \Circle \Circle =\Lk\left(\LoRed,\LoBlue  \right)\)}\label{fig:exa:Pairing:Unlink:Diagram}
	\end{subfigure}
	\begin{subfigure}{\textwidth}
	\centering
    $\PairingUnlinkArcArc$
	$\vc{
    	\begin{tikzpicture}[scale=.5]
        \draw[lightgray] (2.5,-3.5) -- (2.5,2.5);
        \draw[lightgray] (3.5,-3.5) -- (3.5,2.5);
        \draw[lightgray] (-0.5,2) --  (4,2);
        \draw[lightgray] (-0.5,1) --  (4,1);
        \draw[lightgray] (-0.5,0) --  (4,0);
        \draw[lightgray] (-0.5,-1) -- (4,-1);
        \draw[lightgray] (-0.5,-2) -- (4,-2);
        \draw[lightgray] (-0.5,-3) -- (4,-3);
        \draw[line width=1pt,->] (-0.5,-3) -- (4.5,-3) node[right] {\( h \)};
        \draw[line width=1pt,->] (0,-3.5) -- (0,3) node[left,above] {\( q \)};
        \draw (1.25,-3) node[below] {\scriptsize \( 0 \)};
        \draw (3,-3) node[below] {\scriptsize \( 1 \)};
        \draw (0,-2.5) node[left] {\scriptsize \( -7 \)};
        \draw (0,-1.5) node[left] {\scriptsize \( -5 \)};
        \draw (0,-0.5) node[left] {\scriptsize \( -3 \)};
        \draw (0,0.5) node[left] {\scriptsize \( -1 \)};
        \draw (0,1.5) node[left] {\scriptsize \( 1 \)};
        \draw(1.25, 1.5) node {\( w_0 \)};
        \draw(1.25, 0.5) node {\( w_1~w'_1 \)};
        \draw(1.25, -0.5) node {\( w_2~w'_2 \)};
        \draw(1.25, -1.5) node {\( w_3~w'_3 \)};
        \draw(1.25, -2.5) node {\( \raisebox{0.2cm}{\vdots~~~\vdots} \)};
    \end{tikzpicture}}$
    \caption{\( q^1 h^0 \HF\left(\textcolor{red}{\mirror(\BNr(\LoRed))},\textcolor{blue}{\BNr(\LoBlue)} \right) \cong \BNr\left(\Circle \Circle  \right) \)}\label{fig:exa:Pairing:Unlink:ArcArc}
    \end{subfigure}
    \begin{subfigure}{\textwidth}
    \centering
    $\PairingUnlinkLoopArc$
	$\vc{
    	\begin{tikzpicture}[scale=.5]
	    \draw[lightgray,step=1] (-0.5,-0.5) grid (2.5,2.5);
        \draw[line width=1pt,->] (-0.5,0) -- (3,0) node[right] {\( h \)};
        \draw[line width=1pt,->] (0,-0.5) -- (0,3) node[above] {\( q \)};
        \draw (0.5,0) node[below] {\scriptsize \( 0 \)};
        \draw (1.5,0) node[below] {\scriptsize \( 1 \)};
        \draw (0,0.5) node[left] {\scriptsize \( -1 \)};
        \draw (0,1.5) node[left] {\scriptsize \( 1 \)};
        \draw(0.5, 0.5) node {\( w'_1 \)};
        \draw(0.5, 1.5) node {\( w_0 \)};
    \end{tikzpicture}}$
    \caption{\(  \HF\left(\textcolor{red}{\mirror(\Khr(\LoRed))},\textcolor{blue}{\BNr(\LoBlue)}  \right) \cong \Khr\left(\Circle \Circle  \right) \)}\label{fig:exa:Pairing:Unlink:EightArc}
    \end{subfigure}
    \caption{Illustration for Example~\ref{exa:Pairing:Unlink}}\label{fig:exa:Pairing:Unlink}
\end{figure}
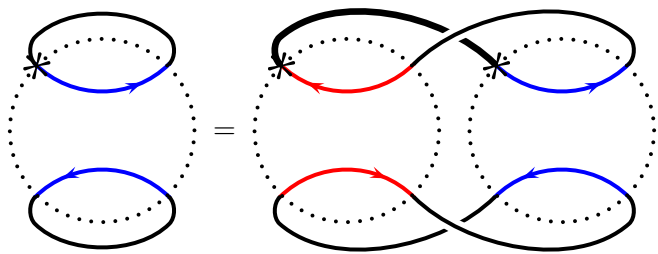
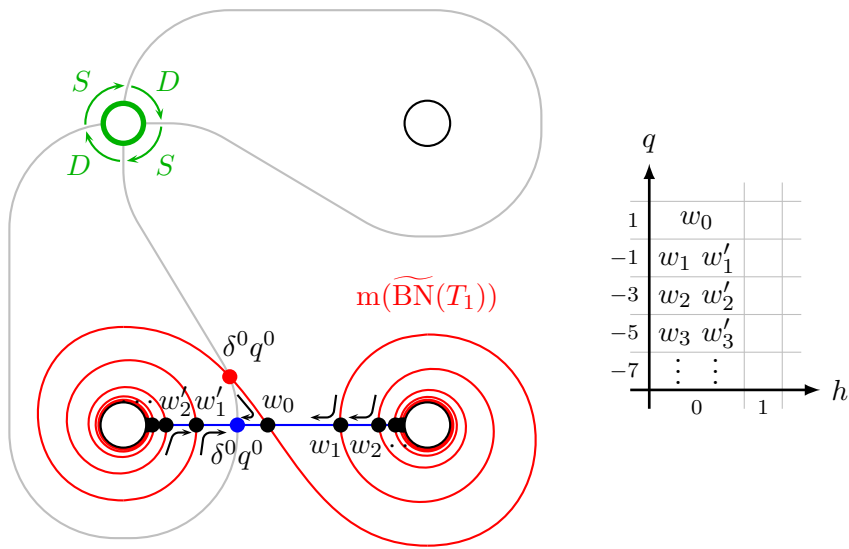
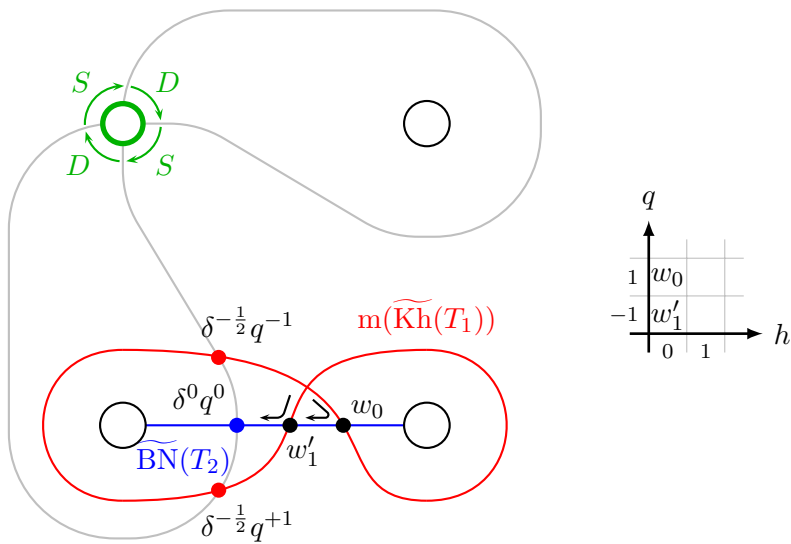

\begin{figure}[ph]
    \centering
    \begin{subfigure}{\textwidth}
		\[ \PairingClosureDiagramTrefoil=\PairingStandardDiagramTrefoil\]
		\caption{A splitting of the right-handed trefoil knot \( T(2,3)= \Lk\left(\LiRed,\ThreeTwistTangleBlue\right) \) into two rational tangles}\label{fig:exa:Pairing:Trefoil:Diagram}
	\end{subfigure}
	\begin{subfigure}{\textwidth}
	\centering
    $\PairingTrefoilArcArc$
	$\vc{
    	\begin{tikzpicture}[scale=.5]
		    \draw[lightgray,step=1] (-0.5,-3.5) grid (4.5,5.5);
            \draw[line width=1pt,->] (-0.5,-3) -- (5,-3) node[right] {\( h \)};
            \draw[line width=1pt,->] (0,-3.5) -- (0,6) node[left,above] {\( q \)};
            \draw (0.5,-3) node[below] {\scriptsize \( 0 \)};
            \draw (1.5,-3) node[below] {\scriptsize \( 1 \)};
            \draw (2.5,-3) node[below] {\scriptsize \( 2 \)};
            \draw (3.5,-3) node[below] {\scriptsize \( 3 \)};
            \draw (0,-2.5) node[left] {\scriptsize \( -6 \)};
            \draw (0,-1.5) node[left] {\scriptsize \( -4 \)};
            \draw (0,-0.5) node[left] {\scriptsize \( -2 \)};
            \draw (0,0.5) node[left] {\scriptsize \( 0 \)};
            \draw (0,1.5) node[left] {\scriptsize \( 2 \)};
            \draw (0,2.5) node[left] {\scriptsize \( 4 \)};
            \draw (0,3.5) node[left] {\scriptsize \( 6 \)};
            \draw (0,4.5) node[left] {\scriptsize \( 8 \)};
            \draw(0.5, 1.5) node {\( w_0 \)};
            \draw(0.5, 0.5) node {\( w_1 \)};
            \draw(0.5, -0.5) node {\( w_2 \)};
            \draw(3.5, 4.5) node {\( v_2 \)};
            \draw(0.5, -1.5) node {\( \raisebox{0.2cm}{\vdots} \)};
    \end{tikzpicture}}$
    \caption{\( q^1 h^0 \HF\left(\textcolor{red}{\mirror(\BNr(\LiRed))},\textcolor{blue}{\BNr(\ThreeTwistTangleBlue)} \right) \cong \BNr\left( T(2,3) \right) \)}\label{fig:exa:Pairing:Trefoil:Resolutions:ArcArc}
    \end{subfigure}
    \begin{subfigure}{\textwidth}
    \centering
    $\PairingTrefoilLoopArc$
$\vc{
    	\begin{tikzpicture}[scale=.5]
		    \draw[lightgray,step=1] (-0.5,-0.5) grid (4.5,5.5);
            \draw[line width=1pt,->] (-0.5,0) -- (5,0) node[right] {\( h \)};
            \draw[line width=1pt,->] (0,-0.5) -- (0,6) node[above] {\( q \)};
            \draw (0.5,0) node[below] {\scriptsize \( 0 \)};
            \draw (1.5,0) node[below] {\scriptsize \( 1 \)};
            \draw (2.5,0) node[below] {\scriptsize \( 2 \)};
            \draw (3.5,0) node[below] {\scriptsize \( 3 \)};
            \draw (0,0.5) node[left] {\scriptsize \( 0 \)};
            \draw (0,1.5) node[left] {\scriptsize \( 2 \)};
            \draw (0,2.5) node[left] {\scriptsize \( 4 \)};
            \draw (0,3.5) node[left] {\scriptsize \( 6 \)};
            \draw (0,4.5) node[left] {\scriptsize \( 8 \)};
            \draw(0.5, 1.5) node {\( w_0 \)};
            \draw(3.5, 4.5) node {\( v_2 \)};
            \draw(2.5, 3.5) node {\( v_1 \)};
    \end{tikzpicture}}$
    \caption{\(  \HF\left(\textcolor{red}{\mirror(\Khr(\LiRed))},\textcolor{blue}{\BNr(\ThreeTwistTangleBlue)}  \right) \cong \Khr\left( T(2,3)  \right) \)}\label{fig:exa:Pairing:Trefoil:Resolutions:EightArc}
    \end{subfigure}
    \caption{Illustration for Example~\ref{exa:Pairing:Trefoil}}\label{fig:exa:Pairing:Trefoil}
\end{figure}
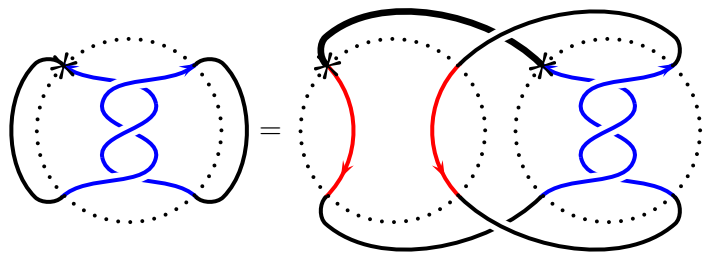
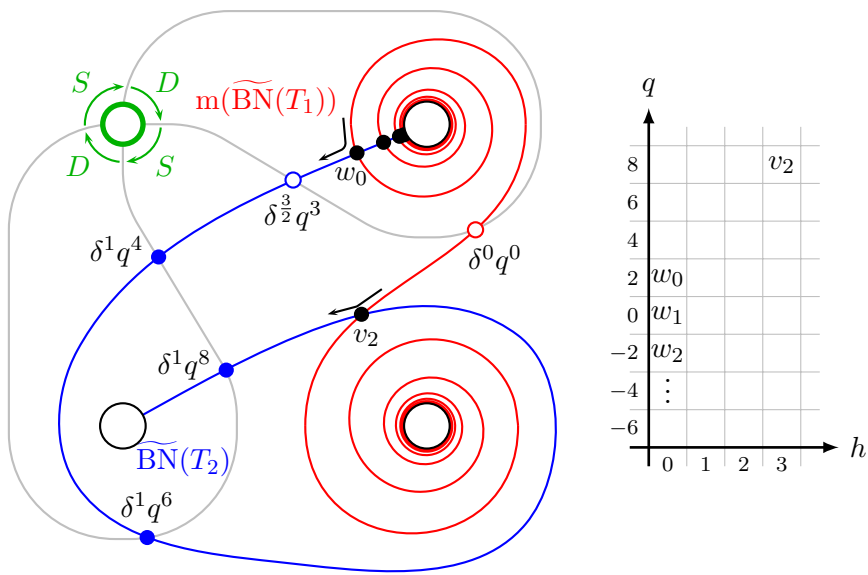
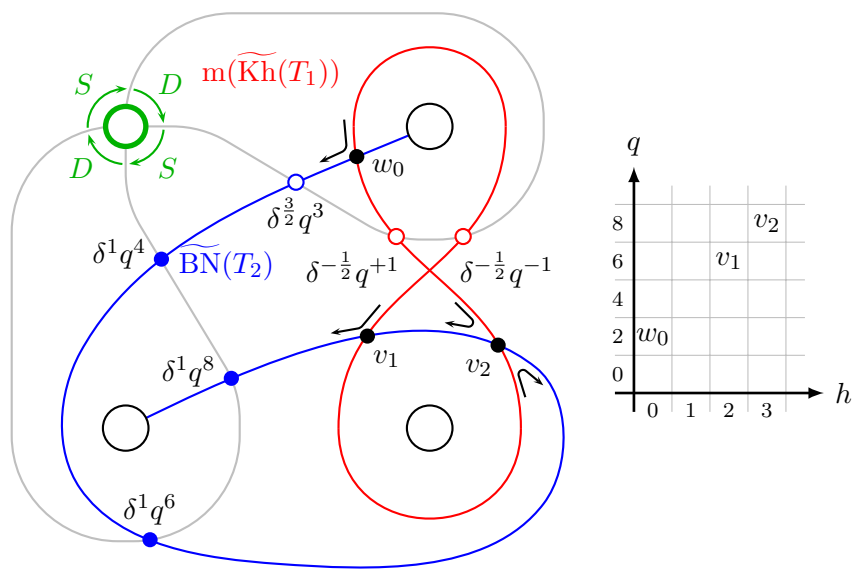

\begin{example}[Unknot]\label{exa:Pairing:Unknot}
	Consider Figure~\ref{fig:exa:Pairing:Unknot:Diagram}, which shows a splitting of the unknot \( \Circle=\Lk(\LiRed,\LoBlue) \) into two trivial tangles \( \LiRed \) and \( \LoBlue \).
	The left-hand side of Figure~\ref{fig:exa:Pairing:Unknot:ArcArc} shows the wrapped Lagrangian Floer homology between \( \textcolor{red}{\mirror(\Arc(\LiRed))} \) and \( \textcolor{blue}{\Arc(\LoBlue)} \), which consists of infinitely many points \( w_i \) where \( i=0,1,2,\dots \). The resolutions of those intersection points, indicated by the black arrows in this figure, are given by the morphisms
	\[ 
	\begin{tikzcd}[ampersand replacement=\&]
	\textcolor{red}{\mirror(\Arc(\LiRed))}
	\arrow{r}{w_i}
	\&
	\textcolor{blue}{\Arc(\LoBlue)}
	\end{tikzcd}
    \quad
	\leftrightarrow
    \quad
    \begin{tikzcd}[ampersand replacement=\&]
    \GGzqh{\DotCred}{0}{0}{0}
    \arrow{r}{S^{1+2i}}
    \&
    \GGzqh{\DotBblue}{0}{0}{0}
    \end{tikzcd}
	\]
	So the bigrading \( \gr(w_i) \) of \( w_i \) is equal to
	\[ 
	\gr\Big(\!\!
	\begin{tikzcd}[ampersand replacement=\&]
	\GGzqh{\DotCred}{0}{0}{0}
	\arrow{r}{S^{1+2i}}
	\&
	\GGzqh{\DotBblue}{0}{0}{0}
	\end{tikzcd}\!\!\Big)
	=
	\gr(\GGzqh{\DotBblue}{0}{0}{0})
	-
	\gr(\GGzqh{\DotCred}{0}{0}{0})
	+
	\gr(S^{1+2i})
	=
	\gr(S^{1+2i})
	=
	q^{-(1+2i)}h^{0}
	\]
	and 
	\[ q^{+1}h^0 \HF(\textcolor{red}{\mirror(\Arc(\LiRed))},\textcolor{blue}{\Arc(\LoBlue)})\cong q^0 h^0\field[H].\]
	This is indeed \( \BNr( \Circle) \), which is also shown on the right-hand side of Figure~\ref{fig:exa:Pairing:Unknot:ArcArc}. Similarly, we can compute \( \Khr(\Circle) \), which according to the Pairing Theorem is equal to 
	\[ \HF(\textcolor{red}{\mirror(\Khr(\LiRed))},\textcolor{blue}{\Arc(\LoBlue)}).\]
	The corresponding calculation is shown in Figure~\ref{fig:exa:Pairing:Unknot:EightArc}. The resolution of the single intersection point \( w_0 \) in this figure is equal to the morphism
	\[ 
	\begin{tikzcd}[ampersand replacement=\&]
	\textcolor{red}{\mirror(\Khr(\LiRed))}
	\arrow{r}{w_0}
	\&
	\textcolor{blue}{\Arc(\LoBlue)}
	\end{tikzcd}	
	\quad
    \leftrightarrow
    \quad
	\begin{tikzcd}[row sep=0pt]
	\GGzqh{\DotCred}{}{-1}{0}
	\arrow[red]{dd}{H}
	\arrow{rd}{S}
	\\
	&
	\GGzqh{\DotBblue}{}{0}{0}
	\\
	\GGzqh{\DotCred}{}{+1}{+1}
	\end{tikzcd}
	\]
	The bigrading \( \gr(w_0) \) of \( w_0 \) is equal to
	\[ 
	\gr\Big(\!\!
	\begin{tikzcd}[ampersand replacement=\&]
	\GGzqh{\DotCred}{0}{-1}{0}
	\arrow{r}{S}
	\&
	\GGzqh{\DotBblue}{0}{0}{0}
	\end{tikzcd}\!\!\Big)
	=
	\gr(\GGzqh{\DotBblue}{0}{0}{0})
	-
	\gr(\GGzqh{\DotCred}{0}{-1}{0})
	+
	\gr(S)
	=
	h^{0}q^{0}.\]
	So we obtain a single generator in bigrading \( h^0q^0 \), which is indeed \( \Khr(\Circle) \).
\end{example}

\begin{example}[Unlink]\label{exa:Pairing:Unlink}
	Consider the splitting of the two-component unlink \( \Circle \Circle = \Lk(\LoRed,\LoBlue) \) from Figure~\ref{fig:exa:Pairing:Unlink:Diagram} and the wrapped Lagrangian Floer theory
	\[ \HF(\textcolor{red}{\mirror(\Arc(\LoRed))},\textcolor{blue}{\Arc(\LoBlue)})\cong\langle\dots,w'_3,w'_2,w'_1,w_0,w_1,w_2,w_3,\dots\rangle\]
	between the two corresponding arc invariants in Figure~\ref{fig:exa:Pairing:Unlink:ArcArc}. Note that the intersection points are in 1:1-correspondence with the standard generators of
	\[ \Mor(\DotBred,\DotBblue)=\langle\dots,D^3,D^2,D,\id,S^2,S^4,S^6,\dots\rangle.\]
	The bigrading of \( w_0 \) is \( \delta^0q^0 \). To see that the above agrees with
	\[ q^{-1}\BNr(\Lk(\LoRed,\LoBlue))\cong q^0 h^0\fieldTwoElements[H]\oplus q^{-2} h^0\fieldTwoElements[H]\]
	as an \( \fieldTwoElements[H] \)-module,
	observe that \( Hw_0=w_1+w'_1 \), so the two \( H \)-towers are generated by \( w_0 \) and \( w'_1 \) (or \( w_0 \) and \( w_1 \)). Finally, Figure~\ref{fig:exa:Pairing:Unlink:EightArc} shows the computation of the reduced Khovanov homology of the unlink: the wrapped Lagrangian Floer theory is generated by two intersection points in bigradings \( q^{1}h^0 \) and \( q^{-1}h^0 \), respectively, which is indeed the reduced Khovanov homology of the two-component unlink.
\end{example}

\begin{figure}[t]
    \centering
    $\PairingTrefoilHeartsArc
    \vc{
    \begin{tikzpicture}[scale=.5]
		\draw[lightgray,step=1] (-0.5,-0.5) grid (4.5,5.5);
        \draw[line width=1pt,->] (-0.5,0) -- (5,0) node[right] {\( h \)};
        \draw[line width=1pt,->] (0,-0.5) -- (0,6) node[above] {\( q \)};
        \draw (0.5,0) node[below] {\scriptsize \( 0 \)};
        \draw (1.5,0) node[below] {\scriptsize \( 1 \)};
        \draw (2.5,0) node[below] {\scriptsize \( 2 \)};
        \draw (3.5,0) node[below] {\scriptsize \( 3 \)};
        \draw (0,0.5) node[left] {\scriptsize \( 1 \)};
        \draw (0,1.5) node[left] {\scriptsize \( 3 \)};
        \draw (0,2.5) node[left] {\scriptsize \( 5 \)};
        \draw (0,3.5) node[left] {\scriptsize \( 7 \)};
        \draw (0,4.5) node[left] {\scriptsize \( 9 \)};
        \draw(0.5, 0.5) node {\( w_1 \)};
        \draw(0.5, 1.5) node {\( w_0 \)};
        \draw(3.5, 4.5) node {\( v_2 \)};
        \draw(2.5, 2.5) node {\( v_1 \)};
    \end{tikzpicture}}$
    \caption{
    \(  \HF  \left( 
    \mirror ( \Kh( \protect\LiRed; \fieldTwoIsNotZero) ),
    \BNr(\protect\ThreeTwistTangleBlue ; \fieldTwoIsNotZero) 
    \right) \cong
    \Kh\left( T(2,3); \fieldTwoIsNotZero  \right)
    \)
    }
    \label{fig:exa:Pairing:Trefoil:unreduced}
\end{figure}

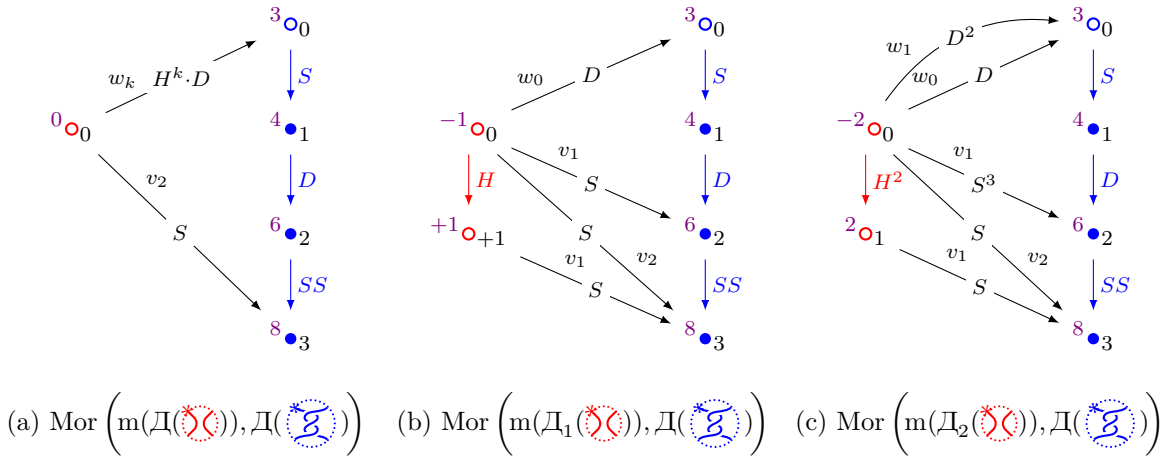
\begin{figure}[t]
    \centering
    \begin{subfigure}{0.32\textwidth}
    	\[ 
	    \begin{tikzcd}[column sep=1cm]
		&
		&
		\GGzqh{\DotCblue}{\frac{3}{2}}{3}{0}
		\arrow[blue]{d}{S}
	    \\
	    \GGzqh{\DotCred}{0}{0}{0}
	    \arrow[urr,"w_k" near start, "H^k\cdot D" description]
	    \arrow[ddrr, "S" description, "v_2" near start]
		&
		&
		\GGzqh{\DotBblue}{1}{4}{1}
		\arrow[blue]{d}{D}
		\\
		&
		&
		\GGzqh{\DotBblue}{1}{6}{2}
		\arrow[blue]{d}{SS}
		\\
		&
		&
		\GGzqh{\DotBblue}{1}{8}{3}
		\end{tikzcd}
		\]
        \caption{\( \Mor\left(\!\mirror(\DD(\LiRed)),\DD(\ThreeTwistTangleBlue)\!\right) \)}\label{fig:exa:Pairing:Trefoil:BNr:Morphisms}
	\end{subfigure}
	\begin{subfigure}{0.32\textwidth}
	\[ 
   	\begin{tikzcd}[column sep=1cm]
		&
		&
		\GGzqh{\DotCblue}{\frac{3}{2}}{3}{0}
		\arrow[blue]{d}{S}
	    \\
	    \GGzqh{\DotCred}{-\frac{1}{2}}{-1}{0}
	    \arrow[red]{d}{H}
	    \arrow[urr,"w_0" near start,"D" description]
	    \arrow[ddrr,"v_2" near end,"S" description]
	    \arrow[drr,"v_1" near start,"S" description]
		&
		&
		\GGzqh{\DotBblue}{1}{4}{1}
		\arrow[blue]{d}{D}
		\\
		\GGzqh{\DotCred}{-\frac{1}{2}}{+1}{+1}
		\arrow[drr,"v_1" near start,"S" description]
		&
		&
		\GGzqh{\DotBblue}{1}{6}{2}
		\arrow[blue]{d}{SS}
		\\
		&
		&
		\GGzqh{\DotBblue}{1}{8}{3}
		\end{tikzcd}
    \]
    \caption{   	\( \Mor\left(\!\mirror(\DD_1(\LiRed)),\DD(\ThreeTwistTangleBlue)\!\right) \)    }\label{fig:exa:Pairing:Trefoil:Khr:Morphisms}
    \end{subfigure}
    \begin{subfigure}{0.32\textwidth}
    \[ 
    \begin{tikzcd}[column sep=1cm]
        &
        &
        \GGzqh{\DotCblue}{\frac{3}{2}}{3}{0}
        \arrow[blue]{d}{S}
        \\
        \GGzqh{\DotCred}{-1}{-2}{0}
        \arrow[red]{d}{H^2}
        \arrow[urr,"w_0" near start,"D" description]
        \arrow[urr, bend left, "w_1" near start,"D^2" description]
        \arrow[ddrr,"v_2" near end,"S" description]
        \arrow[drr,"v_1" near start,"S^3" description]
        &
        &
        \GGzqh{\DotBblue}{1}{4}{1}
        \arrow[blue]{d}{D}
        \\
        \GGzqh{\DotCred}{0}{2}{1}
        \arrow[drr,"v_1" near start,"S" description]
        &
        &
        \GGzqh{\DotBblue}{1}{6}{2}
        \arrow[blue]{d}{SS}
        \\
        &
        &
        \GGzqh{\DotBblue}{1}{8}{3}
        \end{tikzcd}
    \]
    \caption{\( \Mor\left(\!\mirror(\DD_2(\LiRed)),\DD(\ThreeTwistTangleBlue)\!\right) \)}\label{fig:exa:Pairing:Trefoil:Kh:Morphisms}
    \end{subfigure} 
    \caption{Generators of the homology groups of the morphisms spaces for Example~\ref{exa:Pairing:Trefoil} }\label{fig:exa:Pairing:Trefoil:Resolutions}
\end{figure}

\begin{example}[Trefoil]\label{exa:Pairing:Trefoil}
Figure~\ref{fig:exa:Pairing:Trefoil:Diagram} shows a splitting of the right-handed trefoil knot \( T(2,3) \) into two rational tangles. The reduced Khovanov homology \( \Khr(T(2,3)) \) is 3-dimensional; see Figure~\ref{fig:exa:Pairing:Trefoil:Resolutions:EightArc}. The bigradings of the intersection points \( w_0 \), \( v_1 \) and \( v_2 \) are \( h^{0}q^{2} \), \( h^{2}q^{6} \) and \( h^{3}q^{8} \), respectively, so we obtain the 3-dimensional vector space supported in those bigradings illustrated on the right-hand side of Figure~\ref{fig:exa:Pairing:Trefoil:Resolutions:EightArc}. 

The reduced Khovanov homology, viewed as a type~D structure \( \Khr(T(2,3))^{\field[H]} \), carries a map from \( v_1 \) to \( v_2 \) picking up \( H \). So \( \BNr(T(2,3)) \) is the homology of a chain complex consisting of towers at each of \( w_0 \), \( v_1 \), and \( v_2 \), and differentials between \( v_1 \) and \( v_2 \) towers picking up \( H \). The calculation of \( \BNr \) in terms of wrapped Lagrangian Floer homology is shown in Figure~\ref{fig:exa:Pairing:Trefoil:Resolutions:ArcArc}. 

Lastly we discuss the unreduced Khovanov homology \( \Kh(T(2,3)) \). Over \( \fieldTwoElements \) it is equal to \( q^1 h^0\Khr(T(2,3);\fieldTwoElements)\oplus q^{-1} h^0\Khr(T(2,3);\fieldTwoElements) \), and the corresponding intersection picture is the same as in Figure~\ref{fig:exa:Pairing:Trefoil:Resolutions:EightArc}, except there are two red figure-8 curves with different bigradings. Over \( \fieldTwoIsNotZero \) the situation is more interesting; the curve \( \textcolor{red}{\mirror(\Kh(\LiRed))} \) corresponds to type~D structure \( \mirror(\DD_2(\Li)) \). Khovanov homology \( \Kh(T(2,3);\fieldTwoIsNotZero \)) is 4-dimensional, and we describe the corresponding intersection picture in Figure~\ref{fig:exa:Pairing:Trefoil:unreduced}. 

In Figure~\ref{fig:exa:Pairing:Trefoil:Resolutions}, we describe the generators of all three homologies as elements of the corresponding morphism spaces.
\end{example}

\subsection{Closures of (2,-3)-pretzel tangle}
More complicated examples are obtained by considering various closures of the \( (2,-3) \)-pretzel tangle.

\begin{figure}[t]
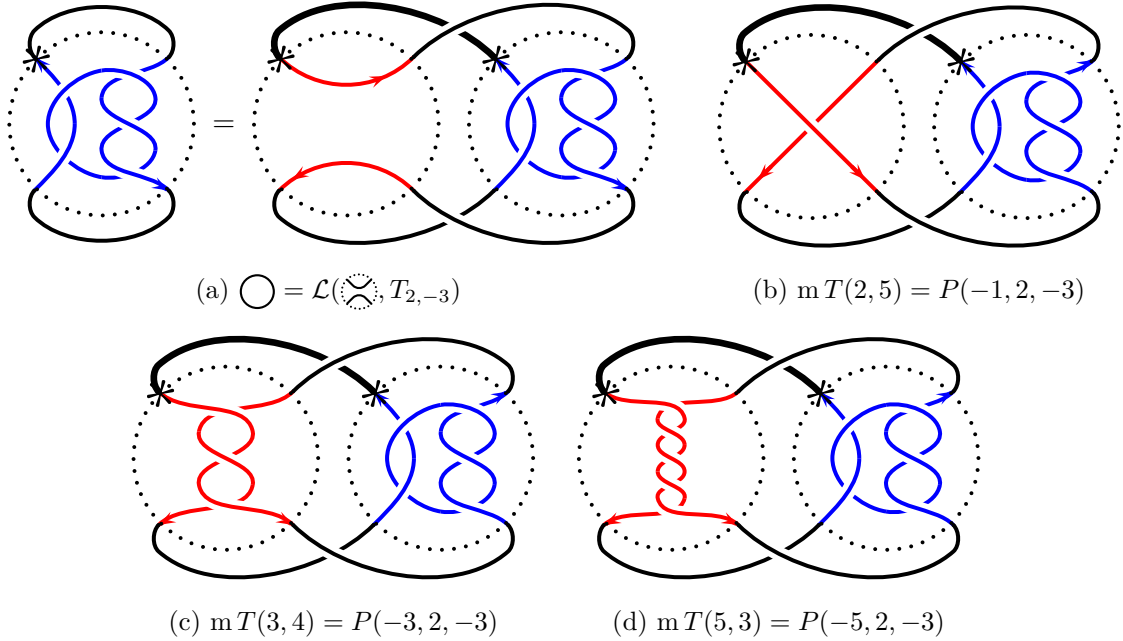

    \centering
    \begin{subfigure}{0.6\textwidth}
    	\centering
    	$ \PairingClosureDiagramUnknotFromPT=\PairingStandardDiagramUnknotFromPT$
    	\caption{\( \Circle=\Lk(\No,T_{2,-3}) \)}\label{fig:exa:Pairing:unknot}
    \end{subfigure}
    \begin{subfigure}{0.36\textwidth}
    \centering
	$\PairingStandardDiagramTorusknotFromPTa$
		\caption{\( \mirror T(2,5)=P(-1,2,-3) \)}\label{fig:exa:Pairing:TorusKnots:DiagramI}
	\end{subfigure}
	\medskip\\
    \begin{subfigure}{0.36\textwidth}
    \centering
	$\PairingStandardDiagramTorusknotFromPTb$
		\caption{\( \mirror T(3,4)=P(-3,2,-3) \)}\label{fig:exa:Pairing:TorusKnots:DiagramII}
	\end{subfigure}
    \begin{subfigure}{0.36\textwidth}
    \centering
	$\PairingStandardDiagramTorusknotFromPTc$
		\caption{\( \mirror T(5,3)=P(-5,2,-3) \)}\label{fig:exa:Pairing:TorusKnots:DiagramIII}
	\end{subfigure}
    \caption{Various knots obtained as closures of the \( (2,-3) \)-pretzel tangle}\label{fig:exa:Pairing:TorusKnots:Diagrams}
\end{figure}

\begin{figure}[t]
	\centering
    $\PairingTorusknotFromPtArcArc$
    \hspace{-3cm}
    \raisebox{-0.5cm}{$\vc{
	    \begin{tikzpicture}[scale=.5]
	    \draw[lightgray,step=1] (0.5,0.5) grid (5.5,4.5);
	    \draw[line width=1pt,->] (0,4) -- (6.5,4) node[right] {\( h \)};
	    \draw[line width=1pt,->] (5,0) -- (5,5.5) node[above] {\( \delta \)};
	    \draw (0.5,4) node[above] {\scriptsize \( -8 \)};
	    \draw (1.5,4) node[above] {\scriptsize \( -6 \)};
	    \draw (2.5,4) node[above] {\scriptsize \( -4 \)};
	    \draw (3.5,4) node[above] {\scriptsize \( -2 \)};
	    \draw (4.5,4) node[above] {\scriptsize \( 0 \)};
	    \draw (5,0.5) node[right] {\scriptsize \( -4 \)};
	    \draw (5,1.5) node[right] {\scriptsize \( -3 \)};
	    \draw (5,2.5) node[right] {\scriptsize \( -2 \)};
	    \draw (5,3.5) node[right] {\scriptsize \( -1 \)};
	    \draw(4.5, 0.7) node {\( \vdots \) };
	    \draw(4.5, 2.5) node {\( w_0 \)};
	    \draw(4.5, 1.5) node {\( w_1 \)};
	    \draw(3.5, 2.5) node {\( \varkappa_2 \)};
	    \draw(2.5, 2.5) node {\( \varkappa_1 \)};
	    \draw(2.5, 3.5) node {\( v_1 \)};
	    \draw(1.5, 3.5) node {\( v_2 \)};
	    \draw(-0.5, 3) node {\(  \underbrace{\ \ \ \ \ \ \cdots}_{\text{if }n>5}   \)};
	    \end{tikzpicture}
	}$}
	\caption{\( 
    \HF(h^0 q^{-n} \textcolor{red}{\mirror(\Arc(T_{-n}))}, \textcolor{blue}{\Arc(T_{2,-3})})
    \cong 
    h^0q^{n-1} \BNr(P(-n,2,-3)) \) for \( n=1,3,5 \)}
    \label{fig:exa:Pairing:TorusKnots:ArcArc}
\end{figure}

\begin{figure}[t]
	\centering
	$\PairingTorusknotFromPtArcEight \vspace{-0.5cm}$
	\begin{tikzpicture}[scale=.5]
	\draw[lightgray,step=1] (0.5,-0.5) grid (9.5,2.5);
	\draw[line width=1pt,->] (0,2) -- (10.5,2) node[right] {\( h \)};
	\draw[line width=1pt,->] (9,-1) -- (9,3.5) node[above] {\( \delta \)};
	\draw (0.5,2) node[above] {\scriptsize \( -8 \)};
	\draw (1.5,2) node[above] {\scriptsize \( -7 \)};
	\draw (2.5,2) node[above] {\scriptsize \( -6 \)};
	\draw (3.5,2) node[above] {\scriptsize \( -5 \)};
	\draw (4.5,2) node[above] {\scriptsize \( -4 \)};
	\draw (5.5,2) node[above] {\scriptsize \( -3 \)};
	\draw (6.5,2) node[above] {\scriptsize \( -2 \)};
	\draw (7.5,2) node[above] {\scriptsize \( -1 \)};
	\draw (8.5,2) node[above] {\scriptsize \( 0 \)};
	\draw (9,0.5) node[right] {\scriptsize \( -3/2 \)};
	\draw (9,1.5) node[right] {\scriptsize \( -1/2 \)};
	\draw(8.5, 0.5) node {\( \varkappa_1 \)};
	\draw(6.5, 0.5) node {\( \varkappa_2 \)};
	\draw(5.5, 0.5) node {\( \varkappa_3 \)};
	\draw(3.5, 0.5) node {\( \varkappa_4 \)};
	\draw(4.5, 0.5) node {\( v_1 \)};
	\draw(4.5, 1.5) node {\( v_2 \)};
	\draw(2.5, 1.5) node {\( v_3 \)};
	\draw(1.5, 1.5) node {\( v_4 \)};
	\draw(-0.5, 1) node {\(  \underbrace{\ \ \ \ \ \ \cdots}_{\text{if }n>5}   \)};
	\end{tikzpicture}
	\caption{\( 
    \HF( h^0 q^{-n} \textcolor{red}{\mirror(\BN(T_{-n})) } , \textcolor{blue}{\Khr(T_{2,-3})} )
    \cong 
    h^0q^{n} \Khr(P(-n,2,-3)) \) for \( n=1,3,5 \)}
    %
    \label{fig:exa:Pairing:TorusKnots:ArcEight}
\end{figure}

\begin{example}[Unknot closure of the \( (2,-3) \)-pretzel tangle]\label{exa:Pairing:UnknotClosure}
	If we pair the \( (2,-3) \)-pretzel tangle \( T_{2,-3} \) with one of the two trivial tangles, we obtain the unknot, see Figure~\ref{fig:exa:Pairing:unknot}.
	And indeed, the wrapped Lagrangian Floer homology of \( \textcolor{red}{\mirror(\Arc(\LoRed))} \) (the bottom arc $\DotBarc$) and \( \textcolor{blue}{\Arc(T_{2,-3})} \) (the curve in Figure~\ref{fig:ArcTypePretzelIII}) consists of a single \( H \)-tower. To compute the correct bigradings of this \( H \)-tower, note that the orientation of \( T_{2,-3} \) is different from the one in Example~\ref{exa:2m3ptBNComplexComputation}. The linking number between the two components of \( T_{2,-3} \) with the orientation from said example is \( -1 \), so by Proposition~\ref{prop:reversingOneComponent}, we only need to multiply the bigrading of its invariant by \( h^{2}q^{6}\delta^{1} \) to obtain the invariant of \( T_{2,-3} \) with the new orientation. And indeed, the tower starts in bigrading \( h^0q^{-1} \). 
	Moreover, \( \textcolor{blue}{\Arc(T_{2,-3})} \) and \( \textcolor{red}{\mirror(\Khr(\LoRed))} \) intersect in a single point of bigrading \( \delta^{0}q^{0} \) and so do \( \textcolor{blue}{\Khr(T_{2,-3})} \) and \( \textcolor{red}{\mirror(\Arc(\LoRed))} \), confirming $\Khr(\Lk(\LoRed, \textcolor{blue}{T_{2,-3}}))=\field$. These are very simply sanity checks one can perform on our immersed curve invariants for any 4-ended tangle that admits a closure to the unknot.
\end{example}

\begin{example}[Torus knot closures of the \( (2,-3) \)-pretzel tangle]\label{exa:Pairing:TorusKnots}

Figure~\ref{fig:exa:Pairing:TorusKnots:Diagrams} shows various knots obtained by gluing \( -n \)-twist tangles \( T_{-n} \) to the \( (2,-3) \)-pretzel tangle for \( n=1,3,5\), which all happen to be torus knots. The corresponding immersed curves of \( \textcolor{red}{\mirror(\Arc(T_{-n}))} \) and \( \textcolor{blue}{\Arc(T_{2,-3})} \) are shown in Figure~\ref{fig:exa:Pairing:TorusKnots:ArcArc}, except that the bigrading of the first curve has been shifted by \( \delta^{-\frac{n}{2}}q^{-n} \). In all three pairings, we have a single tower generated by \( w_0 \) in bigrading \( \delta^{-2}q^{-4} \). Also, the twisting does not affect the two generators \( \varkappa_1 \) (\( \delta^{-2}q^{-12} \)) and \( \varkappa_2 \) (\( \delta^{-2}q^{-8} \)). For \( n=1 \), these are all generators. For \( n=3 \), there is also a generator \( v_1 \) in bigrading \( \delta^{-1}q^{-10} \). For \( n=5 \), there is \( v_1 \), but also another generator \( v_2 \), which has bigrading \( \delta^{-1}q^{-14} \). For each additional full twist, we obtain one more generator, whose bigrading is obtained from that of the previous one by multiplication by \( \delta^{-1}q^{-4} \). This is summarized on the right of Figure~\ref{fig:exa:Pairing:TorusKnots:ArcArc}.

Similarly, we can pair the arc invariants of the \( -n \)-twist tangles with the figure-8 invariant of \( T_{2,-3} \); see Figure~\ref{fig:exa:Pairing:TorusKnots:ArcEight}. Here, the fact that the twisting preserves certain generators is even more obvious, since \( \textcolor{blue}{\Khr(T_{2,-3})} \) splits into two components, one of which stays invariant under twisting. Indeed, we always have the intersection points 
\( \varkappa_1 \) (\( \delta^{-\frac{3}{2}}q^{-3} \)), 
\( \varkappa_2 \) (\( \delta^{-\frac{3}{2}}q^{-7} \)), 
\( \varkappa_3 \) (\( \delta^{-\frac{3}{2}}q^{-9} \)) and 
\( \varkappa_4 \) (\( \delta^{-\frac{3}{2}}q^{-13} \)). 
The other component of \( \textcolor{blue}{\Khr(T_{2,-3})} \) is equal to the invariant of a \( 2 \)-twist rational tangle (up to a shift in bigrading). Therefore, 
\begin{align*}
& \delta^{\frac{n}{2}}q^{n}
\Khr(\Lk(T_{-n},T_{2,-3}))
\\
\cong~& 
\delta^{\frac{n}{2}}q^{n}
\HF(\textcolor{red}{\mirror(\Arc(T_{-n}))},\textcolor{blue}{\Khr(T_{2,-3})})\\
\cong~& 
\HF(\delta^{-\frac{n}{2}}q^{-n}\textcolor{red}{\mirror(\Arc(T_{-n}))},\textcolor{blue}{\Khr(T_{2,-3})})\\
\cong~&
\langle\varkappa_1,\varkappa_2,\varkappa_3,\varkappa_4\rangle
\oplus
\HF(\delta^{-\frac{n}{2}}q^{-n}\textcolor{red}{\mirror(\Arc(T_{-n}))},\delta^{-2}q^{-12}\textcolor{blue}{\Khr(T_{2})})\\
\cong~&
\langle\varkappa_1,\varkappa_2,\varkappa_3,\varkappa_4\rangle
\oplus
\delta^{\frac{n-4}{2}}q^{n-12}\HF(\textcolor{red}{\mirror(\Arc(T_{-n}))},\textcolor{blue}{\Khr(T_{2})})\\
\cong~&
\langle\varkappa_1,\varkappa_2,\varkappa_3,\varkappa_4\rangle
\oplus
\delta^{\frac{n-4}{2}}q^{n-12}\Khr(T(2,2-n))
\end{align*}
as is shown on the right of Figure~\ref{fig:exa:Pairing:TorusKnots:ArcEight}.
The 4-dimensional vector space 
\( \varkappa:=\langle\varkappa_1,\varkappa_2,\varkappa_3,\varkappa_4\rangle\), 
supported in gradings
 \( \delta^{-\frac{3}{2}}(q^{-3}+q^{-7}+q^{-9}+q^{-13}) \), 
agrees with the invariant \( \varkappa \) for the \( (2,-3) \)-pretzel tangle (with the appropriate suture) from~\cite[Figure~5]{Kappa}. This is discussed further in Section \ref{sub:rem-sut-tang} (Note that the quantum grading in~\cite{Kappa} is half of our quantum grading and that the two \( \delta \)-gradings differ by the factor \( -1 \).) The limit \( \displaystyle \lim_{n\to \infty} \Khr(T(2,2-n)) = \langle v_2, v_3, v_4 \ldots \rangle\) for the other summand  is equal to \( \field[x,y]/(x^2=y^3) \); see Example~\ref{ex:Uinf_invariants}.

The observation that there are well-defined limits of \( \BNr(\Lk(T_{-n},T_{2,-3})) \) and \( \Khr(\Lk(T_{-n},T_{2,-3})) \) as \( n \to \infty \) is the starting point for Section~\ref{sec:InfinitelyStranded}.
\end{example}

\section{The mapping class group action}\label{sec:MCGaction}

In this section, we will only work in characteristic 2. However, we expect the main result to hold over any field.

\begin{wrapfigure}{r}{0.33\textwidth}
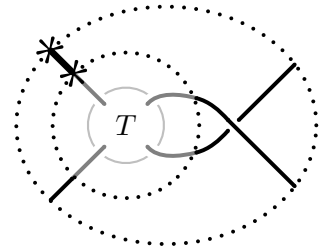

	\centering
	$\DehnTwistTangleDiagram$
	\caption{The tangle \( \tau(T) \)}\label{fig:AddingASingleCrossing}
\end{wrapfigure}

\subsection{Naturality of the action}

In Definition~\ref{def:4-ended_tangle} of a \emph{parameterized} 4-ended tangle \( T \), we fixed a circle-parameterization of the boundary 4-punctured sphere $\partial D^3\smallsetminus \partial T$. The mapping class group of the 4-punctured sphere acts on this parameterization and hence on the tangles, see Figure~\ref{fig:AddingASingleCrossing}. On the other hand, our arc type invariant $\BNr(T)$, as well as the family of compact invariants $\Eight[n](T)$, are curves with local systems on the 4-punctured sphere $\FourPuncturedSphere$, so the mapping class group acts on them as well. The main result from this section is that our immersed curve invariants are compatible with these actions. As a consequence, our curves are invariants of \emph{unparameterized} 4-ended tangles (Definition~\ref{def:4-ended_tangle}), if viewed as multicurves on the tangle boundary $\partial D^3\smallsetminus \partial T$.

\begin{theorem}\label{thm:MCGaction}
	Let \( T \) be a pointed 4-ended tangle and \( \rho \) an element of the mapping class group of the 4-punctured sphere which fixes the reduction basepoint. Then 
	\[ \Arc(\rho(T))=\rho(\Arc(T)) \qquad \text{and} \qquad  \Eight[n](\rho(T))=\rho(\Eight[n](T)) \quad\text{ for all $n\geq 1$.} \] 
\end{theorem}

Note that the analogous result for the curves $\Khr(T)$ and $\Kh(T)$ also holds; this follows from Remark~\ref{rmk:Kh_and_Khr_in_terms_of_Ln}. The proofs of both identities above follow along the same lines as those of the corresponding theorem for the Heegaard Floer type invariant \(\HFT(T)\) in~\cite{pqSym}---the difference is that the computation in bordered sutured Floer theory is replaced by its combinatorial Khovanov theory analogue. It suffices to consider the half Dehn twist (braid move) \( \rho=\tau \) from Figure~\ref{fig:AddingASingleCrossing}; then, the case for general \( \rho \) follows easily.

To begin, we need to compute the type AD structure which corresponds to the action of the Dehn twist \( \tau \) on the type~D structure \( \DD(T)^{\BNAlgH}\).

\begin{proposition}\label{prop:AddingASingleCrossing}
	Let \(T\) be a pointed 4-ended tangle and \(\tau T\) the tangle obtained by adding a single crossing to it as shown in Figure~\ref{fig:AddingASingleCrossing}. Then
	\[ \DD(\tau T)^{\BNAlgH}\simeq \DD(T)^{\BNAlgH}\boxtimes\BNTwisting\] where \( \BNTwisting \) is the type AD structure from Figure~\ref{fig:BNTwistingAD}. 
\end{proposition}

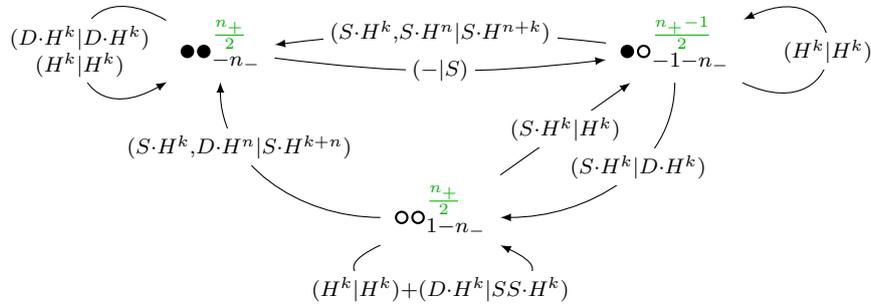
\begin{figure}[b]
	\centering
	\[ 
	\begin{tikzcd}[column sep=40pt,row sep=35pt]
	\GGdzh{\DotB \DotB}{\frac{n_+}{2}}{}{-n_-}
	\arrow[bend right=7]{rr}[description]{(-\vert S)}
	\arrow[in=-150,out=150,looseness=5.5]{rl}[description]{\substack{(D \cdot H^k \vert D \cdot H^k) \\ (H^k \vert H^k)}}
	&&
	\GGdzh{\DotB \DotC}{\frac{n_+-1}{2}}{}{-1-n_-}
	\arrow[bend right=7]{ll}[description]{(S\cdot H^k,S \cdot H^n\vert S \cdot H^{n+k})}
	\arrow[in=0,out=-90,pos=0.4]{dl}[description]{(S \cdot H^k \vert D \cdot H^k )}
	\arrow[in=+25,out=-25,looseness=5]{lr}[description]{ (H^k \vert H^k)}
	\\
	&
	\GGdzh{\DotC \DotC}{\frac{n_+}{2}}{}{1-n_-}
	\arrow[out=180,in=-90,pos=0.7]{lu}[description]{(S \cdot H^k ,D\cdot H^n \vert S\cdot H^{k+n})}
	\arrow[pos=0.5]{ru}[description]{(S\cdot H^k \vert H^k)}
	\arrow[in=-25,out=-155,looseness=3]{rl}[description]{(H^k \vert H^k)+(D \cdot H^k \vert SS \cdot H^k )}
	\end{tikzcd}
	\]
	\caption{The bimodule \( \BNTwisting \); \( (n_+,n_-) \) is either \((1,0)\) or \((0,1)\), depending on the sign of the right crossing in Figure~\ref{fig:AddingASingleCrossing}
	}\label{fig:BNTwistingAD}
\end{figure}


\begin{figure}[t]
	\centering
	\begin{subfigure}
	{\textwidth}
		\newlength{\braceflex}
		\setlength{\braceflex}{-4pt}
		\newlength{\midflex}
		\setlength{\midflex}{-4pt}
	\begin{align*}
		\left(\hspace*{\braceflex}
		\begin{tikzcd}[ampersand replacement = \&]
		\Li
		\arrow{d}{D\cdot H^k}
		\\
		\Li
		\end{tikzcd}\hspace*{\braceflex}\right)
		\mapsto 
		&
		\left(\hspace*{\braceflex}
		\begin{tikzcd}[ampersand replacement = \&,column sep=20pt]
		\Li
		\arrow[swap]{d}{H^{k+1}}
		\arrow[out=-60,in=90,swap,pos=0.7]{drr}{H^k}
		\&
		\Li
		\arrow{r}{1}
		\arrow[swap]{l}{D}
		\arrow[swap,pos=0.2]{d}{D \cdot H^k}
		\&
		\Li
		\\
		\Li
		\&
		\Li
		\arrow{r}{1}
		\arrow[swap]{l}{D}
		\&
		\Li
		\end{tikzcd}\hspace*{\braceflex}\right)\hspace*{\midflex}
		&
		\hspace*{\midflex}\left(\hspace*{\braceflex}
		\begin{tikzcd}[ampersand replacement = \&]
		\Lo
		\arrow{d}{D \cdot H^k}
		\\
		\Lo
		\end{tikzcd}\hspace*{\braceflex}\right)
		\mapsto 
		&
		\left(\hspace*{\braceflex}
		\begin{tikzcd}[ampersand replacement = \&,column sep=20pt]
		\Lo
		\arrow[swap]{d}{D \cdot H^k}
		\arrow{r}{S}
		\&
		\Li
		\\
		\Lo
		\arrow{r}{S}
		\&
		\Li
		\end{tikzcd}\hspace*{\braceflex}\right)
		\\
		\left(\hspace*{\braceflex}
		\begin{tikzcd}[ampersand replacement = \&]
		\Li
		\arrow{d}{H^k}
		\\
		\Li
		\end{tikzcd}\hspace*{\braceflex}\right)
		\mapsto 
		&
		\left(\hspace*{\braceflex}
		\begin{tikzcd}[ampersand replacement = \&,column sep=20pt]
		\Li
		\arrow{d}{H^k}
		\&
		\Li
		\arrow{r}{1}
		\arrow[swap]{l}{D}
		\arrow{d}{H^k}
		\&
		\Li
		\arrow{d}{H^k}
		\\
		\Li
		\&
		\Li
		\arrow{r}{1}
		\arrow[swap]{l}{D}
		\&
		\Li
		\end{tikzcd}\hspace*{\braceflex}\right)\hspace*{\midflex}
		&
		\hspace*{\midflex}\left(\hspace*{\braceflex}
		\begin{tikzcd}[ampersand replacement = \&]
		\Lo
		\arrow{d}{H^k}
		\\
		\Lo
		\end{tikzcd}\hspace*{\braceflex}\right)
		\mapsto 
		&
		\left(\hspace*{\braceflex}
		\begin{tikzcd}[ampersand replacement = \&,column sep=20pt]
		\Lo
		\arrow[swap]{d}{H^k}
		\arrow{r}{S}
		\&
		\Li
		\arrow{d}{H^k}
		\\
		\Lo
		\arrow{r}{S}
		\&
		\Li
		\end{tikzcd}\hspace*{\braceflex}\right)
		\\
		\left(\hspace*{\braceflex}
		\begin{tikzcd}[ampersand replacement = \&]
		\Li
		\arrow{d}{S \cdot H^k}
		\\
		\Lo
		\end{tikzcd}\hspace*{\braceflex}\right)
		\mapsto 
		&
		\left(\hspace*{\braceflex}
		\begin{tikzcd}[ampersand replacement = \&,column sep=20pt]
		\Li
		\arrow[out=-90,in=135,pos=0.3,swap]{drr}{H^k}
		\&
		\Li
		\arrow[pos=0.2,swap]{d}{S \cdot H^k}
		\arrow{r}{1}
		\arrow[swap]{l}{D}
		\&
		\Li
		\arrow{d}{H^{k+1}}
		\\
		\&
		\Lo
		\arrow{r}{S}
		\&
		\Li
		\end{tikzcd}\hspace*{\braceflex}\right)\hspace*{\midflex}
		&
		\hspace*{\midflex}\left(\hspace*{\braceflex}
		\begin{tikzcd}[ampersand replacement = \&]
		\Lo
		\arrow{d}{S \cdot H^k}
		\\
		\Li
		\end{tikzcd}\hspace*{\braceflex}\right)
		\mapsto 
		&
		\left(\hspace*{\braceflex}
		\begin{tikzcd}[ampersand replacement = \&,column sep=20pt]
		\&
		\Lo
		\arrow{r}{S}
		\arrow[swap]{d}{S \cdot H^k}
		\&
		\Li
		\arrow{d}{H^k}
		\\
		\Li
		\&
		\Li
		\arrow{r}{1}
		\arrow[swap]{l}{D}
		\&
		\Li		
		\end{tikzcd}\hspace*{\braceflex}\right)
		\end{align*}
		\caption{The action of the half Dehn twist \( \tau \) on the morphisms in \( \Cob_{/l} \)}\label{fig:BNTwisting}
	\end{subfigure}
\begin{subfigure}{\textwidth}
		\[ 
		\begin{tikzcd}[column sep=10pt,row sep=50pt,scale=2]
		&
		\GGdzh{\Lo \Lo}{0}{0}{0}
		\arrow{rr}[description]{(-\vert S)}
		\arrow[bend left,pos=0.75]{dr}[near end,description]{(S\cdot H^k\vert S\cdot H^k)}
		\arrow[in=-155,out=155,looseness=5]{rl}[description]{\substack{(D \cdot H^k \vert D \cdot H^k) \\ (H^k \vert H^k)}}
		&&
		\GGdzh{\Lo \Li}{-\frac{1}{2}}{+1}{1}
		\arrow[pos=0.7,bend left=20]{dr}[near end,description]{(S \cdot H^k\vert H^k)}
		\arrow[in=25,out=-25,looseness=4]{lr}[description]{(H^k \vert H^k)}
		\\
		\GGdzh{\Li \Li}{0}{+2}{1}
		\arrow[in=-90,out=-90,looseness=0.4]{rrrr}[description]{(D \cdot H^k\vert H^k)}
		\arrow[in=-155,out=155,looseness=5]{rl}[description]{\substack{(D\cdot H^k \vert H^{k+1})\\ (H^k \vert H^k)}}
		\arrow{urrr}[pos=0.15,description]{(S \cdot H^k\vert H^k)}
		&&
		\GGdzh{\Li \Li}{0}{0}{0}
		\arrow[dashed]{rr}[description]{(-\vert 1)}
		\arrow{ll}[description]{(-\vert D)}
		\arrow[pos=0.7,bend left]{ul}[near end,description]{\!\!(S \cdot H^k\vert S \cdot H^k)\!\!}
		\arrow[in=-45,out=-135,looseness=3]{lr}[description]{\substack{(D \cdot H^k\vert D \cdot H^k) + (H^k \vert H^k)}}
		&&
		\GGdzh{\Li \Li}{-1}{0}{1}
		\arrow[pos=0.7,bend left]{ul}[near end,description]{\!\!(S \cdot H^k\vert H^{k+1})\!\!}
		\arrow[in=25,out=-25,looseness=5]{lr}[description]{(H^k \vert H^k)}
		\end{tikzcd}
		\]
		\vspace{-0.5cm}
		\caption{The bimodule \( \BNTwistingTildeCobCob \), summarizing the actions of \( \tau \) from (a)}\label{fig:BNTwistingPrime}
	\end{subfigure}
	\begin{subfigure}{\textwidth}
		\[ 
		\begin{tikzcd}[column sep=30pt,row sep=25pt]
		\GGdzh{\Lo \Lo}{0}{0}{0}
		\arrow[bend right=6]{rr}[description]{(-\vert S)}
		\arrow[in=-155,out=155,looseness=5]{rl}[description]{\substack{(D \cdot H^k \vert D \cdot H^k) \\ (H^k \vert H^k)}}
		&&
		\GGdzh{\Lo \Li}{-\frac{1}{2}}{+1}{1}
		\arrow[bend right=6]{ll}[description]{(S\cdot H^k,S \cdot H^n\vert S \cdot H^{n+k})}
		\arrow[in=0,out=-90,pos=0.4]{dl}[description]{(S \cdot H^k \vert D \cdot H^k )}
		\arrow[in=+25,out=-25,looseness=4]{lr}[description]{ (H^k \vert H^k)}
		\\
		&
		\GGdzh{\Li \Li}{0}{+2}{1}
		\arrow[out=180,in=-90,pos=0.7]{lu}[description]{(S \cdot H^k ,D\cdot H^n \vert S\cdot H^{k+n})}
		\arrow[pos=0.5]{ru}[description]{(S\cdot H^k \vert H^k)}
		\arrow[in=-25,out=-155,looseness=2]{rl}[description]{(H^k \vert H^k)+(D \cdot H^k \vert SS \cdot H^k )}
		\end{tikzcd}
		\]
		\caption{The bimodule \( \BNTwistingCobCob \), obtained from \( \BNTwistingTildeCobCob \) by canceling the dashed arrow }\label{fig:BNTwistingAD:embedded}
	\end{subfigure}
	\caption{Computations for the proof of Proposition~\ref{prop:AddingASingleCrossing}. In both (b) and (c), a label containing the index \( k \) should be understood as the infinite sum of such labels where \( k\geq0 \). Also, to obtain the correct absolute grading, multiply the two bimodules by \(\delta^{\frac{1}{2}n_+}h^{-n_-}\). }\label{fig:BNTwistingMor}
\end{figure}
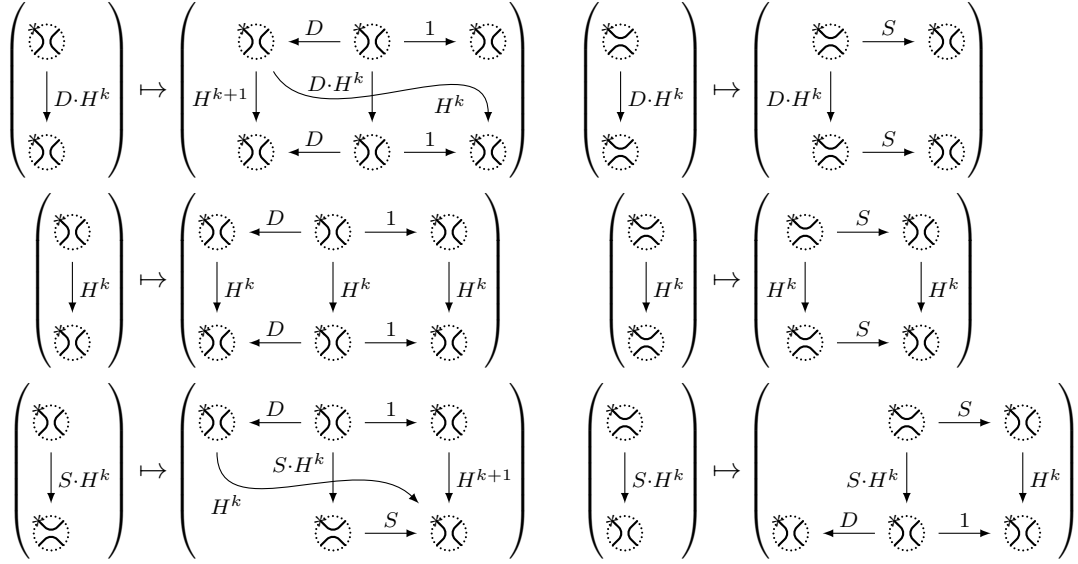
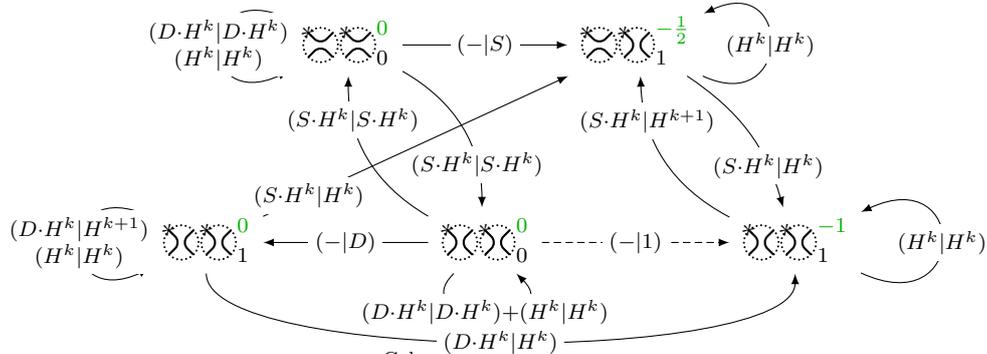
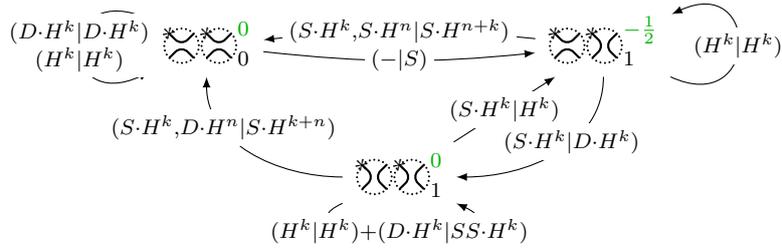

\begin{proof}
	We may compute \( \KhTl{\tau T} \) in steps: (1) resolve all crossings of \( \tau T \) except the new one; (2) deloop; (3) resolve the remaining crossing; (4) deloop once more if necessary. 
	Steps (1) and (2) essentially compute \( \KhTl{T} \). The complex after step (3) can be computed from \( \KhTl{T} \) by gluing the right hand side of \( \KhTl{T} \) to the left of the complex 
	\[ \KhTl{\CrossingR}=\delta^{\frac{1}{2}n_+}h^{-n_-}\left[
	\begin{tikzcd}
	\GGdzh{\No}{0}{0}{0}
	\arrow{r}{\Nol}
	&
	\GGdzh{\Ni}{-\frac{1}{2}}{+1}{1}
	\end{tikzcd}
	\right]\]
	where \( (n_+,n_-) \) is either \((1,0)\) or \((0,1)\), depending on whether the remaining crossing is positive or negative. (This gluing process is explained in  \cite[Section~5]{BarNatanKhT}.) We will describe the gluing \( \KhTl{T} \leftrightarrow \KhTl{\CrossingR}\) in terms of box-tensoring with a certain type AD structure.

	On objects, the gluing is easy to describe. After steps (3) and (4), where we use the first delooping isomorphism from Remark~\ref{obs:delooping}, we have replaced the objects in \( \KhTl{T} \) as follows:
	\[ 
	\GGdzh{\Li}{0}{0}{0}
	\mapsto 
	\delta^{\frac{1}{2}n_+}h^{-n_-}
	\left[
	\begin{tikzcd}[column sep=15pt]
	\GGdzh{\Li}{0}{+2}{1}
	&
	\GGdzh{\Li}{0}{0}{0}
	\arrow{r}{1}
	\arrow[swap]{l}{\LiDotR}
	&
	\GGdzh{\Li}{-1}{0}{1}
	\end{tikzcd}\right]
	\qquad
	\GGdzh{\Lo}{0}{0}{0}
	\mapsto 
	\delta^{\frac{1}{2}n_+}h^{-n_-}
	\left[
	\begin{tikzcd}[column sep=15pt]
	\GGdzh{\Lo}{0}{0}{0}
	\arrow{r}{\Lol}
	&
	\GGdzh{\Li}{-\frac{1}{2}}{+1}{1}
	\end{tikzcd}\right]
	\]
	Similarly, we can compute the action on morphisms; this is done in Figure~\ref{fig:BNTwisting}. Note that we abuse notation by writing \( D \) and \( S \) for the images of these elements under the embedding \( \omega_1 \) from Theorem~\ref{thm:OmegaFullyFaithful}. These actions on objects and morphisms are summarized in the type AD structure \( \BNTwistingTildeCobCob \) in Figure~\ref{fig:BNTwistingPrime}, where we use notation $\Cob_{/l}$ for $\End_{/l}(\Li\oplus\Lo)$. After canceling the identity component, we obtain the type AD structure \( \BNTwistingCobCob \) 
	in Figure~\ref{fig:BNTwistingPrime}. By construction:
	\[
	 \KhTl{\tau T}
	\simeq
	\KhTl{T}^{\Cob_{/l}}\boxtimes\BNTwistingCobCob
	\]
	We can rewrite this expression as follows:
	\[
	\KhTl{T}^{\Cob_{/l}}\boxtimes\BNTwistingCobCob
	\overset{(a)}{\simeq}
	\Omega_1(\DD(T)^{\BNAlgH})\boxtimes\BNTwistingCobCob 
	\overset{(b)}{\simeq}
	\DD(T)^{\BNAlgH}\boxtimes\BNTwistingBCob
	\overset{(c)}{\simeq}
	\Omega_1\left(\DD(T)^{\BNAlgH}\boxtimes\BNTwisting \right)
	\]
	Here, \( \BNTwistingBCob \) is the image of \( \BNTwistingCobCob \) under the functor induced by the embedding \( \omega_1 \) from Theorem~\ref{thm:OmegaFullyFaithful}; analogously \( \BNTwisting \) is obtained from  \( \BNTwistingBCob \). (a) and (c) follow from the definition of \( \DD(T) \), and (b) follows from the Pairing Adjunction~\cite[Theorem~1.21]{pqMod}. Finally, we apply \( \Omega_1^{-1} \) to obtain the desired identity
	\( \DD(\tau T)^{\BNAlgH} \simeq \DD(T)\boxtimes\BNTwisting. \) 
\end{proof}

\begin{figure}[t]
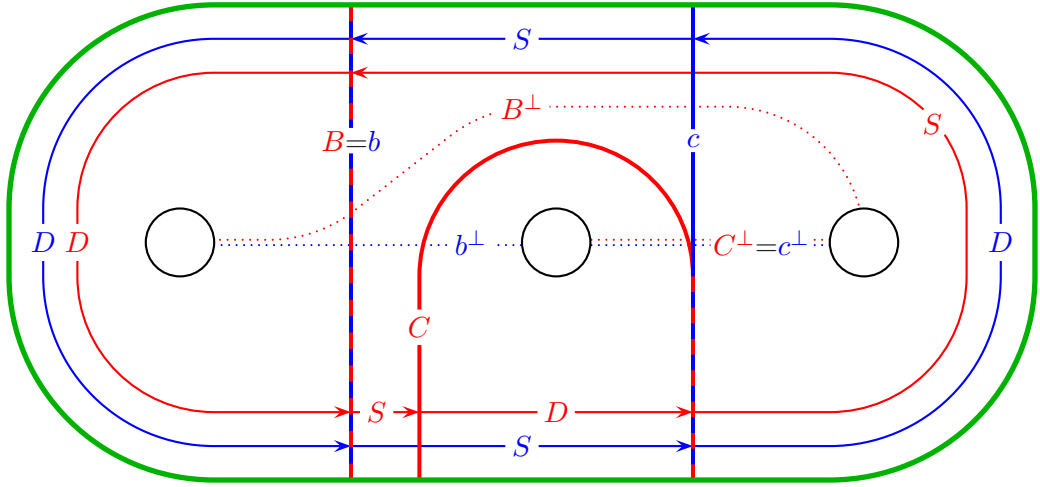

	\centering
	$\MCGactionArcs$
	\caption{The action of the Dehn twist \( \tau \) on the parameterizations of the fixed 3-punctured disc}  \label{fig:MCGactionArcs}
\end{figure}

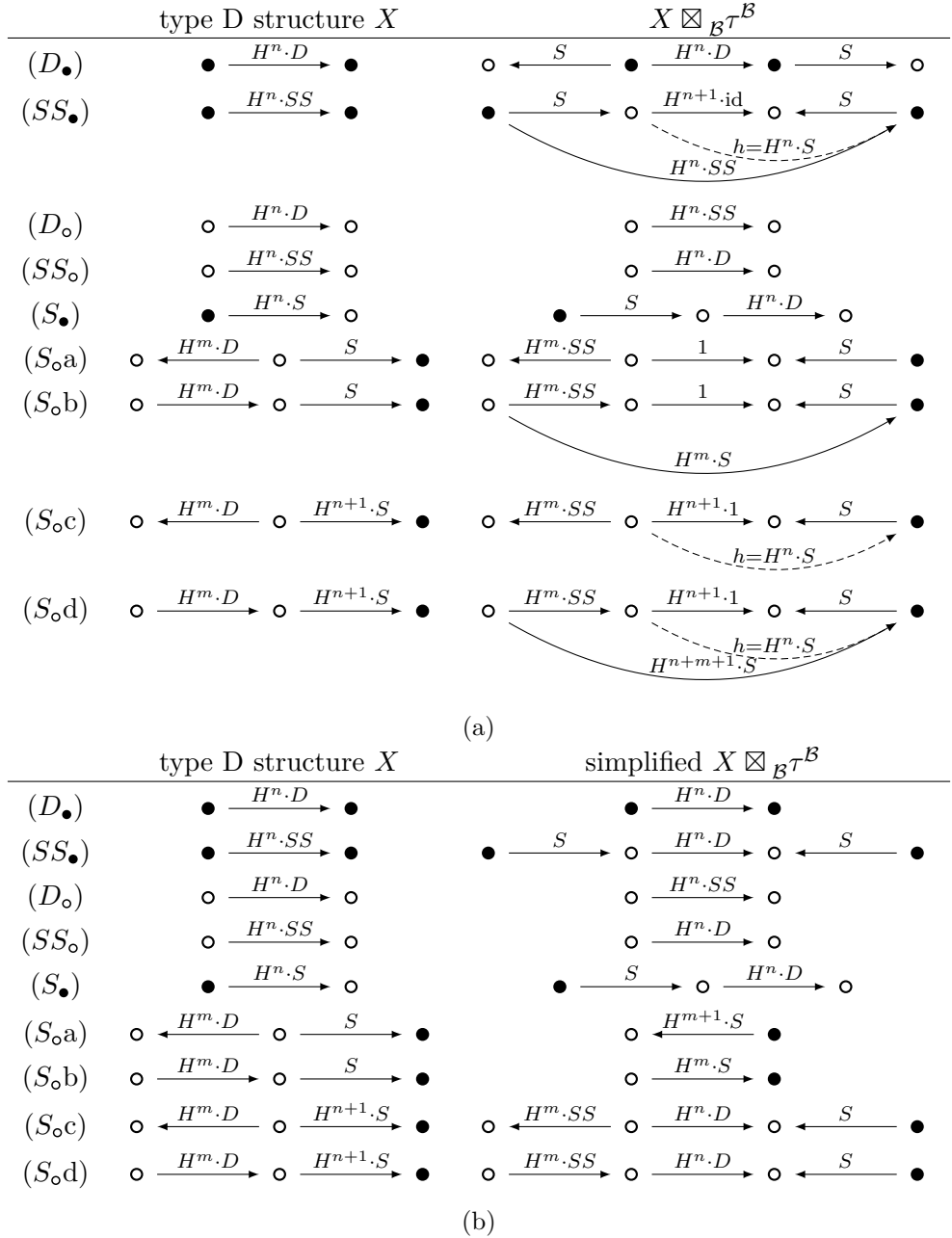
\begin{figure}[t]
	\begin{subfigure}{\textwidth}
		\centering
	\begin{tabular}{ccc}
		&
		type~D structure \( X \) 
		&
		\( X\boxtimes\BNTwisting \)
		\\
		\hline
		(\( \DotcobB \))&
		\( 
		\begin{tikzcd}[column sep=40pt]
		\DotB
		\arrow{r}{H^n\cdot D}
		&
		\DotB
		\end{tikzcd}
		 \)
		&
		\( 
		\begin{tikzcd}[column sep=40pt]
		\DotC
		&
		\DotB
		\arrow{r}{H^n\cdot D}
		\arrow[swap]{l}{S}
		&
		\DotB
		\arrow{r}{S}
		&
		\DotC
		\end{tikzcd}
		 \)
		\\
		(\( S\SaddleCB \))&
		\( 
		\begin{tikzcd}[column sep=40pt]
		\DotB
		\arrow{r}{H^n\cdot SS}
		&
		\DotB
		\end{tikzcd}
		 \)
		&
		\( 
		\begin{tikzcd}[column sep=40pt]
		\DotB
		\arrow{r}{S}
		\arrow[bend right]{rrr}{H^n\cdot SS}
		&
		\DotC
		\arrow{r}{H^{n+1}\cdot \id}
		\arrow[bend right,dashed]{rr}{h=H^n\cdot S}
		&
		\DotC
		&
		\DotB
		\arrow[swap]{l}{S}
		\end{tikzcd}
		 \)
		\\
		(\( \DotcobC \))&
		\( 
		\begin{tikzcd}[column sep=40pt]
		\DotC
		\arrow{r}{H^n\cdot D}
		&
		\DotC
		\end{tikzcd}
		 \)
		&
		\( 
		\begin{tikzcd}[column sep=40pt]
		\DotC
		\arrow{r}{H^{n}\cdot SS}
		&
		\DotC
		\end{tikzcd}
		 \)
		\\
		(\( S\SaddleBC \))&
		\( 
		\begin{tikzcd}[column sep=40pt]
		\DotC
		\arrow{r}{H^n\cdot SS}
		&
		\DotC
		\end{tikzcd}
		 \)
		&
		\( 
		\begin{tikzcd}[column sep=40pt]
		\DotC
		\arrow{r}{H^{n}\cdot D}
		&
		\DotC
		\end{tikzcd}
		 \)
		\\
		(\( \SaddleCB \))&
		\( 
		\begin{tikzcd}[column sep=40pt]
		\DotB
		\arrow{r}{H^n\cdot S}
		&
		\DotC
		\end{tikzcd}
		 \)
		&
		\( 
		\begin{tikzcd}[column sep=40pt]
		\DotB
		\arrow{r}{S}
		&
		\DotC
		\arrow{r}{H^{n}\cdot D}
		&
		\DotC
		\end{tikzcd}
		 \)
		\\
		(\( \SaddleBC \)a)&
		\( 
		\begin{tikzcd}[column sep=40pt]
		\DotC
		&
		\DotC
		\arrow{r}{S}
		\arrow[swap]{l}{H^m\cdot D}
		&
		\DotB
		\end{tikzcd}
		 \)
		&
		\( 
		\begin{tikzcd}[column sep=40pt]
		\DotC
		&
		\DotC
		\arrow[swap]{l}{H^m\cdot SS}
		\arrow{r}{1}
		&
		\DotC
		&
		\DotB
		\arrow[swap]{l}{S}
		\end{tikzcd}
		 \)
		\\
		(\( \SaddleBC \)b)&
		\( 
		\begin{tikzcd}[column sep=40pt]
		\DotC
		\arrow{r}{H^m\cdot D}
		&
		\DotC
		\arrow{r}{S}
		&
		\DotB
		\end{tikzcd}
		 \)
		&
		\( 
		\begin{tikzcd}[column sep=40pt]
		\DotC
		\arrow{r}{H^m\cdot SS}
		\arrow[bend right]{rrr}{H^m\cdot S}
		&
		\DotC
		\arrow{r}{1}
		&
		\DotC
		&
		\DotB
		\arrow[swap]{l}{S}
		\end{tikzcd}
		 \)
		\\
		(\( \SaddleBC \)c)&
		\( 
		\begin{tikzcd}[column sep=40pt]
		\DotC
		&
		\DotC
		\arrow{r}{H^{n+1}\cdot S}
		\arrow[swap]{l}{H^m\cdot D}
		&
		\DotB
		\end{tikzcd}
		 \)
		&
		\( 
		\begin{tikzcd}[column sep=40pt]
		\DotC
		&
		\DotC
		\arrow[swap]{l}{H^m\cdot SS}
		\arrow{r}{H^{n+1}\cdot 1}
		\arrow[bend right,dashed]{rr}{h=H^n\cdot S}
		&
		\DotC
		&
		\DotB
		\arrow[swap]{l}{S}
		\end{tikzcd}
		 \)
		\\
		(\( \SaddleBC \)d)&
		\( 
		\begin{tikzcd}[column sep=40pt]
		\DotC
		\arrow{r}{H^m\cdot D}
		&
		\DotC
		\arrow{r}{H^{n+1}\cdot S}
		&
		\DotB
		\end{tikzcd}
		 \)
		&
		\( 
		\begin{tikzcd}[column sep=40pt]
		\DotC
		\arrow{r}{H^m\cdot SS}
		\arrow[bend right]{rrr}{H^{n+m+1}\cdot S}
		&
		\DotC
		\arrow{r}{H^{n+1}\cdot 1}
		\arrow[bend right,dashed]{rr}{h=H^n\cdot S}
		&
		\DotC
		&
		\DotB
		\arrow[swap]{l}{S}
		\end{tikzcd}
		 \)
	\end{tabular}
	\caption{}\label{tab:MCGactionArcsPuzzlePieces}
	\end{subfigure}
	\begin{subfigure}{\textwidth}
		\centering
	\begin{tabular}{ccc}
		&
		type~D structure \( X \) 
		&
		simplified \( X\boxtimes\BNTwisting \)
		\\
		\hline
		(\( \DotcobB \))&
		\( 
		\begin{tikzcd}[column sep=40pt]
		\DotB
		\arrow{r}{H^n\cdot D}
		&
		\DotB
		\end{tikzcd}
		 \)
		&
		\( 
		\begin{tikzcd}[column sep=40pt]
		\DotB
		\arrow{r}{H^n\cdot D}
		&
		\DotB
		\end{tikzcd}
		 \)
		\\
		(\( S\SaddleCB \))&
		\( 
		\begin{tikzcd}[column sep=40pt]
		\DotB
		\arrow{r}{H^n\cdot SS}
		&
		\DotB
		\end{tikzcd}
		 \)
		&
		\( 
		\begin{tikzcd}[column sep=40pt]
		\DotB
		\arrow{r}{S}
		&
		\DotC
		\arrow{r}{H^{n}\cdot D}
		&
		\DotC
		&
		\DotB
		\arrow[swap]{l}{S}
		\end{tikzcd}
		 \)
		\\
		(\( \DotcobC \))&
		\( 
		\begin{tikzcd}[column sep=40pt]
		\DotC
		\arrow{r}{H^n\cdot D}
		&
		\DotC
		\end{tikzcd}
		 \)
		&
		\( 
		\begin{tikzcd}[column sep=40pt]
		\DotC
		\arrow{r}{H^{n}\cdot SS}
		&
		\DotC
		\end{tikzcd}
		 \)
		\\
		(\( S\SaddleBC \))&
		\( 
		\begin{tikzcd}[column sep=40pt]
		\DotC
		\arrow{r}{H^n\cdot SS}
		&
		\DotC
		\end{tikzcd}
		 \)
		&
		\( 
		\begin{tikzcd}[column sep=40pt]
		\DotC
		\arrow{r}{H^{n}\cdot D}
		&
		\DotC
		\end{tikzcd}
		 \)
		\\
		(\( \SaddleCB \))&
		\( 
		\begin{tikzcd}[column sep=40pt]
		\DotB
		\arrow{r}{H^n\cdot S}
		&
		\DotC
		\end{tikzcd}
		 \)
		&
		\( 
		\begin{tikzcd}[column sep=40pt]
		\DotB
		\arrow{r}{S}
		&
		\DotC
		\arrow{r}{H^{n}\cdot D}
		&
		\DotC
		\end{tikzcd}
		 \)
		\\
		(\( \SaddleBC \)a)&
		\( 
		\begin{tikzcd}[column sep=40pt]
		\DotC
		&
		\DotC
		\arrow{r}{S}
		\arrow[swap]{l}{H^m\cdot D}
		&
		\DotB
		\end{tikzcd}
		 \)
		&
		\( 
		\begin{tikzcd}[column sep=40pt]
		\DotC
		&
		\DotB
		\arrow[swap]{l}{H^{m+1}\cdot S}
		\end{tikzcd}
		 \)
		\\
		(\( \SaddleBC \)b)&
		\( 
		\begin{tikzcd}[column sep=40pt]
		\DotC
		\arrow{r}{H^m\cdot D}
		&
		\DotC
		\arrow{r}{S}
		&
		\DotB
		\end{tikzcd}
		 \)
		&
		\( 
		\begin{tikzcd}[column sep=40pt]
		\DotC
		\arrow{r}{H^m\cdot S}
		&
		\DotB
		\end{tikzcd}
		 \)
		\\
		(\( \SaddleBC \)c)&
		\( 
		\begin{tikzcd}[column sep=40pt]
		\DotC
		&
		\DotC
		\arrow{r}{H^{n+1}\cdot S}
		\arrow[swap]{l}{H^m\cdot D}
		&
		\DotB
		\end{tikzcd}
		 \)
		&
		\( 
		\begin{tikzcd}[column sep=40pt]
		\DotC
		&
		\DotC
		\arrow[swap]{l}{H^m\cdot SS}
		\arrow{r}{H^{n}\cdot D}
		&
		\DotC
		&
		\DotB
		\arrow[swap]{l}{S}
		\end{tikzcd}
		 \)
		\\
		(\( \SaddleBC \)d)&
		\( 
		\begin{tikzcd}[column sep=40pt]
		\DotC
		\arrow{r}{H^m\cdot D}
		&
		\DotC
		\arrow{r}{H^{n+1}\cdot S}
		&
		\DotB
		\end{tikzcd}
		 \)
		&
		\( 
		\begin{tikzcd}[column sep=40pt]
		\DotC
		\arrow{r}{H^m\cdot SS}
		&
		\DotC
		\arrow{r}{H^{n}\cdot D}
		&
		\DotC
		&
		\DotB
		\arrow[swap]{l}{S}
		\end{tikzcd}
		 \)
	\end{tabular}
	\caption{}\label{tab:MCGactionArcsPuzzlePiecesSimplified}
	\end{subfigure}
	\caption{Computations for the proof of Theorem~\ref{thm:MCGaction}.}\label{fig:MCGbimodule}
\end{figure}

\begin{figure}[t]
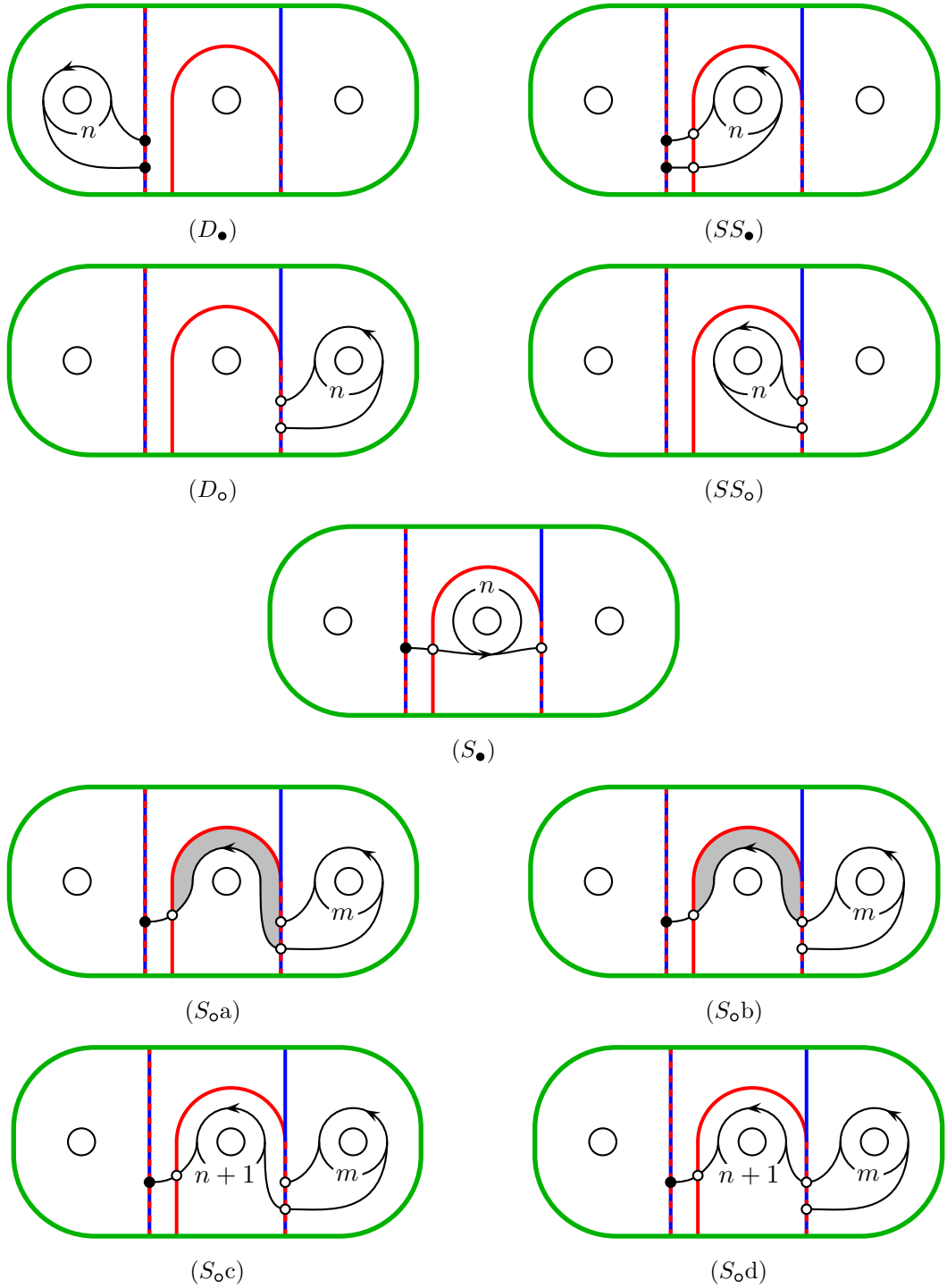

	\centering
	\captionsetup[subfigure]{labelformat=empty}
	\begin{subfigure}{0.45\textwidth}
		\centering
		$\MCGactionArcsBBHnD$
		\caption{(\( \DotcobB \))}
	\end{subfigure}
	\quad
	\begin{subfigure}{0.45\textwidth}
		\centering
		$\MCGactionArcsBBHnSS$
		\caption{(\( S\SaddleCB \))}
	\end{subfigure}
	\medskip\\
	\begin{subfigure}{0.45\textwidth}
		\centering
		$\MCGactionArcsCCHnD$
		\caption{(\( \DotcobC \))}
	\end{subfigure}
	\quad
	\begin{subfigure}{0.45\textwidth}
		\centering
		$\MCGactionArcsCCHnSS$
		\caption{(\( S\SaddleBC \))}
	\end{subfigure}
	\medskip\\
	\begin{subfigure}{0.45\textwidth}
		\centering
		$\MCGactionArcsBCHnS$
		\caption{(\( \SaddleCB \))}
	\end{subfigure}
	\medskip\\
	\begin{subfigure}{0.45\textwidth}
		\centering
		$\MCGactionArcsCBHSD$
		\caption{(\( \SaddleBC \)a)}
	\end{subfigure}
	\quad
	\begin{subfigure}{0.45\textwidth}
		\centering
		$\MCGactionArcsCBHDS$
		\caption{(\( \SaddleBC \)b)}
	\end{subfigure}
	\medskip\\
	\begin{subfigure}{0.45\textwidth}
		\centering
		$\MCGactionArcsCBHnmSD$
		\caption{(\( \SaddleBC \)c)}
	\end{subfigure}
	\quad
	\begin{subfigure}{0.45\textwidth}
		\centering
		$\MCGactionArcsCBHnmDS$
		\caption{(\( \SaddleBC \)d)}
	\end{subfigure}
	\caption{The action of the Dehn twist \( \tau \) on the morphisms in \( \BNAlgH \) (\( m,n\geq0 \))
	}\label{fig:MCGactionArcsPuzzlePieces}
\end{figure}

\begin{proof}[Proof of Theorem~\ref{thm:MCGaction} for \( \rho=\tau \)]	
	Rather than adding a twist to the tangle $T$ and keeping the parameterization the same, we keep the tangle \( T \) fixed and change the parameterization on $\FourPuncturedSphere = \partial D^3 \smallsetminus \partial T$ instead. For this, just as in Section~\ref{sec:classification}, we are representing the 4-punctured sphere $\FourPuncturedSphere$ (Figure~\ref{fig:quiverFuk:surface:dual}) as a 3-punctured disc (Figure~\ref{fig:MCGactionArcs}), equipped with the arc system \( \{\textcolor{blue}{b},\textcolor{blue}{c}\} \) and the dual arc system given by \( \{\textcolor{blue}{b^\perp},\textcolor{blue}{c^\perp}\} \). The new arc system and its dual are given by \( \{\textcolor{red}{B},\textcolor{red}{C}\} \)  and  \( \{\textcolor{red}{B^\perp},\textcolor{red}{C^\perp}\} \), respectively, as shown in Figure~\ref{fig:MCGactionArcs}. One can easily verify that this action on the parameterization has the same effect as applying the Dehn twist \( \tau \) to the tangle \( T \). 
	
	We may now take type~D structure \( \DD(T)^{\BNAlgH} \), find a homotopic type~D structure in the image of \(\Pi\) (Theorem~\ref{thm:EverythingIsLoopTypeUpToLocalSystems}), and draw it as an immersed curve $\BNr(T)$, such that generators in idempotents \( \DotB \) and \( \DotC \) correspond to intersection points with arcs \( \textcolor{blue}{b} \) and \( \textcolor{blue}{c} \), respectively, and differentials correspond to curve segments in the corresponding faces of the 3-punctured disc. 
	
	As we will show, it now suffices to consider the action of the type AD structure \( \BNTwisting \) on simple type~D structures over \( \BNAlgH \), namely, those listed on the left of Figure~\ref{fig:MCGbimodule}, and compare how the corresponding elementary curve segments, shown in Figure~\ref{fig:MCGactionArcsPuzzlePieces}, interact with the original and the new arc system on the fixed 3-punctured disc. 

  Let us consider an arbitrary type~D structure \( X^{\BNAlgH} \) and put it in curve position. For simplicity, we first consider the case when the local systems are trivial. Then, the type~D structure is represented by a collection of cyclic or linear graphs whose arrows are alternatingly labelled by the algebra elements \( D \) multiplied by some power of \( H \) and algebra elements \( S \) or \( SS \) multiplied by some power of \( H \). In Figure~\ref{tab:MCGactionArcsPuzzlePieces}, we consider the action of the bimodule \( \BNTwisting \) on all six different types of arrows, splitting the last case
     \(
     \begin{tikzcd}[column sep=30pt]
	 	\DotC
	 	\arrow{r}{H^{n}\cdot S}
	 	&
	 	\DotB
	 \end{tikzcd}
	 \)
	into four separate subcases. Note that for the computation of the right column the type AD structure $\BNTwisting$ is written with respect to a certain basis, so in order to compute the action on an arrow labelled \( H^n\cdot SS \), we use the identity 
	\[ H^n\cdot SS=H^{n+1}\cdot \id+H^n \cdot D.\] The subdivision of the last case is necessary because of one of the two higher actions of the type AD structure \( \BNTwisting \), namely \( (S\cdot H^k,D\cdot H^n\vert S\cdot H^{k+n}) \). The other higher action \( (S\cdot H^k,S\cdot H^n\vert S\cdot H^{n+k}) \) does not contribute any differential in the \( \boxtimes \)-tensor product, since there are no two consecutive arrows $S\cdot H^n$ and $S\cdot H^k$ in our type~D structures. 
	So, to obtain \( X\boxtimes\BNTwisting \), we can now piece the computed type~D structures from the right column of Figure~\ref{tab:MCGactionArcsPuzzlePieces} together, along the vertices corresponding to the generator \( \DotC\DotC \) in \( \BNTwisting \) and along the arrows coming from the differential \( (-\vert S) \) in \( \BNTwisting \). Unless \( X \) is only built out of arrows \( (\DotcobB) \), \( (\DotcobC) \), \( (S\SaddleBC) \) and \( (\SaddleCB) \), \( X\boxtimes\BNTwisting \) is not in curve position. However, this can be easily fixed by doing certain homotopies, indicated by the dashed arrows in the right column of Figure~\ref{tab:MCGactionArcsPuzzlePieces}. Using the fact that arrows in \( X \) are alternatingly labelled by the algebra elements \( D \) multiplied by some power of \( H \) and algebra elements \( S \) or \( SS \) multiplied by some power of \( H \), we can see that these homotopies have no effect on the rest of the type~D structure \( X\boxtimes\BNTwisting \). Finally, we do some cancellations in the two cases (\( \SaddleBC a \)) and (\( \SaddleBC b \)). The simplified building blocks are shown in Figure~\ref{tab:MCGactionArcsPuzzlePiecesSimplified}. These can be pieced together along the vertices corresponding to the generators \( \DotC\DotC \) and  \( \DotB\DotB \) in \( \BNTwisting \), noting that in case \( (\DotcobB) \), we have omitted the arrows coming from the differential \( (-\vert S) \) in \( \BNTwisting \). 
	
  It now only remains to compare the building blocks from Figure~\ref{tab:MCGactionArcsPuzzlePiecesSimplified} to Figure~\ref{fig:MCGactionArcsPuzzlePieces} and observe that the shown curve segments with respect to the original (blue) arc system correspond to the type~D structures in the left column of Figure~\ref{tab:MCGactionArcsPuzzlePiecesSimplified}, and with respect to the new (red) arc system they correspond precisely to the type~D structures in the right column of Figure~\ref{tab:MCGactionArcsPuzzlePiecesSimplified}. 
  
  The case when a curve has a non-trivial local system on one of its components can be reduced to the previous case. Since any local system on a non-compact component is equivalent to a trivial local system, we may assume that the component in question is compact. Suppose the component contains no generator \( \DotB \). Then we only have arrows corresponding to the two cases \( (\DotcobC) \) and \( (S\SaddleBC) \). If a component contains a generator \( \DotB \), then
  we see an arrow from case \( (\DotcobB) \). Therefore we can push the local system into one of the three components \( (\DotcobC) \), \( (S\SaddleBC) \), or \( (\DotcobB) \), and in each case it is clear (from Figure~\ref{tab:MCGactionArcsPuzzlePiecesSimplified}) that the local systems before and after applying \( \BNTwisting \) agree. 
  This concludes the proof of naturality of \( \Arc(T) \) for \( \rho=\tau \).
  
  Let us now address the proof of naturality of \( \Eight[n](T) \). First, observe that for any type~D structure \( X^{\BNAlgH} \) one has \( 
  \big[
  \begin{tikzcd}[column sep=30pt]
  X
  \arrow{r}{H^n}
  &
  X
  \end{tikzcd}
  \big]^{\BNAlgH}= X^{\BNAlgH} \boxtimes {}_{\BNAlgH} 
  \big[
  \begin{tikzcd}[column sep=30pt]
  \mathbb{I}
  \arrow{r}{(-|H^n)}
  &
  \mathbb{I}
  \end{tikzcd}
  \big]^{\BNAlgH}\), where \( {}_{\BNAlgH}
  \big[
  \begin{tikzcd}[column sep=30pt]
  \mathbb{I}
  \arrow{r}{(-|H^n)}
  &
  \mathbb{I}
  \end{tikzcd}
  \big]^{\BNAlgH} \) is the mapping cone of the \emph{identity AD bimodule} \( {}_{\BNAlgH}\mathbb{I}^{\BNAlgH} \) (see \cite[Definition~2.2.48]{LOT-bim}). Thus it is sufficient to prove 
  \[ {}_{\BNAlgH}\tau^{\BNAlgH} \boxtimes {}_{\BNAlgH}\big[
  \begin{tikzcd}[column sep=30pt]
  \mathbb{I}
  \arrow{r}{(-|H^n)}
  &
  \mathbb{I}
  \end{tikzcd}
  \big]^{\BNAlgH} \simeq {}_{\BNAlgH}\big[
  \begin{tikzcd}[column sep=30pt]
  \mathbb{I}
  \arrow{r}{(-|H^n)}
  &
  \mathbb{I}
  \end{tikzcd}
  \big]^{\BNAlgH} \boxtimes {}_{\BNAlgH}\tau^{\BNAlgH}\]
  which follows easily from the actions \( (H^k \vert H^k ) \) in \( {}_{\BNAlgH}\tau^{\BNAlgH} \).
\end{proof}

\begin{proof}[Proof of Theorem~\ref{thm:MCGaction} for general \( \rho \)]
    We have shown so far that the theorem is true for \( \rho=\tau \) and all tangles \( T \). Then, we can apply the above statement to the tangle \( T'=\tau^{-1}(T) \) and obtain
	\[ \tau^{-1}(\Arc(T))=\tau^{-1}(\Arc(\tau(T')))=\tau^{-1}(\tau(\Arc(T')))=\Arc(\tau^{-1}(T))\]
	(and similarly for $\Eight[n](T)$),
	so the theorem also holds for the Dehn twist \( \rho=\tau^{-1} \). Next we use a rotation by \( 180^\circ \) about the axis through 
	the North-West and South-East tangle ends (for visualization see Figure~\ref{fig:intro:ArcParametrization})
	to deduce  Theorem~\ref{thm:MCGaction} for \( \rho \) equal to the Dehn twist (and its inverse) in the neighbourhood of the arc \( b^\perp \). Indeed, on the algebra \( \BNAlgH \), this rotation has the effect of switching idempotents. 
	Then, naturally, the corresponding versions of Proposition~\ref{prop:AddingASingleCrossing} and Theorem~\ref{thm:MCGaction} hold. 
	Finally, we use the fact that the two Dehn twists along the arcs \( b^\perp \) and \( c^\perp \) generate the subgroup of the mapping class group of the 4-punctured sphere $\FourPuncturedSphere$, elements of which fix the reduction basepoint $\ast$. 
\end{proof}

\FloatBarrier

\subsection{Identification of the pillowcase with the tangle boundary}\label{sec:pillowcase}
 Hedden, Herald, and Kirk associate with 4-ended tangles \( T \) various curves in the space
\[
\RepVarSphere\coloneqq \{\text{traceless representations } \pi_1(S^2\smallsetminus \text{4 points})\rightarrow SU(2)\}/\!\!\sim
\]
where \emph{traceless} means that the loops around the four punctures are sent to traceless elements, and $\sim$ denotes conjugation \cite{HHK1,HHK2}. The basic idea of the construction is as simple as it is beautiful: the inclusion of the tangle boundary $(S^2\smallsetminus \text{4 points})$ into the tangle complement $B^3\smallsetminus T$ induces a map from
\[
R(B^3,T)\coloneqq \{\text{traceless representations } \pi_1(B^3\smallsetminus T)\rightarrow SU(2)\}/\!\!\sim
\]
to the representation variety $\RepVarSphere$; \emph{traceless} for $R(B^3,T)$ means that meridians of all tangle components are sent to traceless elements. 
This map, in particular its image, is a tangle invariant.
Hedden, Herald, and Kirk identify $\RepVarSphere$ with the pillowcase $P$, the quotient of the 2-dimensional torus by the hyperelliptic involution. This is an orbifold homeomorphic to a 2-dimensional sphere with four singular points.
We show the following:
\begin{proposition}\label{prp:pillowcase}
	Distinguish one of the punctures in the tangle boundary \(\PFPS\) by an asterisk~\(\ast\). Then, up to homotopy, there is a unique identification of \(P\cong \RepVarSphere\) minus the four marked points with the tangle boundary such that \(R(B^3,T)\) agrees with \(\BNr(T;\fieldTwoElements)\) for all pointed rational tangles. Moreover, this identification is compatible with Dehn twists. 
\end{proposition} 
As a result, we may regard $R(B^3,T)$---as well as the other immersed curves from~\cite{HHK1,HHK2,HHHK}---as an invariant of unparametrized pointed 4-ended tangles that live on the tangle boundary, placing them on the same footing as $\BNr(T)$, $\Khr(T)$, and $\Kh(T)$.

\begin{figure}[t]
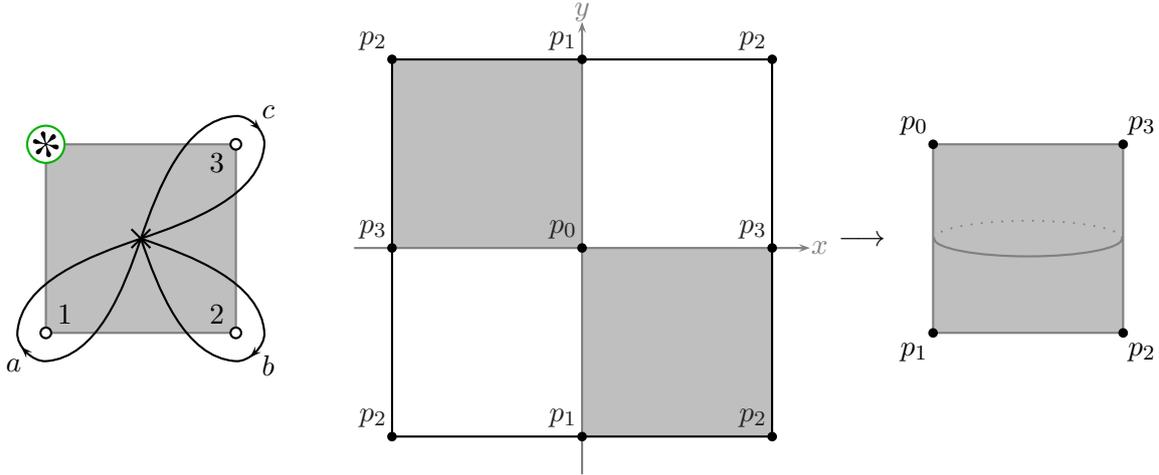

	\centering
	$\PillowcaseAsTangleBoundary
	\qquad
	\PillowcaseAsTorusQuotient\longrightarrow\PillowcaseAsPillowcase$
	\caption{The tangle boundary $\PFPS$ (left) and the pillowcase as the quotient of the torus (right)}
	\label{fig:pillowcase}
\end{figure}

\begin{proof}[Proof of Proposition \ref{prp:pillowcase}]
 Fix generators $a$, $b$, and $c$ for the free group $\pi_1(S^2\smallsetminus \text{4 points})$ as shown in Figure~\ref{fig:pillowcase}. The  map defining the pillowcase as the quotient of the 2-dimensional torus under the hyperelliptic involution is illustrated on the right of Figure~\ref{fig:pillowcase}, where the torus is drawn in terms of its fundamental domain in the universal cover. 
These two observations give rise to an explicit description of the homeomorphism between the pillowcase $P$ and $\RepVarSphere$:
\begin{align*} 
\varphi\co
P = [-1,1]^2/(\pm1)
&\rightarrow
\Big\{\rho\co\langle a, b, c\rangle\rightarrow SU(2)\Big\vert \rho(a),\rho(b),\rho(c)\in su(2)\Big\}\Big/\!\!\sim~ = \RepVarSphere   \\
(x,y)
&\mapsto 
\Big(\rho_{x,y}\co a \mapsto \ii e^{\k\pi y}, b \mapsto \ii, c \mapsto \ii e^{\k\pi x}\Big)
\end{align*}
Here, we are using the usual identification of $SU(2)$ with the unit quaternions, under which elements in the subset of traceless matrices $su(2)$ correspond to the unit quaternions with vanishing real component.
The homeomorphism $\varphi$ differs slightly from the one in~\cite[Proposition~3.1]{HHK1}. However, both homomorphisms share the following desirable property: the four marked points on the pillowcase, that is, the branched points $p_0$, $p_1$, $p_2$, and $p_3$ of the torus cover in Figure~\ref{fig:pillowcase}, correspond to reducible non-central representations, and all other points correspond to irreducible representations. More explicitly:
\begin{align*}
p_0=(0,0) &\leftrightarrow  (a \mapsto \ii, b \mapsto \ii, c \mapsto \ii ) \\
p_1=(0,1) &\leftrightarrow  (a \mapsto -\ii, b \mapsto \ii, c \mapsto \ii ) \\
p_2=(1,1) &\leftrightarrow  (a \mapsto \ii, b \mapsto -\ii, c \mapsto \ii ) \\ 
p_3=(1,0) &\leftrightarrow  (a \mapsto \ii, b \mapsto \ii, c \mapsto -\ii )
\end{align*}
We now compute $R(B^3,\Lo)$. The inclusion of the tangle boundary into the tangle complement induces the quotient map 
on fundamental groups which corresponds to adding the relation $ab=1$. So the image of $R(B^3,\Lo)$ in the pillowcase under the identification $\varphi$ is given by the straight line $y=1$ connecting $p_1$ and $p_2$. Similarly, if we add the relation $bc=1$, we obtain the image of $R(B^3,\Li)$, which is equal to the line $x=1$ connecting $p_2$ and $p_3$. 
So it is natural to identify the pillowcase minus the marked points with the tangle boundary in the following way: $p_0$ corresponds to the special puncture and $p_1$, $p_2$ and $p_3$ correspond to the punctures 1, 2 and 3 in Figure~\ref{fig:pillowcase}, respectively, and the front (back) part of the pillowcase is mapped to the grey (white) part of the tangle boundary. Under this identification, \(R(B^3,T)\) agrees with \(\BNr(T;\fieldTwoElements)\) for the two trivial tangles. 

It now suffices to show that this identification is compatible with the action of the mapping class group of the 4-punctured sphere that fixes the special puncture $\ast$. This mapping class group is generated by the Dehn twists $\tau_{12}$ and $\tau_{23}$ switching the punctures (1 and 2) and (2 and 3), respectively. On the fundamental group $\pi_1(S^2\smallsetminus \text{4 points})=\langle a,b,c\rangle$, $\tau_{12}$ acts as follows:
\[
\tau_{12}\co 
a\mapsto b,~
b\mapsto b^{-1}ab,~
c\mapsto c
\]
Now, on one hand, $\tau_{12}$ acts on $\RepVarSphere$ by precomposition; we obtain
\[
\Big(
\tau_{12}.\rho_{x,y}\co 
a\mapsto \ii,~
b\mapsto \ii^{-1}\ii e^{\k\pi y}\ii,~
c\mapsto \ii e^{\k\pi x}
\Big)
\sim
\Big(
a\mapsto \ii e^{\k\pi y},~
b\mapsto \ii,~
c\mapsto \ii e^{\k\pi (x+y)}
\Big)
\]
by conjugation with the element $e^{\k\pi y/2}$.
On the other hand, the action of $\tau_{12}$ on the tangle boundary is equal to the map induced by the linear transformation 
$
\left(\begin{smallmatrix}
1 & 1 \\
0 & 1
\end{smallmatrix}\right)
$
on the covering torus, so the two actions agree. 
The same argument works for the Dehn twist $\tau_{23}$. 
\end{proof}
Let us compute the invariant $R(B^3,\CrossingLDot)$ to check our conventions for the Dehn twist actions on the representation variety and the pillowcase: indeed, adding the relation $ac=1$ gives the line $y=x-1$, which agrees with $\BNr(\CrossingLDot;\fieldTwoElements)$.

\section{Signs and mutation}\label{sec:mutation}

Khovanov homology behaves differently over \( \fieldTwoElements \) and \( \Q \). This can be seen, for instance, when considering mutant pairs of knots. In this section we explore this phenomenon.

\subsection{Dependence of invariants on the field}\label{subsec:field_dependence}
We first note that  \( \DD(T) \) and its corresponding curve \( \BNr(T) \) depend on the choice of \( \field \).
\begin{example}\label{example:diff_curves}
Figure~\ref{fig:kh-field-dependence} shows the reduced Khovanov homology of the torus knot \( T(4,5) \), along with the differentials of the type~D structure \( \Khr(T(4,5))^{\Rcomm[H]} \), for \( \Rcomm=\fieldTwoElements \) and \( \Rcomm=\Q \). For these computations, we used the program \cite{knotkit} due to Seed, which computes the Bar-Natan spectral sequence. Obtaining the type D structure from the spectral sequence turns out to be easy in this case owing to grading constraints.
\begin{figure}[t]
\centering 
\begin{tikzpicture}[scale=.5]
  \draw[lightgray,step=1] (-0.5,-0.5) grid (11.5,9.5);
  \draw[line width=1pt,->] (-0.5,0) -- (12,0) node[right] {\( h \)};
  \draw[line width=1pt,->] (0,-0.5) -- (0,10) node[above] {\( q \)};
  \draw (0.5,0) node[below] {\scriptsize \( 0 \)};
  \draw (1.5,0) node[below] {\scriptsize \( 1 \)};
  \draw (2.5,0) node[below] {\scriptsize \( 2 \)};
  \draw (3.5,0) node[below] {\scriptsize \( 3 \)};
  \draw (4.5,0) node[below] {\scriptsize \( 4 \)};
  \draw (5.5,0) node[below] {\scriptsize \( 5 \)};
  \draw (6.5,0) node[below] {\scriptsize \( 6 \)};
  \draw (7.5,0) node[below] {\scriptsize \( 7 \)};
  \draw (8.5,0) node[below] {\scriptsize \( 8 \)};
  \draw (9.5,0) node[below] {\scriptsize \( 9 \)};
  \draw (10.5,0) node[below] {\scriptsize \( 10 \)};
  \draw (0,0.5) node[left] {\scriptsize \( 12 \)};
  \draw (0,1.5) node[left] {\scriptsize \( 14 \)};
  \draw (0,2.5) node[left] {\scriptsize \( 16 \)};
  \draw (0,3.5) node[left] {\scriptsize \( 18 \)};
  \draw (0,4.5) node[left] {\scriptsize \( 20 \)};
  \draw (0,5.5) node[left] {\scriptsize \( 22 \)};
  \draw (0,6.5) node[left] {\scriptsize \( 24 \)};
  \draw (0,7.5) node[left] {\scriptsize \( 26 \)};
  \draw (0,8.5) node[left] {\scriptsize \( 28 \)};
  \draw (0.5, 0.5) node {\( \gen \)};
  \draw (2.5, 2.5) node {\( \gen \)};
  \draw (3.5, 3.5) node {\( \gen \)};
  \draw (4.5, 3.5) node {\( \gen \)};
  \draw (5.5, 5.5) node {\( \gen \)};
  \draw (6.5, 4.5) node {\( \gen \)};
  \draw (6.5, 5.5) node {\( \gen \)};
  \draw (7.5, 5.5) node {\( \gen \)};
  \draw (7.5, 6.5) node {\( \gen \)};
  \draw (8.5, 6.5) node {\( \gen \)};
  \draw (9.5, 7.5) node {\( \gen \)};
  \draw (9.5, 8.5) node {\( \gen \)};
  \draw (10.5, 8.5) node {\( \gen \)};
  \draw[->] (2.712, 2.712) -- (3.288, 3.288) node[midway,below,scale=.8]{\( H \)};
  \draw[->] (4.712, 3.712) -- (5.288, 5.288) node[midway,left,scale=.8]{\( H^2 \)};
  \draw[->] (8.712, 6.712) -- (9.288, 8.288) node[midway,left,scale=.8]{\( H^2 \)};
  \draw[->] (6.712, 4.712) -- (7.288, 5.288) node[midway,below,scale=.8]{\( H \)};
  \draw[->] (6.712, 5.712) -- (7.288, 6.288) node[midway,above,scale=.8]{\( H \)};
  \draw[->] (9.712, 7.712) -- (10.288, 8.288) node[midway,below,scale=.8]{\( H \)};
\end{tikzpicture}
\begin{tikzpicture}[scale=.5]
  \draw[lightgray,step=1] (-0.5,-0.5) grid (11.5,9.5);
  \draw[line width=1pt,->] (-0.5,0) -- (12,0) node[right] {\( h \)};
  \draw[line width=1pt,->] (0,-0.5) -- (0,10) node[above] {\( q \)};
  \draw (0.5,0) node[below] {\scriptsize \( 0 \)};
  \draw (1.5,0) node[below] {\scriptsize \( 1 \)};
  \draw (2.5,0) node[below] {\scriptsize \( 2 \)};
  \draw (3.5,0) node[below] {\scriptsize \( 3 \)};
  \draw (4.5,0) node[below] {\scriptsize \( 4 \)};
  \draw (5.5,0) node[below] {\scriptsize \( 5 \)};
  \draw (6.5,0) node[below] {\scriptsize \( 6 \)};
  \draw (7.5,0) node[below] {\scriptsize \( 7 \)};
  \draw (8.5,0) node[below] {\scriptsize \( 8 \)};
  \draw (9.5,0) node[below] {\scriptsize \( 9 \)};
  \draw (10.5,0) node[below] {\scriptsize \( 10 \)};
  \draw (0,0.5) node[left] {\scriptsize \( 12 \)};
  \draw (0,1.5) node[left] {\scriptsize \( 14 \)};
  \draw (0,2.5) node[left] {\scriptsize \( 16 \)};
  \draw (0,3.5) node[left] {\scriptsize \( 18 \)};
  \draw (0,4.5) node[left] {\scriptsize \( 20 \)};
  \draw (0,5.5) node[left] {\scriptsize \( 22 \)};
  \draw (0,6.5) node[left] {\scriptsize \( 24 \)};
  \draw (0,7.5) node[left] {\scriptsize \( 26 \)};
  \draw (0,8.5) node[left] {\scriptsize \( 28 \)};
  \draw (0.5, 0.5) node {\( \gen \)};
  \draw (2.5, 2.5) node {\( \gen \)};
  \draw (3.5, 3.5) node {\( \gen \)};
  \draw (4.5, 3.5) node {\( \gen \)};
  \draw (5.5, 5.5) node {\( \gen \)};
  \draw (6.5, 4.5) node {\( \gen \)};
  \draw (7.5, 6.5) node {\( \gen \)};
  \draw (8.5, 6.5) node {\( \gen \)};
  \draw (9.5, 7.5) node {\( \gen \)};
  \draw[->] (2.712, 2.712) -- (3.288, 3.288) node[midway,below,scale=.8]{\( H \)};
  \draw[->] (8.712, 6.712) -- (9.288, 7.288) node[midway,below,scale=.8]{\( H \)};
  \draw[->] (4.712, 3.712) -- (5.288, 5.288) node[midway,left,scale=.8]{\( H^2 \)};
  \draw[->] (6.712, 4.712) -- (7.288, 6.288) node[midway,left,scale=.8]{\( H^2 \)};
\end{tikzpicture}
\caption{Left: \( \Khr(T(4,5))^{\fieldTwoElements[H]} \). Right: \( \Khr(T(4,5))^{\Q[H]} \).}
\label{fig:kh-field-dependence}
\end{figure}
With this information we can obtain \( \DD( T(4,5) \# \Li) \) as in Example~\ref{example:non-prime-tangles}: Figure~\ref{fig:curve-field-dependence} shows 
\( \DD( T(4,5) \# \Li )\) over \( \fieldTwoElements \) and over \( \Q \).
\begin{figure}[t]
\centering 
\begin{tikzpicture}[scale=.5]
  \draw[lightgray,step=1] (-0.5,-0.5) grid (11.5,9.5);
  \draw[line width=1pt,->] (-0.5,0) -- (12,0) node[right] {\( h \)};
  \draw[line width=1pt,->] (0,-0.5) -- (0,10) node[above] {\( q \)};
  \draw (0.5,0) node[below] {\scriptsize \( 0 \)};
  \draw (1.5,0) node[below] {\scriptsize \( 1 \)};
  \draw (2.5,0) node[below] {\scriptsize \( 2 \)};
  \draw (3.5,0) node[below] {\scriptsize \( 3 \)};
  \draw (4.5,0) node[below] {\scriptsize \( 4 \)};
  \draw (5.5,0) node[below] {\scriptsize \( 5 \)};
  \draw (6.5,0) node[below] {\scriptsize \( 6 \)};
  \draw (7.5,0) node[below] {\scriptsize \( 7 \)};
  \draw (8.5,0) node[below] {\scriptsize \( 8 \)};
  \draw (9.5,0) node[below] {\scriptsize \( 9 \)};
  \draw (10.5,0) node[below] {\scriptsize \( 10 \)};
  \draw (0,0.5) node[left] {\scriptsize \( 12 \)};
  \draw (0,1.5) node[left] {\scriptsize \( 14 \)};
  \draw (0,2.5) node[left] {\scriptsize \( 16 \)};
  \draw (0,3.5) node[left] {\scriptsize \( 18 \)};
  \draw (0,4.5) node[left] {\scriptsize \( 20 \)};
  \draw (0,5.5) node[left] {\scriptsize \( 22 \)};
  \draw (0,6.5) node[left] {\scriptsize \( 24 \)};
  \draw (0,7.5) node[left] {\scriptsize \( 26 \)};
  \draw (0,8.5) node[left] {\scriptsize \( 28 \)};
  \draw (0.5, 0.5) node{\( \DotC \)};
  \draw (2.5, 2.5) node{\( \DotC \)};
  \draw (3.5, 3.5) node{\( \DotC \)};
  \draw (4.5, 3.5) node{\( \DotC \)};
  \draw (5.5, 5.5) node{\( \DotC \)};
  \draw (6.5, 4.5) node{\( \DotC \)};
  \draw (6.5, 5.5) node{\( \DotC \)};
  \draw (7.5, 5.5) node{\( \DotC \)};
  \draw (7.5, 6.5) node{\( \DotC \)};
  \draw (8.5, 6.5) node{\( \DotC \)};
  \draw (9.5, 7.5) node{\( \DotC \)};
  \draw (9.5, 8.5) node{\( \DotC \)};
  \draw (10.5, 8.5) node{\( \DotC \)};
  \draw[->] (2.712, 2.712) -- (3.288, 3.288) node[midway,below,scale=.8]{\( H \)};
  \draw[->] (4.712, 3.712) -- (5.288, 5.288) node[midway,left,scale=.8]{\( H^2 \)};
  \draw[->] (8.712, 6.712) -- (9.288, 8.288) node[midway,left,scale=.8]{\( H^2 \)};
  \draw[->] (6.712, 4.712) -- (7.288, 5.288) node[midway,below,scale=.8]{\( H \)};
  \draw[->] (6.712, 5.712) -- (7.288, 6.288) node[midway,above,scale=.8]{\( H \)};
  \draw[->] (9.712, 7.712) -- (10.288, 8.288) node[midway,below,scale=.8]{\( H \)};
\end{tikzpicture}
\begin{tikzpicture}[scale=.5]
  \draw[lightgray,step=1] (-0.5,-0.5) grid (11.5,9.5);
  \draw[line width=1pt,->] (-0.5,0) -- (12,0) node[right] {\( h \)};
  \draw[line width=1pt,->] (0,-0.5) -- (0,10) node[above] {\( q \)};
  \draw (0.5,0) node[below] {\scriptsize \( 0 \)};
  \draw (1.5,0) node[below] {\scriptsize \( 1 \)};
  \draw (2.5,0) node[below] {\scriptsize \( 2 \)};
  \draw (3.5,0) node[below] {\scriptsize \( 3 \)};
  \draw (4.5,0) node[below] {\scriptsize \( 4 \)};
  \draw (5.5,0) node[below] {\scriptsize \( 5 \)};
  \draw (6.5,0) node[below] {\scriptsize \( 6 \)};
  \draw (7.5,0) node[below] {\scriptsize \( 7 \)};
  \draw (8.5,0) node[below] {\scriptsize \( 8 \)};
  \draw (9.5,0) node[below] {\scriptsize \( 9 \)};
  \draw (10.5,0) node[below] {\scriptsize \( 10 \)};
  \draw (0,0.5) node[left] {\scriptsize \( 12 \)};
  \draw (0,1.5) node[left] {\scriptsize \( 14 \)};
  \draw (0,2.5) node[left] {\scriptsize \( 16 \)};
  \draw (0,3.5) node[left] {\scriptsize \( 18 \)};
  \draw (0,4.5) node[left] {\scriptsize \( 20 \)};
  \draw (0,5.5) node[left] {\scriptsize \( 22 \)};
  \draw (0,6.5) node[left] {\scriptsize \( 24 \)};
  \draw (0,7.5) node[left] {\scriptsize \( 26 \)};
  \draw (0,8.5) node[left] {\scriptsize \( 28 \)};
  \draw (0.5, 0.5) node{\( \DotC \)};
  \draw (2.5, 2.5) node{\( \DotC \)};
  \draw (3.5, 3.5) node{\( \DotC \)};
  \draw (4.5, 3.5) node{\( \DotC \)};
  \draw (5.5, 5.5) node{\( \DotC \)};
  \draw (6.5, 4.5) node{\( \DotC \)};
  \draw (7.5, 6.5) node{\( \DotC \)};
  \draw (8.5, 6.5) node{\( \DotC \)};
  \draw (9.5, 7.5) node{\( \DotC \)};
  \draw[->] (2.712, 2.712) -- (3.288, 3.288) node[midway,below,scale=.8]{\( H \)};
  \draw[->] (8.712, 6.712) -- (9.288, 7.288) node[midway,below,scale=.8]{\( H \)};
  \draw[->] (4.712, 3.712) -- (5.288, 5.288) node[midway,left,scale=.8]{\( H^2 \)};
  \draw[->] (6.712, 4.712) -- (7.288, 6.288) node[midway,left,scale=.8]{\( H^2 \)};
\end{tikzpicture}
\caption{Left: \( \DD( T(4,5) \# \protect\Li; \fieldTwoElements ) \). Right: \( \protect \DD( T(4,5) \# \protect\Li; \Q ) \).}
\label{fig:curve-field-dependence}
\end{figure}
\end{example}

Example~\ref{example:diff_curves} shows that \( \BNr(T) \), and thus also \( \Khr(T) \) and \( \Kh(T) \), can depend on the field of coefficients \( \field \). In  view of Theorem~\ref{thm:pairing}, this provides a geometric way to see torsion in the link invariants \( \BNr(\Lk)\), \(\Khr(\Lk)\), and \(\Kh(\Lk) \). Recall that there is an  abundance of 2-torsion in \( \Kh(\Lk;\Z) \) (see \cite[Conjecture 1]{Shum_torsion}). Contrasting this, for \( \Khr(\Lk;\Z) \) 2-torsion appears first for \( 13 \)-crossing knots \cite{Shum_torsion}. 
\begin{observation}\label{obs:2_torsion_in_Kh}
The ranks of the unreduced Khovanov homologies \( \Kh(\Knot;\Q) \) and \( \Kh(\Knot;\fieldTwoElements) \) differ when $\Knot$ is the trefoil knot, since  \( \Kh(\Knot;\Z) \) contains 2-torsion. The tangle invariants \( \Kh(T;\Q) \) and \( \Kh(T;\fieldTwoElements) \) already differ for the trivial tangle:
\[ \Kh(\Li;\fieldTwoElements)
 = q^{-1}h^{0}\Eight[1](\Li;\fieldTwoElements) \oplus q^{+1}h^{0}\Eight[1](\Li;\fieldTwoElements)
 , \qquad \Kh(\Li;\Q)=\Eight[2](\Li ;\Q )
,\]
where the first curve is drawn in Figure~\ref{fig:eight_for_trivial}, and the second curve is drawn in Figure~\ref{fig:heart_for_trivial}. 
\begin{figure}[t]
  \bigskip
  \centering
  \begin{subfigure}[t]{0.45\textwidth}
  \centering
  \includegraphics[scale=0.6]{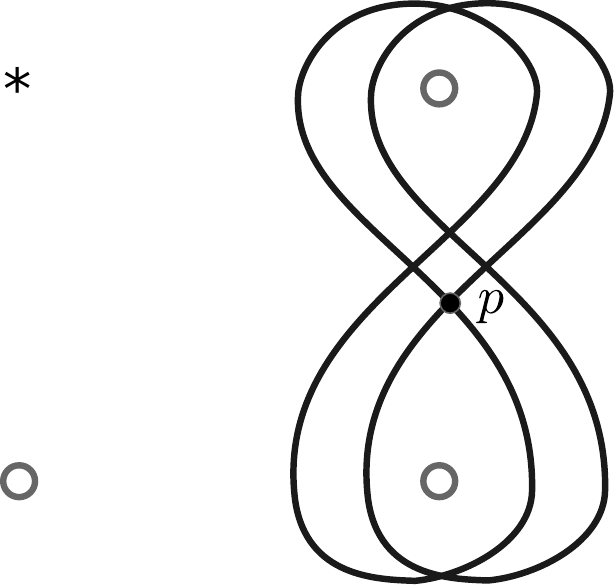}
  \caption{\( \Kh( \Li;\fieldTwoElements)
  \)}
  \label{fig:eight_for_trivial}
  \end{subfigure}
  \begin{subfigure}[t]{0.45\textwidth}
  \centering
  \includegraphics[scale=0.6]{figures/heart_for_trivial}
  \caption{\( \Kh( \Li;\Q) \)}
  \label{fig:heart_for_trivial}
  \end{subfigure}
  \caption{}
  \label{fig:comparison-of-Kh-curves-for-triv-tangle}
\end{figure}

Observe that the curve \( \Eight[1](\Li) \oplus \Eight[1](\Li) \) is obtained from the curve \( \Eight[2](\Li) \) by smoothing a self-intersection \( p \). Ignoring coefficients for a moment, this shows that there is a rank inequality for any test curve $L$: 
\[
\rk \HF(\Eight[1](\Li)\oplus \Eight[1](\Li), L)\geq\rk \HF(\Eight[2](\Li), L) 
\]
Any link $\Lk$ can be written as \(\Lk(\Li,T) \). Setting $L=\BNr(T)$, the above inequality illustrates that the Khovanov homology of $\Lk$ with \( \fieldTwoElements \) coefficients \( \Kh(\Lk;\fieldTwoElements)=\HF(\Eight[1](\Li) \oplus \Eight[1](\Li) , \BNr(T)) \) can be easily larger than the Khovanov homology with \( \Q \) coefficients  \( \Kh(\Lk;\Q )= \HF( \Eight[2](\Li ) , \BNr(T) ) \). In this case, the integral Khovanov homology \( \Kh(\Lk;\Z) \) would have 2-torsion. However, this argument only works if \(L=\BNr(T)\) is independent of the field of coefficients, which is not true in general.
\end{observation}

It is tempting to apply the observation above towards a proof of \cite[Conjecture 1]{Shum_torsion} which says that every knot \( \Knot\ne \Circle \) has 2-torsion in \( \Kh(\Knot;\Z) \). More precisely, one could try to prove that for any decomposition \( \Knot=\Lk(\Li,T) \), the curve \( \BNr(T;\Q) \) intersects \( \Kh(\Li;\Q) \) in fewer points than the curve  \( \BNr(T;\fieldTwoElements) \) intersects \( \Kh(\Li;\fieldTwoElements) \). This lead us to the following:
\begin{proposition}\label{prop:H1_implies_2-torsion}
If \( \ \BNr(\Lk)_{\Q[H]} \) contains \( \Q[H]/(H) \), or equivalently, if \( \ \Khr(\Lk)^{\Q[H]} \) contains
\(
\begin{tikzcd}[nodes={inner sep=2pt}, column sep=13pt,ampersand replacement=\&]
[\gen
\arrow{r}{H} 
\&
\gen]
\end{tikzcd}
\), then \( \Kh(\Lk;\Z) \) contains 2-torsion.
\end{proposition}
\begin{proof}
While there is a geometric proof using Observation~\ref{obs:2_torsion_in_Kh}, we choose an easier, algebraic route.  Ignoring gradings, Proposition~\ref{prop:kh_mapping_cones} implies:
\[ \Khr(\Lk;\Q)\cong \Homology \Big[
\begin{tikzcd}[nodes={inner sep=2pt}, column sep=13pt,ampersand replacement=\&]
\BNr(\Lk;\Q)
\arrow{r}{H} 
\&
\BNr(\Lk;\Q) 
\end{tikzcd} 
\Big] \qquad \Kh(\Lk;\Q)\cong \Homology \Big[
\begin{tikzcd}[nodes={inner sep=2pt}, column sep=13pt,ampersand replacement=\&]
\BNr(\Lk;\Q)
\arrow{r}{H^2} 
\&
\BNr(\Lk;\Q) 
\end{tikzcd}
\Big]\]
Recall that we have a decomposition \( \BNr(\Lk)_{\Q[H]} = \bigoplus_j [\Q [H]]_j \bigoplus_i \left[\Q[H]/(H^{m_i}) \right]_i \). The key point is that the mapping cone formulas above allow us to understand the rank of Khovanov homology from the structure of \( \BNr(\Lk)_{\Q[H]} \): 
\begin{itemize}
\item \( \Q[H] \) contributes 1 generator to \( \Khr(\Lk;\Q) \) and 2 generators to \( \Kh(\Lk;\Q) \); 
\item \( \Q[H]/(H)  \) contributes 2 generators to \( \Khr(\Lk;\Q) \) and 2 generators to \( \Kh(\Lk;\Q) \);
\item \( \Q[H]/(H^m) ,m\geq 2 \) each contribute 2 generators to \( \Khr(\Lk;\Q) \) and 4 generators to \( \Kh(\Lk;\Q) \). 
\end{itemize}
Therefore, having \( \Q[H]/(H) \) present in \( \BNr(\Lk)_{\Q[H]} \) implies  \( \rk\Kh(\Lk;\Q) < 2 \rk\Khr(\Lk;\Q) \). Then the universal coefficient formula implies \( \rk\Khr(\Lk;\Q) \leq \rk\Khr(\Lk;\fieldTwoElements) \). Combining this with the fact that \( 2 \rk\Khr(\Lk;\fieldTwoElements)=\rk\Kh(\Lk;\fieldTwoElements) \), we conclude that having \( \Q[H]/(H) \) in \( \BNr(\Lk)_{\Q[H]} \) implies
\[ \rk\Kh(\Lk;\Q) < 2 \rk\Khr(\Lk;\Q) \leq 2 \rk\Khr(\Lk;\fieldTwoElements) = \rk\Kh(\Lk;\fieldTwoElements),\]
which in turn implies that there must be 2-torsion in \( \Kh(\Lk;\Z) \).
\end{proof}

\begin{question}
Does \( \BNr(\Knot)_{\Q[H]} \) contain a summand \( \Q[H]/(H) \) for every non-trivial knot \( \Knot \)? Or equivalently, does \( \Khr(\Knot)^{\Q[H]} \) contain a summand 
\(
\begin{tikzcd}[nodes={inner sep=2pt}, column sep=13pt,ampersand replacement=\&]
[\gen
\arrow{r}{H} 
\&
\gen]
\end{tikzcd}
\) for every non-trivial knot \( \Knot \)?
\end{question}
If the answer is yes, this would imply the existence of 2-torsion in \( \Kh(\Knot;\Z) \) for every knot \( \Knot\ne \Circle \). We remark that all examples of \( \BNr(\Knot)_{\Q[H]} \) that the authors are aware of include \( \Q[H]/(H) \). On the other hand, computations suggested that the knight-move conjecture \cite[Conjecture~1]{BarNatanKh} was true; in our setting, in the light of Proposition~\ref{prop:kh_mapping_cones}, this states that \( \BNr(\Knot)_{\Q[H]} \) consists of only \( \Q[H] \), \( \Q[H]/(H) \) and \( \Q[H]/(H^2) \). Recently, Manolescu and Marengon disproved the knight-move conjecture, and thus found a knot containing \( \Q[H]/(H^m),m\geq 3 \) in its reduced Bar-Natan homology \cite{Man_Mar_knight_move_false}. 
The following result first appeared in \cite[Corollary 5]{Shum_torsion}, and was generalized in \cite{AP_torsion_thickness}.
\begin{corollary}
All homologically thin non-trivial knots \( \Knot \) (ie such that \( \Khr(\Knot;\Z) \) is concentrated in a single \( \delta \)-grading) have 2-torsion in \( \Kh(\Knot;\Z) \). In particular, all non-trivial quasi-alternating knots have 2-torsion in \( \Kh(\Knot;\Z) \). 
\end{corollary}
\begin{proof}
The fact that \( \Khr(\Knot;\Z) \) is concentrated in a single \( \delta \)-grading implies that 
\( 
\begin{tikzcd}[nodes={inner sep=2pt}, column sep=13pt,ampersand replacement=\&]
[\gen
\arrow{r}{H^m} 
\&
\gen]
\end{tikzcd}\),
\(m\geq 2 \) 
does not appear in \( \Khr(\Knot)^{\Q[H]} \). Non-triviality  \( \Knot\neq \Circle \) implies that there must be 
\(
\begin{tikzcd}[nodes={inner sep=2pt}, column sep=13pt,ampersand replacement=\&]
[\gen
\arrow{r}{H} 
\&
\gen]
\end{tikzcd}
\) 
in \( \Khr(\Knot)^{\Q[H]} \) due to the fact that the rank of \( \Khr \) detects the unknot \cite{K-M_unknot_detector}. Thus there is \( \Q[H]/(H) \) in \( \BNr(\Knot)_{\Q[H]} \), and so by Proposition~\ref{prop:H1_implies_2-torsion}, Khovanov homology \( \Kh(\Knot;\Z) \) must have 2-torsion.
The fact that quasi-alternating knots are homologically thin is proved in \cite{Man-Ozs_quasi}.
\end{proof}

\begin{wrapfigure}{r}{0.2\textwidth}
    \centering
    \begin{tikzpicture}[scale=.65]
      \draw[lightgray,step=1] (-0.5,-0.5) grid (2.5,2.5);
      \draw[line width=1pt,->] (-0.5,0) -- (3,0) node[right] {\( h \)};
      \draw[line width=1pt,->] (0,-0.5) -- (0,3) node[above] {\( q \)};
      \draw(0.5, 0.5) node {\( \gen \)};
      \draw(0.5, 1.5) node {\( \gen \)};
      \draw(1.5, 1.5) node {\( \gen \)};
      \draw(-0.5, 0.5) node {\scriptsize\( 0 \)};
      \draw(-0.5, 1.5) node {\scriptsize\( 2 \)};
      \draw(0.5, -0.5) node {\scriptsize\( 0 \)};
      \draw(1.5, -0.5) node {\scriptsize\( 1 \)};
      \draw[->] (0.5, 0.5) -- (1.5, 1.5);
      \draw(1.25, 0.75) node {\( H \)};
      \draw[->] (0.5, 1.5) -- (1.5, 1.5) node[midway,above] {\( 2 \)};
    \end{tikzpicture}\vspace{-8pt}
    \caption{} \label{fig:s-invs-different}
\end{wrapfigure}
Finally, we illustrate the dependence of the \( s \)-invariant on the choice of field. Again, we consider the cases \( \Q \) and \( \fieldTwoElements \). Seed found a knot \( \Knot 14n19265 \) for which \( s^\fieldTwoElements\ne s^\Q \); see \cite[Remark 6.1]{LipSar}. We attempted to determine \( \Khr(\Knot 14n19265)^{\Z[H]} \) from the existing computer programs, in order to see how the \( q \)-grading of the single free summand \( [\gen] \) could differ with the choice of field. Based on computations of the Bar-Natan spectral sequence \( \Khr(\Knot 14n19265) \rightrightarrows \Homology (\CBNr(\Knot)|_{H=1})= \field \) (using  \cite{knotkit}), for \( \field=\fieldTwoElements \) and  \( \field=\Q \), we were only able to make an educated guess for \( \Khr(\Knot 14n19265)^{\Z[H]} \) and, in particular, how this type~D structure behaves  near the free summand with grading \( s(\Knot 14n19265) \). We illustrate the relevant part of the complex in Figure~\ref{fig:s-invs-different}: notice that, working over \( \fieldTwoElements \) the free summand \( [\gen] \) has \( q \)-grading \( 2 \), while working over \( \Q \) allows one to cancel the horizontal arrow and the left free summand \( [\gen] \) has \( q \)-grading \( 0 \).

\subsection{Conway mutation}

	Given a tangle \( R \), let \( \mutx(R) \), \( \muty(R) \) and \( \mutz(R) \) be the tangles obtained from \( R \) by  $180^\circ$ rotations about the $x$-, $y$-, and $z$-axis, respectively; see Figure~\ref{fig:mutation}.  We say that  \( \mutx(R) \), \( \muty(R) \) and \( \mutz(R) \) are obtained from \( R \) by mutation. 
  If \( R \) is oriented, we orient these three tangles such that the orientations agree at the tangle ends. If this means that we need to reverse the orientation of the two open components of~\( R \), then we also reverse the orientation of all other components; otherwise we do not change any orientation. 
	
\begin{definition}\label{def:mutation}	We say that two links $\Lk$ and $\Lk'$ are \textbf{Conway mutants} if they agree outside a 3-ball whose boundary intersects the links transversely in four points, and the tangles inside the 3-ball are related by mutation; in other words $\Lk=\Lk(T_0,T)$ and  $\Lk'=\Lk(T_0,\muti(T))$.
\end{definition}

\begin{definition}
Let \( \varphi_{x} \), \( \varphi_{y} \), and \( \varphi_{z} \) be the three  automorphisms of \( \BNAlgH \) defined by
	\begin{align*}
	\varphi_{x}\co & 
	\DotcobB \mapsto -\DotcobB,\quad
	\DotcobC \mapsto +\DotcobC,\quad
	\SaddleCB \mapsto \SaddleCB,\quad
	\SaddleBC \mapsto \SaddleBC,\\
	\varphi_{y}\co & 
	\DotcobB \mapsto +\DotcobB,\quad
	\DotcobC \mapsto -\DotcobC,\quad
	\SaddleCB \mapsto \SaddleCB,\quad
	\SaddleBC \mapsto \SaddleBC,\\
	\varphi_{z}\co & 
	\DotcobB \mapsto -\DotcobB,\quad
	\DotcobC \mapsto -\DotcobC,\quad
	\SaddleCB \mapsto \SaddleCB,\quad
	\SaddleBC \mapsto \SaddleBC.
	\end{align*}
	Let us denote the induced functors on type~D structures by \( \mutx \), \( \muty \), and \( \mutz \), respectively.
\end{definition}

\begin{figure}[b]
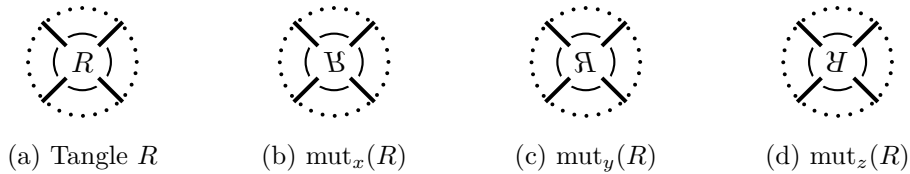

  \centering
  \begin{subfigure}{0.2\textwidth}
    \centering
    $\MutationTangle$
    \caption{Tangle \( R \)}
  \end{subfigure}
  \begin{subfigure}{0.2\textwidth}
    \centering
    $\MutationTangleX$
    \caption{\( \mutx(R) \)}
  \end{subfigure}
  \begin{subfigure}{0.2\textwidth}
    \centering
    $\MutationTangleY$
    \caption{\( \muty(R) \)}
  \end{subfigure}
  \begin{subfigure}{0.2\textwidth}
    \centering
    $\MutationTangleZ$
    \caption{\( \mutz(R) \)}
  \end{subfigure}
  \caption{Conway mutation.}\label{fig:mutation}
\end{figure}

\begin{theorem}\label{thm:mutation:DD}
	For any pointed 4-ended tangle \(T\) and \( i\in\{x,y,z\} \), 
	\(\DD(\muti(T))=\muti(\DD(T))\).
\end{theorem}

\begin{proof}
	Let us first consider the case \( i=z \). Let \( \rotz \) denote the operation which rotates both the objects and morphisms in \( \Cob_{/l} \) by \( 180^\circ \) in the plane. Then, by construction, \( \KhTl{\mutz(T)}=\rotz(\KhTl{T}) \). Note, in particular, that the bigradings agree, regardless of the orientations on \(T\). Together with Definition~\ref{def:MainTangleInvariantTypeD} of \( \DD(T) \) and \( \DD(\mutz(T)) \), this gives
	\[ 
	\Omega_1(\DD(\mutz(T)))
	=
	\KhTl{\mutz(T)}
	=
	\rotz(\KhTl{T})
	=
	\rotz(\Omega_1(\DD(T))).
	\]
	The key idea now is to invoke Observation~\ref{obs:mutation}, which gives \( \rotz\circ~\omega_1=\omega_1\circ \varphi_z \). This, in turn, implies 
	\( \rotz\circ~\Omega_1=\Omega_1\circ \mutz \), so the above is equal to \( \Omega_1(\mutz(\DD(T))) \). Finally, faithfulness of \( \Omega_1 \) implies the desired identity. 
	
	In the cases \( i=x \) and \( i=y \), let \( \refx \) and \( \refy \) denote the reflection of objects and morphisms in \( \Cob_{/l} \) along the horizontal and vertical axis, respectively. Then, the identities 
	\( \KhTl{\muti(T)}=\refi(\KhTl{T}) \)
	can be seen on the level of the cube of resolutions. The rest of the argument is the same as above. We obtain
	\[ 
	\Omega_1(\DD(\muti(T)))
	=
	\KhTl{\muti(T)}
	=
	\refi(\KhTl{T})
	=
	\refi(\Omega_1(\DD(T)))
	=
	\Omega_1(\muti(\DD(T)))
	\]
	by Definition~\ref{def:MainTangleInvariantTypeD} and Observation~\ref{obs:mutation}, which implies the desired identities.
\end{proof}

\begin{theorem}\label{thm:mutation:BNr}
	For any pointed 4-ended tangle \(T\) and \( i\in\{x,y,z\} \), 
	\(\BNr(\muti(T))\) and \(\BNr(T)\) agree up to multiplication of the local systems by \(\pm1\). In particular, 	\(\BNr(\muti(T);\fieldTwoElements)=\BNr(T;\fieldTwoElements)\) and hence \(\BNr(\Lk;\fieldTwoElements)=\BNr(\Lk';\fieldTwoElements)\) for any mutant links \(\Lk\) and \(\Lk'\). 
\end{theorem}

\begin{proof}
	The first statement follows immediately from Theorem~\ref{thm:mutation:DD}. Then, if the field of coefficients has characteristic 2, the local systems actually agree. The final statement follows from the previous one together with Theorem~\ref{thm:pairing} for reduced Bar-Natan homology. 
\end{proof}

This gives a new proof of a result due to Bloom \cite{Bloom} and Wehrli \cite{wehrli2010mutation}:

\begin{corollary} Khovanov homology is mutation invariant over $\fieldTwoElements$.
\qed
\end{corollary}

\begin{figure}
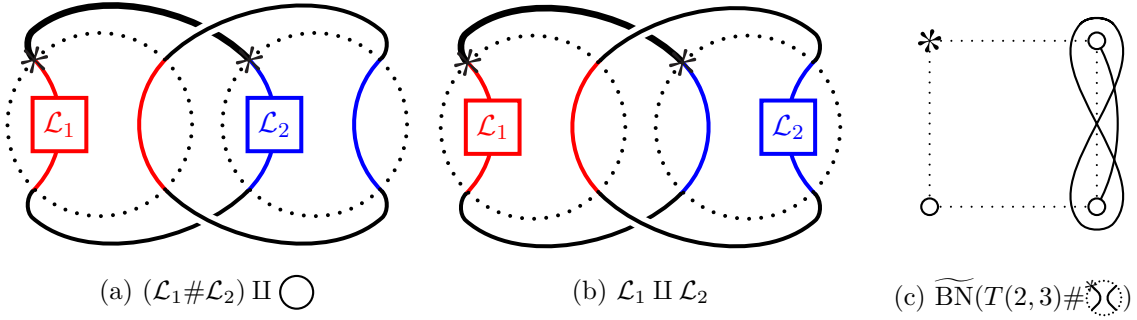

	\centering
	\begin{subfigure}{0.35\textwidth}
		\centering
		$\MutationCounterexampleII$
		\caption{\( (\Lk_1\#\Lk_2)\amalg \Circle \)}\label{fig:mutation:counterexample:II}
	\end{subfigure}
	\begin{subfigure}{0.35\textwidth}
		\centering
		$\MutationCounterexampleI$
		\caption{\( \Lk_1\amalg \Lk_2 \)}\label{fig:mutation:counterexample:I}
	\end{subfigure}
	\begin{subfigure}{0.25\textwidth}
		\centering
		$\MutationCounterexampleCurve$
		\caption{$\BNr(T(2,3)\# \Li)$}\label{fig:mutation:counterexample:curve}
	\end{subfigure}
	\caption{Counterexample for mutation invariance of Khovanov homology over $\Q$.}\label{fig:mutation:counterexample}
\end{figure}

\begin{example}\label{exa:mutation:counterexample}
	Over $\Q$, however, it is known that \(\BNr\), \(\Khr\), and \(\Kh\) are {\em not} mutation invariant. The counterexample, found by Wehrli, illustrated in Figures~\ref{fig:mutation:counterexample:II} and~\ref{fig:mutation:counterexample:I}. He observed that the disjoint union of two links \( \Lk_1 \) and \( \Lk_2 \) and the disjoint union of an unknot with the connected sum of \( \Lk_1 \) and \( \Lk_2 \) are mutant links whose Khovanov homologies can, in general, be different \cite[Theorem~3]{WehrliCounterexample}. We can explain this failure of mutation invariance using our curve invariants. Assume for simplicity that both links are equal to the trefoil knot \( \Lk_1=\Lk_2=T(2,3) \). 
Let \( T = T(2,3)\# \Li \) be the connected sum of the trefoil knot with the component of the trivial tangle \(\Li\) containing the basepoint $\ast$. We  compute
\[
\Khr(T(2,3))^{\Q[H]}
= 
[\gen]
\oplus
\begin{tikzcd}[nodes={inner sep=2pt}, column sep=30pt,ampersand replacement=\&]
[\gen
\arrow{r}{H} 
\&
\gen]
\end{tikzcd}
\]
Then by Example~\ref{example:non-prime-tangles},
\[ 
\DD(T)
=
[\DotC]
\oplus
\begin{tikzcd}[nodes={inner sep=2pt}, column sep=30pt,ampersand replacement=\&]
[\DotC
\arrow{r}{D-SS} 
\&
\DotC]
\end{tikzcd}
\]
The corresponding multicurve $\BNr(T)$ is illustrated in Figure~\ref{fig:mutation:counterexample:curve}. Its compact component carries the local system \( (-1) \). By Theorem~\ref{thm:mutation:DD}, the curve $\BNr(\mutz(T))$ corresponds to the type~D structure
	\[ 
   \DD(\mutz(T))
   =
   [\DotC]
   \oplus
   \begin{tikzcd}[nodes={inner sep=2pt}, column sep=30pt,ampersand replacement=\&]
   [\DotC
   \arrow{r}{-D-SS} 
   \&
   \DotC].
   \end{tikzcd}
  \]
	Notice that the local system on the compact curve has changed to \( (+1) \). We now see that the additional two generators of
	\[ 
	\BNr(T(2,3)\#T(2,3)\amalg \Circle; \Q)
	\cong
	\HF(\textcolor{red}{\mirror(\BNr(T))},\textcolor{blue}{\BNr(T)})
	\]
	in comparison to 
	\[ 
	\BNr(T(2,3)\amalg T(2,3); \Q)
	\cong
	\HF(\textcolor{red}{\mirror(\BNr(T))},\textcolor{blue}{\BNr(\mutz(T))})
	\]	
	come from the contribution of the local systems, ie the second summand of Equation~\eqref{eqn:MorSpacesParallel} from Theorem~\ref{thm:PairingFormula}. In the first case, the two local systems are both \( (-1) \), so the dimension of 
	\[ \ker\left((X^{-1})^t\otimes X'-\id\right)=\ker\left((-1)\otimes (-1)-\id\right)=\ker(0)\] 
	is 1. In the second case, the local systems are \( (-1) \) and \( (+1) \), so
	\[ \ker\left((X^{-1})^t\otimes X'-\id\right)=\ker\left((-1)\otimes (+1)-\id\right)=\ker(-2),\] 
	which is zero, unless we work over $\fieldTwoElements$. We can argue similarly for the reduced and unreduced Khovanov homology of these two links.
\end{example}

\begin{remark}
This example has a satisfying consequence: while non-trivial local systems have not appeared to date in any of the immersed curve invariants arising in Heegaard Floer theory \cite{HRW, HRW-prop,pqSym,pqMod}, we see that there is a meaningful contribution from local systems in the context of Khovanov homology. 
\end{remark}

\begin{definition}
	A Conway mutation is said to be component preserving if each tangle end belongs to the same tangle component before and after the mutation. 
\end{definition}

The following conjecture remains open in general.

\begin{conjecture}
	\(\BNr(\Lk;\Z)\), \(\Khr(\Lk;\Z)\), and \(\Kh(\Lk;\Z)\)  are invariant under component preserving Conway mutation. In particular, \(\BNr(\Knot;\Z)\), \(\Khr(\Knot;\Z)\), and \(\Kh(\Knot;\Z)\) are preserved by mutation if $\Knot$ is a knot.
\end{conjecture}

\begin{definition}\label{def:arc-and-compact-components}
	Recall from Section~\ref{sec:FigureEight} that the reduced Bar-Natan homology \( \BNr(T) \) of a pointed 4-ended tangle \( T \) necessarily contains at least one non-compact component, and may or may not contain compact components. 
	We denote the non-compact part (or \( arc \) part) of \( \BNr(T) \) by \( \BNra(T) \), and the compact part by \( \BNrc(T) \). We also write \( \DDa(T) \) and \( \DDc(T) \) for the two summands of \( \DD(T) \) which correspond to \( \BNra(T) \) and \( \BNrc(T) \), respectively.
\end{definition}

\begin{corollary}\label{cor:mutation_preserves_arc_component}
Mutation preserves \(\BNra(T)\), and therefore also preserves \(\DDa(T)\).
\end{corollary}
\begin{proof}
This follows from Theorem~\ref{thm:mutation:BNr}, since \( \BNra(T) \) consists of only immersed arcs, and the local system on any immersed arc is trivial.
\end{proof}

\begin{theorem}
For any two mutant knots \(\Knot\) and \(\Knot'\) and any field \(\field\), \(s^\field(\Knot)=s^\field(\Knot')\).
\end{theorem}
\begin{proof}
Suppose \( \Knot=\Lk(T_0,T) \) and \( \Knot'=\Lk(T_0,T') \), where \( T' \) is obtained from \( T \) by Conway mutation.
Then by Proposition~\ref{prop:alg_pairing_kh}:
\[ 
\BNr(\Knot) 
\cong 
\Homology(\Mor(\mirror(\DD(T_0)),\DD(T)))
\cong
\bigoplus_{i,j \in \{a,c\}} 
\Homology(\Mor(\mirror(\DD_i(T_0)),\DD_j(T)))
\]
Note that this splitting respects the \( \field[H] \)-module structure. Now the key observation is the following: the \( \field[H] \)-tower of \( \BNr(\Knot) \) (which must exist by Proposition~\ref{prop:towersReduced-field}) is contained in the $\DDa$-$\DDa$ summand
\[ 
\Homology(\Mor(\mirror(\DDa(T_0)),\DDa(T))) 
.\]
This is because the other three summands compute the Lagrangian Floer homology of a multicurve with \emph{compact} curves, which is necessarily finite dimensional. Similarly, the \( \field[H] \)-tower of \( \BNr(\Knot') \) is also contained in the $\DDa$-$\DDa$ summand
\[ 
\Homology(\Mor(\mirror(\DDa(T_0)),\DDa(T'))) 
.\]
By Proposition~\ref{prop:s-inv-as-grading}, \( s^\field(\Knot) \) and \( s^\field(\Knot') \) are the \( q \)-gradings of the top generators of those two towers, respectively. 
By Corollary~\ref{cor:mutation_preserves_arc_component}, \(\DDa(T)\) and \(\DDa(T')\) are graded isomorphic. Therefore, the two $\DDa$-$\DDa$ summands above are isomorphic as \( \field[H] \)-modules,
and so, in particular, the two \( \field[H] \)-towers start in the same quantum grading.
\end{proof}

\section{\texorpdfstring{Khovanov homology of infinitely twisted knots}{Khovanov homology of infinitely twisted knots}}\label{sec:InfinitelyStranded}
In this section, we define new invariants by applying certain representable functors to our algebraic tangle invariants $\DD(T)$. Geometrically, this corresponds to pairing $\BNr(T)$ with particular immersed curves. While we will emphasize this geometric perspective, we will carry out all constructions in this section algebraically. 
We work over \( \fieldTwoElements \) throughout.

\subsection{Invariants of infinitely twisted knots}
The central objects for this section will be \textbf{\( (2,-\infty) \)-twisted knots}. These are defined as unoriented knot-like objects in \( S^3 \), which admit a decomposition \( \Knot=\Lk(T_{-\infty},T) \) where \( T \) is an ordinary unoriented 4-ended tangle
and \( T_{-\infty} \) is a 4-ended tangle with \( -\infty \) twists (this is a formal object, illustrated in Figure~\ref{fig:Kinf_intro}). 
Denote a \( (2,-\infty) \)-twisted knot by \( \Knot_\infty \). These are topological objects in the sense that they are defined up to isotopy of the smooth part \( T \) (that is away from a ball containing  \( T_{-\infty} \)) and up to adding and removing twists \( \CrossingL \) and \( \CrossingR \) on the top and on the bottom of \( T \). The latter move is explained by the fact that the \( \infty \) part \( T_{-\infty} \) can ‘‘produce'' and ‘‘absorb'' arbitrarily many twists, both at the top and on the bottom. The \( (2,+\infty) \)-twisted knots are defined in the same way; we focus on the \( -\infty \) case noting that the \( +\infty \) case is similar.

From Example~\ref{ex:curves_for_rational_tangles} we observe that there exists a limit \( \lim_{n\to \infty}\DD(\mirror T_{-n})\) and take this as a definition of the type~D structure \( \DD(\mirror T_{-\infty}) \), which we denote by \( \DDinf \) below. 
Analogously, there exists a limit \( \lim_{n\to \infty}
\big[
\begin{tikzcd}
\DD(\mirror T_{-n})
\arrow{r}{H}
&
\DD(\mirror T_{-n})
\end{tikzcd}
\big]
\) and we use this limit as a prototype for 
\(\big[
\begin{tikzcd}
\DD(\mirror T_{-\infty})
\arrow{r}{H}
&
\DD(\mirror T_{-\infty})
\end{tikzcd}
\big]\), which we denote by \( \DDinfinf \).

\begin{definition}
We specify the following type~D structures over $\BNAlgH$:
\begin{equation*}
\begin{split}
\DDinf &\coloneqq  
\left[
\begin{tikzcd}[ampersand replacement=\&]
\DotC \arrow{r}{S}  \& \DotB \arrow{r}{D} \& \DotB  \arrow{r}{SS} \& \DotB  \arrow{r}{D} \& \DotB  \arrow{r}{SS} \& \cdots 
\end{tikzcd}
\right]\\
\DDinfinf  & \coloneqq
    \left[
    \begin{tikzcd}[ampersand replacement=\&]
    \DotC \arrow{r}{S}\arrow{d}{D}  \& \DotB \arrow{r}{D} \& \DotB  \arrow{r}{SS} \& \DotB  \arrow{r}{D} \& \DotB  \arrow{r}{SS} \& \cdots \\
    \DotC \arrow{r}{S} \& \DotB \arrow{r}{D} \& \DotB  \arrow{r}{SS} \& \DotB  \arrow{r}{D} \& \DotB  \arrow{r}{SS} \& \cdots 
    \end{tikzcd}
    \right]
\end{split}
\end{equation*}
\end{definition}

Note that \( \DDinf \) and \( \DDinfinf \) are infinite dimensional, and so they are objects of \( \Mod_\infty^{\BNAlgH} \), rather than \( \Mod^{\BNAlgH} \); see Definition~\ref{def:infinite_D_structures}. The following definition is motivated by the Pairing Theorem (Theorem~\ref{thm:pairing}) and the decomposition \( \Knot_\infty=\Lk(T_{-\infty},T) \).

\begin{definition}\label{def:BNr_Khr_inf_invariants}
With a \( (2,-\infty) \)-twisted knot \( \Knot_\infty \) we associate the vector spaces
\[ 
\BNr(\Knot_\infty)
\coloneqq 
\Homology (\Mor_{\fs}(\DDinf,\DD(T))) 
\qquad
\Khr(\Knot_\infty)
\coloneqq 
\Homology (\Mor_{\fs}(\DDinfinf,\DD(T)))
\]
where \( \Mor_{\fs} \) specifies that only compact (and therefore finite) support elements of the morphism spaces are considered.
\end{definition}

We can view \(\DDinf\) and \(\DDinfinf\) as the complexes corresponding to the curves 
$\Linf$ and $\Linfinf$ in Figure~\ref{fig:inf_curves}, which in turn are the limits of the curves \(\BNr(\mirror(T_{-n}))\) and \(\Khr(\mirror(T_{-n}))\), respectively.
\begin{figure}[t]
  \centering
  \begin{subfigure}[t]{0.4\textwidth}
  \includegraphics[scale=0.6]{figures/Linf}
  \caption{\( \Linf =  \lim_{n\to \infty} \mirror(\BNr (T_{-n})) \)}
  \label{fig:inf_curve_1}
  \end{subfigure}
  \begin{subfigure}[t]{0.4\textwidth}
  \includegraphics[scale=0.6]{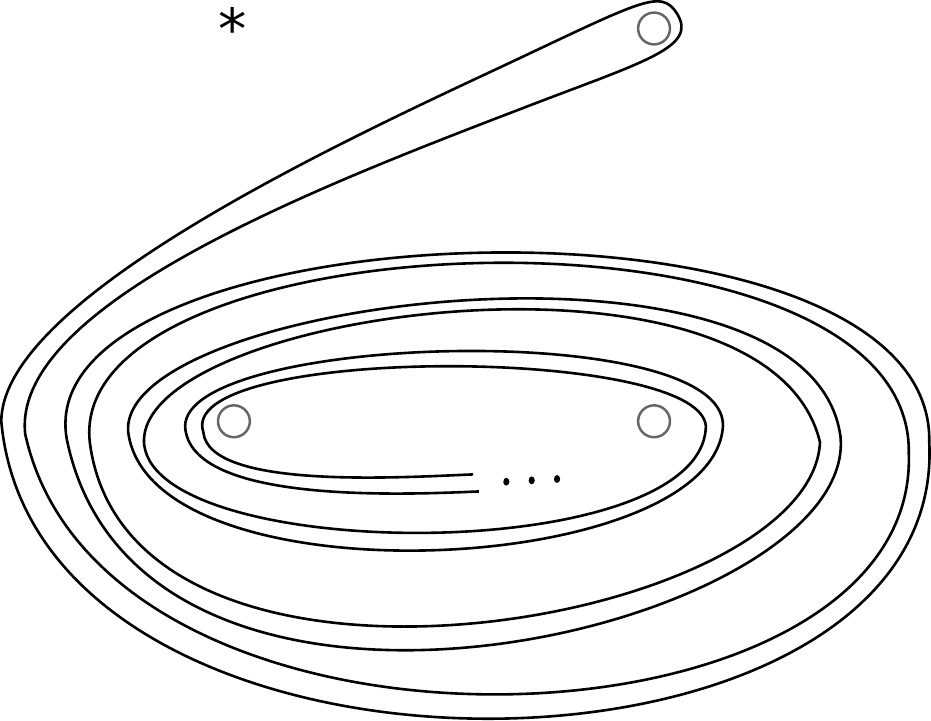}
  \caption{\( \Linfinf = \lim_{n\to \infty} \mirror(\Khr (T_{-n})) \)}
  \label{fig:inf_curve_2}
  \end{subfigure}
  \caption{}
  \label{fig:inf_curves}
\end{figure}
This suggests that we may regard \( \BNr(\Knot_\infty) \) and \( \Khr(\Knot_\infty) \) as the wrapped Lagrangian Floer homologies \( \HF(\Linf,\BNr (T)) \) and \( \HF(\Linfinf,\BNr (T)) \), respectively. Below, in Theorem~\ref{thm:lifting_to_D_str}, we will see how this point enables us to compute  \( \BNr(\Knot_\infty) \) in terms of $\DD(T)$.

The infinite twisting in \( \Knot_\infty \) endows the invariants \( \BNr(\Knot_\infty) \) and \( \Khr(\Knot_\infty) \) with rich module structures. Consider the following bigraded algebras:
\begin{align*}
&\RxH \coloneqq \fieldTwoElements[x,H]/(\xben H), \ \gr(\xben)=q^{-4}\delta^0, \gr(H)=q^{-2}\delta^{-1} \\
&\Rxy \coloneqq \fieldTwoElements[\xben,\yben]/(\xben^3=\yben^2), \ \gr(\xben)=q^{-4}\delta^{0}, \gr(\yben)=q^{-6}\delta^{0}\\
&\RzT \coloneqq \fieldTwoElements\langle z,T \rangle/ (z^2,T^2), \ \gr(T)=q^{2}\delta^0, \gr(z)=q^{-6}\delta^{0}
\end{align*}
Note that there is an inclusion \( \Rxy \subset \RzT, \ \xben \mapsto zT+Tz, \yben \mapsto z+TzTzT \). These algebras arise as endomorphism algebras:
\begin{align*}
\RxH &\subset \Mor(\DDinf,\DDinf) , \ 1\mapsto \id,
\xben \mapsto \{ \text{dotted arrows below} \},
H \mapsto \{ \text{dashed arrow below} \},\\
&\begin{tikzcd}[ampersand replacement = \&,column sep =1.4cm]
\DotC \arrow{r}{S} \arrow[lr, in=+135,out=+45,looseness=4, blue, dashed , "D" description]  \& \DotB \arrow{r}{D} \& \DotB \arrow[ll, bend right, dotted, "S" description]  \arrow{r}{SS} \& \DotB \arrow[ll, bend right, dotted, "1" description]  \arrow{r}{D} \& \DotB \arrow[ll, bend right,dotted, "1" description]  \arrow{r}{SS} \& \cdots \arrow[ll, bend right,dotted, "\cdots" description] 
\end{tikzcd}
\\
\Rxy &\subset  \Mor(\DDinfinf,\DDinfinf) , \ 1\mapsto \id,
\xben \mapsto \{ \text{dotted arrows below} \},
\yben \mapsto \{ \text{dashed arrows below} \},\\
  &\begin{tikzcd}[ampersand replacement = \&, row sep=1.4cm,column sep =1.4cm]
  \DotC \arrow{r}{S}\arrow{d}{D}  
  \& \DotB\arrow{r}{D} 
  \& \DotB\arrow[ll, bend right, dotted, "S" description]  \arrow{r}{SS} 
  \& \DotB\arrow[ll, bend right, dotted, "1" description] \arrow{r}{D} 
  \& \DotB \arrow{r}{SS} \arrow[ll, bend right, dotted, "1" description] \arrow[dllll, blue, dashed, near end, "S" description]  
  \& \DotB  \arrow{r}{D} \arrow[ll, bend right, dotted, "1" description] \arrow[dllll, blue, dashed, near end, "1" description] 
  \& \cdots \arrow[dllll, blue, dashed, near end, "1" description] \arrow[ll, bend right, dotted, "\cdots" description] \&  \arrow[dllll, blue, dashed,  near end, "\cdots" description] \\
  \DotC \arrow[r,"S" below]
  \& \DotB \arrow[r,"D" below]
  \& \DotB \arrow[ll, bend left, dotted, "S" description] \arrow[r,"SS" below] \arrow[llu, blue, dashed, near end, "S" description] 
  \& \DotB \arrow[ll, bend left, dotted, "1" description]  \arrow[r,"D" below] \arrow[llu, blue, dashed, near end, "1" description]  
  \& \DotB \arrow[ll, bend left, dotted, "1" description] \arrow[r,"SS" below] \arrow[llu, blue, dashed, very near end, "1" description] 
  \& \DotB  \arrow[r,"D" below] \arrow[ll, bend left, dotted, "1" description] \arrow[llu, blue, dashed, very near end, "1" description]  
  \& \cdots \arrow[ll, bend left, dotted, "\cdots" description] \arrow[llu, blue, dashed, near end,"\cdots" description] 
\end{tikzcd}
\\
\RzT &\subset  \Mor(\DDinfinf,\DDinfinf) , \ 1\mapsto \id,
z \mapsto \{ \text{dotted arrows below} \},
T \mapsto \{ \text{dashed arrows below} \},\\
&\begin{tikzcd}[ampersand replacement = \&, row sep=1.4cm,column sep =1.4cm]
  \DotC \arrow[d, blue, dashed, bend right=45, "1" description]  \arrow{r}{S}\arrow{d}{D}  \& \DotB \arrow[d, blue, dashed, bend right=45, "1" description] \arrow{r}{D} \& \DotB \arrow[d, blue, dashed, bend right=45, "1" description]  \arrow{r}{SS} \& \DotB \arrow[d, blue, dashed, bend right=45, "1" description]  \arrow{r}{D} \& \DotB \arrow[d, blue, dashed, bend right=45, "1" description]  \arrow{r}{SS} \& \cdots \arrow[d, blue, dashed, bend right=45, "\cdots" description]  \\
  \DotC \arrow{r}{S} \& \DotB \arrow{r}{D} \& \DotB \arrow[llu, dotted, bend right=10, "S" description] \arrow{r}{SS} \& \DotB \arrow[llu, dotted, bend right=10, "1" description]  \arrow{r}{D} \& \DotB \arrow[llu, dotted, bend right=10, "1" description] \arrow{r}{SS} \& \cdots \arrow[llu, dotted, bend right=10, "\cdots" description] 
\end{tikzcd}
\end{align*}

\begin{definition}\label{def:module_structure}
Endow \( \BNr(\Knot_\infty)\) with an \( \RxH \)-module structure 
\(\BNr(\Knot_\infty)\otimes \RxH \rightarrow \BNr(\Knot_\infty) \) 
by  precomposition with elements of \( \RxH \subset \Mor(\DDinf,\DDinf) \).
Similarly, we can regard \(\Khr(\Knot_\infty)\) as an \(\RzT\)-module or an \( \Rxy \)-module.
\end{definition}

We find it useful to think of these actions as products in the Fukaya category: remembering that, morally, we have \(\BNr(\Knot_\infty)=\HF(\Linf,\BNr (T))\) and \(\Khr(\Knot_\infty)=\HF(\Linfinf,\BNr (T))\), the product structure in the Fukaya category 
\begin{align*}
&\HF(\Linf,\BNr (T))\otimes \HF(\Linf,\Linf) \rightarrow \HF(\Linf,\BNr (T))  \\
&\HF(\Linfinf,\BNr (T))\otimes \HF(\Linfinf,\Linfinf) \rightarrow \HF(\Linfinf,\BNr (T))
\end{align*}
explains the action
\(
\BNr(\Knot_\infty)\otimes \RxH \rightarrow \BNr(\Knot_\infty)\) and the action 
\(\Khr(\Knot_\infty)\otimes \RzT \rightarrow \Khr(\Knot_\infty).
\)

\begin{theorem}\label{thm:inf_twisted_invariants_well_defined}
The modules \( \BNr(\Knot_\infty)_{\RxH} \) and \( \Khr(\Knot_\infty)_{\RzT} \) are well-defined invariants of \( (2,-\infty) \)-twisted knots. 
\end{theorem}
\begin{proof}
We address the issues of Definition~\ref{def:module_structure} one by one.

\textbf{1) The absence of marking.} The definitions of \( \BNr(\Knot_\infty)\) and \( \Khr(\Knot_\infty) \) depend on \( \DD(T) \); recall that according to Definition~\ref{def:MainTangleInvariantTypeD} the latter object depends on the choice of marked point. However, since we are working over \( \fieldTwoElements \), the marked point does not affect the type~D structure \( \DD(T) \); see Observation~\ref{obs:mutation}. 

\textbf{2) Invariance as vector spaces.} The homotopy equivalence class of \( \DD(T) \) does not change with Reidemeister moves of the diagram of \( T \), and thus the invariants \( \BNr(\Knot_\infty) \) and \( \BNr(\Knot_\infty) \) are preserved by smooth isotopies of \( T \). Moreover, the invariants are also preserved by addition and removal of twists \( \CrossingL \) and \( \CrossingR \) on the top and on the bottom of \( T \). For \( \BNr(\Knot_\infty) \) this follows from 
  \[ 
  \Mor_{\fs}(\DDinf ,\DD(T)) 
  \overset{(a)}{\simeq} 
  \Mor_{\fs}(\DDinf \bt {}_{\BNAlgH}{\tau^{\pm 1}}^{\BNAlgH},\DD(T) \bt {}_{\BNAlgH}{\tau^{\pm 1}}^{\BNAlgH}) 
  \overset{(b)}{\simeq} 
  \Mor_{\fs}(\DDinf ,\DD(T) \bt {}_{\BNAlgH}{\tau^{\pm 1}}^{\BNAlgH}) 
  \]
  where \( _{\BNAlgH}\tau^{\BNAlgH} \) is the bimodule corresponding to the bottom half Dehn twist, which is agrees with the bimodule in Figure~\ref{fig:BNTwistingAD} but with \( \DotB \) and \( \DotC \) interchanged. The chain homotopy (a) follows from \cite[Lemma~2.3.3, Lemma 2.4.9]{LOT-bim};
  (b) follows from \( \DDinf\bt {}_{\BNAlgH}{\tau}^{\BNAlgH} = \DDinf \), which implies \( \DDinf \simeq \DDinf \bt {}_{\BNAlgH}{\tau^{-1}}^{\BNAlgH}  \). 
  Note that the \(\bt\)-tensor products are well-defined since the type~AD bimodule \( _{\BNAlgH}\tau^{\BNAlgH} \) is left operationally bounded in the sense of \cite{LOT-bim}. 
  
  Analogous statements hold for \( \DDinfinf \), and thus the vector space \( \Khr(\Knot_\infty) \) is also well-defined.

\textbf{3) The actions are well-defined algebraically.} Notice that the elements of \( \RxH\), \(\Rxy\), and \(\RzT \) are cycles in the relevant morphism spaces, and thus descend to \(\Homology(\Mor(\DDinf,\DDinf))\) and \(\Homology(\Mor(\DDinfinf,\DDinfinf))\). This implies that the actions in Definition~\ref{def:module_structure} are well-defined not only on the level of chain complexes \( \CBNr(\Knot_\infty), \ \CKhr(\Knot_\infty)  \), but also on the level of homologies \( \BNr(\Knot_\infty), \ \Khr(\Knot_\infty) \).

\textbf{4) Invariance of the module structure.} We also need to prove that the module structure \( \BNr(\Knot_\infty)\otimes \RxH \rightarrow \BNr(\Knot_\infty)  \) is invariant with respect to adding twists, ie that the maps
  \begin{align*} 
  &\Homology(\Mor_{\fs}(\DDinf,\DD(T)))\otimes \RxH \rightarrow \Homology(\Mor_{\fs}(\DDinf,\DD(T))) \text{ and}\\ 
  &\Homology(\Mor_{\fs}(\DDinf,\DD(T)\boxtimes \tau)) \otimes \RxH \rightarrow \Homology(\Mor_{\fs}(\DDinf,\DD(T) \boxtimes \tau ))
  \end{align*} 
  coincide after the identification
  \[
  \Homology(\Mor_{\fs}(\DDinf,\DD(T))) \cong \Homology(\Mor_{\fs}(\DDinf,\DD(T)\bt \tau))
  \]
  from Step 2. Consider the following commutative diagram:
  $$
  \begin{tikzcd}[column sep=0.5cm]
  \Homology(\Mor(\DDinf,\DD(T)))\otimes \Homology(\Mor(\DDinf,\DDinf)) \arrow{r} \arrow{d}{\cong}& 
  \Homology(\Mor(\DDinf,\DD(T))) \arrow{d}{\cong}  \\
  \Homology(\Mor(\DDinf \bt \tau,\DD(T) \bt \tau))\otimes \Homology(\Mor(\DDinf \bt \tau,\DDinf \bt \tau)) \arrow{r} \arrow[equal]{d} & 
  \Homology(\Mor(\DDinf\bt \tau ,\DD(T)\bt \tau)) \arrow[equal]{d}   \\
  \Homology(\Mor(\DDinf,\DD(T) \bt \tau))\otimes \Homology(\Mor(\DDinf,\DDinf)) \arrow{r} & 
  \Homology(\Mor(\DDinf,\DD(T)\bt \tau))
  \end{tikzcd}
  $$
  The horizontal maps represent compositions of morphisms. 
  The top two vertical maps are induced by the quasi-auto-equivalence $-\bt {}_{\BNAlgH}\tau^{\BNAlgH}$ of the category $\Mod_{\infty}^{\BNAlgH}$. 
  The bottom two equalities are due to $\DDinf \bt \tau = \DDinf$. The right vertical arrow induces the identification $\Homology(\Mor_{\fs}(\DDinf,\DD(T))) \cong \Homology(\Mor_{\fs}(\DDinf,\DD(T)\bt \tau))$ from Step 2. The actions $- \otimes \xben$ and $- \otimes H$ from the top row become actions $- \otimes (\xben \boxtimes \id)$ and $- \otimes (H \boxtimes \id )$ in the second row.  Hence it suffices to show that these actions coincide, namely 
  \( \xben \bt \id  \) and \( \xben \) act identically on \( \DDinf\bt {\tau} = \DDinf \), and the same for actions \( H \bt \id  \) and \( H \). This is a straightforward computation using the formula for \( \xben \bt \id  \) and \( H \bt \id  \), which can be found in the discussion preceding \cite[Lemma~2.3.3]{LOT-bim}. (Note in particular that \( \xben \bt \id  \) and \( H \bt \id  \) are well-defined since the type~AD bimodule \( _{\BNAlgH}\tau^{\BNAlgH} \) is left operationally bounded.)

  The invariance of the module structure \( \Khr(\Knot_\infty)\otimes \RzT \rightarrow \Khr(\Knot_\infty) \) with respect to adding twists is proved analogously. 
\end{proof}
The fact that the infinite part \( T_{-\infty} \) can freely produce/absorb twists complicates the definition of the absolute bigrading on \( \BNr(\Knot_\infty) \) and \( \Khr(\Knot_\infty) \).  It is still possible to pin down an absolute bigrading on our invariants, by adding a correction term $q^{-(n_+ -n_-)}\delta^{-\frac{1}{2}(n_+ -n_-)}$ in front of $\Mor_{\fs}(\DDinf,\DD(T))=\BNr(\Knot_\infty)$, where $n_+,n_-$ are positive and  negative crossing between \emph{arc} components of tangle $T$. For simplicity we choose to work with the relative bigrading.

\begin{example}\label{ex:Uinf_invariants}
In the simple case of the infinite torus knot \( T(2,-\infty)=\Lk(T_{-\infty},\Ni) \), our invariants turn out to be equal to the algebras considered as modules:
\begin{equation*}
\BNr(T(2,-\infty))_\RxH= \RxH_\RxH \qquad \Khr(T(2,-\infty))_\Rxy = \Rxy_\Rxy
\end{equation*}
The corresponding calculation is illustrated in Figure~\ref{fig:Uinf_invariants}. Let us explain how to do this calculation
‘‘locally''. Take, for example, the bottom left generator $\DotC$ of $\DDinfinf$ in Figure~\ref{fig:Benheddi_comp}; near that generator the complex $\DDinfinf$ is $[\cdots\xrightarrow{D}\DotC\xrightarrow{S}\cdots]$. The morphisms from $\DotC$ to $\DD(\Ni)=\DotC$ form a vector space $\langle \ldots,D^2,D,1,S^2,S^4,\ldots \rangle$. On the one hand, due to the arrow $\xrightarrow{D}\DotC$, only the generators $S^2,S^4,\ldots$ are in the kernel of the differential in
$\Mor(\DDinfinf,\DotC)$. 
On the other hand, due to the arrow $\DotC\xrightarrow{S}$, the generators $S^2,S^4,\ldots$ are in the image of the differential. Thus, no morphisms from the bottom left generator $\DotC$ of $\DDinfinf$ in Figure~\ref{fig:Benheddi_comp} survive in homology. Next, one can take another generator of $\DDinfinf$, and argue as before; following this strategy completes the calculation in Figure~\ref{fig:Uinf_invariants}. Note that this example recovers a result of Benheddi \cite{Benheddi}. We can also view these computations geometrically; see Figure~\ref{fig:Uinf_invariants_geometrically}.
\end{example}

\begin{figure}[t]
  \centering
  \begin{subfigure}[t]{0.45\textwidth}
  \begin{tikzcd}[ampersand replacement = \&,column sep=10pt, row sep=50pt]
  \DotC \arrow{r}{S} \arrow[drr, purple,bend left=10, "1" near start, "1" description] \arrow[drr, bend right=40, purple, "H" near start, "D" description] \arrow[drr, bend right=80, purple, "\cdots" left, "\cdots" description]   \& \DotB \arrow{r}{D} \& \DotB \arrow[d,  purple, "\xben" near start, "S" description]  \arrow{r}{SS} \& \DotB \arrow{r}{D} \& \DotB \arrow[lld, purple, "\xben^2" near start,"S" description]  \arrow{r}{SS} \& \cdots \& \arrow[lllld, purple, bend left, "\cdots" near start,"\cdots" description] \\
  \& \& \DotC
  \end{tikzcd}
  \caption{\( \Homology \left( \Mor_{\fs}(\DDinf,\DD(\Li)) \right)=\RxH_\RxH \) }
  \end{subfigure}
  \begin{subfigure}[t]{0.45\textwidth}
  \begin{tikzcd}[ampersand replacement = \&, column sep=10pt, row sep=25pt]
  \DotC  \arrow{r}{S}\arrow{d}{D} \arrow[dddrrr, bend right, purple, "1" near end, "1" description]
  \& \DotB \arrow{r}{D} 
  \& \DotB  \arrow{r}{SS}  \arrow[dddr, purple, "\xben" very near start, "S" description]
  \& \DotB  \arrow{r}{D} 
  \& \DotB  \arrow{r}{SS} \arrow[dddl, purple, "\xben^2" very near start, "S" description] 
  \& \cdots 
  \&  \arrow[dddlll, purple, "\xben^3" very near start, "S" description]  \\
  \DotC \arrow{r}{S} 
  \& \DotB \arrow{r}{D} 
  \& \DotB \arrow{r}{SS} \arrow[ddr, bend right, purple, "\yben" very near start, "S" description] 
  \& \DotB \arrow{r}{D} 
  \& \DotB  \arrow{r}{SS} \arrow[ddl, purple, "\yben \xben" very near start, "S" description] 
  \& \cdots  
  \&  \arrow[ddlll, bend left, purple, "\yben \xben^2" very near start, "S" description]  \\ \\
  \& \& \& \DotC
  \end{tikzcd}
  \caption{\( \Homology \left( \Mor_{\fs}(\DDinfinf,\DD(\Li)) \right)=\Rxy_\Rxy \)}
  \label{fig:Benheddi_comp}
  \end{subfigure}
  \caption{}
  \label{fig:Uinf_invariants}
\end{figure}
\begin{figure}[t]
  \centering
  \begin{subfigure}[t]{0.45\textwidth}
  \includegraphics[scale=0.6]{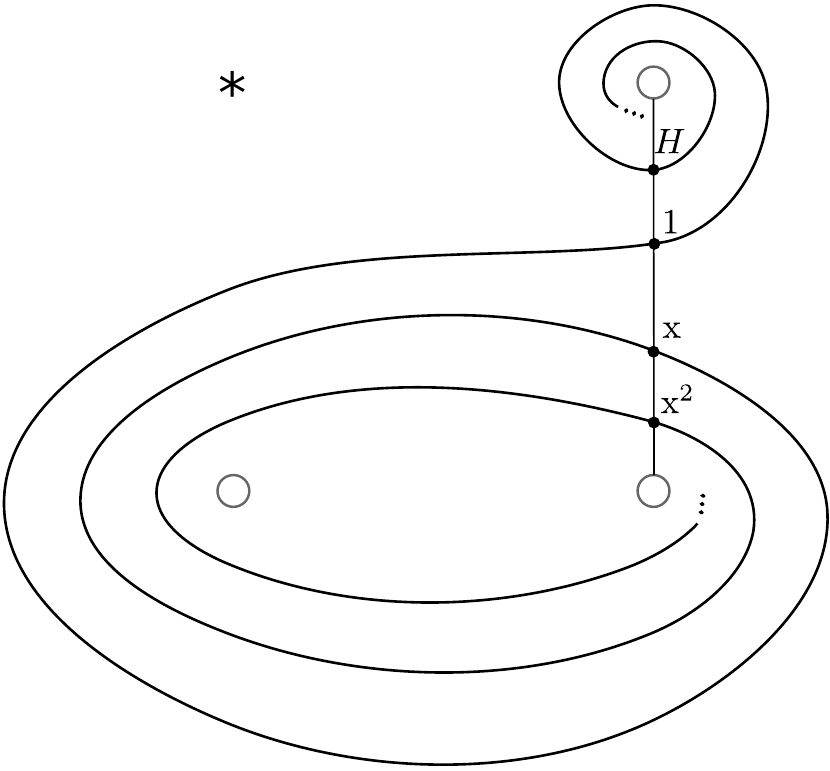}
  \caption{ \( \HF(\Linf,\BNr (\Li))=\RxH \)}
  \label{fig:Uinf_arc}
  \end{subfigure}
  \begin{subfigure}[t]{0.45\textwidth}
  \includegraphics[scale=0.6]{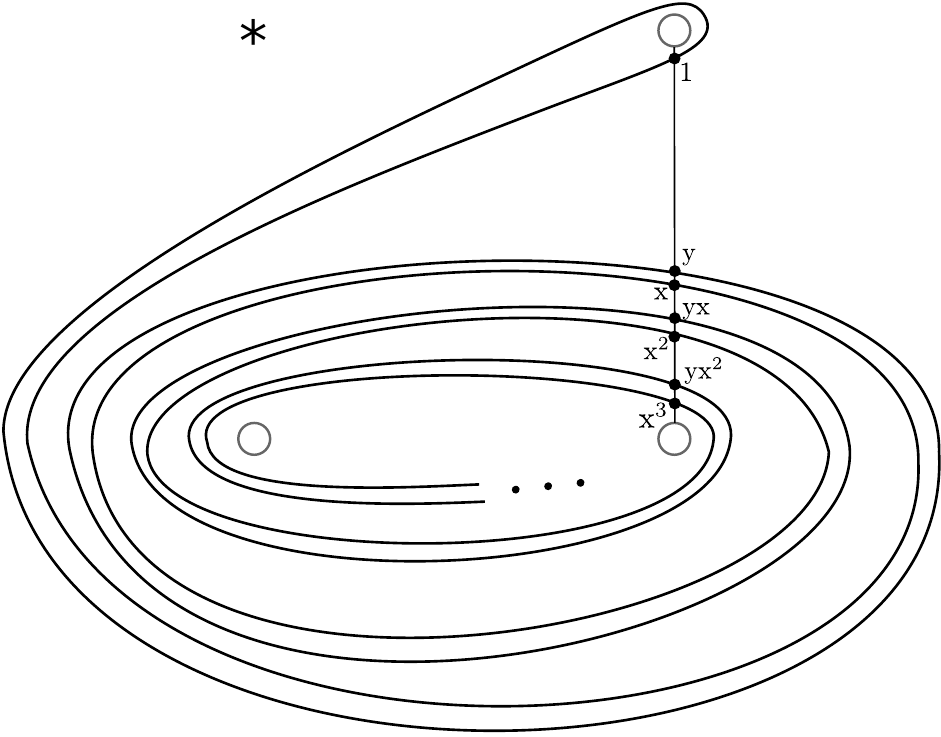}
  \label{fig:Uinf_eight}
  \caption{\( \HF(\Linfinf,\BNr (\Li))=\Rxy \)}
  \end{subfigure}
  \caption{}
  \label{fig:Uinf_invariants_geometrically}
\end{figure}

We will explain how to compute \( \BNr(\Knot_{\infty})_\RxH \) for arbitrary \( \Knot_{\infty} \) in Theorem~\ref{thm:lifting_to_D_str}.

\subsection{A remark on sutured tangles}\label{sub:rem-sut-tang}
There is an alternate definition of the invariants \( \BNr(\Knot_\infty)\) and \(\Khr(\Knot_\infty) \): consider the direct system of complexes \[  \cdots \rightarrow \CBNr(\Lk(T_{-n},T)) \rightarrow \CBNr(\Lk(T_{-(n+1)},T)) \rightarrow \CBNr(\Lk(T_{-(n+2)},T)) \rightarrow \cdots\] where each arrow corresponds to the map from the long exact sequence associated with the resolution of the bottom crossing of the tangles \( T_{-n} \). Using this direct system, one can define \( \BNr(\Lk(T_{-\infty},T)) \) and, analogously, \( \Khr(\Lk(T_{-\infty},T)) \) as a direct limit, and the resulting invariants will be isomorphic to the ones we defined.

This point of view appears in another setting: in \cite{Kappa}, the second author considers {\em sutured} tangles arising as the quotients of (equivariantly meridional sutured) strongly invertible knot exteriors. These objects are equivalent to the $\Knot_{\infty}$ described above, but come with the additional restriction that, in our conventions, the  $\Lk(\Lo,T)$-closure of the tangle is the unknot. The utility of this setup is that it allows any invariant of the sutured tangle (or suitably restricted class of \( (2,-\infty) \)-twisted knots) to be used as an invariant of the strongly invertible knot $(\Knot,h)$ in the cover. To fix notation, given a strongly invertible knot $(\Knot,h)$, we let $T_{\Knot,h}$ be the associated quotient tangle; note that the complement of $\Knot$ is the double branched cover of $T_{\Knot,h}$, by construction. 

The invariant $\varkappa(\Knot,h)$ is a finite dimensional graded vector space that is defined as the quotient of an inverse limit $\underset{\raisebox{3pt}{$\longleftarrow$}}{\operatorname{Kh}}(T_{\Knot,h})=\varprojlim \widetilde{\operatorname{Kh}}(T_{\Knot,h}(i))$ \cite[Definition 10]{Kappa}; this uses the system above, interpreted as an inverse system. The homological grading endows $\underset{\raisebox{3pt}{$\longleftarrow$}}{\operatorname{Kh}}(T_{\Knot,h})$ with an absolute integer grading and, in fact, $\underset{\raisebox{3pt}{$\longleftarrow$}}{\operatorname{Kh}}(mT_{\Knot,h})$ and the analogous {\em direct} limit $\underset{\raisebox{3pt}{$\longrightarrow$}}{\operatorname{Kh}}(T_{\Knot,h})$ agree as graded vector spaces; compare \cite[Section 3]{Kappa}. As a result, writing $\varkappa(T_{\Knot,h})=\varkappa(\Knot,h)$, we can extract $\varkappa(T_{\Knot,h})$ from $\underset{\raisebox{3pt}{$\longrightarrow$}}{\operatorname{Kh}}(mT_{\Knot,h})$ and consider $\Khr(\Knot_\infty)$ for $\Knot_\infty=\Lk(T_{-\infty},mT_{\Knot,h})$.

Another and much simpler means of extracting a finite dimensional graded vector space is to consider $\Khr(\Knot_\infty)$ as an \( \Rxy \)-module, and define $\ker(\xben^N)$ for sufficiently large $N$. This has the obvious additional advantage that it carries the structure of an \( \Rxy \)-module. Given $\Khr(\Knot_\infty)$  with $\Knot_\infty=\Lk(T_{-\infty},mT_{\Knot,h})$, the torsion submodule $\ker(\xben^N)$ corresponds to $\varkappa(T_{\Knot,h})$, as vector spaces. In particular, the module $\ker(\xben^N)\subset\Khr(T_{\Knot,h})$ is an interesting invariant of strong inversions. (In Example \ref{exa:Pairing:TorusKnots}, it is calculated for the left-hand trefoil, which arises as the double branched cover of the tangle $T_{2,-3}$.) For instance, this module vanishes if and only if the strongly invertible knot is the trivial knot; compare \cite[Theorem 1]{Kappa} and Example \ref{ex:Uinf_invariants}.

It seems surprising that, in practice, no more information is gained by working with the (apparently) more general $\RzT$-modules. We end with:

\begin{question} Is the $\RzT$-module structure on $\Khr(\Knot_\infty)$ determined by the \( \Rxy \)-module structure? \end {question}

\subsection{\texorpdfstring{Lifting \( \BNr(\Knot_{\infty})_\RxH \) to a type~D structure}{Lifting BN(K∞) to a type~D structure}}

In Section~\ref{subsec:BN-deformation} we observed  that \( \CBN(\Lk)_{\fieldTwoElements[H]} \) lifts to a type~D structure so that \( \CBN(\Lk)_{\fieldTwoElements[H]} \simeq \CKh(\Lk)^{\fieldTwoElements[H]} \bt {\fieldTwoElements[H]} \). Such a statement is true for other invariants in low-dimensional topology, for example, the minus version of Heegaard Floer homology \( \mathit{HF}^{-}(Y)_{\fieldTwoElements[U]} \). 
In this section we prove that the module \( \BNr(\Knot_{\infty})_\RxH \) also naturally lifts to a type~D structure and,
along the way, we will obtain a recipe for computing \( \BNr(\Knot_{\infty})_\RxH \). 

We begin with a geometric observation: having in mind the interpretation \( \BNr(\Knot_\infty)=\HF(\Linf,\BNr (T)) \), we can deform \( \BNr (T) \) into the neighborhood of two arcs \( \DotBarc \) and \( \DotCarc \), and also deform \( \Linf \) as in Figure~\ref{fig:Uinf_arc}.
Notice that \( \Linf \) only intersects the vertical arcs \( \DotCarc \) of \( \BNr (T) \). This suggests that only the generators \( \DotC \) of \( \DD(T) \) affect \( \BNr(\Knot_\infty) \). Indeed, this is what we show below.

\begin{theorem}\label{thm:lifting_to_D_str}
Suppose \( \Knot_{\infty}=\Lk(T_{-\infty},T) \). Then
\begin{equation*}
\label{eq:lifting_to_D_str}
\BNr(\Knot_{\infty})_\RxH \simeq \DD(T)^{\BNAlgH} \bt {}_{\BNAlgH}\bimoduleG^\RxH \bt \RxH
\end{equation*}
where the type AD bimodule \( {}_{\BNAlgH}\bimoduleG^\RxH \) is equal to 
\[
\begin{tikzcd}
    \DotC\gen
    \arrow[in=-45,out=45,looseness=4]{rl}{\substack{(S,D,S | \xben) \\ (S,D,SS,D,S | \xben^2) \\ (S,D,SS,D,SS,D,S|\xben^3) \\ \cdots}}
    \arrow[rl, in=-135,out=135,looseness=4,"(D^k  | H^k)" left]
\end{tikzcd}
 \]
\end{theorem}

\begin{figure}[t]
  \[ 
  \begin{tikzcd}[scale=0.5,column sep=2.2cm, ,row sep=1.5cm]
  \boxed{\DotC\gen \bt H }
  \arrow[r, "\substack{(-|H)\\(D|-)}" description] 
  \arrow[rr, bend left = 16, "\substack{(-|H^2)\\(D^2|-)}" description]
  \arrow[rrr, bend left = 22, "\ldots" description]
  & \boxed{\DotC\gen \bt H^2  }
  \arrow[r, "\substack{(-|H)\\(D|-)}" description]
  \arrow[rr, bend left = 12, "\ldots" description]
  & \boxed{\DotC\gen \bt H^3 }
  \arrow[r, "\cdots" description] 
  & \cdots \\
  \boxed{\DotC\gen \bt 1}
  \arrow[u, "\substack{(-|H)\\(D|-)}" description]
  \arrow[ru, "\substack{(-|H^2)\\(D^2|-)}" description]
  \arrow[rru, pos=0.7, bend right = 6, "\substack{(-|H^3)\\(D^3|-)}" description]
  \arrow[rrru, pos=0.7, bend right = 8, "\cdots" description]
  \arrow[d, "\substack{(-|\xben) \\ (S,D,S|-)}" description]
  \arrow[rd, "\substack{(-|\xben^2) \\ (S,D,SS,D,S|-)}" description]
  \arrow[rrd, pos=0.7, bend left = 6, "\substack{(-|\xben^3) \\ (S,D,SS,D,SS,D,S|-)}" description]
  \arrow[rrrd, pos=0.7, bend left = 8, "\cdots" description]
  \\
  \boxed{\DotC\gen \bt \xben  }
  \arrow[r, "\substack{(-|\xben) \\ (S,D,S|-)}" description] 
  \arrow[rr, bend right = 14, "\substack{(-|\xben^2) \\ (S,D,SS,D,S|-)}" description]
  \arrow[rrr, bend right = 22, "\ldots" description]
  & \boxed{\DotC\gen \bt \xben^2  }
  \arrow[r, "\substack{(-|\xben) \\ (S,D,S|-)}" description]
  \arrow[rr, bend right = 12, "\ldots" description]
  & \boxed{\DotC\gen \bt \xben^3}
  \arrow[r, "\cdots" description] 
  & \cdots
  \end{tikzcd}
  \]
  \caption[hm]{\( {}_{\BNAlgH}{\bimoduleG}^\RxH \bt {}_{\RxH}\RxH_{\RxH} \)}\label{fig:left_bimodule}
\end{figure} 

\begin{figure}[t]
  \[ 
  \begin{tikzcd}[column sep=0.5cm,row sep=2cm]
  \vdots& \vdots & \vdots & \vdots & \vdots & \vdots & \vdots\\
  \boxed{D^2\bt \DotC\gen}
  \arrow[u, "\cdots" description] 
  & 
  D^2 \bt \DotB\gen 
  \arrow[u, "\cdots" description]
  \arrow[rr, gray,bend right,very near start, sloped, "(-|\xben)" description]
  \arrow[rrrr, gray,bend right, sloped, "(-|\xben^2)" description]
  & 
  D^2 \bt \DotB\gen
  \arrow[ul, blue,dashed, very near end,"1" description]
  \arrow[u, "\cdots" description]
  \arrow[rr, gray,bend right,very near start, sloped, "(-|\xben)" description]
  & 
  D^2 \bt \DotB\gen
  \arrow[u, "\cdots" description]
  \arrow[rr, gray,bend right,very near start, sloped, "(-|\xben)" description]
  &  
  D^2 \bt \DotB\gen
  \arrow[ul, blue,dashed, very near end,"1" description]
  \arrow[u, "\cdots" description]
  & 
  D^2 \bt \DotB\gen     
  \arrow[u, "\cdots" description]
  & 
  \cdots 
  \arrow[ul, blue,dashed, very near end,"1" description]
  \\
  \boxed{D^1 \bt \DotC\gen} 
  \arrow[u, "\substack{(-|H)\\(D|-)}" description]
  & 
  D^1 \bt \DotB\gen 
  \arrow[u, "(D|-)" description]
  \arrow[rr, gray,bend right,very near start, sloped, "(-|\xben)" description]
  \arrow[rrrr, gray,bend right, sloped, "(-|\xben^2)" description]
  & 
  D^1 \bt \DotB\gen
  \arrow[ul, blue,dashed, very near end,"1" description]
  \arrow[u, "(D|-)" description]
  \arrow[rr, gray,bend right,very near start, sloped, "(-|\xben)" description]
  & 
  D^1 \bt \DotB\gen
  \arrow[u, "(D|-)" description]
  \arrow[rr, gray,bend right,very near start, sloped, "(-|\xben)" description]
  &  
  D^1 \bt \DotB\gen
  \arrow[ul, blue,dashed, very near end,"1" description]
  \arrow[u, "(D|-)" description]
  & 
  D^1 \bt \DotB\gen 
  \arrow[u, "(D|-)" description]
  & 
  \cdots
  \arrow[ul, blue,dashed, very near end,"1" description]
  \\
  \boxed{1 \bt \DotC\gen} 
  \arrow[u, "\substack{(-|H)\\(D|-)}" description]
  \arrow[uu, bend left = 40,near end, "\substack{(-|H^2)\\(D^2|-)}" description]
  \arrow[d, "(S|-)" description]
  \arrow[dd, bend right = 40,near end, "(S^2|-)" description]
  \arrow[drr, gray,bend right,very near start, sloped, "(-|\xben)" description]
  \arrow[drrrr, gray,very near start, sloped, "(-|\xben^2)" description]
  & 
  1 \bt \DotB\gen 
  \arrow[dl, blue,dashed, very near end,"1" description]
  \arrow[u, "(D|-)" description]
  \arrow[uu, bend left = 40,near end, "(D^2|-)" description]
  \arrow[d, "(S|-)" description]
  \arrow[dd, bend right = 40,near end, "(S^2|-)" description]
  \arrow[rr, gray,bend right,very near start, sloped, "(-|\xben)" description]
  \arrow[rrrr, gray,bend right, sloped, "(-|\xben^2)" description]
  & 
  1 \bt \DotB\gen
  \arrow[ul, blue,dashed, very near end,"1" description]
  \arrow[u, "(D|-)" description]
  \arrow[uu, bend left = 40,near end, "(D^2|-)" description]
  \arrow[d, "(S|-)" description]
  \arrow[dd, bend right = 40,near end, "(S^2|-)" description]
  \arrow[rr, gray,bend right,very near start, sloped, "(-|\xben)" description]
  & 
  1 \bt \DotB\gen
  \arrow[ddl, blue,dashed, very near end,"1" description]
  \arrow[u, "(D|-)" description]
  \arrow[uu, bend left = 40,near end, "(D^2|-)" description]
  \arrow[d, "(S|-)" description]
  \arrow[dd, bend right = 40,near end, "(S^2|-)" description]
  \arrow[rr, gray,bend right,very near start, sloped, "(-|\xben)" description]
  &  
  1 \bt \DotB\gen
  \arrow[ul, blue,dashed, very near end,"1" description]
  \arrow[u, "(D|-)" description]
  \arrow[uu, bend left = 40,near end, "(D^2|-)" description]
  \arrow[d, "(S|-)" description]
  \arrow[dd, bend right = 40,near end, "(S^2|-)" description]
  & 
  1 \bt \DotB\gen 
  \arrow[ddl, blue,dashed, very near end,"1" description]
  \arrow[u, "(D|-)" description]
  \arrow[uu, bend left = 40,near end, "(D^2|-)" description]
  \arrow[d, "(S|-)" description]
  \arrow[dd, bend right = 40,near end, "(S^2|-)" description]
  & 
  \cdots
  \arrow[ul, blue,dashed, very near end,"1" description]
  \\
  S^1 \bt \DotC\gen 
  \arrow[d, "(S|-)" description]
  \arrow[drr, gray,bend right,very near start, sloped, "(-|\xben)" description]
  \arrow[drrrr, gray,very near start, sloped, "(-|\xben^2)" description]
  & 
  S^1 \bt \DotB\gen 
  \arrow[dl, blue,dashed, very near end,"1" description]
  \arrow[d, "(S|-)" description]
  \arrow[rr, gray,bend right,very near start, sloped, "(-|\xben)" description]
  \arrow[rrrr, gray,bend right, sloped, "(-|\xben^2)" description]
  & 
  \boxed{S^1 \bt \DotB\gen}
  \arrow[d, "(S|-)" description]
  \arrow[rr, gray,bend right,very near start, sloped, "(-|\xben)" description]
  & 
  S^1 \bt \DotB\gen
  \arrow[ddl, blue,dashed, very near end,"1" description]
  \arrow[d, "(S|-)" description]
  \arrow[rr, gray,bend right,very near start, sloped, "(-|\xben)" description]
  &  
  \boxed{S^1 \bt \DotB\gen}
  \arrow[d, "(S|-)" description]
  & 
  S^1 \bt \DotB\gen 
  \arrow[ddl, blue,dashed, very near end,"1" description]
  \arrow[d, "(S|-)" description]
  & 
  \cdots
  \\
  S^2 \bt \DotC\gen 
  \arrow[d, "\cdots" description]
  \arrow[drr, gray,bend right,very near start, sloped, "(-|\xben)" description]
  \arrow[drrrr, gray,very near start, sloped, "(-|\xben^2)" description]
  & 
  S^2 \bt \DotB\gen 
  \arrow[dl, blue,dashed, very near end,"1" description]
  \arrow[d, "\cdots" description]
  \arrow[rr, gray,bend right,very near start, sloped, "(-|\xben)" description]
  \arrow[rrrr, gray,bend right, sloped, "(-|\xben^2)" description]
  & 
  S^2 \bt \DotB\gen
  \arrow[d, "\cdots" description]
  \arrow[rr, gray,bend right,very near start, sloped, "(-|\xben)" description]
  & 
  S^2 \bt \DotB\gen
  \arrow[d, "\cdots" description]
  \arrow[rr, gray,bend right,very near start, sloped, "(-|\xben)" description]
  &  
  S^2 \bt \DotB\gen
  \arrow[d, "\cdots" description]
  & 
  S^2 \bt \DotB\gen 
  \arrow[d, "\cdots" description]
  & 
  \cdots
  \\
  \vdots& \vdots & \vdots & \vdots & \vdots & \vdots & \vdots
  \end{tikzcd}
  \]
  \caption{\( {}_{\BNAlgH}\BNAlgH_{\BNAlgH} \bt {}^{\BNAlgH}\dual{\DDinf}_{\RxH} \)}\label{fig:right_bimodule}
\end{figure} 

\begin{proof}
  Let us rephrase the module structure from Definition~\ref{def:module_structure}; for this, we first promote the type~D structure \( \DDinf^{\BNAlgH} \) to a type AD bimodule:
  \[ 
  _\RxH \DDinf^{\BNAlgH} =
  \begin{tikzcd}[column sep=1.5cm]
  \gen \DotC \arrow{r}{S} \arrow[lr, in=+135,out=+45,looseness=9 , "(H^k | D^k)" description]  
  & 
  \gen \DotB \arrow{r}{(- | D))} 
  & 
  \gen \DotB \arrow[ll, bend right, "(\xben | S)" description]  \arrow{r}{(- | SS)}
  & 
  \gen \DotB \arrow[ll, bend right, "(\xben | 1)" description]  \arrow{r}{(- | D)} 
  & 
  \gen \DotB \arrow[ll, bend right, "(\xben | 1)" description] \arrow[llll, bend right, "(\xben^2 | S)" description]   \arrow{r}{(- | SS)} 
  & 
  \cdots \arrow[ll, bend right, "\cdots" description] \arrow[llll, bend right, "(\xben^2 | 1)" description] 
  \end{tikzcd}
  \]
  Now we can view the module structure as follows:
  \[ \BNr(\Knot_\infty)_{\RxH} = \Homology ( \Mor_{\fs}(_{\RxH}\DDinf^{\BNAlgH},\DD(T)^{\BNAlgH} )) \simeq \DD(T)^{\BNAlgH} \bt {}_{\BNAlgH}\BNAlgH_{\BNAlgH} \bt {}^{\BNAlgH}\dual{\DDinf}_{\RxH}\]
  where the second homotopy equivalence follows from \cite[Proposition~2.7]{LOT-mor} and the fact that we consider compactly supported elements of the morphism space. Thus, for the proof of the theorem, it remains to establish the following homotopy equivalence of bimodules:
  \[ {}_{\BNAlgH}{\bimoduleG}^\RxH \bt {}_{\RxH}\RxH_{\RxH} \simeq {}_{\BNAlgH}\BNAlgH_{\BNAlgH} \bt {}^{\BNAlgH}\dual{\DDinf}_{\RxH}\]
  We describe these bimodules explicitly; see Figure~\ref{fig:left_bimodule} for \( {}_{\BNAlgH}{\bimoduleG}^\RxH \bt {}_{\RxH}\RxH_{\RxH} \) and Figure~\ref{fig:right_bimodule} for \( {}_{\BNAlgH}\BNAlgH_{\BNAlgH} \bt {}^{\BNAlgH}\dual{\DDinf}_{\RxH} \). We now apply \cite[Lemma 12.3]{OS-bordered-matchings} to the bimodule \( {}_{\BNAlgH}\BNAlgH_{\BNAlgH} \bt {}^{\BNAlgH}\dual{\DDinf}_{\RxH} \), which we denote by \( {}_{\BNAlgH}Y_{\RxH} \) for short. Namely, we notice that only the boxed generators in Figure~\ref{fig:right_bimodule} survive in homology, and denote the vector space generated by them by \( Z=\Homology (Y) \). Denote the map corresponding to the reversed dashed arrows in Figure~\ref{fig:right_bimodule} by \( T\co Y\rightarrow Y \). We also have an inclusion \( f\co Z\rightarrow Y \), and a projection  \( g\co Y\rightarrow Z \). One can now check that all the conditions of \cite[Lemma 12.3]{OS-bordered-matchings} are satisfied:
  \[ g \circ f=\operatorname{Id}_{Z} \quad \operatorname{Id}_{Y}+\partial T+T \partial=f \circ g \quad T \circ T=0 \quad f \circ T=0 \quad T \circ g=0\]
  Thus we obtain a bimodule 
  \( {}_{\BNAlgH}Z_{\RxH} \simeq {}_{\BNAlgH}Y_{\RxH}={}_{\BNAlgH}\BNAlgH_{\BNAlgH} \bt {}^{\BNAlgH}\dual{\DDinf}_{\RxH}. \) Moreover, by \cite[Lemma 12.3]{OS-bordered-matchings} all actions in \( {}_{\BNAlgH}Z_{\RxH}  \) come from the zigzags
  \[ 
  \begin{tikzcd}
  \boxed{*} \arrow[rd]&         & * \arrow[ld, blue, dashed]\arrow[rd] &         & \cdots \arrow[ld, blue, dashed]\arrow[rd]  &        & *\arrow[ld, blue, dashed]\arrow[rd]  \\
                 & * &         & * &       &* &       & \boxed{*}
  \end{tikzcd}
  \]
  in Figure~\ref{fig:right_bimodule}. A routine search through this diagram shows that the only zigzags involving the dashed arrows are the ones depicted in Figure~\ref{fig:zigzags}.

  \begin{figure}[t]
  \centering
  \begin{align*}
    &
    \begin{tikzpicture}[scale=0.4]
    \draw [fill] (1,1) circle [radius=0.07];
    \draw [fill] (1,2) circle [radius=0.07];
    \draw [fill] (1,3) circle [radius=0.07];
    \draw [fill] (1,4) circle [radius=0.07];
    \draw [fill] (1,5) circle [radius=0.07];
    \draw [fill] (2,1) circle [radius=0.07];
    \draw [fill] (2,2) circle [radius=0.07];
    \draw [fill] (2,3) circle [radius=0.07];
    \draw [fill] (2,4) circle [radius=0.07];
    \draw [fill] (2,5) circle [radius=0.07];
    \draw [fill] (3,1) circle [radius=0.07];
    \draw [fill] (3,2) circle [radius=0.07];
    \draw [fill] (3,3) circle [radius=0.07];
    \draw [fill] (3,4) circle [radius=0.07];
    \draw [fill] (3,5) circle [radius=0.07];
    \draw [fill] (4,1) circle [radius=0.07];
    \draw [fill] (4,2) circle [radius=0.07];
    \draw [fill] (4,3) circle [radius=0.07];
    \draw [fill] (4,4) circle [radius=0.07];
    \draw [fill] (4,5) circle [radius=0.07];
    \draw [fill] (5,1) circle [radius=0.07];
    \draw [fill] (5,2) circle [radius=0.07];
    \draw [fill] (5,3) circle [radius=0.07];
    \draw [fill] (5,4) circle [radius=0.07];
    \draw [fill] (5,5) circle [radius=0.07];
    \draw [fill] (6,1) circle [radius=0.07];
    \draw [fill] (6,2) circle [radius=0.07];
    \draw [fill] (6,3) circle [radius=0.07];
    \draw [fill] (6,4) circle [radius=0.07];
    \draw [fill] (6,5) circle [radius=0.07];
    \draw (0.8,2.8) rectangle (1.2,3.2);
    \draw (0.8,3.8) rectangle (1.2,4.2);
    \draw (0.8,4.8) rectangle (1.2,5.2);
    \draw (2.8,1.8) rectangle (3.2,2.2);
    \draw (4.8,1.8) rectangle (5.2,2.2);
    \draw[->] (1,3) -- (1,2);
    \draw[->] (2,3) -- (2,4);
    \draw[->] (3,3) -- (3,2);
    \draw[->, densely dashed,blue] (2,3) -- (1,2);
    \draw[->, densely dashed,blue] (3,3) -- (2,4);
    \draw (7, 1) -- (7,5);
    \end{tikzpicture}
    \quad 
    \begin{tikzpicture}[scale=0.4]
    \draw [fill] (1,1) circle [radius=0.07];
    \draw [fill] (1,2) circle [radius=0.07];
    \draw [fill] (1,3) circle [radius=0.07];
    \draw [fill] (1,4) circle [radius=0.07];
    \draw [fill] (1,5) circle [radius=0.07];
    \draw [fill] (2,1) circle [radius=0.07];
    \draw [fill] (2,2) circle [radius=0.07];
    \draw [fill] (2,3) circle [radius=0.07];
    \draw [fill] (2,4) circle [radius=0.07];
    \draw [fill] (2,5) circle [radius=0.07];
    \draw [fill] (3,1) circle [radius=0.07];
    \draw [fill] (3,2) circle [radius=0.07];
    \draw [fill] (3,3) circle [radius=0.07];
    \draw [fill] (3,4) circle [radius=0.07];
    \draw [fill] (3,5) circle [radius=0.07];
    \draw [fill] (4,1) circle [radius=0.07];
    \draw [fill] (4,2) circle [radius=0.07];
    \draw [fill] (4,3) circle [radius=0.07];
    \draw [fill] (4,4) circle [radius=0.07];
    \draw [fill] (4,5) circle [radius=0.07];
    \draw [fill] (5,1) circle [radius=0.07];
    \draw [fill] (5,2) circle [radius=0.07];
    \draw [fill] (5,3) circle [radius=0.07];
    \draw [fill] (5,4) circle [radius=0.07];
    \draw [fill] (5,5) circle [radius=0.07];
    \draw [fill] (6,1) circle [radius=0.07];
    \draw [fill] (6,2) circle [radius=0.07];
    \draw [fill] (6,3) circle [radius=0.07];
    \draw [fill] (6,4) circle [radius=0.07];
    \draw [fill] (6,5) circle [radius=0.07];
    \draw (0.8,2.8) rectangle (1.2,3.2);
    \draw (0.8,3.8) rectangle (1.2,4.2);
    \draw (0.8,4.8) rectangle (1.2,5.2);
    \draw (2.8,1.8) rectangle (3.2,2.2);
    \draw (4.8,1.8) rectangle (5.2,2.2);
    \draw[->] (1,3) -- (1,2);
    \draw[->] (2,3) -- (2,4);
    \draw[->] (3,3) -- (3,1);
    \draw[->] (4,3) -- (4,4);
    \draw[->] (5,3) -- (5,2);
    \draw[->, densely dashed,blue] (2,3) -- (1,2);
    \draw[->, densely dashed,blue] (4,3) -- (3,1);
    \draw[->, densely dashed,blue] (3,3) -- (2,4);
    \draw[->, densely dashed,blue] (3,3) -- (2,4);
    \draw[->, densely dashed,blue] (5,3) -- (4,4);
    \draw (7, 1) -- (7,5);
    \end{tikzpicture}
    \quad
    \begin{tikzpicture}[scale=0.4]
    \draw [fill] (1,1) circle [radius=0.07];
    \draw [fill] (1,2) circle [radius=0.07];
    \draw [fill] (1,3) circle [radius=0.07];
    \draw [fill] (1,4) circle [radius=0.07];
    \draw [fill] (1,5) circle [radius=0.07];
    \draw [fill] (2,1) circle [radius=0.07];
    \draw [fill] (2,2) circle [radius=0.07];
    \draw [fill] (2,3) circle [radius=0.07];
    \draw [fill] (2,4) circle [radius=0.07];
    \draw [fill] (2,5) circle [radius=0.07];
    \draw [fill] (3,1) circle [radius=0.07];
    \draw [fill] (3,2) circle [radius=0.07];
    \draw [fill] (3,3) circle [radius=0.07];
    \draw [fill] (3,4) circle [radius=0.07];
    \draw [fill] (3,5) circle [radius=0.07];
    \draw [fill] (4,1) circle [radius=0.07];
    \draw [fill] (4,2) circle [radius=0.07];
    \draw [fill] (4,3) circle [radius=0.07];
    \draw [fill] (4,4) circle [radius=0.07];
    \draw [fill] (4,5) circle [radius=0.07];
    \draw [fill] (5,1) circle [radius=0.07];
    \draw [fill] (5,2) circle [radius=0.07];
    \draw [fill] (5,3) circle [radius=0.07];
    \draw [fill] (5,4) circle [radius=0.07];
    \draw [fill] (5,5) circle [radius=0.07];
    \draw [fill] (6,1) circle [radius=0.07];
    \draw [fill] (6,2) circle [radius=0.07];
    \draw [fill] (6,3) circle [radius=0.07];
    \draw [fill] (6,4) circle [radius=0.07];
    \draw [fill] (6,5) circle [radius=0.07];
    \draw [fill] (7,1) circle [radius=0.07];
    \draw [fill] (7,2) circle [radius=0.07];
    \draw [fill] (7,3) circle [radius=0.07];
    \draw [fill] (7,4) circle [radius=0.07];
    \draw [fill] (7,5) circle [radius=0.07];
    \draw [fill] (8,1) circle [radius=0.07];
    \draw [fill] (8,2) circle [radius=0.07];
    \draw [fill] (8,3) circle [radius=0.07];
    \draw [fill] (8,4) circle [radius=0.07];
    \draw [fill] (8,5) circle [radius=0.07];
    \draw (0.8,2.8) rectangle (1.2,3.2);
    \draw (0.8,3.8) rectangle (1.2,4.2);
    \draw (0.8,4.8) rectangle (1.2,5.2);
    \draw (2.8,1.8) rectangle (3.2,2.2);
    \draw (4.8,1.8) rectangle (5.2,2.2);
    \draw (6.8,1.8) rectangle (7.2,2.2);
    \draw[->] (1,3) -- (1,2);
    \draw[->] (2,3) -- (2,4);
    \draw[->] (3,3) -- (3,1);
    \draw[->] (4,3) -- (4,4);
    \draw[->] (5,3) -- (5,1);
    \draw[->] (6,3) -- (6,4);
    \draw[->] (7,3) -- (7,2);
    \draw[->, densely dashed,blue] (2,3) -- (1,2);
    \draw[->, densely dashed,blue] (4,3) -- (3,1);
    \draw[->, densely dashed,blue] (6,3) -- (5,1);
    \draw[->, densely dashed,blue] (7,3) -- (6,4);
    \draw[->, densely dashed,blue] (3,3) -- (2,4);
    \draw[->, densely dashed,blue] (3,3) -- (2,4);
    \draw[->, densely dashed,blue] (5,3) -- (4,4);
    \end{tikzpicture} 
    \quad \raisebox{0.7cm}{$\cdots$}
    \\
    &
    \begin{tikzpicture}[scale=0.4]
    \draw [fill] (1,1) circle [radius=0.07];
    \draw [fill] (1,2) circle [radius=0.07];
    \draw [fill] (1,3) circle [radius=0.07];
    \draw [fill] (1,4) circle [radius=0.07];
    \draw [fill] (1,5) circle [radius=0.07];
    \draw [fill] (2,1) circle [radius=0.07];
    \draw [fill] (2,2) circle [radius=0.07];
    \draw [fill] (2,3) circle [radius=0.07];
    \draw [fill] (2,4) circle [radius=0.07];
    \draw [fill] (2,5) circle [radius=0.07];
    \draw [fill] (3,1) circle [radius=0.07];
    \draw [fill] (3,2) circle [radius=0.07];
    \draw [fill] (3,3) circle [radius=0.07];
    \draw [fill] (3,4) circle [radius=0.07];
    \draw [fill] (3,5) circle [radius=0.07];
    \draw [fill] (4,1) circle [radius=0.07];
    \draw [fill] (4,2) circle [radius=0.07];
    \draw [fill] (4,3) circle [radius=0.07];
    \draw [fill] (4,4) circle [radius=0.07];
    \draw [fill] (4,5) circle [radius=0.07];
    \draw [fill] (5,1) circle [radius=0.07];
    \draw [fill] (5,2) circle [radius=0.07];
    \draw [fill] (5,3) circle [radius=0.07];
    \draw [fill] (5,4) circle [radius=0.07];
    \draw [fill] (5,5) circle [radius=0.07];
    \draw [fill] (6,1) circle [radius=0.07];
    \draw [fill] (6,2) circle [radius=0.07];
    \draw [fill] (6,3) circle [radius=0.07];
    \draw [fill] (6,4) circle [radius=0.07];
    \draw [fill] (6,5) circle [radius=0.07];
    \draw (0.8,2.8) rectangle (1.2,3.2);
    \draw (0.8,3.8) rectangle (1.2,4.2);
    \draw (0.8,4.8) rectangle (1.2,5.2);
    \draw (2.8,1.8) rectangle (3.2,2.2);
    \draw (4.8,1.8) rectangle (5.2,2.2);
    \draw[->] (3,2) -- (3,1);
    \draw[->] (4,3) -- (4,4);
    \draw[->] (5,3) -- (5,2);
    \draw[->, densely dashed,blue] (4,3) -- (3,1);
    \draw[->, densely dashed,blue] (5,3) -- (4,4);
    \draw (7, 1) -- (7,5);
    \end{tikzpicture}
    \quad 
    \begin{tikzpicture}[scale=0.4]
    \draw [fill] (1,1) circle [radius=0.07];
    \draw [fill] (1,2) circle [radius=0.07];
    \draw [fill] (1,3) circle [radius=0.07];
    \draw [fill] (1,4) circle [radius=0.07];
    \draw [fill] (1,5) circle [radius=0.07];
    \draw [fill] (2,1) circle [radius=0.07];
    \draw [fill] (2,2) circle [radius=0.07];
    \draw [fill] (2,3) circle [radius=0.07];
    \draw [fill] (2,4) circle [radius=0.07];
    \draw [fill] (2,5) circle [radius=0.07];
    \draw [fill] (3,1) circle [radius=0.07];
    \draw [fill] (3,2) circle [radius=0.07];
    \draw [fill] (3,3) circle [radius=0.07];
    \draw [fill] (3,4) circle [radius=0.07];
    \draw [fill] (3,5) circle [radius=0.07];
    \draw [fill] (4,1) circle [radius=0.07];
    \draw [fill] (4,2) circle [radius=0.07];
    \draw [fill] (4,3) circle [radius=0.07];
    \draw [fill] (4,4) circle [radius=0.07];
    \draw [fill] (4,5) circle [radius=0.07];
    \draw [fill] (5,1) circle [radius=0.07];
    \draw [fill] (5,2) circle [radius=0.07];
    \draw [fill] (5,3) circle [radius=0.07];
    \draw [fill] (5,4) circle [radius=0.07];
    \draw [fill] (5,5) circle [radius=0.07];
    \draw [fill] (6,1) circle [radius=0.07];
    \draw [fill] (6,2) circle [radius=0.07];
    \draw [fill] (6,3) circle [radius=0.07];
    \draw [fill] (6,4) circle [radius=0.07];
    \draw [fill] (6,5) circle [radius=0.07];
    \draw [fill] (7,1) circle [radius=0.07];
    \draw [fill] (7,2) circle [radius=0.07];
    \draw [fill] (7,3) circle [radius=0.07];
    \draw [fill] (7,4) circle [radius=0.07];
    \draw [fill] (7,5) circle [radius=0.07];
    \draw [fill] (8,1) circle [radius=0.07];
    \draw [fill] (8,2) circle [radius=0.07];
    \draw [fill] (8,3) circle [radius=0.07];
    \draw [fill] (8,4) circle [radius=0.07];
    \draw [fill] (8,5) circle [radius=0.07];
    \draw (0.8,2.8) rectangle (1.2,3.2);
    \draw (0.8,3.8) rectangle (1.2,4.2);
    \draw (0.8,4.8) rectangle (1.2,5.2);
    \draw (2.8,1.8) rectangle (3.2,2.2);
    \draw (4.8,1.8) rectangle (5.2,2.2);
    \draw (6.8,1.8) rectangle (7.2,2.2);
    \draw[->] (3,2) -- (3,1);
    \draw[->] (4,3) -- (4,4);
    \draw[->] (5,3) -- (5,1);
    \draw[->] (6,3) -- (6,4);
    \draw[->] (7,3) -- (7,2);
    \draw[->, densely dashed,blue] (4,3) -- (3,1);
    \draw[->, densely dashed,blue] (6,3) -- (5,1);
    \draw[->, densely dashed,blue] (7,3) -- (6,4);
    \draw[->, densely dashed,blue] (5,3) -- (4,4);
    \draw (9, 1) -- (9,5);
    \end{tikzpicture}
    \quad
    \begin{tikzpicture}[scale=0.4]
    \draw [fill] (1,1) circle [radius=0.07];
    \draw [fill] (1,2) circle [radius=0.07];
    \draw [fill] (1,3) circle [radius=0.07];
    \draw [fill] (1,4) circle [radius=0.07];
    \draw [fill] (1,5) circle [radius=0.07];
    \draw [fill] (2,1) circle [radius=0.07];
    \draw [fill] (2,2) circle [radius=0.07];
    \draw [fill] (2,3) circle [radius=0.07];
    \draw [fill] (2,4) circle [radius=0.07];
    \draw [fill] (2,5) circle [radius=0.07];
    \draw [fill] (3,1) circle [radius=0.07];
    \draw [fill] (3,2) circle [radius=0.07];
    \draw [fill] (3,3) circle [radius=0.07];
    \draw [fill] (3,4) circle [radius=0.07];
    \draw [fill] (3,5) circle [radius=0.07];
    \draw [fill] (4,1) circle [radius=0.07];
    \draw [fill] (4,2) circle [radius=0.07];
    \draw [fill] (4,3) circle [radius=0.07];
    \draw [fill] (4,4) circle [radius=0.07];
    \draw [fill] (4,5) circle [radius=0.07];
    \draw [fill] (5,1) circle [radius=0.07];
    \draw [fill] (5,2) circle [radius=0.07];
    \draw [fill] (5,3) circle [radius=0.07];
    \draw [fill] (5,4) circle [radius=0.07];
    \draw [fill] (5,5) circle [radius=0.07];
    \draw [fill] (6,1) circle [radius=0.07];
    \draw [fill] (6,2) circle [radius=0.07];
    \draw [fill] (6,3) circle [radius=0.07];
    \draw [fill] (6,4) circle [radius=0.07];
    \draw [fill] (6,5) circle [radius=0.07];
    \draw [fill] (7,1) circle [radius=0.07];
    \draw [fill] (7,2) circle [radius=0.07];
    \draw [fill] (7,3) circle [radius=0.07];
    \draw [fill] (7,4) circle [radius=0.07];
    \draw [fill] (7,5) circle [radius=0.07];
    \draw [fill] (8,1) circle [radius=0.07];
    \draw [fill] (8,2) circle [radius=0.07];
    \draw [fill] (8,3) circle [radius=0.07];
    \draw [fill] (8,4) circle [radius=0.07];
    \draw [fill] (8,5) circle [radius=0.07];
    \draw [fill] (9,1) circle [radius=0.07];
    \draw [fill] (9,2) circle [radius=0.07];
    \draw [fill] (9,3) circle [radius=0.07];
    \draw [fill] (9,4) circle [radius=0.07];
    \draw [fill] (9,5) circle [radius=0.07];
    \draw [fill] (10,1) circle [radius=0.07];
    \draw [fill] (10,2) circle [radius=0.07];
    \draw [fill] (10,3) circle [radius=0.07];
    \draw [fill] (10,4) circle [radius=0.07];
    \draw [fill] (10,5) circle [radius=0.07];
    \draw (0.8,2.8) rectangle (1.2,3.2);
    \draw (0.8,3.8) rectangle (1.2,4.2);
    \draw (0.8,4.8) rectangle (1.2,5.2);
    \draw (2.8,1.8) rectangle (3.2,2.2);
    \draw (4.8,1.8) rectangle (5.2,2.2);
    \draw (6.8,1.8) rectangle (7.2,2.2);
    \draw (8.8,1.8) rectangle (9.2,2.2);
    \draw[->] (3,2) -- (3,1);
    \draw[->] (4,3) -- (4,4);
    \draw[->] (5,3) -- (5,1);
    \draw[->] (6,3) -- (6,4);
    \draw[->] (7,3) -- (7,1);
    \draw[->] (8,3) -- (8,4);
    \draw[->] (9,3) -- (9,2);
    \draw[->, densely dashed,blue] (4,3) -- (3,1);
    \draw[->, densely dashed,blue] (6,3) -- (5,1);
    \draw[->, densely dashed,blue] (7,3) -- (6,4);
    \draw[->, densely dashed,blue] (5,3) -- (4,4);
    \draw[->, densely dashed,blue] (8,3) -- (7,1);
    \draw[->, densely dashed,blue] (9,3) -- (8,4);
    \end{tikzpicture} 
    \quad \raisebox{0.7cm}{$\cdots$}
    \\
    &
     \begin{tikzpicture}[scale=0.4]
    \draw [fill] (1,1) circle [radius=0.07];
    \draw [fill] (1,2) circle [radius=0.07];
    \draw [fill] (1,3) circle [radius=0.07];
    \draw [fill] (1,4) circle [radius=0.07];
    \draw [fill] (1,5) circle [radius=0.07];
    \draw [fill] (2,1) circle [radius=0.07];
    \draw [fill] (2,2) circle [radius=0.07];
    \draw [fill] (2,3) circle [radius=0.07];
    \draw [fill] (2,4) circle [radius=0.07];
    \draw [fill] (2,5) circle [radius=0.07];
    \draw [fill] (3,1) circle [radius=0.07];
    \draw [fill] (3,2) circle [radius=0.07];
    \draw [fill] (3,3) circle [radius=0.07];
    \draw [fill] (3,4) circle [radius=0.07];
    \draw [fill] (3,5) circle [radius=0.07];
    \draw [fill] (4,1) circle [radius=0.07];
    \draw [fill] (4,2) circle [radius=0.07];
    \draw [fill] (4,3) circle [radius=0.07];
    \draw [fill] (4,4) circle [radius=0.07];
    \draw [fill] (4,5) circle [radius=0.07];
    \draw [fill] (5,1) circle [radius=0.07];
    \draw [fill] (5,2) circle [radius=0.07];
    \draw [fill] (5,3) circle [radius=0.07];
    \draw [fill] (5,4) circle [radius=0.07];
    \draw [fill] (5,5) circle [radius=0.07];
    \draw [fill] (6,1) circle [radius=0.07];
    \draw [fill] (6,2) circle [radius=0.07];
    \draw [fill] (6,3) circle [radius=0.07];
    \draw [fill] (6,4) circle [radius=0.07];
    \draw [fill] (6,5) circle [radius=0.07];
    \draw [fill] (7,1) circle [radius=0.07];
    \draw [fill] (7,2) circle [radius=0.07];
    \draw [fill] (7,3) circle [radius=0.07];
    \draw [fill] (7,4) circle [radius=0.07];
    \draw [fill] (7,5) circle [radius=0.07];
    \draw [fill] (8,1) circle [radius=0.07];
    \draw [fill] (8,2) circle [radius=0.07];
    \draw [fill] (8,3) circle [radius=0.07];
    \draw [fill] (8,4) circle [radius=0.07];
    \draw [fill] (8,5) circle [radius=0.07];
    \draw (0.8,2.8) rectangle (1.2,3.2);
    \draw (0.8,3.8) rectangle (1.2,4.2);
    \draw (0.8,4.8) rectangle (1.2,5.2);
    \draw (2.8,1.8) rectangle (3.2,2.2);
    \draw (4.8,1.8) rectangle (5.2,2.2);
    \draw (6.8,1.8) rectangle (7.2,2.2);
    \draw[->] (5,2) -- (5,1);
    \draw[->] (6,3) -- (6,4);
    \draw[->] (7,3) -- (7,2);
    \draw[->, densely dashed,blue] (6,3) -- (5,1);
    \draw[->, densely dashed,blue] (7,3) -- (6,4);
    \draw (9, 1) -- (9,5);
    \end{tikzpicture}
    \quad
    \begin{tikzpicture}[scale=0.4]
    \draw [fill] (1,1) circle [radius=0.07];
    \draw [fill] (1,2) circle [radius=0.07];
    \draw [fill] (1,3) circle [radius=0.07];
    \draw [fill] (1,4) circle [radius=0.07];
    \draw [fill] (1,5) circle [radius=0.07];
    \draw [fill] (2,1) circle [radius=0.07];
    \draw [fill] (2,2) circle [radius=0.07];
    \draw [fill] (2,3) circle [radius=0.07];
    \draw [fill] (2,4) circle [radius=0.07];
    \draw [fill] (2,5) circle [radius=0.07];
    \draw [fill] (3,1) circle [radius=0.07];
    \draw [fill] (3,2) circle [radius=0.07];
    \draw [fill] (3,3) circle [radius=0.07];
    \draw [fill] (3,4) circle [radius=0.07];
    \draw [fill] (3,5) circle [radius=0.07];
    \draw [fill] (4,1) circle [radius=0.07];
    \draw [fill] (4,2) circle [radius=0.07];
    \draw [fill] (4,3) circle [radius=0.07];
    \draw [fill] (4,4) circle [radius=0.07];
    \draw [fill] (4,5) circle [radius=0.07];
    \draw [fill] (5,1) circle [radius=0.07];
    \draw [fill] (5,2) circle [radius=0.07];
    \draw [fill] (5,3) circle [radius=0.07];
    \draw [fill] (5,4) circle [radius=0.07];
    \draw [fill] (5,5) circle [radius=0.07];
    \draw [fill] (6,1) circle [radius=0.07];
    \draw [fill] (6,2) circle [radius=0.07];
    \draw [fill] (6,3) circle [radius=0.07];
    \draw [fill] (6,4) circle [radius=0.07];
    \draw [fill] (6,5) circle [radius=0.07];
    \draw [fill] (7,1) circle [radius=0.07];
    \draw [fill] (7,2) circle [radius=0.07];
    \draw [fill] (7,3) circle [radius=0.07];
    \draw [fill] (7,4) circle [radius=0.07];
    \draw [fill] (7,5) circle [radius=0.07];
    \draw [fill] (8,1) circle [radius=0.07];
    \draw [fill] (8,2) circle [radius=0.07];
    \draw [fill] (8,3) circle [radius=0.07];
    \draw [fill] (8,4) circle [radius=0.07];
    \draw [fill] (8,5) circle [radius=0.07];
    \draw [fill] (9,1) circle [radius=0.07];
    \draw [fill] (9,2) circle [radius=0.07];
    \draw [fill] (9,3) circle [radius=0.07];
    \draw [fill] (9,4) circle [radius=0.07];
    \draw [fill] (9,5) circle [radius=0.07];
    \draw [fill] (10,1) circle [radius=0.07];
    \draw [fill] (10,2) circle [radius=0.07];
    \draw [fill] (10,3) circle [radius=0.07];
    \draw [fill] (10,4) circle [radius=0.07];
    \draw [fill] (10,5) circle [radius=0.07];
    \draw (0.8,2.8) rectangle (1.2,3.2);
    \draw (0.8,3.8) rectangle (1.2,4.2);
    \draw (0.8,4.8) rectangle (1.2,5.2);
    \draw (2.8,1.8) rectangle (3.2,2.2);
    \draw (4.8,1.8) rectangle (5.2,2.2);
    \draw (6.8,1.8) rectangle (7.2,2.2);
    \draw (8.8,1.8) rectangle (9.2,2.2);
    \draw[->] (5,2) -- (5,1);
    \draw[->] (6,3) -- (6,4);
    \draw[->] (7,3) -- (7,1);
    \draw[->] (8,3) -- (8,4);
    \draw[->] (9,3) -- (9,2);
    \draw[->, densely dashed,blue] (6,3) -- (5,1);
    \draw[->, densely dashed,blue] (7,3) -- (6,4);
    \draw[->, densely dashed,blue] (8,3) -- (7,1);
    \draw[->, densely dashed,blue] (9,3) -- (8,4);
    \end{tikzpicture}
    \quad \raisebox{0.7cm}{$\cdots$}
    \\ & \hspace{2cm} \cdots \hspace{2cm} \cdots 
  \end{align*}
  \caption{Some of the zigzags in Figure~\ref{fig:right_bimodule} that contribute to the actions in ${}_{\BNAlgH}Z_{\RxH}=\Homology( {}_{\BNAlgH}\BNAlgH_{\BNAlgH} \bt {}^{\BNAlgH}\dual{\DDinf}_{\RxH} )$ }\label{fig:zigzags}
  \end{figure} 

  Combining these zigzags with the one-arrow zigzags (not containing any dashed arrows), we see that the resulting actions in \( {}_{\BNAlgH}Z_{\RxH}  \) coincide with the ones in \( {}_{\BNAlgH}{\bimoduleG}^\RxH \bt {}_{\RxH}\RxH_{\RxH} \) in Figure~\ref{fig:left_bimodule}.
\end{proof}

\begin{example} 
Let us write \( \BNr(\Knot_{\infty})^\RxH \coloneqq \DD(T)^{\BNAlgH} \bt {}_{\BNAlgH}\bimoduleG^\RxH \). For the knot in Figure~\ref{fig:Kinf_intro} we compute: 
\[ 
\BNr(\Lk(T_{-\infty},T_{2,-3}))^\RxH \simeq
\left[
\begin{tikzcd}[column sep=20pt,row sep=10pt,ampersand replacement=\&]  
    \GGdzh{\DotC}{-1}{-12}{-5}
      \arrow{r}{D}
      \arrow{d}{S}
      \&
      \GGdzh{\DotC}{-1}{-10}{-4}
      \arrow{r}{S}
      \&
      \GGdzh{\DotB}{-\frac{3}{2}}{-9}{-3}
      \arrow{r}{D}
      \&
      \GGdzh{\DotB}{-\frac{3}{2}}{-7}{-2}
      \\
      \GGdzh{\DotB}{-\frac{3}{2}}{-11}{-4}
      \arrow{r}{D}
      \&
      \GGdzh{\DotB}{-\frac{3}{2}}{-9}{-3}
      \arrow{r}{SS}
      \&
      \GGdzh{\DotB}{-\frac{3}{2}}{-7}{-2}
      \arrow{r}{D}
      \&
      \GGdzh{\DotB}{-\frac{3}{2}}{-5}{-1}
      \arrow{r}{S}
      \&
      \GGdzh{\DotC}{-2}{-4}{0}
\end{tikzcd}
\right] \bt {}_{\BNAlgH}\bimoduleG^\RxH 
\]
This is equal to the type~D structure with three generators on the right of Figure~\ref{fig:Geom_interpretation}.
Recall that \(\gr(\xben)=h^{-2}\delta^{0}\) and \(\gr(H)=h^{0}\delta^{-1}\). 
From this, we see that \( \Homology \big(\BNr(\Lk(T_{-\infty},T_{2,-3}))^\RxH \bt \RxH \big) \) is equal to the limit of the vector space in Figure~\ref{fig:exa:Pairing:TorusKnots:ArcArc}, confirming Theorem~\ref{thm:lifting_to_D_str}.

Observe that \( \BNr(\Lk(T_{-\infty},T_{2,-3}))^\RxH \) (and in general any \( \BNr(\Knot_\infty)^\RxH \)) can be read off from the curve \( \BNr(T_{2,-3}) \): considering Figure~\ref{fig:Geom_interpretation}, the intersections of the curve with the arc \( c \) are the three generators, and holomorphic discs between them give the differential. The yellow disc corresponds to the differential picking up \( H \). The number of times this disc goes through the \( H \)-chord at the top left corner indicates the exponent of \( H \) in the differential. A similar disc on the other side of \( c \) picks up \( \xben^2 \) because it goes through the \( \xben \)-chord twice.
\end{example}
\begin{figure}
\centering
    \begin{subfigure}{0.45\textwidth}
      \centering
      \includegraphics[scale=0.7]{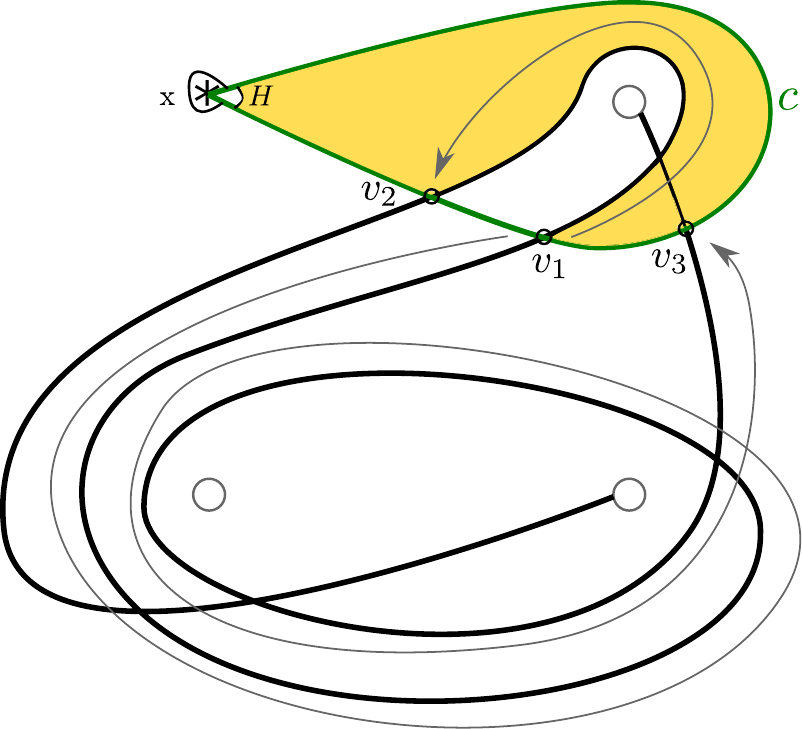}
    \end{subfigure}
    \begin{subfigure}{0.45\textwidth}
    \centering
    \begin{tikzpicture}[scale=.65]
    \draw[lightgray,step=1] (-1.5,1.5) grid (5.5,4.5);
    \draw[line width=1pt,->] (-2,4) -- (6.5,4) node[right] {\( h \)};
    \draw[line width=1pt,->] (5,1) -- (5,5.5) node[above] {\( \delta \)};
     \draw (-0.5,4) node[above] {\scriptsize \( -5 \)};
    \draw (0.5,4) node[above] {\scriptsize \( -4 \)};
    \draw (1.5,4) node[above] {\scriptsize \( -3 \)};
    \draw (2.5,4) node[above] {\scriptsize \( -2 \)};
    \draw (3.5,4) node[above] {\scriptsize \( -1 \)};
    \draw (4.5,4) node[above] {\scriptsize \( 0 \)};
    \draw (5,2.5) node[right] {\scriptsize \( -2 \)};
    \draw (5,3.5) node[right] {\scriptsize \( -1 \)};
    \draw(4.5, 2.5) node {\( v_3 \)};
    \draw(-0.5, 3.5) node {\( v_1 \)};
    \draw(0.5, 3.5) node {\( v_2 \)};
    \draw[->] (-0.25, 3.5) -- (0.25, 3.5) node[midway,below] {\( H \)};
    \draw[->] (-0.5, 3.25) .. controls (-0.5, 2.5) .. (4.15, 2.5) node[midway,below] {\( \xben^2 \)};
    \end{tikzpicture}
    \end{subfigure}
\caption{Geometric interpretation of \( \BNr(\Lk(T_{-\infty},T_{2,-3}))^\RxH = \DD(T_{2,-3})^{\BNAlgH} \bt {}_{\BNAlgH}\bimoduleG^\RxH \)}\label{fig:Geom_interpretation}
\end{figure} 


\newcommand*{\arxivPreprint}[1]{ArXiv preprint \href{http://arxiv.org/abs/#1}{#1}}
\newcommand*{\arxiv}[1]{(ArXiv:\ \href{http://arxiv.org/abs/#1}{#1})}
\bibliographystyle{alpha}
\bibliography{KhCurves}
\end{document}